\documentclass[12pt, a4paper,twoside,openright]{book}

\pdfoutput=1
\usepackage[a4paper, left=2.5cm, right=2.5cm, top=2.5cm, bottom=2.5cm]{geometry}

\usepackage[utf8]{inputenc}
\usepackage[T1]{fontenc}
\usepackage{lmodern}
\usepackage{graphicx}
\usepackage{array} 
\usepackage{paralist} 
\usepackage{amsmath}
\usepackage{amsthm}
\usepackage{amssymb}
\usepackage{color}
\usepackage{xcolor}
\usepackage{mathabx}
\usepackage{hyperref}
\usepackage{longtable}
\usepackage{booktabs}
\usepackage{stackrel}
\usepackage{calc}
\usepackage{numprint}
\usepackage{setspace}

\theoremstyle{plain}
\newtheorem{thm}{Theorem}[chapter]
\newtheorem{prop}[thm]{Proposition}
\newtheorem{lem}[thm]{Lemma}
\newtheorem{coro}[thm]{Corollary}
\newtheorem{cor}[thm]{Corollary}

\theoremstyle{definition}
\newtheorem{defn}{Definition}[chapter]
\newtheorem{rmk}{Remark}[chapter]
\newtheorem{rem}[rmk]{Remark}
\newtheorem{exa}[thm]{Example}

\newcommand{\figpath}{fig/}
\newcommand{\insertfigure}[3][1]{\includegraphics[page=#3,scale=#1]{\figpath #2}}
\newcommand{\mydef}[1]{\underline{\smash{\emph{#1}}}}
\newcommand{\tdef}[1]{\underline{\smash{\emph{#1}}}}
\newcommand{\eqdef}{\stackrel{\textup{\tiny def}}{=}}
\newcommand{\horizrule}{\begin{center}\rule{4cm}{0.4pt}\end{center}}

\newcommand{\Sg}{$\mathbb{S}_g$}
\newcommand{\id}{\mathrm{id}}
\newcommand{\vecvar}[1]{\underline{\smash{\mathbf{#1}}}}
\newcommand{\varvec}[1]{\underline{\smash{\mathbf{#1}}}}
\newcommand{\ord}{\operatorname{ord}}

\newcommand{\sgn}{\operatorname{sgn}}

\newcommand{\integers}{\mathbb{Z}}
\newcommand{\naturals}{\mathbb{N}}
\newcommand{\rationals}{\mathbb{Q}}
\newcommand{\greeks}{\mathcal{G}}
\newcommand{\rL}{\mathbb{L}}

\newcommand{\tam}[1]{$\textsc{Tam}({#1})$}

\newcommand{\fdeg}{\operatorname{fdeg}}

\newcommand{\torus}[1]{\mathbb{S}_{#1}}
\newcommand{\core}{\operatorname{core}}
\newcommand{\kernel}{\operatorname{ker}}

\newcommand{\fw}{\operatorname{fw}}
\newcommand{\ew}{\operatorname{ew}}

\begin{document}
\thispagestyle{empty}

\begin{center}
\Large{\textbf{UNIVERSITE PARIS DIDEROT (PARIS 7) \\ SORBONNE PARIS CITE}}
\end{center}

\begin{center}
\vspace{\stretch{1}}
{\Large \textbf{École doctorale de sciences mathématiques}}
\medbreak
{\Large \textbf{de Paris centre}}
\vspace{\stretch{2}}

{\Huge \textsc{Thèse de doctorat}}

\vspace{\stretch{1}}

{\LARGE Discipline : Informatique}

\vspace{\stretch{2}}

{\large présentée par}
\vspace{\stretch{1}}

{\LARGE \textsc{Wenjie} \textsc{Fang}}

\vspace{\stretch{2}}
\hrule
\vspace{\stretch{1}}
\begin{spacing}{2}
{\LARGE \textbf{Aspects énumératifs et bijectifs des cartes combinatoires : généralisation, unification et application}}
\end{spacing}

\begin{spacing}{2}
{\Large \textbf{Enumerative and bijective aspects of combinatorial maps: generalization, unification and application}}
\end{spacing}

\vspace{0cm}
\hrule
\vspace{\stretch{2}}

{\Large dirigée par Guillaume \textsc{Chapuy} et Mireille \textsc{Bousquet-Mélou}}

\vspace{\stretch{4}}

{\Large Soutenue le 11 octobre 2016 devant le jury composé de :}

\vspace{\stretch{2}}
{
\begin{tabular}{lll}

M\textsuperscript{me} Frédérique \textsc{Bassino} & Université Paris-Nord & examinatrice \\
M\textsuperscript{me} Mireille \textsc{Bousquet-Mélou} & CNRS, Université de Bordeaux & directrice \\
M. Guillaume \textsc{Chapuy} & CNRS, Université Paris Diderot & directeur \\
M. Valentin \textsc{Féray} & Universität Zürich & examinateur \\
M. Emmanuel \textsc{Guitter} & CEA, IPhT & examinateur \\
M. Christian \textsc{Krattenthaler} & Universität Wien & rapporteur \\
M. Bruno \textsc{Salvy} & INRIA, ENS Lyon & examinateur \\
M. Gilles \textsc{Schaeffer} & CNRS, École polytechnique & rapporteur \\
\end{tabular}
}

\end{center}

\newpage

\vspace*{\fill}

\noindent\begin{center}
\begin{minipage}[t]{(\textwidth-2cm)/2}
Institut de Recherche en Informatique Fondamentale \\
CNRS UMR 8243 \\
Université Paris Diderot - Paris 7 \\
Case 7014 \\
8 place Aurélie Nemours \\
\numprint{75205} Paris Cedex 13
\end{minipage}
\vspace{1.5cm} \\
\begin{minipage}[t]{(\textwidth-2cm)/2}
Laboratoire Bordelais de Recherche en Informatique \\
CNRS UMR 5800 \\
Université de Bordeaux \\
351 cours de la Libération\\
\numprint{33405} Talence
\end{minipage}
\vspace{1.5cm} \\
\begin{minipage}[t]{(\textwidth-2cm)/2}
Université Paris Diderot - Paris 7. \\
École doctorale de sciences mathématiques de Paris centre (ED 386)\\
Case courrier 7012 \\
8 place Aurélie Nemours \\
\numprint{75205} Paris Cedex 13
\end{minipage}
\end{center}

\newpage

\chapter*{Abstract}

This thesis deals with the enumerative study of combinatorial maps, and its application to the enumeration of other combinatorial objects.

Combinatorial maps, or simply maps, form a rich combinatorial model. They have an intuitive and geometric definition, but are also related to some deep algebraic structures. For instance, a special type of maps called \emph{constellations} provides a unifying framework for some enumeration problems concerning factorizations in the symmetric group. Standing on a position where many domains meet, maps can be studied using a large variety of methods, and their enumeration can also help us count other combinatorial objects. This thesis is a sampling from the rich results and connections in the enumeration of maps.

This thesis is structured into four major parts. The first part, including Chapter~1 and 2, consist of an introduction to the enumerative study of maps. The second part, Chapter~3 and 4, contains my work in the enumeration of constellations, which are a special type of maps that can serve as a unifying model of some factorizations of the identity in the symmetric group. The third part, composed by Chapter~5 and 6, shows my research on the enumerative link from maps to other combinatorial objects, such as generalizations of the Tamari lattice and random graphs embeddable onto surfaces. The last part is the closing chapter, in which the thesis concludes with some perspectives and future directions in the enumerative study of maps. 

We now give a more precise description of the content in each chapter. Chapter~1 is a brief review of different directions in map enumeration and their relations with other domains in mathematics. It also includes an overview of tools we can use in map enumeration. Chapter~2 is a technical preliminary that explains, in details and with examples, the tools we will be using in the later chapters.

In Chapter~3, we will see a simple enumerative relation about constellations, which generalizes the \emph{quadrangulation relation} between bipartite maps and general maps first proved in \cite{JV1990a}. It is also a relatively rare occasion to see how we can use characters in the symmetric group to obtain enumerative results on maps. 

In Chapter~4, we will consider the enumeration of constellations in all genera by writing and solving Tutte equations. The planar case was already solved in \cite{BMS} via bijective method, but had resisted being solved using a functional equation approach. We will revisit the planar case by solving a Tutte equation with a method applied in \cite{BMCPR2013representation} to the enumeration of intervals in the $m$-Tamari lattice. It is also the first appearance of the mysterious link between planar maps and intervals in the Tamari lattice and its generalizations, which will be the subject of the next chapter. We will also solve our functional equation for higher genus $g>0$ in the bipartite case. In the solution, we adapt some ideas from the \emph{topological recursion} (\textit{cf.} \cite{Eynard:book}), a highly convoluted yet powerful technique to solve problems including map enumeration in higher genus. 

In Chapter~5, we will look further at the link between planar maps and intervals in the Tamari lattice and its generalizations. More precisely, we will establish a bijection between intervals in generalized Tamari lattices introduced in \cite{PRV2014extension} and non-separable planar maps. As an application of our bijection, we give an enumeration formula for intervals in generalized Tamari lattices, which is the same as that of non-separable planar maps, obtained in \cite{Tutte:census} by Tutte. We will also discuss other implications of our bijection.

In Chapter~6, we will study the enumeration of cubic graphs embeddable into a surface with given genus, using existing enumeration results of several types of triangulations. More precisely, we will be interested in weighted cubic graphs, where loops, single edges and double edges will receive different weights. Our proof adapts the same basic ideas as in \cite{graphfive, bender2011asymptotic}. With our approach, we are able to give asymptotic enumeration results for several classes of weighted cubic graphs. This enumeration is motivated by the study of phase transitions of random graphs embeddable onto surfaces with higher genus, similar to those in \cite{planar-phase} for planar random graphs.

In Chapter~7, we will conclude by some perspectives and discussions about possible future research directions in the enumerative study of maps. We start by an overview of several aspects of map enumeration that are not treated in this thesis, then we will look at some possible extensions of results presented in previous chapters. Finally, we will consider a future research direction in map enumeration. 

\chapter*{Résumé}

Le sujet de cette thèse est l'étude énumérative des cartes combinatoires et ses applications à l'énumération d'autres objets combinatoires.

Les cartes combinatoires, aussi appelées simplement «~cartes~», sont un modèle combinatoire riche. Elles sont définies d'une manière intuitive et géométrique, mais elles sont aussi liées à des structures algébriques plus complexes. Par exemple, l'étude d'une famille de cartes appelées des «~constellations~» donne un cadre unifié à plusieurs problèmes d'énumération de factorisations dans le groupe symétrique. À la rencontre de différents domaines, les cartes peuvent être analysées par une grande variété de méthodes, et leur énumération peut aussi nous aider à compter d'autres objets combinatoires. Cette thèse présente un ensemble de résultats et de connexions très riches dans le domaine de l'énumération des cartes.

Cette thèse se divise en quatre grandes parties. La première partie, qui correspond aux chapitres 1 et 2, est une introduction à l'étude énumérative des cartes. La deuxième partie, qui correspond aux chapitres 3 et 4, contient mes travaux sur l'énumération des constellations, qui sont des cartes particulières présentant un modèle unifié de certains types de factorisation de l'identité dans le groupe symétrique. La troisième partie, qui correspond aux chapitres 5 et 6, présente ma recherche sur le lien énumératif entre les cartes et d'autres objets combinatoires, par exemple les généralisations du treillis de Tamari et les graphes aléatoires qui peuvent être plongés dans une surface donnée. La dernière partie correspond au chapitre 7, dans lequel je conclus cette thèse avec des perspectives et des directions de recherche dans l'étude énumérative des cartes.

Voici maintenant une description plus précise de chaque chapitre. Le chapitre 1 est un résumé de différentes directions prises dans l'énumération des cartes et leurs relations avec d'autres domaines des mathématiques. Il contient également une liste d'outils utilisés dans l'énumération des cartes. Le chapitre 2 est un préliminaire technique aux chapitres suivants ; il présente de manière détaillée les outils utilisés dans ceux-ci, avec des exemples.

Dans le chapitre 3, nous voyons une relation énumérative simple concernant les triangulations. Cette relation généralise la \emph{relation des quadrangulations} entre les cartes biparties et les cartes générales, démontrée dans \cite{JV1990a}. C'est aussi une occasion relativement rare de voir l'utilisation des caractères du groupe symétrique dans l'énumération des cartes.

Dans le chapitre 4, nous considérons l'énumération des constellations en genre arbitraire en écrivant et en résolvant des équations de Tutte. Le cas planaire est résolu dans \cite{BMS} avec la méthode bijective, mais pas encore avec la méthode symbolique. On revient au cas planaire en résolvant une équation de Tutte avec la méthode inventée dans \cite{BMCPR2013representation} pour l'énumération des intervalles dans le treillis de $m$-Tamari. C'est aussi la première apparence du lien entre les cartes planaires et les intervalles dans le treillis de Tamari et ses généralisations, qui est le sujet du chapitre suivant. Nous résoudrons aussi notre équation fonctionnelle en genre supérieur $g>0$ dans le cas biparti. Pour cette résolution, nous adaptons quelques idées de la \emph{récurrence topologique} (\textit{cf.} \cite{Eynard:book}), qui est une technique complexe mais puissante de résolution de divers problèmes, y compris l'énumération des cartes en genre supérieur.

Dans le chapitre 5, nous examinons le lien entre les cartes planaires et les intervalles dans le treillis de Tamari et ses généralisations. Plus précisément, nous établions une bijection entre les intervalles dans les treillis de Tamari généralisés introduit dans \cite{PRV2014extension} et les cartes planaires non-séparables. En appliquant notre bijection, nous donnons une formule d'énumération des intervalles dans les treillis de Tamari généralisés, qui est la même que celle des cartes planaires non-séparables, obtenue dans \cite{Tutte:census}. Nous discutons aussi des autres implications de notre bijection.

Dans le chapitre 6, nous étudions l'énumération des graphes cubiques qui peuvent être plongés dans une surface en genre fixé, en utilisant des résultats d'énumération existants sur plusieurs types de triangulations. Plus précisément, nous examinons les graphes cubiques pondérés, dans lesquels les boucles, les arêtes simples et les arêtes doubles reçoivent différents poids. Notre preuve est fondée sur les mêmes idées de base que celles de \cite{graphfive, bender2011asymptotic}. Avec notre approche, nous sommes capables de donner des résultats d'énumération asymptotique pour plusieurs classes de graphes cubiques pondérés. Cette énumération est motivée par l'étude des transitions de phase dans les cartes aléatoires qui peuvent être plongées dans une surface fixée en genre supérieur, qui sont similaires à celles données dans \cite{planar-phase} pour les graphes planaires aléatoires.

Dans le chapitre 7, nous concluons par quelques perspectives et discussions sur les directions possibles à prendre dans l'avenir de l'étude énumérative des cartes. Nous commençons par un résumé de certains aspects de l'énumération des cartes qui ne sont pas traités dans cette thèse, puis nous examinons quelques extensions possibles de résultats présentés dans les chapitres précédents. Nous terminons avec une direction de recherche possible pour l'énumération des cartes.

\chapter*{Acknowledgements}

I would like to express my gratitude here to everyone who has helped me in the preparation of this thesis, no matter directly or indirectly, and no matter scientifically or morally. I will try to avoid name enumeration here, but not due to lack of gratitude. Instead, since so many people are so important in my journey towards a thesis, a fair enumeration will quickly drown my sincere expressions.

Firstly, I would like to thank Guillaume and Mireille who guided me into real research in combinatorics. I first met Guillaume as the advisor of my Master internship. Given the nice stay, it was natural to further my study under his supervision, which proves to be an excellent choice. Later I met Mireille, and I found it a delight to work with her. Her detailed way of working is really impressive and educational for me. I feel lucky to be able to discuss and work with them, and I am so grateful to be able to benefit from their insights in research and their help. They also helped me a lot in the writing and the redaction of my thesis.

Secondly, I would like to thank my thesis committee. I am grateful to Christian Krattenthaler and Gilles Schaeffer for examining my thesis. It is a heavy task to go over more than 200 pages of anything. I have known both Christian and Gilles before, first by reading their impressive work, then by actually talking to them. Some of their work gave a lot of inspiration to my own research. Therefore, I am more than happy that they agreed to read my thesis and give their precious comments. I am also grateful to Fr\'ed\'erique Bassino, Valentin F\'eray, Emmanuel Guitter and Bruno Salvy for joining my thesis committee.

During my thesis, I have spent my research time at IRIF (formerly LIAFA) of Université Paris Diderot and LaBRI of Université de Bordeaux. I have much enjoyed the research environment of both places, which would be impossible without the people there, both researchers and administration staff. I would like to thank all members of the Combinatorics team at IRIF and of the Enumerative and Algebraic Combinatorics team at LaBRI. I would also like to thank my collaborators. I really enjoyed working with them, and they are all delightful people with interesting ideas and deep insights.

I am also grateful for the vibrant French research community in combinatorics. I have drawn a huge amount of inspiration from conversations with these people in various occasions, such as the annual meeting of ALEA, the Flajolet seminars and ``Journ\'ees Cartes''. They showed me a great variety of current research in combinatorics, which was indeed eye-opening and stimulating for someone beginning his own research. This also applies to everyone I met in international conferences and workshops.

My thanks also go to my friends and fellow PhD students, with whom I shared useful information and mutual moral support. Their accompany means a lot to me.

Lastly, I am in debt to my parents for their unconditional support. They may not understand anything in this thesis, but they sure know how to raise someone who does.

\tableofcontents
\listoffigures
\listoftables

\chapter*{Notation}
\markboth{NOTATION}{}

\renewcommand*{\arraystretch}{1.4}
\begin{longtable}[l]{p{2cm} p{15cm}}

$\mathbb{N}$ & Set of natural numbers \\
$\mathbb{N}_+$ & Set of strictly positive natural numbers \\
$\mathbb{Q}$ & Set of rational numbers \\
$\mathbb{K}[x_1, x_2, \ldots]$ & Polynomial ring with base ring $\mathbb{K}$ and indeterminates $x_1, x_2, \ldots$ \\
$\mathbb{K}(x)$ & Field of rational fractions with base field $\mathbb{K}$ and indeterminate $x$ \\
$\mathbb{K}[[t]]$ & Ring of formal power series with base ring $\mathbb{K}$ in the variable $t$ \\
$\mathbb{K}((t))$ & Field of Laurent series with base field $\mathbb{K}$ in the variable $t$ \\
$\mathbb{K}((t^*))$ & Field of Puiseux sereis with base field $\mathbb{K}$ in the variable $t$ \\
\Sg & Orientable surface of genus $g$ \\
$S_n$ & Symmetric group of size $n$ \\
$\sigma, \phi, \rho, \tau$ & Permutations \\
$\id_n$ & Neutral element in $S_n$ \\
$\lambda, \mu, \theta$ & Integer partitions \\
$\epsilon$ & The empty integer partition \\
$\mathbb{C}S_n$ & Group algebra of the symmetric group $S_n$ \\
$J_k$ & Jucys-Murphy element indexed by $k$ in the symmetric group \\
$\chi^\lambda_\mu$ & Irreducible character indexed by $\lambda$ evaluated at $\mu$ \\
$\mathcal{A},\mathcal{B},\mathcal{C},\ldots$ & Classes of combinatorial objects
\end{longtable}

\chapter{Introduction}
\chaptermark{Introduction}

In our daily life, a map is something that we use to get directions from one place to another. It can be seen as a graph, where vertices are roundabouts, crossings and points of interest, and edges are streets, roads and avenues. We use a map to see by which route we can get from one point to another. However, a map is more than all these links. At the corner of a crossing, even if we know the next street to follow, we need to figure out how streets neighbor one another around the crossing to determine whether to turn left or right. Therefore, it is also important to know the cyclic order  of links around a point. For direction, these are all we have to know.

When distilled into a mathematical object, the concept of a quotidian map becomes that of a \textbf{combinatorial map}, the study subject of this thesis. In this thesis, we will deal with the enumeration of combinatorial maps, that is, counting combinatorial maps with given properties. We are going to see how to solve various enumeration problems related to combinatorial maps through some of their interplays with other combinatorial objects. Due to the so many faces of maps, the landscape that I am going to paint for the study of maps can only be incomplete, or even ignorant. Nevertheless, I hope that we will all enjoy this journey to the magnificence of maps.

This chapter will be a brief qualitative survey of what we (or more precisely, I) know about combinatorial maps and their enumeration. Most precise definitions and technical details will be delayed to the next chapter.

\section{The many faces of maps}

As many interesting mathematical objects, combinatorial maps have several definitions, each reveals one of its intriguing aspects. One may guess that they have a geometric and intuitive definition that gives rise to many beautiful combinatorial bijections with other important mathematical objects, which then lead to elegant enumerative and probabilistic results. One may not guess that they also have an algebraic but powerful definition that makes them a highway interchange between various topics in enumerative and algebraic combinatorics and even physics. In this section, we will see two definitions of combinatorial maps and the related connections to other fields of mathematics and physics.

\subsection{Maps as graph embeddings}

We start by the basics. A \mydef{graph} $G$ is composed by a finite set $V$ of vertices and a finite multiset $E$ with elements from $\{ \{u,v\} \mid u,v \in V \}$ of edges, and we write $G=(V,E)$. Usually we see vertices as points and edges as line segments between two points. Here multiple edges are allowed since $E$ is a multiset, and edges that link one vertex to itself (also called \mydef{loops}) are also allowed. The graphs we consider here are called ``finite multigraphs'' in the jargon of graph theory. A graph is \mydef{connected} if for every pair of vertices $(u,v) \in V^2$, there is a sequence of edges $\{ u_0 = u, u_1\}, \{u_1, u_2\} ,\ldots, \{u_{k-1}, u_k = v \}$ that leads from $u$ to $v$.

Although a graph is inherently a discrete object, it can nevertheless be endowed with a topological structure by looking at vertices as distinct points and edges as distinct copies of the real interval $[0,1]$ with ends identified with points corresponding to their adjacent vertices. We can then view graphs as topological spaces (and even metric spaces), and it is now reasonable to talk about how they can be embedded into other topological spaces, especially surfaces.

In this thesis, by ``surfaces'' we mean surfaces that are connected, closed and oriented. A \mydef{homeomorphism} is a bijective and bi-continuous mapping. Since we are looking at topological spaces, everything is defined up to homeomorphism. There is a well-known classification of surfaces according to their genus (see \cite{Mohar-book}), which we describe briefly as follows. For any integer $g \geq 0$, we denote by \Sg{} the surface obtained by adding $g$ handles to the sphere. For example, $\mathbb{S}_1$ is the torus. Then a surface must be homeomorphic to a certain \Sg{}. In this case, we say that the \mydef{genus} of this surface is $g$. All surfaces we consider are thus classified by their genera.

We can now give our first definition of combinatorial maps as embeddings of graphs.

\begin{defn}[Combinatorial maps] \label{def:1:map}
A \mydef{combinatorial map} (or simply \mydef{map}) $M$ is an embedding (\textit{i.e.}, the image of an injective continuous mapping) of a connected graph $G$ onto a surface $\mathbb{S}$ such that all \mydef{faces}, \textit{i.e.} connected components of $\mathbb{S} \setminus M$, are topological disks. If there is an orientation-preserving homeomorphism between two maps, then we consider these two maps as identical.
\end{defn}

Maps naturally inherit all terminologies of graphs, and they also have their own terminologies. The \mydef{size} of a map $M$ is the number of edges of its underlying graph. The \mydef{genus} $g$ of a map $M$ is the genus of the surface $\mathbb{S}$ onto which it embeds. A map is \mydef{planar} if it is of genus $0$, \textit{i.e.} it is embedded on the sphere. The name ``planar'' comes from the fact that we can draw a planar map on the plane by taking a point in one of the faces as the point at infinity of the plane. For vertices and faces, their \mydef{degree} is the number of adjacent edges, counted with multiplicity. Figure~\ref{fig:1:map-ex} shows two examples of maps, the left one of genus $2$, and the right one is planar.

\begin{figure}
  \centering
  \insertfigure[0.6]{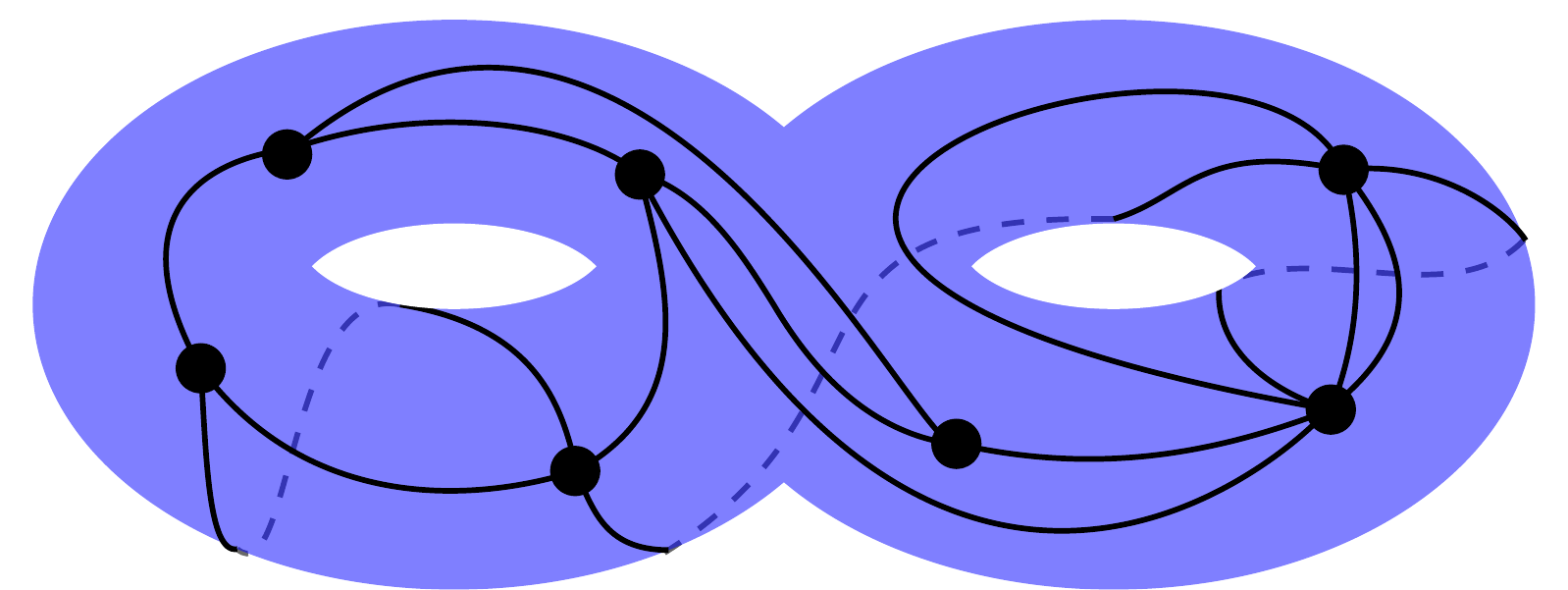}{1} \quad \insertfigure{planar-ex.pdf}{1}
  \caption{Examples of maps}
  \label{fig:1:map-ex}
\end{figure}

We have the following useful relation named ``Euler's relation'' after Euler.

\begin{lem}
  For a map of genus $g$ with $v$ vertices, $e$ edges and $f$ faces, we have
  \begin{equation} \label{eq:1:euler}
    v - e + f = 2 - 2g.
  \end{equation}
\end{lem}

This definition of maps as graph embeddings can be seen as an exact translation of our intuitive concept of ``everyday maps''. For a map $M$, besides its inherited graph structure, it also has a topological structure from being an embedding, which determines the cyclic order of edges around each vertex. Since we are only interested in the topological information around vertices and edges, we assume that nothing with particular topological interest happens elsewhere. Therefore, faces should be topologically as mundane as possible, and the most natural choice is a topological disk.

\begin{rmk}
  Maps can also be defined as embeddings on surfaces which are not necessarily orientable, such as the Klein bottle. These maps are called \emph{non-orientable maps} and they also have interesting connections with other algebraic objects. However, we only discuss the orientable case in this thesis. Readers interested by non-orientable maps are referred to \cite{non-orient-map, BC0, chapuy-dolega} for a taste of results in their enumeration.
\end{rmk}

Due to the extra topological data about how edges neighbor each other around a vertex, maps with the same graph structure can be different. Figure~\ref{fig:1:maps-diff} illustrates such a situation, where three maps sharing the same graph structure are different and even with different genus (the first two are planar, and the last is of genus $1$).

\begin{figure}
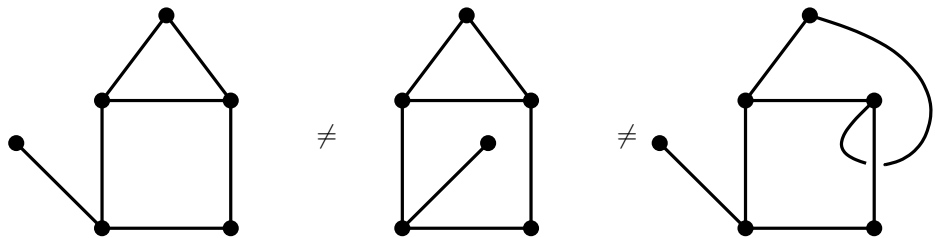

  \centering
  \insertfigure{ch1.pdf}{1}
  \caption{Three different maps with the same underlying graph}
  \label{fig:1:maps-diff}
\end{figure}

For the ease of enumeration, we often consider a variant of maps called \mydef{rooted maps}. There are two equivalent ways to root a map, each practical in different cases. The first way is to distinguish an edge $e = \{u, v\}$ and to give it an orientation, for example from $u$ to $v$. When the edge $e$ is a loop, we distinguish each end to allow two possible orientations. The distinguished edge is called the \mydef{root}, and $u$ is called the \mydef{root vertex}. The second way is to distinguish a corner in the map, where a \mydef{corner} is a pair of neighboring half-edges around a vertex. A map with $n$ edges has $2n$ rooting possibilities in either ways. To show that rooting on edges and on corners are equivalent, for a map rooted at a corner around a vertex $u$ between two adjacent edges $e_1$ and $e_2$ in clockwise order, we can re-root it at the edge $e_2$ with $u$ as root vertex, and this procedure gives a bijection between corner-rooted and edge-rooted maps. The left side of Figure~\ref{fig:1:rooting-dual} gives an example of the two ways of rooting. The automorphism group of a rooted map is trivial. To see this, we can perform a depth-first search on the rooted map starting from the root vertex along the root, and when a new vertex is discovered, its adjacent edges are visited in clockwise order. An element in the automorphism group of the rooted map should also be an automorphism of the spanning tree obtained in our search, conserving for all vertices the order of adjacent edges. The only automorphism that satisfies this condition is the identity, which forms a trivial group. This fact makes enumeration easier since we don't need to factor in extra symmetries that can occur for some unrooted maps. Furthermore, it is under this convention that maps are the most interesting in the enumerative eyes. From now on, we will mainly consider rooted maps. In rare occasions where unrooted maps are considered, such as in Chapter~4, they will be weighted by the inverse of the size of their automorphism group as a compensation. As a convention, when we draw a map on the plane, the face that contains the root corner (also called \mydef{outer face}) will always contain the point at infinity. Alternatively, we can also say that the root is always adjacent to the outer face and oriented in the clockwise direction along the outer face.

\begin{figure}
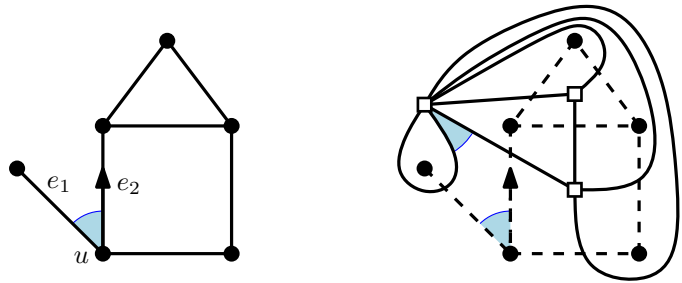

  \centering
  \insertfigure{ch1.pdf}{2}
  \caption{Rooting and dual of a map}
  \label{fig:1:rooting-dual}
\end{figure}

We now present an important involution on maps. For a map $M$ embedded on a surface \Sg{}, we construct its \mydef{dual} denoted by $M^*$ as follows: firstly, inside each face $f$ of $M$ we place a vertex $f^*$ (called the \mydef{dual vertex} of $f$); secondly, for each edge $e$ of $M$ that borders two (not necessarily distinct) faces $f_1, f_2$, we place an edge $e^*$ (called the \mydef{dual edge} of $e$) between the dual vertices of $f_1$ and $f_2$. For the root of the dual map $M^*$, we take the convention that $M^*$ is rooted at the dual corner composed by the dual edges of the two edges of the marked corner in $M$. Figure~\ref{fig:1:rooting-dual} gives an example of a planar map and its dual. It is clear that a map and its dual have the same genus.

Knowing the basic concepts about maps, we are interested in how maps interacts with other domains of mathematics.

Since maps are essentially embeddings of graphs, they are also related to \emph{topological graph theory}, which is a branch of graph theory that takes an interest in whether and how a graph can be embedded onto a surface. Using results from topological graph theory, it is possible to use enumerative results on maps to count the number of graphs that can be embedded onto a surface of a certain genus $g$. This was first done in \cite{planar-law} for planar graphs, then in \cite{bender2011asymptotic,graphfive} and \cite{fang-graz} for graphs embeddable on surfaces of higher genera. Furthermore, topological graph theory is also a source of intriguing questions on maps, such as the chromatic number of a typical map of given genus (proposed in \cite{graphfive} for graphs) or the typical face-width of a map of given genus as a function of the size (\textit{cf.} \cite{log-facewidth} and \cite{guitter2010distance}), \textit{etc}.

If we move our eyes from maps to the surface onto which they embed, we can see maps as ``discretizations'' of surfaces, each map representing one way to discretize its embedding surface. It is then natural to ask the following question: what does the ``typical discretization'' of a given surface look like when it becomes more and more fine-grained? To answer this question, we must get into the realm of probability to define the notion of a \emph{random map}, which is a probability measure on a given class of maps. An example of a well-studied random map model is the uniform planar quadrangulation. We can then ask interesting questions on these random maps, such as the asymptotic behavior of their radius and how they look like as a metric space. Asymptotic enumeration of maps plays a crucial role in answering these questions. Up to now, most random map models that have been studied are planar. It has been discovered that, for many planar random map models, when the size $n$ tends to infinity, there is an appropriate scaling (usually of order $n^{1/4}$) such that the \emph{scaling limit} of the random map is a well-defined continuous object \cite{brownian-map-def}. It turns out that this scaling limit is the same for many kinds of random planar maps. It seems that the scaling limit is \emph{universal}, because it does not depend on the precise construction details of random planar maps \cite{universality-3-4}, but rather on the fact that they are random discretizations of the sphere. This scaling limit is called the \emph{Brownian map}, and serves as a model of random surface. The Brownian map provides a vision of the global structure of very large random planar maps, and its study has attracted many researchers in probability. There are also some studies on the scaling limit of large maps of higher genus \cite{guitter2010distance, dist-higher-genus, limit-higher-genus, Chapuy:PTRF}.

We can also ask another question about random planar maps: how does the neighborhood of the root of a random map look like when the size of the random map tends to infinity? To answer this question, we need another notion of limit object called the \emph{weak local limit}. Such limit objects are studied for several random planar map models, such as uniform infinite planar triangulation (UIPT, see \cite{uipt}) and uniform infinite planar quadrangulation (UIPQ, see \cite{uipq}). These limit objects, while having an interesting fractal structure on their own \cite{uipt-more}, also serve as a random lattice, on which we can study various stochastic processes such as percolation \cite{uipt-more} and random walks \cite{uipq-walk}.

Some theoretical physicists are also interested in random maps as a model of discrete geometry. This interest is related to the quest for a unified theory of fundamental physics, which needs a reconciliation between two successful theories in conflict, namely general relativity and quantum mechanics. Thus comes quantum gravity, a branch of theoretical physics that attempts to bridge these two theories by quantization of gravity and space-time. Brownian maps can serve as a model of a ``quantized'' 2-dimensional space, on which a quantum gravity theory can be built (see \cite{quantum-geometry}). The use of maps in 2-dimensional quantum gravity has also been extended to higher dimensions via a generalization of random maps called ``random tensor model'' (see \cite{random-tensor, random-tensor-adv}).


Other than random maps, maps are also studied in physics for another reason. In the study of particle physics, we are led to the computation of \emph{matrix integrals}, which are integrals over certain kinds of random matrices. These integrals are found to be expressible as an infinite sum of the weight of all maps of a certain type that depends on the integrand. Readers are referred to \cite{LandoZvonkine} and \cite{Eynard2011integral} for introductions. We can regard the enumeration of planar maps as a way to compute matrix integrals, while techniques once developed for matrix integrals can be extended and adapted to the enumeration of maps in general. One notable example is the \emph{topological recursion} technique invented by Eynard and Orantin in \cite{EO}, which was abstracted from techniques for computing expansions of matrix integrals, and then successfully applied to enumerations of various families of maps (\textit{e.g.}, \cite{KZ} on Grothendieck's \textit{dessin d'enfant}, also see Eynard's book \cite{Eynard:book}).

\subsection{Polygon gluing and rotation system}

We now present another way to define maps by gluing polygons to form a surface. This definition leads to a deep connection between maps and factorizations in the symmetric group.

Suppose that we have a finite set of oriented polygons, and we want to ``glue up'' these polygons by edges to form a closed compact surface without boundary. To glue polygons, we pair up and identify edges. There is only one way to glue two edges such that the orientations of polygons are preserved across the border. Since we want no boundary, all edges have to be paired up. Furthermore, since we only want one surface, there must be a way to go between any two edges by visiting neighbors in the same polygon and by going from one edge of a polygon to the other with which it is glued. We thus obtain a surface with a graph embedded on it formed by vertices and edges of polygons. We call this process a \mydef{gluing process}.

\begin{defn}[Combinatorial maps, alternative definition]
A \mydef{combinatorial map} is a surface with a graph embedded resulting from a gluing process, defined up to orientation-preserving homeomorphism.
\end{defn}

This definition, which dates back to Cori \cite{map-perm-code}, is equivalent to our previous definition of maps as graph embeddings. To see heuristically the equivalence, we first observe that a gluing process always gives a graph embedded onto a surface whose faces are all polygons, which are topological disks. Then for a graph embedded onto a surface, its faces all have finite degree, thus they are polygons, and we can see the embedded graph as a result of a gluing process of these face-polygons. A detailed proof can be found in Chapter~3.2 of \cite{Mohar-book}.

There is a way to encode gluing processes using permutations. For such a process with polygons with a total number $2n$ of edges, we suppose that each edge receives a distinct label from $1$ to $2n$. The set of polygons can thus be encoded by a permutation $\phi$ in the symmetric group $S_{2n}$ in which cycles consist of labels of edges of the same polygon in clockwise order. Since edges are all matched up by gluing, their matching can be encoded by a fixed-point-free involution $\rho \in S_{2n}$. The fact that the process leads to a connected surface is expressed by the transitivity of the pair $(\phi, \rho)$. A pair of permutations $(\phi, \rho)$ is \mydef{transitive} if the orbit of any $i$ from $1$ to $2n$ in the subgroup generated by $\phi$ and $\rho$ is the entire set of integers from $1$ to $2n$. Such a transitive pair of permutations $(\phi, \rho)$ where $\rho$ is a fixed-point-free involution is called a \mydef{rotation system} of general maps. Figure~\ref{fig:1:gluing} gives an example of a map given by the gluing process encoded by a rotation system. We take the convention to root the map at the edge with label $1$, oriented counter-clockwise inside its polygon.

\begin{figure}
  \centering
  \insertfigure[0.85]{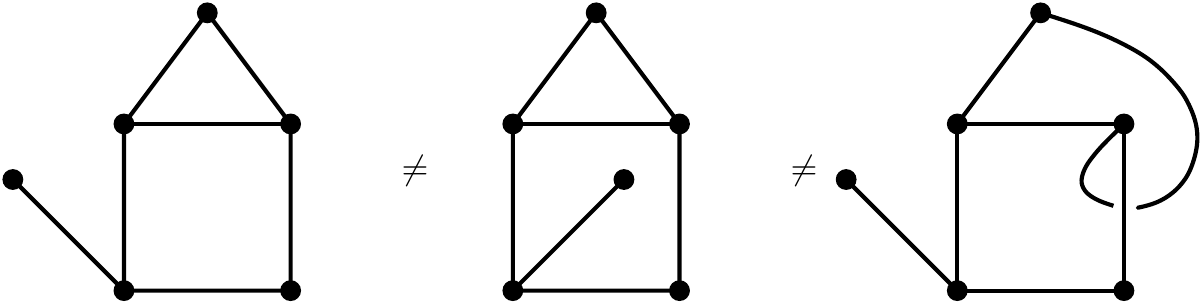}{3}
  \caption{A planar map resulting from the gluing process encoded by $(\phi, \sigma)$. Note that the point at infinity is inside the heptagon.}
  \label{fig:1:gluing}
\end{figure}

In some literature, instead of labeling the two sides of an edge, we cut the edge in half and label the \emph{half-edges}. It is easy to see that this alternative labeling gives an equivalent definition (see, \textit{e.g.}, \cite{JV1990a}).

Every rotation system encodes a gluing process, thus a map. But a map can have many rotation systems, because rotation systems rely on labeling of edges, but maps don't. Every rooted map with $n$ edges corresponds to $(2n-1)!$ rotation systems, since labels can be attached to edges arbitrarily, except for the label $1$ that indicates the root. Two rotation systems $(\phi_1, \rho_1)$ and $(\phi_2, \rho_2)$ give the same rooted map if and only if there exists a permutation $\pi \in S_{2n}$ with $\pi(1) = 1$ such that $\phi_1 = \pi \phi_2 \pi^{-1}$ and $\rho_1 = \pi \rho_2 \pi^{-1}$.

Given a rotational system $(\phi, \rho)$, which is also a pair of permutations, it is natural to look at the product $\sigma = \rho \phi$. Here we take the convention $(\rho \phi)(i) = \phi(\rho(i))$, \textit{i.e.} we multiply permutations from left to right. It turns out that the product $\sigma$ encodes how edges surround vertices. Around a vertex $v$, the product $\sigma$ permutes labels on the right side (seen from $v$) of adjacent edges in counter-clockwise order. Figure~\ref{fig:1:perm-product} provides an illustration of how this works, and Figure~\ref{fig:1:gluing} provides a concrete example. Therefore, the lengths of cycles of $\phi$ and $\sigma$ encode respectively two important statistics of the map: the degrees of faces and vertices. This fact urges us to tap into the algebraic structure of the symmetric group for hints to enumerations of maps. Due to the importance of the product $\sigma$, we sometimes also write the rotation system $(\phi, \rho)$ as $(\phi, \rho, \sigma)$.

\begin{figure}
  \centering
  \insertfigure{ch1.pdf}{4}
  \caption{The action of $\textcolor{blue}{\sigma} = \textcolor{red}{\rho} \textcolor{green!50!black}{\phi}$ for a rotation system $(\phi, \rho)$}
  \label{fig:1:perm-product}
\end{figure}

In this thesis, we will encounter various restricted families or variants of maps, and we will see how they correspond to different types of rotation systems, all in forms of pairs or tuples of permutations of the same size. Their definitions follow the same idea as the type of rotation systems we presented above: since maps are glued up from topologically trivial polygons, all topological information lies in the local topology around vertices and edges; thus to encode the whole map, we only need to label some small structures (such as edges), and then write down how they lie around vertices and edges in the form of permutations. Later we will see several types of rotation systems for different kinds of maps, including bipartite maps and constellations (detailed definitions will be given in the next section).

A \mydef{bipartite map} is a map with a coloring of vertices in black and white such that every edge is adjacent to two vertices with different colors. A rotation system $(\sigma_\bullet, \sigma_\circ, \phi)$ for a bipartite map with $n$ edges is a transitive triple of permutations in $S_n$ (\textit{i.e.} the group generated by $\sigma_\bullet, \sigma_\circ, \phi$ acts transitively on all numbers from $1$ to $n$) such that 
\[ \sigma_\bullet \sigma_\circ \phi = \id_n. \]
Here, $\id_n$ is the neutral element of $S_n$. Exact details of how such a rotation system encodes a bipartite map will be given in Chapter~\ref{sec:2:map-const}. Therefore, rotation systems for bipartite maps are exactly \mydef{transitive factorizations of the identity} in the symmetric group into three elements. Bipartite maps have a generalization called \emph{constellations} with a parameter $m$ (or $r$ in some literature, or more scarcely $p$), whose vertices come in $m$ colors, and whose rotation systems correspond to transitive factorizations of the identity in the following form:
\[ \sigma_1 \sigma_2 \cdots \sigma_m \phi = \id_n.\]
A detailed definition will be given in Chapter~\ref{sec:2:map-const}. Figure~\ref{fig:1:const} shows an example of a constellation. Thanks to combinatorial techniques such as bijections \cite{BMS, BDFG, chapuy2009asymptotic}, characters \cite{poulalhon2002factorizations} and functional equations \cite{BMJ}, we are able to enumerate constellations in some cases, thus also some constellation-type factorizations of the identity.

\begin{figure}
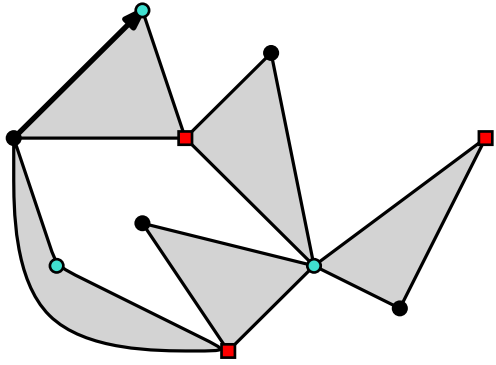

  \centering
  \insertfigure{ch1.pdf}{6}
  \caption{An example of a $3$-constellation}
  \label{fig:1:const}
\end{figure}

There are other types of factorizations of the identity that have been studied in different contexts. A \mydef{transposition} is a permutation that exchanges exactly two elements while keeping others intact. We consider transitive factorizations of the identity in the following form:
\[ \tau_1 \tau_2 \cdots \tau_m \phi = \id_n, \]
where all $\tau_i$ are transpositions and $\phi$ an arbitrary permutation. These factorizations are counted by the \emph{Hurwitz numbers}, which also count \emph{branched coverings} of the \emph{Riemann sphere} (or \emph{extended complex plane}). These numbers were first studied in the context of branched coverings by Hurwitz, who also brought the connection with factorizations of the identity to the sight of combinatorialists \cite{hurwitz, hurwitz-ree}. Such factorizations of the identity involving transpositions quickly grabbed the attention of combinatorialists, and much effort was poured into the enumeration of these objects \cite{classical-hurwitz} and their various generalizations \cite{monotone-hurwitz-0, GGPN}. As part of the effort, researchers proposed map models for these objects \cite{hurwitz-bij}, which makes them officially happy members of the map family.

As usual, some theoretical physicists are also attracted to this encoding of maps in terms of permutations, but for a slightly different reason. It turns out that Hurwitz numbers and their variants are also related to a type of matrix integral called Harish-Chandra-Itzykson-Zuber (HCIZ) integrals \cite{hciz} and to the Gromov-Witten theory \cite{classical-hurwitz}, both playing an important role in 2-dimensional quantum gravity. On the other hand, it has been proved by Goulden and Jackson \cite{goulden2008kp} that the generating functions of general maps are solutions of a well-studied set of equations called the \emph{KP hierarchy}, which is a physical model that falls into the category of \emph{integrable systems} that is actively studied. This connection leads physicists to investigate further generalized map models as solutions to the KP hierarchy and its variant \emph{2-Toda lattice} equipped with further structures \cite{Okounkov}. Using some equations in the KP hierarchy, Goulden and Jackson \cite{goulden2008kp}, and later Carrell and Chapuy \cite{kp-quad}, were able to derive simple and elegant recurrences for the number of triangulations and quadrangulations respectively.

Rotation systems also have a concrete application which we may not expect. In 3D modeling, the surface of a real-world object is often approximated by a mesh glued up of small polygons, which is essentially a map with extra data. Some encodings of these meshes, such as the quad-edge representation \cite{guibas1985primitives}, use exactly the same idea of rotation systems. There are also researches on using enumeration results to design more succinct map encodings, such as \cite{encoding-triangulation}, that might have an impact on computer graphics.

\section{Tools for map enumeration}

As shown in the previous section, map enumeration is related to many other interesting fields and problems, which gives us a strong incentive to count maps. In this section, we will survey some general tools for both exact and asymptotic map enumeration. Precise definitions and examples of the tools we need will be given in Chapter~2.

\subsection{Generating functions} \label{sec:1:gen}

Generating functions are standard tools in enumerative combinatorics. The generating function $F_{\mathcal{C}}(t)$ of a class $\mathcal{C}$ with a size statistics $|\cdot|$ is defined as
\[ F_{\mathcal{C}}(t) = \sum_{c \in \mathcal{C}} t^{|c|}.  \]
We assume here that there are finitely many objects with any given size. A more general definition of generating functions will be given in the next chapter. To enumerate objects in $\mathcal{C}$ of different sizes, we can translate a decomposition of objects in $\mathcal{C}$ into a functional equation that characterizes $F_{\mathcal{C}}$, then we can solve for the generating function which contains all the enumerative information we want. In the seminal series of papers \cite{tutte-0, tutte-1, tutte-2, Tutte:census}, Tutte applied this method to the enumeration of many classes of planar maps, and his way of writing functional equations for maps is still actively used today. Most map enumeration results were first obtained by solving one of these equations.

For instance, consider general planar maps. We start by a simple question: what would a planar map become if its root edge were removed? We allow the ``empty map'' that consists of a single vertex but no edge for the moment. Now, for a planar map with at least one edge, its root is either a \emph{bridge} (\emph{isthmus}) or not. In the case of a bridge, its removal gives two smaller planar maps with appropriate rooting. Otherwise, its removal will merge the two adjacent faces, one of them the outer face, and give a new re-rooted map with one less edge. Figure~\ref{fig:1:tutte-eq} shows this decomposition by root removal in detail. By introducing the extra parameter of outer face degree, we are able to write the following functional equation on the generating functions $M(t,x)$ of planar maps, where $x$ marks the degree of the outer face and $t$ marks the number of edges:
\[ M(t,x) = 1 + tx^2M(t,x)^2 + tx\frac{xM(t,x) - M(t,1)}{x-1}. \]
By solving this equation, Tutte obtained the following simple formula for the number $M_n$ of planar maps with $n$ edges:
\begin{equation} \label{eq:1:tutte-planar}
  M_n = \frac{2 \cdot 3^n}{(n+1)(n+2)} \binom{2n}{n}.
\end{equation}
Details of this resolution will be given in the Chapter~\ref{sec:2:example} as an example of the generating function method. With the idea of root removal, we can write functional equations for other families of planar maps and even maps of higher genus (see, \textit{e.g.}, Chapter~4). These equations are generally called \mydef{Tutte equations} or \emph{cut-and-join equations}.

\begin{figure}
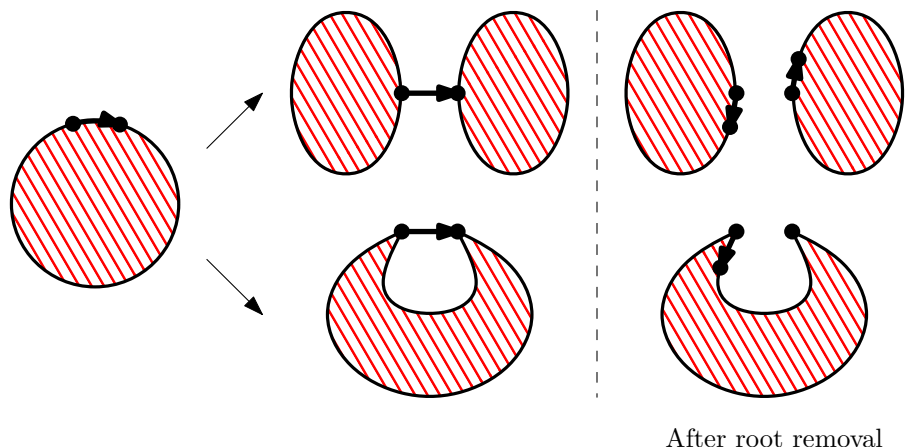

  \centering
  \insertfigure{ch1.pdf}{5}
  \caption{Planar map decomposition by root removal}
  \label{fig:1:tutte-eq}
\end{figure}

For many families of maps, the Tutte equations often involve, in addition to the size variable $t$, an extra variable, dubbed \mydef{catalytic variable}. Many (but not all) Tutte equations thus fall into the category of \mydef{polynomial functional equations with one catalytic variable}. Quadratic equations in this category are often solvable by the \emph{quadratic method}, which was first proposed by Brown in \cite{brown-quadratic}. This method was then used to enumerate many families of planar maps. Another method called \emph{kernel method} has also been used for linear equations occurring in map enumeration \cite{BC0}. Later, Bousquet-M\'elou and Jehanne generalized both methods in \cite{BMJ} to a systematic way to solve any polynomial functional equation with one catalytic variable and confirmed that their solutions are all algebraic under suitable assumptions. They then applied this method to prove that the generating function of $m$-constellation is algebraic, \textit{i.e.} is a solution of a polynomial equation.

Although powerful, sometimes the quadratic method and its generalization in \cite{BMJ} cannot deal easily with some map-counting generating functions with further refinements, such as the exact profile of face degrees for maps of higher genus. Nevertheless, Tutte equations can be written for such refined generating functions for planar maps and maps of higher genus, sometimes involving various operators. Examples include the cut-and-join equations written for Hurwitz numbers by Goulden and Jackson in \cite{classical-hurwitz-planar}, which they solved in a guess-and-check manner for the genus $0$ case. This approach was then extended to monotone Hurwitz numbers of genus $0$ by Goulden, Guay-Paquet and Novak in \cite{monotone-hurwitz-0}. However, for higher genus, since there was no known explicit formula, the resolution needs either more involved algebraic tools \cite{classical-hurwitz} or a different type of analysis \cite{GGPN}. The topological recursion method can also be used to solve Tutte equations (also called \emph{loop equations} in this context) of many families of maps of all genera, for example bipartite maps with given face degree profile \cite{KZ} and even maps (see \cite[Chapter~3]{Eynard:book}, unconventionally called ``bipartite maps'' therein). It is also possible to write the Tutte equation for maps of higher genus using multiple catalytic variables, which led Bender and Canfield to the asymptotic behavior \cite{BC0} of maps of higher genus, and to prove that the generating functions of these maps can be expressed as a rational fraction in an explicit algebraic series \cite{BC}.

Strangely, the generating functions of different families of maps have a lot in common. For instance, they are often rational functions of a few algebraic series \cite{BCRvf}. Moreover, Gao showed in \cite{Gao1991-2-connected-projective, Gao1992-2-connected-surface, Gao1993-pattern} that, for many classes $\mathcal{M}$ of maps, the number of maps in $\mathcal{M}$ of genus $g$ of size $n$ grows asymptotically as $ c_{\mathcal{M}} t_g \gamma_{\mathcal{M}}^n (\alpha_{\mathcal{M}} n)^{5(g-1)/2}$ when $n$ tends to infinity, where $c_{\mathcal{M}}, \gamma_{\mathcal{M}}$ and $\alpha_{\mathcal{M}}$ are constants that depend only on the class $\mathcal{M}$, and $t_g$ a constant depending only on the genus $g$ but staying the same across different map classes. Later Chapuy also obtained results of the same form in \cite{chapuy2009asymptotic} for some other families of maps related to constellations. This concordance hints at some kind of universality for maps. It is worth mentioning that the asymptotic behavior of $t_g$, when $g$ tends to infinity, was determined in \cite{constant-tg} using a recurrence on the number of triangulations obtained from the KP hierarchy \cite{goulden2008kp}, which is essentially an infinite sequence of differential equations satisfied by the generating functions of maps.

\subsection{Bijections} \label{sec:1:bij}

Due to the geometric and intuitive definition of maps as graph embeddings and elegant closed formulas for enumeration, we may imagine that there are many kinds of bijections that can be used for enumeration. This is indeed the case. Pioneered by Cori and Vauquelin \cite{cori-vauquelin}, bijective methods have always been playing an important role in map enumeration ever since. Bijections also allow us to understand map enumeration results in a more intuitive way. However, most bijections are for classes of planar maps. For maps of higher genus, only few bijections are available for enumeration (for instance \cite{CMS, chapuy-dolega}), and they are mostly only useful for enumeration in very specific cases. Nevertheless, for enumeration of planar maps, the bijective method is widely used. We can also discover deep connections between maps and other combinatorial objects such as permutations using bijections.

The bijective study of maps is a well-developed field, and there are many types of bijections relating different classes of maps to various objects, which we could not exhaust. Therefore, we will just point out two major categories of bijections for map enumeration in the following. Of course, there is still a vast ocean of map bijections that cannot be put into any of these categories, but they all contribute towards our combinatorial understanding of maps.

\paragraph{Blossoming trees}
In his thesis \cite{schaeffer-thesis}, Schaeffer observed that the number of planar maps with $n$ edges \eqref{eq:1:tutte-planar} is closely related to that of binary trees. He defined a type of binary trees called \emph{blossoming trees} with an extra blossom on each vertex. Conjugacy classes of blossoming tress are in bijection with planar $4$-valent maps, which are duals of planar \emph{quadrangulations} (\textit{i.e.} maps whose faces are all of degree 4), again in bijection with general planar maps. He then extended this approach to several other types of maps, developing a family of bijections. All these blossoming-tree bijections essentially rely on a breadth-first search on the dual of the map that breaks edges in the original map until what is left is a tree. Broken edges are then turned into blossoms, and we obtain a blossoming tree. Figure~\ref{fig:1:blossom} shows an example of this bijection. A particularly important application of this approach is to constellations \cite{BMS}, where Bousquet-M\'elou and Schaeffer obtained an explicit enumeration formula for planar constellations. A variant of blossoming trees also has application in efficient encoding of planar maps \cite{encoding-triangulation}. A unified scheme of various blossoming-tree bijections on planar maps was given in \cite{ap-blossoming} by Albenque and Poulalhon.

\begin{figure}
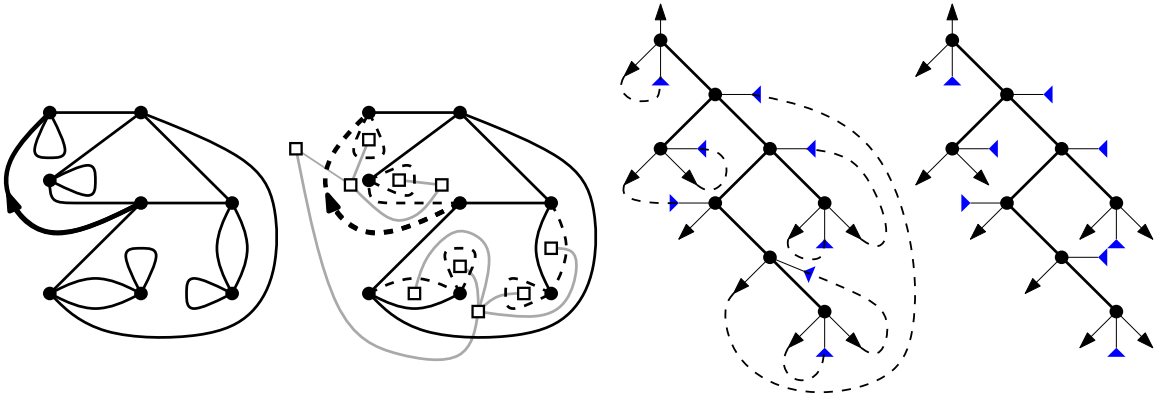

  \centering
  \insertfigure[0.85]{ch1.pdf}{7}
  \caption{An example of the blossoming tree bijection}
  \label{fig:1:blossom}
\end{figure}

\paragraph{Well-labeled trees}
Cori and Vauquelin \cite{cori-vauquelin} first gave a bijection from planar quadrangulations to the so-called \emph{well-labeled trees}, but in a recursive form. Again in his thesis, Schaeffer deconvoluted this recursive bijection and found that it can be described as a set of simple local rules located at each face, depending on the distance of adjacent vertices to a distinguished vertex. Figure~\ref{fig:1:well-labeled} shows an example of this bijection. Later Bouttier, Di Francesco and Guitter generalized this approach to \emph{mobiles} for the case of \emph{face-bicolored map}, which includes in particular constellations. Miermont also introduced in \cite{dist-higher-genus} a version of the bijection of Cori, Vauquelin and Schaeffer to quadrangulations of arbitrary genus with multiple distinguished vertices, for the study of random maps of higher genus. Recently Ambj{\o}rn and Budd \cite{ambjorn-budd} found a similar bijection in the sense that it can also be described in a similar set of simple local rules. Since well-labeled trees contain the distance information of the original maps, they are particularly suitable for the study of the limit of large random maps as a metric space, such as in \cite{ise} and \cite{dist-higher-genus}. A unified scheme of many bijections in this family was given in \cite{bernardi-fusy} by Bernardi and Fusy. A generalization to quadrangulations of higher genus (orientable or not) based on a previous generalization by Marcus and Schaeffer (see \cite{CMS}) was given in \cite{chapuy-dolega} by Chapuy and Do{\l}{\k{e}}ga, then extended in \cite{bettinelli2015bijection} by Bettinelli to some other classes of maps.

\begin{figure}
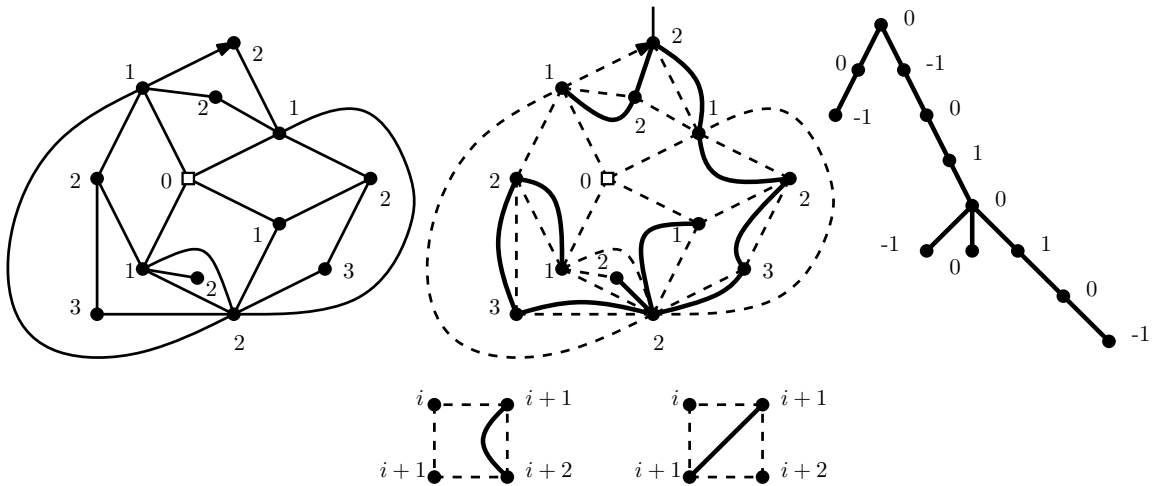

  \centering
  \insertfigure[0.85]{ch1.pdf}{8}
  \caption{An instance of well-labeled tree bijection, with local rules explained}
  \label{fig:1:well-labeled}
\end{figure}

Other than counting, bijective methods can also be used to relate maps to other combinatorial objects in a way that allows us to transfer enumerative results and structural properties. Among these bijections, many are based on the notion of orientations. An \tdef{orientation} on a map is an assignment of orientation to all edges in the map, and it is used in many map bijections. Using a special kind of orientation called \emph{bipolar orientation}, in \cite{fusy-nsp} Fusy gave several bijections between different classes of planar maps, including non-separable planar maps and irreducible triangulations. For the connection between maps and other combinatorial objects, Bernardi and Bonichon gave in \cite{BB2009intervals} a bijection between three families of planar triangulations and intervals in three lattices of Catalan objects, using a notion called \emph{Schnyder wood} which can be considered as a special orientation of triangulations. In \cite{baxter-nsp-map}, Bonichon, Bousquet-M\'elou and Fusy gave a bijection between bipolar orientation of planar maps and Baxter permutations. It is worth mentioning that Bernardi and Fusy also used a generalized version of orientation in the unified scheme of some well-labeled tree bijections in \cite{bernardi-fusy}.

\subsection{Character methods}\label{sec:1:char}

Since maps can be encoded by permutations whose product is the identity, it is natural to think of using the representation theory of the symmetric group to count maps. Although transitive factorizations of the identity cannot be directly counted using characters, we notice that we can split a general factorization of the identity into a bunch of its transitive components, which will solve our problem at the generating function level. Details will be discussed in Chapter~\ref{sec:2:sym} and Chapter~\ref{sec:2:class}. 

However, although characters of the symmetric group have nice combinatorial interpretations (see, for example, \cite{vershik2004new} and \cite{sagan2001symmetric}), they have a simple formula only in special cases, which greatly limits their application to map enumeration. A quintessential example of such an application is the enumeration of \emph{unicellular maps}, \textit{i.e.}, maps with only one face. These maps correspond to factorizations involving a full-cycle permutation, which are always transitive. We then need to evaluate characters in $S_n$ at the partition $(n)$, where only characters indexed by a \emph{hook-partition} (\textit{i.e.}, a partition of the form $(n-a, 1^a)$) can have a non-zero value, which vastly simplifies the computations. Using this approach, Jackson proved that the number of unicellular bipartite maps on any genus $g$ has a certain form in \cite{jackson1987counting}. An explicit and refined expression was then given by Goupil and Schaeffer in \cite{goupil1998factoring}, which was extended in \cite{poulalhon2002factorizations} by Poulalhon and Schaeffer to constellations. It is worth mentioning that there has been an explicit formula \cite{LW} and a nice recurrence \cite{HZ} for the number of unicellular maps that were obtained without using characters of the symmetric group.

It is also possible to obtain enumerative relations between different classes of maps using characters. In \cite{JV1990a}, Jackson and Visentin used characters to obtain a simple enumerative relation between general maps of genus $g$ and bipartite quadrangulations with some marked vertices of genus at most $g$. This relation is called the \mydef{quadrangulation relation}, which was then generalized in \cite{JV1990b} and \cite{jackson1999combinatorial} to general bipartite maps. Despite its elegant and innocent appearance, the quadrangulation relation has resisted all attempts of a bijective proof, making itself something that is only achievable using characters. Similarly, results relying on the KP hierarchy, such as \cite{goulden2008kp} and \cite{kp-quad}, share the same position.

\section{A road map of our tour}

We have seen many faces of maps and their connections to other fields of study. We have also cited some powerful enumeration tools for maps that come essentially from the intuitive definitions and the versatility of maps. But this is just a start of our journey to the magnificence of maps, a teaser if you will. For interested readers, \cite{LandoZvonkine}, \cite{MBM-survey} and \cite{Schaeffer:survey} contain more detailed accounts of the panorama of map enumeration. In the rest of this thesis, we are going to take a more in-depth tour to the domain of maps, in the prism of my own research. The rest of this section is a road map of this thesis.

This chapter and the next one are for preparation. In Chapter~2, we will prepare ourselves with precise definitions of various objects that we will come across during our tour. We will also try to wield some tools we will use, such as resolution of functional equations and analytic methods for asymptotic counting.

In the next two sections, we will see some results on map enumeration, especially of constellations. As we have mentioned in the previous section, there is an enumerative relation of maps called the quadrangulation relation, proved using characters of the symmetric group. In Chapter~3, a generalization of the quadrangulation relation will be presented. This chapter is based on my paper \cite{fang2014generalization}, which generalized the character approach in \cite{JV1990a, JV1990b} to constellations and hypermaps. In Chapter~4, we will see how to write a Tutte equation for constellations and how to solve it in the bipartite case using some ideas from topological recursion. We then obtain a rationality result for the generating functions of bipartite maps and the corresponding rotation systems of genus $g>1$. This result is similar to those in \cite{classical-hurwitz} for Hurwitz numbers and in \cite{GGPN} for monotone Hurwitz numbers. It is thus interesting to investigate a unified proof. This chapter is based on a collaboration article \cite{cf-bipartite} with my advisor Guillaume Chapuy.


The two chapters that follow will concern applications of map enumeration to other fields. In Chapter~5, we will look at a bijection between non-separable planar maps and intervals in generalized Tamari lattices, which have their root in algebraic combinatorics. In the course, we will also give the first combinatorial proof of a theorem concerning self-dual non-separable planar maps published in \cite{kitaev-nsp}. This chapter is partially based on a collaboration with Louis-Fran{\c c}ois Pr\'eville-Ratelle \cite{wf-lfpr}. In Chapter~6, we will enumerate cubic \emph{multigraphs} embeddable on the surface \Sg{} with a fixed genus $g$. This is done by taking a detour to the enumeration of various triangulations of higher genus, similar to the strategy in \cite{graphfive}. This chapter is based on a collaboration with Mihyun Kang, Michael Mo{\ss}hammer and Philipp Spr{\"u}ssel \cite{fang-cubicgraph, fang-graz}, which has implications in the study of random graphs of higher genus.

Finally, in the last chapter, Chapter~7, we will end our tour by some discussions of possible further developments of previously presented results and map enumeration in general.

\chapter{First steps in map enumeration}

To craft a fine work, one must first sharpen the tools. The purpose of this chapter is to prepare ourselves for our tour in the realm of maps and to get familiar with tools that we will use. This chapter will be divided into two parts: definitions of classes of maps and introduction of tools that we will use to enumerate these maps. Readers familiar with maps, generating functions and/or characters of the symmetric group can skip this chapter and use it as a reference.

Since the notions we will introduce are so intertwined, it is difficult to streamline all definitions in a perfect logical order without hurting the presentation. Therefore, some well-known notions will be used before they are defined, but there will be notices and directions.

\section{The many classes of maps} \label{sec:2:map}

In the previous chapter, we have given two definitions of maps. We will refer to this most general class of maps as \mydef{general maps}. All classes of maps we will define later can be considered as sub-classes of general maps. Since we are also going to talk about rotation systems of maps, which live in the symmetric group $S_n$, we will assume for the moment that readers are familiar with permutations, especially their cycle presentation. A detailed introduction to the symmetric group will be given in Section~\ref{sec:2:sym}.

\subsection{Maps involving degree and connectivity}

Some elementary sub-classes of general maps are defined by properties on their faces. In terms of polygon gluing process, the \mydef{degree} of a face is the number of edges in the corresponding polygon before gluing. Therefore, if an edge borders a face $f$ twice, it is also counted twice in the degree of $f$. The map on the left side of Figure~\ref{fig:2:triang} contains two edges of this type.

\begin{figure}
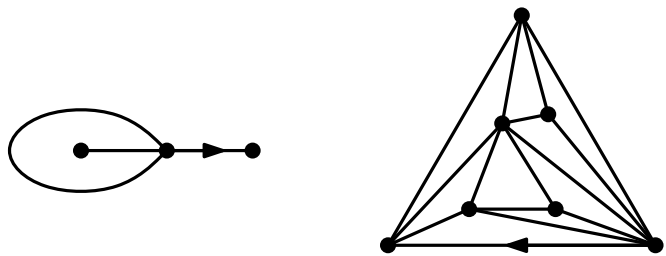

  \centering
  \insertfigure{ch2-fig.pdf}{1}
  \caption{A general planar triangulation and a simple planar triangulation}
  \label{fig:2:triang}
\end{figure}

We can now define classes of maps using the notion of face degree. A \mydef{triangulation} is a map whose faces are all of degree 3. A \mydef{quadrangulation} is a map whose faces are all of degree 4. More generally, a \mydef{$p$-angulation} is a map whose faces are all of degree $p$, and an \mydef{even map} is a map whose faces are all of even degree. Figure~\ref{fig:2:triang} shows two planar triangulations. We further define a subclass of triangulations called \mydef{simple triangulations}, which are triangulations without loop nor multiple edges. The triangulation on the right side of Figure~\ref{fig:2:triang} is simple.

Other than face degree, we can also restrict maps by their connectivity. All maps are connected, but some maps are more connected than others, in the sense that they remain connected even if we remove something from them. There are two different notions of connectivity in graph theory, edge-connectivity and vertex-connectivity, which can be transplanted directly to maps. A map is \mydef{$k$-edge-connected} if the map remains connected after the removal of any $k-1$ edges. All maps are $1$-edge-connected but not necessarily $2$-edge-connected. Figure~\ref{fig:2:conn}(a) is an example of a map that is not $2$-edge-connected. In this case, an edge whose removal disconnects the map is called a \mydef{bridge}. A map is \mydef{$k$-vertex-connected} (or simply \mydef{$k$-connected}) if, for any partition $E_1, E_2$ of the edge set $E$ of the map, there are at least $k$ vertices that have adjacent edges both in $E_1$ and in $E_2$. Similarly, all maps are $1$-connected but not necessarily $2$-connected. A map is called \mydef{separable} if it is not $2$-connected, and \mydef{non-separable} if it is. Figure~\ref{fig:2:conn}(b) is an example of a separable map. In this case, a vertex $v$ is called a \mydef{cut vertex} if there exists a partition $E_1, E_2$ of the edge set $E$ of the map such that $v$ is the only vertex that has adjacent edges in both $E_1$ and $E_2$. There can be several cut vertices in the same separable map. Figure~\ref{fig:2:conn}(c) is an example of a non-separable map, which is also $2$-edge-connected, but is neither $3$-connected nor $3$-edge-connected.

\begin{figure}
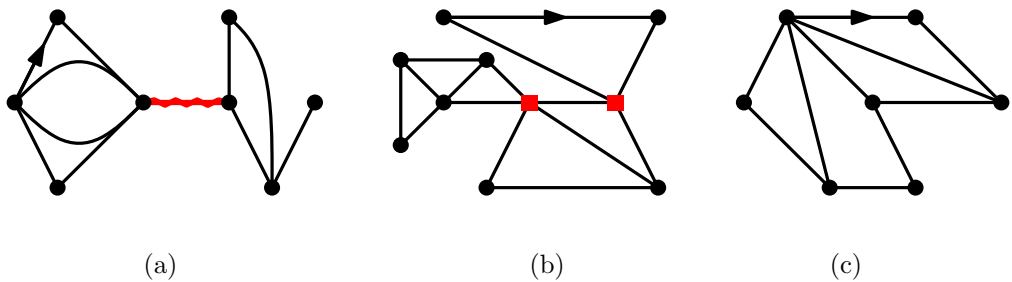

  \centering
  \insertfigure{ch2-fig.pdf}{2}
  \caption{Some examples of connectivity in maps, with marked bridges and cut vertices}
  \label{fig:2:conn}
\end{figure}

\subsection{Bipartite maps and constellations} \label{sec:2:map-const}

We can also define classes of maps using vertex colorings, yet another notion from graph theory. Let $M$ be a map and $V$ the set of its vertices. A \mydef{vertex coloring} (or simply a \mydef{coloring}) is a function $f$ from $V$ to a color set $C$. The color set is usually a finite set of natural numbers $\{ 1, 2, \ldots, c \}$, in this case the coloring is also called a \mydef{$c$-coloring}. A vertex coloring $f$ of the map $M$ is \mydef{proper} if for any edge $e = \{ v_1, v_2 \}$ in $M$ we have $f(v_1) \neq f(v_2)$. A map is said to be \mydef{bipartite} if it has a proper $2$-coloring. Colors in a bipartite map is colloquially referred to as \emph{black} and \emph{white}. By convention, the root vertex of a bipartite map is always black, which fixes the $2$-coloring. Figure~\ref{fig:2:bip} shows a bipartite map of genus $1$. Notice that, although all bipartite maps are even maps, the converse is not true for even maps of higher genus. For instance, by identifying opposite sides of an $m \times n$ rectangular grid, we obtain an even map (in fact a quadrangulation) on the torus, but it is bipartite if and only if both $m$ and $n$ are even.

\begin{figure}
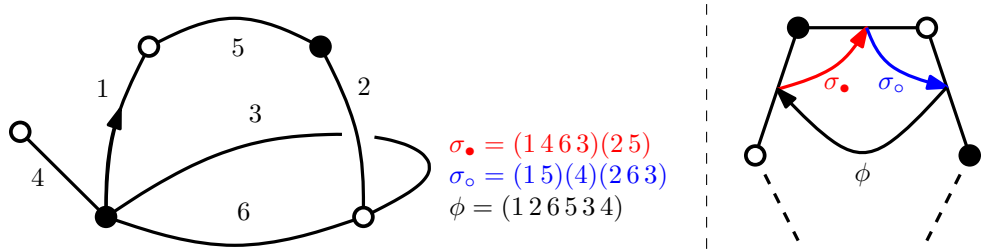

  \centering
  \insertfigure{ch2-fig.pdf}{3}
  \caption{A bipartite map and its rotation system}
  \label{fig:2:bip}
\end{figure}

There is a particularly simple way to define rotation systems for bipartite maps. Let $M$ be a bipartite map with $n$ edges, with distinct labels from $1$ to $n$. The root receives the label $1$ by convention. Since $M$ is bipartite, each edge is adjacent to exactly one black vertex and one white vertex. For each black vertex, we read out the labels of its adjacent edges in counter-clockwise order to obtain a cycle, and all these cycles form a permutation $\sigma_\bullet$ since each edge belongs to exactly one cycle. We similarly define $\sigma_\circ$ for white vertices. For faces, we also consider adjacent edges in counter-clockwise order, or equivalently, we can imagine that we are taking a tour inside a face while keeping adjacent edges always on the right. By only picking edges that point from black vertices to white ones in the tour inside a face, we construct a cycle for each face, and all cycles for faces form a permutation $\phi$. We say that the rotation system for the bipartite map $M$ is $(\sigma_\bullet, \sigma_\circ, \phi)$, which is a transitive triple of permutations in $S_n$ that gives the following factorization of the identity:
\[ \sigma_\bullet \sigma_\circ \phi = \id_n. \]
We recall that we multiply permutations from left to right. Figure~\ref{fig:2:bip} shows an example of such a rotation system, and an illustration on why a rotation system of a bipartite map gives a factorization of the identity. We see that each bipartite map with $n$ edges gives $(n-1)!$ different rotation systems by changing labeling.

We observe that rotation systems of bipartite maps give transitive factorizations of the identity of length $3$. It is natural to try to find a class of maps whose rotation systems are factorizations of the identity of arbitrary fixed length, generalizing bipartite maps. Indeed, such a generalization exists, and is called \emph{constellations}. Constellations are also more ``colorful'' than bipartite maps, in the sense that they come with an integral parameter $m \geq 2$ for the number of vertex colors, and their definition involves proper coloring with $m$ colors.

\begin{defn} \label{def:2:constellation}
An \mydef{$m$-constellation} is a map $M$ with a proper $m$-coloring $f$ that satisfies the following conditions:
\begin{itemize}
\item The dual of $M$ is bipartite, which induces colors on faces of $M$, where black faces are called \mydef{hyperedges}, and white faces \mydef{hyperfaces};
\item Each hyperedge has degree $m$, and each hyperface has degree a multiple of $m$;
\item In the proper $m$-coloring $f$ of $M$, vertices adjacent to each hyperedge are colored from $1$ to $m$ in counter-clockwise order.
\end{itemize}
An \mydef{$m$-hypermap} is a map $M$ without coloring that satisfies the first two conditions .
\end{defn}

For a more detailed treatment of constellations, readers are referred to \cite{LandoZvonkine}. In some literature, the color parameter is denoted by $k$ (\textit{e.g.} \cite{LandoZvonkine}), $r$ (\textit{e.g.} \cite{GGPN}) or $p$ (\textit{e.g.} \cite{BDFG}). The term ``hypermap'' has also been used in some literature for different classes of maps, and we follow here the definition in \cite{chapuy2009asymptotic}. Figure~\ref{fig:2:constellation} gives an example of a planar $3$-constellation. As a convention, the root of a constellation must point from one vertex of color $1$ to a vertex of color $m$ while leaving its adjacent hyperedge on its right. Since such edges are in bijection with hyperedges, the hyperedge adjacent to the root is also called the \mydef{root hyperedge}. We observe that an $m$-constellation with $n$ hyperedges has $mn$ edges.

\begin{figure}
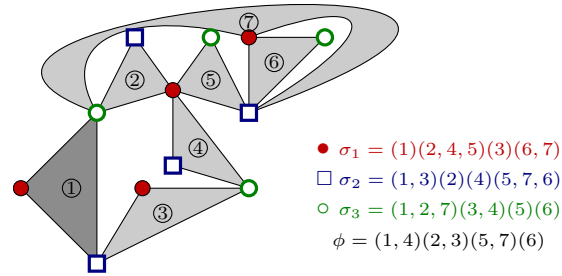

  \centering
  \insertfigure{constellation-def.pdf}{1}
  \caption{A example of a $3$-constellation and its rotation system}
  \label{fig:2:constellation}
\end{figure}

We see in Figure~\ref{fig:2:bip-const} that constellations indeed generalize bipartite maps. Given a bipartite map, we ``blow'' its edges into hyperedges of degree $2$, and we obtain a $2$-constellation. We thus see that hyperedges in a constellation play the same role as edges in bipartite maps. Since rotation systems of bipartite maps are defined as tuples of permutations of edges, we are tempted to define rotation systems of constellations as tuples of permutations of hyperedges.

\begin{figure}
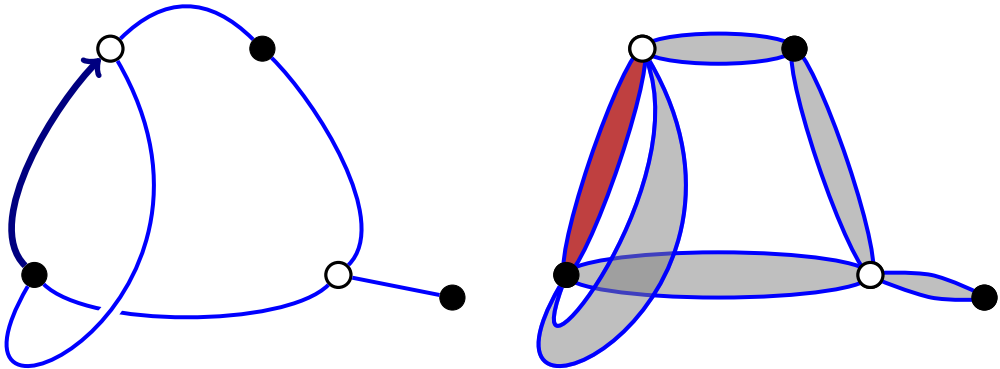

  \centering
  \insertfigure{bipartite-const.pdf}{1}
  \caption{Bipartite map as $2$-constellation}
  \label{fig:2:bip-const}
\end{figure}

We now define a rotation system for constellations. Let $M$ be a constellation containing $n$ hyperedges with distinct labels from $1$ to $n$. As a convention, the root hyperedge always receives label $1$. For each color $i$, we define a permutation $\sigma_i$, whose cycles are formed by hyperedges in counter-clockwise order around each vertex. We also construct similarly a permutation $\phi$ whose cycles correspond to hyperfaces, but for a hyperface we only consider hyperedges that it borders on across an edge with vertices of color $1$ and $m$. This definition of $\phi$ is reminiscent to that of bipartite maps, where we only consider edges pointing from a black vertex to a white vertex in the tour inside each face. The $(m+1)$-tuple $(\sigma_1, \sigma_2, \ldots, \sigma_m, \phi)$ is the rotation system of the given constellation $M$ and it is also a transitive factorization of the identity in $S_n$:
\[ \sigma_1 \sigma_2 \cdots \sigma_m \phi = \id_n. \]
Figure~\ref{fig:2:const-rot} illustrates the construction of such rotation systems and the reason why they are transitive factorizations of the identity.

\begin{figure}
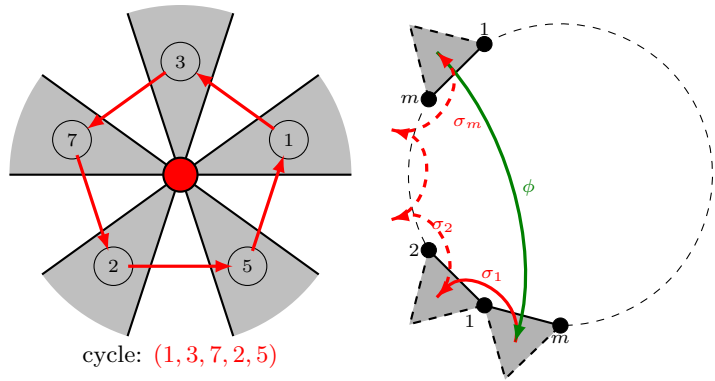

  \centering
  \insertfigure{const-rot.pdf}{1}
  \caption{Rotation systems of constellations as factorizations of the identity in the symmetric group}
  \label{fig:2:const-rot}
\end{figure}

In the form of transitive factorizations of the identity into an arbitrary fixed number of permutations, constellations can serve as a unifying scheme for different kinds of factorizations, such as those enumerated by classical or monotone Hurwitz numbers. However, we will delay this connection until having savored the delicacy of the symmetric group.

\section{Symmetric group} \label{sec:2:sym}

In this section, we will talk about the symmetric group and its representation theory. Since they are all immense fields of research, we will only scratch their surfaces for what we need in the following chapters. For a more detailed treatment of these fields, readers are referred to various sources: \cite{serre1977linear} for the general theory of group representation, and \cite{vershik2004new, sagan2001symmetric} and \cite[Chapter~7]{Stanley:EC2} for a combinatorial treatment of the representation theory of the symmetric group.

\subsection{Group algebra and characters of the symmetric group}

A \mydef{permutation} of size $n$ is a bijective function from the set of integers from $1$ to $n$ to itself. All permutations of size $n$ form a group called the \mydef{symmetric group} of degree $n$, denoted by $S_n$, where the group law is function composition: $(\sigma \tau)(i) = \tau(\sigma(i))$. Notice that while functions compose from right to left, the group law we use multiplies from left to right. A permutation $\sigma \in S_n$ can be presented as a word $\sigma(1) \sigma(2) \ldots \sigma(n)$. For example, the permutation $\sigma = 35241$ sends $1$ to $3$, $2$ to $5$, \textit{etc.}

Now, for a permutation $\sigma$, we consider the orbits of its action on integers from $1$ to $n$, which are also called \mydef{cycles}. Each cycle is for the form $(i, \sigma(i), \sigma^2(i), \ldots, \sigma^{k-1}(i))$, where $k$ is the minimal value such that $\sigma^k(i)=i$. We can then present a permutation as the collection of its cycles. For instance, the permutation $\sigma = 35241$ can also be written as $\sigma = (1,3,2,5)(4)$.

We now consider conjugacy classes in $S_n$. Two elements $\sigma, \tau$ in $S_n$ are \mydef{conjugate} if there exists another element $\pi \in S_n$ such that $\tau = \pi \sigma \pi^{-1}$. Conjugacy is an equivalence relation, therefore we can define the \mydef{conjugacy class} of an element $\sigma \in S_n$ to be the set of elements conjugate to $\sigma$. If we consider $\sigma$ as a set of cycles, the conjugate operation can be consider as a relabeling of integers in the cycles in all possible way while preserving the cycle structures. Therefore, a conjugacy class in $S_n$ is the set of permutations with a fixed cycle structure.

There is a notion that precisely captures the cycle structure of permutations. An \mydef{integer partition} $\lambda$ (or simply \mydef{partition}) is a finite non-increasing sequence of positive integers, \textit{i.e.}, $\lambda = [\lambda_1, \lambda_2, \ldots, \lambda_k]$, where $\lambda_i > 0$ and $\lambda_i \geq \lambda_{i+1}$ for all index $i$. Each $\lambda_i$ is called a \mydef{part} of $\lambda$. The empty partition that has no part is denoted by $\epsilon$. We denote by $\ell(\lambda)$ the length of the partition $\lambda$, \textit{i.e.} the number of its parts. We say that $\lambda$ is a partition of a natural number $n$, also denoted as $\lambda \vdash n$, if $\sum_i \lambda_i = n$. Equivalently, we say that the \mydef{size} of a partition $\lambda$, denoted by $|\lambda|$, is $n$ if $\lambda \vdash n$. To alleviate the notation, when there are multiple identical parts in a partition, we write them as a ``power'', \textit{e.g.} we may write $[4,2,2,2,1,1]$ as $[4,2^3,1^2]$. For a permutation $\sigma$, the partition $\lambda$ obtained by listing lengths of all cycles in $\sigma$ in increasing order is called the \mydef{cycle type} of $\sigma$. Since the cycle type completely describes the cycle structure of a permutation, we can index conjugacy classes of $S_n$ by partitions of $n$. We denote by $Cl(\lambda)$ the conjugacy class of $S_n$ formed by permutations with cycle type $\lambda \vdash n$. There are $n!z^{-1}_\lambda$ elements in $Cl(\lambda)$, where $z_\lambda$ is the size of the centralizer of any element $\phi \in Cl(\lambda)$, \textit{i.e.} the number of permutations $\sigma$ such that $\sigma \phi = \phi \sigma$. If $\lambda$ has $m_i$ parts of size $i$, we have $z_\lambda = \prod_{i} i^{m_i} m_i!$.

We will see the importance of conjugacy classes in $S_n$ through another object. The \mydef{group algebra} $\mathbb{C}S_n$ of $S_n$ is the complex vector space with a canonical basis indexed by elements in $S_n$ and a multiplication extending distributively the group law of $S_n$. By abuse of notation, we will identify the basis vector indexed by an element in $S_n$ and the element itself. As an example, for a partition $\lambda \vdash n$, we define $K_\lambda = \sum_{\sigma \in Cl(\lambda)} \sigma$, which is an element in $\mathbb{C}S_n$. The group algebra $\mathbb{C}S_n$ has a natural inner product defined by regarding the scaled canonical basis $((n!)^{-1/2}\sigma)_{\sigma \in S_n}$ as an orthonormal basis.

We now consider a sub-algebra of $\mathbb{C}S_n$ called the \mydef{center}, denoted by $Z(\mathbb{C}S_n)$, which is formed by elements that commute with every element in $\mathbb{C}S_n$. The center is therefore a commutative algebra, but how do its elements look like? The answer is given in the following proposition.

\begin{prop}
For $n \geq 1$, the set $\{ K_\lambda \mid \lambda \vdash n \}$ is a linear basis of $Z(\mathbb{C}S_n)$.
\end{prop}
\begin{proof}
We first prove that $K_\lambda$ is an element in $Z(\mathbb{C}S_n)$. For any $\tau \in S_n$, we have $\tau K_\lambda \tau^{-1} = \sum_{\sigma \in Cl(\lambda)} \tau \sigma \tau^{-1} = K_\lambda$, since $Cl(\lambda)$ is a conjugacy class. We thus have $\tau K_\lambda = K_\lambda \tau$, and by linear combination, we see that $K_\lambda$ indeed commutes with all elements in $\mathbb{C}S_n$.

It is clear that all $K_\lambda$ are linearly independent. We now prove that all $K_\lambda$ span linearly $Z(\mathbb{C}S_n)$. Let $a = \sum_{\sigma \in S_n} a_\sigma \sigma$ be an element of $Z(\mathbb{C}S_n)$. Since $a$ is in the center, we have $a = \tau a \tau^{-1} = \sum_{\sigma \in S_n} a_\sigma \tau \sigma \tau^{-1}$ for any $\tau \in S_n$, from which we have $a_\sigma = a_{\tau \sigma \tau^{-1}}$. Therefore, $a_\sigma$ is the same for all $\sigma$ in the same conjugacy class, thus $a$ is a linear combination of some $K_\lambda$. Hence, we conclude that $\{ K_\lambda \mid \lambda \vdash n \}$ is a linear basis of $Z(\mathbb{C}S_n)$.
\end{proof}

As a consequence, the dimension of $Z(\mathbb{C}S_n)$ is the number of partitions of $n$. By the representation theory (\textit{cf.} \cite{serre1977linear}, or an elegant treatment specialized to the symmetric group in \cite{vershik2004new}), there is an orthonormal basis $(F_\theta)_{\theta \vdash n}$ of $Z(\mathbb{C}S_n)$ whose elements are also indexed by partitions of $n$ and satisfy $F_\theta^2 = F_\theta$ for all $\theta \vdash n$. Furthermore, we also have $F_{\theta_1}F_{\theta_2} = \delta_{\theta_1, \theta_2} F_{\theta_1}$, which makes them particularly suitable for computation.

We now have two bases of $Z(\mathbb{C}S_n)$, and we want to know how to change from one to the other. It is here that we see characters. For an element $a \in Z(\mathbb{C}S_n)$, we denote by $[K_\lambda]a$ (resp. $[F_\theta]a$) the coefficient of $K_\lambda$ (resp. $F_\theta$) in the expression of $a$ as a linear combination of elements in the basis $(K_\lambda)_{\lambda \vdash n}$ (resp. $(F_\theta)_{\theta \vdash n}$. We denote by $f^\theta$ the dimension of the vector space $F_\theta \mathbb{C}S_n$, which is also the dimension of the irreducible representation of $S_n$ indexed by $\theta$. The \mydef{character} $\chi^\theta_\lambda$ of $S_n$ indexed by $\theta$ and evaluated at $\lambda$ is defined as the coefficient in the following change of basis between $(K_\lambda)_{\lambda \vdash n}$ and $(F_\theta)_{\theta \vdash n}$:
\begin{equation} \label{eq:2:basis-change}
F_\theta = \frac{f^\theta}{n!} \sum_{\lambda \vdash n} \chi^\theta_\lambda K_\lambda, \quad K_\lambda = \frac{n!}{z_\lambda} \sum_{\theta \vdash n} \frac{\chi^\theta_\lambda}{f^\theta} F_\theta.
\end{equation}
The two formulas of change of basis give the same values of $\chi^\theta_\lambda$, which can be seen by substituting one in the other and using the fact that $(F_\theta)_{\theta \vdash n}$ is an orthonormal basis. We also have $\chi^\theta_{[1^n]} = f^\theta$.


Despite their algebraic definition, the characters in the symmetric group can also be defined in a purely combinatorial way. Given a partition $\lambda = (\lambda_1, \ldots, \lambda_k)$, its \mydef{Ferrers diagram} is a graphical representation of $\lambda$ consisting of left-aligned rows of boxes (also called \mydef{cells}), in which the $i$-th line has $\lambda_i$ boxes. Figure~\ref{fig:2:diagrams}(a) shows an example of a Ferrers diagram, drawn in French convention, where the first row lies at the bottom.

\begin{figure}
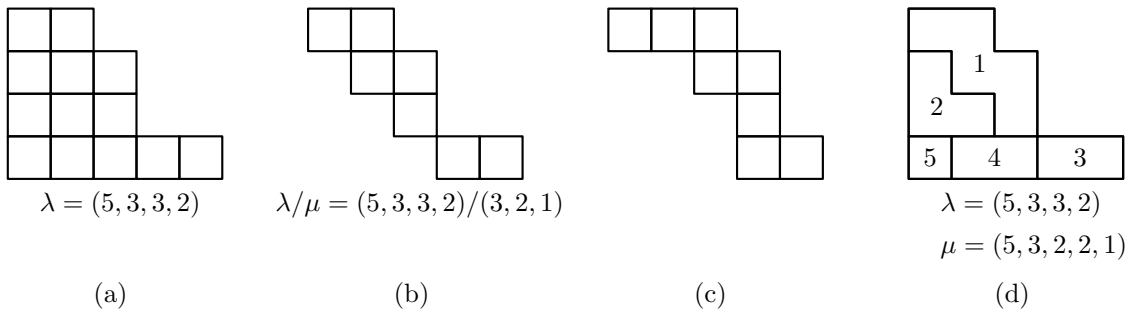

  \centering
  \insertfigure{ch2-fig.pdf}{8}
  \caption{Examples of a Ferrers diagram, a skew diagram, a ribbon and a ribbon tableau}
  \label{fig:2:diagrams}
\end{figure}

The notion of partitions can be slightly generalized. A \mydef{skew-partition} $\lambda/\mu$ is a pair of partitions $(\lambda, \mu)$ such that for all $i>0$, $\lambda_i \geq \mu_i$. Graphically, it is equivalent to that the Ferrers diagram of $\lambda$ covers totally that of $\mu$. We then define the \mydef{skew diagram} of the form $\lambda/\mu$ as the difference of the Ferrers diagrams of $\lambda$ and of $\mu$, \textit{i.e.} the Ferrers diagram of $\lambda$ without cells that also appear in that of $\mu$. Figure~\ref{fig:2:diagrams}(b) shows an example of a skew diagram. A \mydef{ribbon} of a Ferrers diagram $\lambda$ is a skew diagram of the form $\lambda/\mu$ for some $\mu$ that is connected and without any $2 \times 2$ cells. The \mydef{size} of a ribbon is the number of cells it contains. The \mydef{height} of a ribbon $\lambda / \mu$, denoted by $ht(\lambda/\mu)$, is the number of rows that the skew diagram of $\lambda/\mu$ occupies \emph{minus one}. Figure~\ref{fig:2:diagrams} shows an example of a ribbon of size $8$ and of height $3$. 

We can now define the main objects in the combinatorial interpretation of characters in the symmetric group that we will be using. We denote by $\epsilon$ the empty partition. A \mydef{ribbon tableau} of \mydef{shape} $\lambda$ and \mydef{type} $\mu$ is a sequence of partitions $(\lambda^{(0)} = \lambda, \lambda^{(1)}, \ldots \lambda^{\ell(\mu)}=\epsilon)$ such that for all $i$, the skew tableau $\lambda^{(i)} / \lambda^{(i+1)}$ is a ribbon of size $\mu_i$. It is easy to see that the shape and the type of ribbon tableau must be partitions of the same integer. The \mydef{sign} of a ribbon tableau $T = (\lambda^{(0)} = \lambda, \lambda^{(1)}, \ldots, \lambda^{k}=\epsilon)$, denoted by $\sgn(T)$, is defined by
\[ \sgn(T) = \prod_{i=0}^{k-1} (-1)^{ht(\lambda^{(i)} / \lambda^{(i+1)})}. \]
Figure~\ref{fig:2:diagrams} shows a ribbon tableau of shape $\lambda=(5,3,3,2)$ and type $\mu = (5,3,2,2,1)$, which has sign $-1$. The partition $\lambda^{(i)}$ is given by the diagram formed by ribbons with label strictly smaller than $i$. We can now state the Murnaghan-Nakayama rule that expresses characters in the symmetric group with ribbon tableaux.

\begin{thm}[Murnaghan-Nakayama rule] \label{thm:2:mn-rule}
For two partitions $\lambda, \mu$ of a natural number $n \geq 1$, we have
\[ \chi^{\lambda}_{\mu} = \sum_{\substack{T \; \mathrm{ribbon\;tableau} \\ \mathrm{shape}(T)=\lambda, \; \mathrm{type}(T)=\mu}} \sgn(T). \]
\end{thm}

\begin{figure}
  \centering
  \insertfigure[0.95]{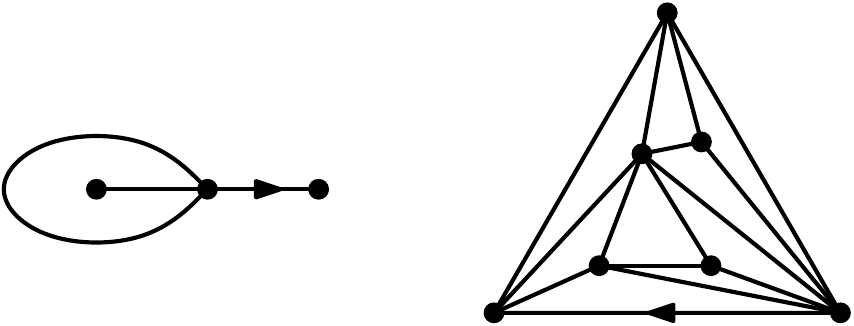}{9}
  \caption{Murnaghan-Nakayama rule for $\chi^{[5,4,3]}_{[3,3,2,2,1,1]}=-2$}
  \label{fig:2:mn-example}
\end{figure}

Figure~\ref{fig:2:mn-example} shows how to compute the character $\chi^{[5,4,3]}_{[3,3,2,2,1,1]}$ using the Murnaghan-Nakayama rule. In fact, the condition that the type $\mu$ of a ribbon tableau must be a partition is not necessary for the Murnaghan-Nakayama rule. We can relax the definition of ribbon tableaux to allow arbitrary finite sequence $\mu$ as the type. However, we can prove that the ordering of elements in $\mu$ does not affect the sum of the sign of all ribbon tableaux of type $\mu$ with the same shape. Therefore, here we fix the order of $\mu$ to be decreasing, which is equivalent to saying that $\mu$ is a partition. 

For more information on ribbon tableaux and the Murnaghan-Nakayama rule, readers are referred to \cite[Section~7.17]{Stanley:EC2} and \cite[Section~4.10]{sagan2001symmetric}.


\subsection{Counting factorizations using characters}

In the previous section, we have presented characters as coefficients in the change of basis between $(K_\lambda)_{\lambda \vdash n}$ and idempotents $(F_\theta)_{\theta \vdash n}$ in the center $Z(\mathbb{C}S_n)$ of the group algebra of $S_n$. This formalism allows us to easily express the number of factorization of the identity into permutations with given cycle types. Suppose that we want to find out the number $B(\lambda^\bullet, \lambda^\circ, \mu)$ of factorizations $\sigma_\bullet \sigma_\circ \phi = \id_n$ in $S_n$, where $\sigma_\bullet$, $\sigma_\circ$ and $\phi$ have cycle types $\lambda^\bullet$, $\lambda^\circ$, $\mu$ respectively. We notice that the factorizations we count here are simply rotation systems of bipartite maps, with transitivity requirement dropped. We can compute $B(\lambda^\bullet, \lambda^\circ, \mu)$ in $Z(\mathbb{C}S_n)$ in the following way, using the fact that all $F_\theta$ are idempotent and $F_\theta F_{\theta'} = 0$ if $\theta \neq \theta'$.

\begin{align} 
\begin{split}
B(\lambda^\bullet, \lambda^\circ, \mu) &= \sum_{\substack{\sigma_\bullet \sigma_\circ \phi = \id_n \\ \sigma_\bullet \in Cl(\lambda^\bullet), \sigma_\circ \in Cl(\lambda^\circ), \phi \in Cl(\mu)}} 1 = [K_{[1^n]}] K_{\lambda^\bullet} K_{\lambda^\circ} K_\mu \\
&= [K_{[1^n]}] \frac{(n!)^3}{z_{\lambda^\bullet} z_{\lambda^\circ} z_\mu} \sum_{\theta \vdash n} \frac{\chi^\theta_{\lambda^\bullet} \chi^\theta_{\lambda^\circ} \chi^\theta_\mu}{(f^\theta)^3} F_\theta = \frac{(n!)^2}{z_{\lambda^\bullet} z_{\lambda^\circ} z_\mu} \sum_{\theta \vdash n} \chi^\theta_{\lambda^\bullet} \chi^\theta_{\lambda^\circ} \chi^\theta_\mu (f^\theta)^{-1}.
\end{split}
\end{align}

Enumeration of factorizations using characters can be dated back to Frobenius, and the relation above is sometimes called the \emph{Frobenius formula}. The computation above can be easily generalized to factorizations involving more permutations. Let $C(\lambda^{(1)}, \ldots, \lambda^{(m)}, \mu)$ be the number of factorizations $\sigma_1 \sigma_2 \cdots \sigma_m \phi = \id_n$ in $S_n$ such that $\phi \in Cl(\mu)$ and $\sigma_i \in Cl(\lambda^{(i)})$ for all $i$. Such factorizations are said to be of $m$-constellations type. We have the following expression of $C(\lambda^{(1)}, \ldots, \lambda^{(m)}, \mu)$:

\begin{equation} \label{eq:2:const}
C(\lambda^{(1)}, \ldots, \lambda^{(m)}, \mu) = \frac{(n!)^m}{z_\mu \prod_{i=1}^m z_{\lambda^{(i)}}} \sum_{\theta \vdash n} (f^\theta)^{1-m} \chi^\theta_\mu \prod_{i=1}^m \chi^\theta_{\lambda^{(i)}}.
\end{equation}

We now look at two other factorization models we have mentioned in the previous chapter. A \mydef{transposition} is a permutation with only one cycle of length $2$ and all other cycles of length $1$. In other words, transpositions in $S_n$ are exactly elements in $Cl([2,1^{n-2}])$. We often  write down a transposition as its unique 2-cycle. A \mydef{transposition factorization} with $r$ transpositions in $S_n$ is a tuple $(\tau_1, \ldots, \tau_r, \phi)$ with $\phi$ an arbitrary permutation in $S_n$ and all $\tau_i$ transpositions, which satisfies
\[ \tau_1 \cdots \tau_r \phi = \id_n. \]
A transposition factorization $(\tau_1, \ldots, \tau_r, \phi)$ is \mydef{monotone} if, when we denote by $(a_i, b_i)$ with $a_i < b_i$ the unique 2-cycle of $\tau_i$, we have $b_1 \leq b_2 \leq \cdots \leq b_r$. Given a partition $\lambda \vdash n$, the \mydef{classical Hurwitz number} $H_r(\lambda)$ is the number of transposition factorizations $(\tau_1, \ldots, \tau_r, \phi)$ with $\phi \in Cl(\lambda)$ that are transitive, divided by $n!$. Similarly, we define the \mydef{monotone Hurwitz number} $\vec{H}_r(\lambda)$ with monotone transposition factorizations. Readers might be confused with the notation for monotone Hurwitz number, since there is no reason that elements in the same conjugacy class have exactly the same number of monotone transposition factorizations, but we will see later that it is indeed the case.

Although with seemingly different nature, transposition factorizations can in fact be put under a unified framework with factorizations of constellation type using the group algebra. The \mydef{Jucys-Murphy elements} $(J_k)_{1 \leq k \leq n}$ are sums of transpositions defined by
\[ J_1 = 0, J_k = (1,k) + (2,k) + \cdots + (k-1,k) \; \; \; \mathrm{for} \; \; \; k \geq 2. \]
Jucys-Murphy elements commute with each other, and they play an important role in the representation theory of the symmetric group (\textit{cf.} \cite{vershik2004new}). They can also be used to describe factorizations of the identity thanks to the following proposition. We define a product $\Pi_n(t)$ by the equation
\begin{equation} \label{eq:2:prod-pi}
\Pi_n(t) = \prod_{k=1}^n (\id_n + t J_k).
\end{equation}

\begin{prop} \label{prop:2:prod-pi}
Let $t$ be an indeterminate that commutes with elements in $\mathbb{C}S_n$, we have
\[ \Pi_n(t) = \sum_{\sigma \in S_n} t^{n-\#cycle(\sigma)} \sigma = \sum_{\lambda \vdash n} t^{n-\ell(\lambda)} K_\lambda. \]
\end{prop}
\begin{proof}
We proceed by induction on $n$. The case $n=1$ is trivial. Suppose that the identity holds for a certain $n$. For $\sigma \in S_{n+1}$, let $k = \sigma(n+1)$. There are two cases: either $k=n+1$ or not. In the first case, we can also see $\sigma$ as a permutation in $S_n$, and by passing to $S_{n+1}$ we gain one cycle $(n+1)$. In the second case, we want to find a transposition $\tau$ of the form $(k',n+1)$ with an unknown $k'$ such that $\sigma' = \sigma \tau$ satisfies $\sigma'(n+1)=n+1$, \textit{i.e.}, $\sigma'$ can be considered as a permutation in $S_n$. We have $k' = \tau(\sigma'(n+1)) = k$, therefore the choice of $\tau$ is unique. Furthermore, $\sigma$ and $\sigma'$ have the same number of cycles. We thus conclude that
\[ \Pi_{n+1}(t) = \sum_{\sigma' \in S_n} t^{n-\#cycle(\sigma')} \sigma' (\id_{n+1} + t J_{n+1}) = \sum_{\sigma \in S_{n+1}} t^{n+1-\#cycle(\sigma)} \sigma. \]
By induction, the identity holds for all $n$.
\end{proof}

It is interesting that, even though $J_k$ is not in $Z(\mathbb{C}S_n)$ in general, the related product $\Pi_n(t)$ is. Let $F^C_{m,k}(\lambda)$ be the number of factorizations $\sigma_1 \cdots \sigma_m \phi = \id_n$ in $S_n$ such that $\phi \in Cl(\lambda)$ and the total number of cycles in $\sigma_i$ for all $i$ is $k$. Similarly, let $F^{H}_r(\lambda)$ and $F^{\vec{H}}_r(\lambda)$ be respectively the number of (not necessarily transitive) general and monotone transposition factorizations $\tau_1 \cdots \tau_r \phi = \id_n$ in $S_n$ such that $\phi \in \lambda$. By Proposition~\ref{prop:2:prod-pi}, we can express $F^{C}_{m,k}(\lambda)$, $F^{H}_r(\lambda)$ and $F^{\vec{H}}_r(\lambda)$ using $\Pi_n(t)$. We start by $F^{C}_{m,k}(\lambda)$:
\begin{align*}
F^C_{m,k}(\lambda) &= [t^{mn-k}]\sum_{\substack{\sigma_1 \cdots \sigma_m = \phi^{-1} \\ \sigma_i \in Cl(\mu^{(i)}), \phi \in Cl(\lambda)}} \prod_{i=1}^m t^{n-\ell(\sigma_m)} \\
&= |Cl(\lambda)| \cdot [t^{mn-k} K_\lambda] \prod_{i=1}^m \left( \sum_{\lambda \vdash n} t^{n-\ell(\mu^{(i)})} K_{\mu^{(i)}} \right)^m \\
&= n! z_\lambda^{-1} [t^{mn-k}K_\lambda] \Pi_n(t)^m.
\end{align*}
For $F^{H}_r(\lambda)$, we observe that a sequence $\tau_1, \ldots, \tau_r$ of transpositions in $S_n$ can be seen as a multiset of elements in $Cl([2,1^{n-2}])$ with labels from $1$ to $r$. By the multiset construction of labeled combinatorial classes in Section~\ref{sec:2:class}, we have
\begin{align*}
F^H_r(\lambda) &= n! z_\lambda^{-1} [t^r K_\lambda] \exp \left( t \sum_{\tau \in Cl([2,1^{n-2}])} \tau \right) = n! z_\lambda^{-1} [t^r K_\lambda] \exp \left( t \sum_{k=1}^n J_k \right)  \\
&= n! z_\lambda^{-1} [t^r K_\lambda] \lim_{a \to \infty} \Pi_n(a^{-1}t)^{a}.
\end{align*}
The case for $F^{\vec{H}}_r(\lambda)$ is a bit simpler, since in this case the sequence of transpositions can be divided into $n$ segments according to the larger element in the only $2$-cycle:
\begin{align*}
F^{\vec{H}}_r(\lambda) &= [t^r] \sum_{r \geq 0} \sum_{\substack{(a_1, b_1) \cdots (a_r, b_r) = \phi^{-1} \\ \phi \in Cl(\lambda), b_1 \leq \cdots \leq b_r}} t^r = |Cl(\lambda)| \cdot [t^r K_\lambda] \prod_{i=1}^{m} \sum_{a\geq 0} t^a J_i^a \\
&= |Cl(\lambda)| \cdot [t^r K_\lambda] \prod_{i=1}^{m} \frac1{1-tJ_i} = n! z_\lambda^{-1} [t^r K_\lambda] \Pi_n(-t)^{-1}.
\end{align*}

The factor $n! z_\lambda^{-1}$ comes from the number of choices of $\phi$ from $Cl(\lambda)$ in each case. It is worth noting that, since the sum of monotone transposition sequences weighted by length can be expressed using $\Pi_n(t)$, which is an element of the center $Z(\mathbb{C}S_n)$, it is clear that permutations in the same conjugacy class share the same number of such factorizations.

Of course, these are not exactly the number of constellations or monotone Hurwitz numbers that we want, since the transitivity constraint is dropped. But in Section~\ref{sec:2:gen}, we will see how to add back that constraint by a simple manipulation of generating functions.

\section{Generating functions} \label{sec:2:gen}

In this section, we will discuss how to enumerate combinatorial objects using their generating function. Since the use of generating functions in enumeration is a vast and fruitful topic, we can only cover the basics here. The philosophy of generating functions is that, to enumerate a certain class of combinatorial objects constructed recursively, we only need to construct a power series whose coefficients are the numbers of such objects with different size, then translate the recursive construction of the objects into a functional equation of the power series, and finally use algebraic and analytic methods to obtain information about the wanted coefficients. There are several kinds of generating functions, and the study of how to extract enumerative information from these generating functions has grown into a huge and important field. Readers interested in a more complete treatment of generating functions are referred to the book \cite{flajolet}.

Let $\mathbb{K}$ be a field. We denote by $\mathbb{K}[x_1, x_2, \ldots, x_n]$ the \mydef{polynomial ring} with variables $x_1, x_2, \ldots, x_n$ which is spanned linearly by monomials $x_1^{k_1} x_2^{k_2} \cdots x_n^{k_n}$. Sometimes we need to deal with polynomials ``with infinitely many variables''. For an infinite sequence of variables $x_1, x_2, \ldots$, we consider the \emph{projective limit} of $(\mathbb{K}[x_1, x_2, \ldots, x_n])_{n > 1}$, denoted by $\mathbb{K}[x_1, x_2, \ldots]$:
\[ \mathbb{K}[x_1, x_2, \ldots] = \left\{ \left. (P_n)_{n \geq 1} \in \prod_{n \geq 1} \mathbb{K}[x_1, \ldots, x_n] \; \right| \; \forall n, \forall k, P_n = P_{n+k}(x_1, \ldots, x_n, 0,\ldots,0) \right\}. \]
Although scary at appearance, we can simply imagine a member of $\mathbb{K}[x_1,x_2,\ldots]$ as a (potentially infinite) sum of monomials containing finitely many variables.

We now set the playground of our generating functions. We first fix a \mydef{base ring} (in fact, a field in many cases) $\mathbb{K}$. We denote by $\mathbb{K}[[t]]$ the \mydef{formal power series ring} with variable $t$, which is formed by \mydef{formal power series} of the form
\[ \sum_{i \geq 0} c_i t^i, \;\;\; \mathrm{where} \;\;\; \forall i, c_i \in \mathbb{K}. \]
As a shorthand, we denote by $\mathbb{K}[[t_1,t_2]] = \mathbb{K}[[t_1]][[t_2]]$ the ring of power series in $t_2$ whose coefficients are power series in $t_1$. Furthermore, we extend this notation to an arbitrary finite number of variables. Other than polynomials and formal power series, we will also need some more complicated objects. We now require $\mathbb{K}$ to be a field, which is called the \mydef{base field}. We denote by $\mathbb{K}(x_1, x_2, \ldots, x_n)$ the \mydef{rational fraction field} with variables $x_1, x_2, \ldots, x_n$, which is formed by fractions $P/Q$ of polynomials $P, Q \in \mathbb{K}[x_1, \ldots, x_n]$ with $Q \neq 0$. We denote by $\mathbb{K}((t))$ the \mydef{Laurent series field} with variable $t$, which is formed by \mydef{Laurent series} of the form
\[ \sum_{i \geq -d} c_i t^i \;\;\; \mathrm{with} \;\;\; \forall i, c_i \in \mathbb{K} \]
where $d \in \mathbb{Z}$. In other words, we can allow a finite number of terms with negative powers in $t$ in the Laurent series. A Laurent series in $t$ with finitely many non-zero coefficients is also called a \mydef{Laurent polynomial} in $t$. In other words, a Laurent polynomial in $t$ is a polynomial in $t$ divided by some power $t^d$ of $t$. Finally, we denote by $\mathbb{K}((t^*))$ the \mydef{Puiseux series field} of variable $t$, which is the union of all Laurent series fields $\mathbb{K}((t^{1/d}))$ for all integers $d \geq 1$. In other words, we allow fractional powers in our Puiseux series (but with one common denominator). Puiseux series are important in the generating function method due to the following theorem (\textit{cf.} \cite[Chapter~VII.7.1]{flajolet}).

\begin{thm}[Newton-Puiseux] \label{thm:2:newton-puiseux}
For $\mathbb{K}$ an algebraically closed field with characteristic $0$, and $P(x,y) \in \mathbb{K}[x,y]$ of degree $k$ in $y$, the equation $P(x,y)=0$ has $k$ solutions (counted with multiplicity) in $\mathbb{K}((x^*))$.
\end{thm}


\begin{table}
  \centering
  \begin{tabular}{ccc}
    \toprule
    Notation & Name & Elements \\
    \midrule
    $\mathbb{K}[x_1, x_2, \ldots, x_n]$ & Polynomial ring & Linear combinations of $x_1^{k_1} \cdots x_n^{k_n}$ \\
    \midrule
    $\mathbb{K}(x_1, x_2, \ldots, x_n)$ & Rational fraction field & $P/Q$, where $P,Q \in \mathbb{K}[x_1, x_2, \ldots, x_n]$ \\
    \midrule
    $\mathbb{K}[[t]]$ & Formal power series ring & $\displaystyle{\sum_{n \geq 0} c_n t^n}$ \\
    \midrule
    $\mathbb{K}((t))$ & Laurent series field & $\displaystyle{\sum_{n \geq -k} c_n t^n}$ for any $k \geq 0$ \\
    \midrule
    $\mathbb{K}((t^*))$ & Puiseux series field & $\displaystyle{\sum_{n \geq -k} c_n t^{n/d}}$ for any $k \geq 0, d \geq 1$ \\
    \bottomrule
  \end{tabular}
  \caption{Some rings and fields used in generating function method}
  \label{tab:2:series}
\end{table}

Table~\ref{tab:2:series} offers a comparison of all the objects mentioned above. In this thesis, we will use formal power series to write generating functions of combinatorial objects. By abuse of notation, when a generating function converges in a neighborhood of $0$, we will identify it with the analytic function to which it converges. We also introduce the following useful notation: for $F \in \mathbb{K}[[t]]$ or $F \in \mathbb{K}((t))$, we denote by $[t^n]F$ the coefficient of $t^n$ in $F$; for $F \in \mathbb{K}((t))$, we denote by $[t^{\geq 0}]F$ the \emph{positive part} of $F$, that is, the part with non-negative powers in $t$. We also notice the following inclusions: $\mathbb{K}[[t]] \subset \mathbb{K}((t)) \subset \mathbb{K}((t^*))$ and $\mathbb{K}(t) \subset \mathbb{K}((t))$.

\subsection{Combinatorial classes and their construction} \label{sec:2:class}

We now give a brief exposition of how to write the generating function of a class of combinatorial objects, and how to extract a functional equation from a recursive decomposition (or a \emph{combinatorial specification} as in \cite{flajolet}).

Let $\mathcal{C}$ be a set. A \mydef{statistic} is a function $st: \mathcal{C} \to \mathbb{N}$. We say that $\mathcal{C}$ equipped with a statistic (called the \emph{size}) $| \cdot |_{\mathcal{C}}$ is a \mydef{combinatorial class} (or simply a \mydef{class}) if, for all $n \in \mathbb{N}$, the set $\mathcal{C}_i$ of elements of size $n$ is finite. When writing the size statistics, we often leave out the class that it belongs to and write $| \cdot |$ when there is no ambiguity. There are two fundamental combinatorial classes: $\mathcal{E}$, containing one element of size $0$, and $\mathcal{Z}$, containing one element of size $1$. We define a \mydef{variable scheme} as a set of pairs $(st_i, x_i)$ of a statistic and a variable. For a class $\mathcal{C}$, given a variable scheme $\{ (st_1, x_1), \ldots, (st_k, x_k)\}$ with statistics of $\mathcal{C}$ and an extra indeterminate $t$ for the size statistic, we can define two types of formal power series as generating functions of $\mathcal{C}$:
\begin{itemize}
\item \textbf{Ordinary generating function:} (or \mydef{OGF} for short)
\[ F_{\mathcal{C},\ord}(t,x_1,\ldots,x_k) = \sum_{c \in \mathcal{C}} t^{|c|} x_1^{st_1(c)} \cdots x_k^{st_k(c)}, \]
\item \textbf{Exponential generating function:} (or \mydef{EGF} for short)
\[ F_{\mathcal{C},\exp}(t,x_1,\ldots,x_k) = \sum_{c \in \mathcal{C}} \frac{t^{|c|}}{|c|!} x_1^{st_1(c)} \cdots x_k^{st_k(c)}. \]
\end{itemize}
We say that the variable $x_i$ \emph{marks} the statistics $st_i$. As a shorthand, when there is no ambiguity, we write $\vecvar{x}$ instead of $x_1,\ldots,x_k$ in arguments of a generating function. We will also omit the index $ord$ and $exp$ when there is no ambiguity. 

The reason why we have two types of generating functions is that the number $\# \mathcal{C}_n$ of objects of size $n$ in a class $\mathcal{C}$ can grow at different speeds, and when it exceeds exponential growth, the OGF of $\mathcal{C}$ will cease to be convergent in any neighborhood of $t=0$. In this case, the OGF cannot be the series expansion of an analytic function, and we are denied the use of powerful analytic methods (however, it is still well-defined as a formal power series). This situation happens mostly with \emph{labeled classes}, where the building parts counted by the size of the object receive distinct labels. For instance, the class of permutations is a labeled class. We usually use EGF for labeled classes.

By combining combinatorial classes, we can obtain new classes, whose generating functions can be expressed in the generating functions of the original classes in some cases. Let $\mathcal{A}, \mathcal{B}$ be two classes. We can construct the following new classes.
\begin{itemize}
\item \textbf{Disjoint union}: Denoted by $\mathcal{A} + \mathcal{B}$, the disjoint union of $\mathcal{A}$ and $\mathcal{B}$ is a combinatorial class with the size statistics that gives the size in $\mathcal{A}$ for objects in $\mathcal{A}$ and the size in $\mathcal{B}$ for objects in $\mathcal{B}$. 
\item \textbf{Cartesian product}: Denoted by $\mathcal{A} \cdot \mathcal{B}$, the Cartesian product of $\mathcal{A}$ and $\mathcal{B}$ is a combinatorial class with the size statistics $| \cdot |_{\mathcal{A} \cdot \mathcal{B}}$ given by $|(a,b)|_{\mathcal{A} \cdot \mathcal{B}} = |a|_{\mathcal{A}} + |b|_{\mathcal{B}}$. We also use $\mathcal{A}^2$ as a shorthand of $\mathcal{A} \cdot \mathcal{A}$, and similarly $\mathcal{A}^k$ with $k \geq 2$ for successive Cartesian products of the same class $\mathcal{A}$. For labeled classes, there should be a relabeling of building blocks that preserves the order of labels in $a,b$. Due to relabeling, the Cartesian product $(a,b)$ stands for a set of objects in $\mathcal{A} \cdot \mathcal{B}$ with the same underlying combinatorial structure but different labels.
\item \textbf{Sequence construction}: Denoted by $\textsc{Seq}(\mathcal{A})$, the sequence class of $\mathcal{A}$ is the set $\{ (a_1, a_2, \ldots, a_k) \mid k \in \mathbb{N}, \forall i, a_i \in \mathcal{A} \}$ of sequences with elements in $\mathcal{A}$, with the size given by $|(a_1, \ldots, a_k)| = \sum_{i=1}^k |a_i|$. In this construction, for labeled classes we also have to consider relabeling as in Cartesian product. We can also consider $\textsc{Seq}(\mathcal{A})$ as a short hand of $\mathcal{E} + \sum_{i \geq 1} \mathcal{A}^i$.
\item \textbf{Multiset construction}: Denoted by $\textsc{Mset}(\mathcal{A})$, the multiset class of $\mathcal{A}$ is the set of multisets with elements in $\mathcal{A}$. For labeled classes, $\textsc{Mset}(\mathcal{A})$ can be seen as a shorthand for $\mathcal{E} + \sum_{i \geq 1} \frac1{i!} \mathcal{A}^i$, where the extra factor means that we are not interested in the order of components (but they distinguish themselves with labels of their building blocks).
\item \textbf{Pointing construction}: Denoted by $\mathcal{A}^{\bullet}$, the pointed class of $\mathcal{A}$ is the set $\cup_{i \geq 1} \mathcal{A}_i \times \{1, 2, \ldots, i \}$. Elements of this class may be seen as an object in $\mathcal{A}$ with one of its building blocks (counted by the size statistics) pointed.
\end{itemize}

To avoid confusion on labeled classes, we will now see an example of relabeling. Let $\mathcal{A} = \textsc{Seq}(\mathcal{Z})$ be the labeled class of sequences of points. We can imagine an element $a \in \mathcal{A}$ of size $n$ as a sequence of distinct labels from $1$ to $n$. Now we consider the labeled class $\mathcal{B} = \mathcal{A}^2$. Let $b$ be a member of $\mathcal{B}$, we can see $b$ as two sequences of labels from $1$ to $m+n$, the first of length $m$, the second of length $n$. The order of labels in each sequence gives two elements $a_1, a_2 \in \mathcal{A}$ of size $m$ and $n$ respectively. However, several elements in $\mathcal{B}$ give the same couple $(a_1, a_2)$ with this operation, each corresponds to the subset of labels used in the first sequence. Therefore, each couple $(a_1, a_2)$ of size $m$ and $n$ leads to $\binom{m+n}{m}$ elements in $\mathcal{B}$. This extra factor is due to the relabeling process.

A statistic $st$ defined on all the powers $\mathcal{C}^k$ of a combinatorial class $\mathcal{C}$ is called \mydef{additive} if on the successive Cartesian product $\mathcal{C}^k$, it is defined by $st((c_1,\ldots,c_k)) = \sum_{i=1}^k st(c_i)$. The size statistic is additive by definition. Many useful statistics are additive, for example the number of faces of degree $k$ and the number of vertices in a map, and they can appear in many cases. For instance, in the Tutte equation of planar maps (see Example~3 in Section~\ref{sec:2:example}), there is a case where the root is a bridge. Maps in this case can be seen as a pair of planar maps linked by an edge, and in this case the number of vertices is an additive statistics. There are also statistics that are not additive, especially statistics that are marked by ``catalytic variables'' that we will introduce later in examples. We now consider generating functions defined with extra statistics that are all additive. Let $\mathcal{A}$, $\mathcal{B}$ be two classes with common additive statistics $st_1, \ldots, st_k$. Table~\ref{tab:2:construction} gives the generating functions of different constructions based on $\mathcal{A}$ and $\mathcal{B}$ under the common variable scheme $(st_1, x_1), \ldots, (st_k, x_k)$.

\begin{table}
  \centering
  \begin{tabular}{ccc}
    \toprule
    Construction & Unlabeled(OGF) & Labeled(EGF) \\
    \midrule
    $\mathcal{A} + \mathcal{B}$ & $F_{\mathcal{A},\ord} + F_{\mathcal{B},\ord}$ & $F_{\mathcal{A},\exp} + F_{\mathcal{B},\exp}$ \\
    \midrule
    $\mathcal{A} \cdot \mathcal{B}$ & $F_{\mathcal{A},\ord} F_{\mathcal{B},\ord}$ & $F_{\mathcal{A},\exp} F_{\mathcal{B},\exp}$ \\
    \midrule
    $\textsc{Seq}(\mathcal{A})$ & $\displaystyle{\frac1{1-F_{\mathcal{A},\ord}}}$ & $\displaystyle{\frac1{1-F_{\mathcal{A},\exp}}}$ \\
    \midrule
    $\textsc{Mset}(\mathcal{A})$ & Messy and not needed & $\displaystyle{\exp(F_{\mathcal{A},\exp})}$ \\
    \midrule
    $\mathcal{A}^{\bullet}$ & $\displaystyle{\frac{t\partial}{\partial t} F_{\mathcal{A},\ord}}$ & $\displaystyle{\frac{t\partial}{\partial t} F_{\mathcal{A},\exp}}$ \\
    \midrule
  \end{tabular}
  \caption{Translation from constructions to generating functions}
  \label{tab:2:construction}
\end{table}

We will not give detailed proofs of all the generating functions here. To see the general proof idea, we will give a short proof for expressions of the OGF and the EGF of $\mathcal{A} \cdot \mathcal{B}$ and the EGF of $\textsc{Mset}(\mathcal{A})$. For the OGF of $\mathcal{A} \cdot \mathcal{B}$, we have
\begin{align*}
F_{\mathcal{A} \cdot \mathcal{B},\ord}(t,\vecvar{x}) &= \sum_{(a,b) \in \mathcal{A} \cdot \mathcal{B}} t^{|(a,b)|} x_1^{st_1((a,b))} \cdots x_k^{st_k((a,b))} \\
\quad &= \sum_{a \in \mathcal{A}} t^{|a|} x_1^{st_1(a)} \cdots x_k^{st_k(a)} \sum_{b \in \mathcal{B}} t^{|b|} x_1^{st_1(b)} \cdots x_k^{st_k(b)} = F_{\mathcal{A},\ord}(t,\vecvar{x}) F_{\mathcal{A},\ord}(t,\vecvar{x}).
\end{align*}
For the EGF of $\mathcal{A} \cdot \mathcal{B}$, the computation is similar, but we have to take care of the binomial factor stemming from the relabeling:
\begin{align*}
F_{\mathcal{A} \cdot \mathcal{B},\exp}(t,\vecvar{x}) &= \sum_{a \in \mathcal{A}} \sum_{b \in \mathcal{B}} \binom{|(a,b)|}{|a|} \frac{t^{|(a,b)|}}{|(a,b)|!} x_1^{st_1((a,b))} \cdots x_k^{st_k((a,b))} \\
\quad &= \sum_{a \in \mathcal{A}} \sum_{b \in \mathcal{B}} \frac{t^{|a|}}{|a|!} x_1^{st_1(a)} \cdots x_k^{st_k(a)} \frac{t^{|b|}}{|b|!} x_1^{st_1(b)} \cdots x_k^{st_k(b)} = F_{\mathcal{A},\exp}(t,\vecvar{x}) F_{\mathcal{A},\exp}(t,\vecvar{x}).
\end{align*}
For the EGF of $\textsc{Mset}(\mathcal{A})$, we notice that its elements are now multisets of elements in $\mathcal{A}$ that are all distinguishable thanks to labels. In this case, a multiset of size $k$ corresponds to $k!$ ordered tuples of size $n$, and we have
\[ F_{\textsc{Mset}(\mathcal{A}),\exp} = \sum_{k \geq 0} \frac{F_{\mathcal{A},\exp}^k}{k!} = \exp(F_{\mathcal{A},\exp}). \]

As a further remark, with additive statistics that mark the number of some substructures in a combinatorial object such as triangular faces in a map, it is also sensible to consider pointed classes where one of the marked substructures is pointed. In this case, the generating function can be obtained by differentiating the appropriate variable.

We will now see how we can express the generating functions of transitive factorizations of the identity in the symmetric group using the construction $\textsc{Mset}$. Let us take as example factorizations of the form $\sigma_\bullet \sigma_\circ \phi = \id$. When transitivity is imposed, these factorizations become rotation systems of bipartite maps. For a factorization tuple $s = (\sigma_\bullet, \sigma_\circ, \phi)$, let $O(i)$ be the orbit of the integer $i$ in the group generated by $\sigma_\bullet$, $\sigma_\circ$ and $\phi$. The restriction $s|_{O(i)} = (\sigma_\bullet|_{O(i)}, \sigma_\circ|_{O(i)}, \phi|_{O(i)})$ of $s$ in the orbit $O(i)$ is still a factorization of the identity, but in $S_{|O(i)|}$ and is now transitive. By decomposing $s$ along all orbits, we can see $s$ as a multiset of transitive factorizations with relabeling. Let $\mathcal{A}$ be the class of such factorizations, and $\mathcal{B}$ the class of rotation systems of bipartite maps. The size of a factorization tuple $s$ in $S_n$ is $n$. By the construction $\textsc{Mset}$, we know that the EGFs of $\mathcal{A}$ and $\mathcal{B}$ satisfy $F_{\mathcal{A},\exp} = \exp(F_{\mathcal{B},\exp})$. This relation also applies to other types of factorizations and rotation systems. Therefore, to obtain the generating function of rotation systems, we only need to take the logarithm of the generating function of factorizations, and we say colloquially that taking the logarithm ``enforces'' transitivity.

These constructions can be used to describe \emph{recursive decompositions} (also called \emph{combinatorial specifications} in \cite{flajolet}) of combinatorial classes, which can then be translated into functional equations on the corresponding generating functions. If we manage to solve the functional equation for an exact expression of the generating function, we can try to extract its coefficients to obtain an explicit expression of the number of elements of size $n$ in the combinatorial class, possibly refined by various statistics. If the generating function satisfies some analytic condition, we can also extract the asymptotic behavior of its coefficients using analytic methods by looking at singularities of the generating function.

\subsection{Analytic method for asymptotics} \label{sec:2:analytic}

In the following introduction of the analytic method, we assume that readers have a general knowledge of complex analysis. Readers are referred to \cite[Chapter~IV]{flajolet} for details.

On the complex plane, a \mydef{domain} $D$ is a connected open region, and we denote by $\partial D$ its boundary. For a function $f$ defined over a domain $D$, we say that $f$ is \mydef{analytic} in $D$ if for any $z \in D$, there is an open set $U_z$ containing $z$ such that $f$ is complex differentiable in $U_z$, or equivalently $f$ has a Taylor expansion at $z$ that coincides with itself in $U_z$. Let $z_0$ be a point on the boundary $\partial D$, we say that $f$ is \mydef{analytically continuable} at $z_0$ if there exists an open set $U$ containing $z_0$ and another function $f^*$ that is analytic in $U$ such that $f(z)=f^*(z)$ in $U \cap D$. In fact, given the domain $D$, there is a unique function $f^*$ that satisfies all conditions. In this case, we say that $f$ can be continued to $U$, and we can define the function $f$ to be identical to $f^*$ in $U$. This process is called the \mydef{analytic continuation} of $f$ over $D \cap U$. A \mydef{singular point} (or simply \mydef{singularity}) of $f$ is a point $z_0 \in \partial D$ where $f$ is not analytically continuable.

Let $f(t)$ be a formal power series in $t$ such that $f(t)$ is analytic in an open disk $D(r)$ of radius $r>0$ centered the origin $t=0$. We can extend the domain of definition of $f(t)$ to disks with larger and larger radius by analytic continuation, until we hit the first singularity. Let $D(R)$ be the largest disk centered at $t=0$ such that $f(t)$ is analytic in $D(R)$. The radius $R$ is called the \mydef{convergence radius} of $f(t)$, and there is at least one singularity of modulus $R$. For generating functions we considered in combinatorial enumeration, there is a theorem called Pringsheim's theorem \cite[Theorem~IV.5]{flajolet} that locates at least one of these singularities with minimal modulus.

\begin{thm}[Pringsheim's theorem]
For a formal power series $f(t)$ with non-negative coefficients, if $f(t)$ is analytic in some open disk containing $t=0$ and with finite convergence radius $R$, then $t=R$ is a singularity of $f(t)$.
\end{thm}

However, this theorem does not forbid singularities to appear elsewhere. When analytically continued, a formal power series with finite convergence radius may encounter singularities on and outside its disk of convergence. But as we will see later, singularities outside the disk of convergence have no influence on the asymptotic behavior of the coefficients. Meanwhile, we need to worry about other singularities with the same modulus $R$, especially in the case of coefficient periodicity, \textit{e.g.} $[t^n]f(t) \neq 0$ only when $n$ is in some congruence classes. Such singularities whose modulus is equal to the convergence radius of the function are called \emph{dominant singularities}. If we impose stronger analytic conditions, the analysis of these dominant singularities will tell us the asymptotic behavior of the coefficients. We now describe these conditions in the case of a unique dominant singularity.

\begin{defn} \label{def:2:delta-domain}
For $0 < \rho < R$ and $\theta \in (0,\pi/2)$, the open set
\[ \Delta(\rho,R,\theta) = \{ z \in \mathbb{C} \mid |z| < R, |\operatorname{arg}(z-\rho)| > \theta \} \]
is called a \mydef{$\Delta$-domain}. A power series $f(t)$ is called \mydef{$\Delta$-analytic} if all its coefficients are non-negative and it is analytic in some $\Delta$-domain with $\rho$ its real dominant singularity assured by Pringsheim's theorem.
\end{defn}

Many generating functions we encounter in combinatorics are $\Delta$-analytic. For instance, any rational fraction in $t$ that does not diverge at $t=0$, satisfies the conditions in Pringsheim's theorem and has only one dominant singularity is $\Delta$-analytic. We have the following transfer theorem that determines the asymptotic behavior of coefficients in a $\Delta$-analytic power series $f$ from the behavior of $f$ near the dominant singularity (see \cite[Chapter~VI]{flajolet} for a more general result).

\begin{thm}[Transfer theorem]\label{thm:2:transfer}
For a power series $f(t)$ that is $\Delta$-analytic with real dominant singularity $\rho$, if there exist $\alpha, \beta \in \mathbb{R} \setminus \mathbb{Z}_{\leq 0}$ with $\alpha > \beta$ such that
\[ f(t) = c (1-t/\rho)^{-\alpha} + O((1-t/\rho)^{-\beta}) \quad \mathrm{when} \quad t \to \rho \quad \mathrm{in} \quad \Delta(\rho,R,\theta),\]
then we have
\[ [t^n] f(t) = \frac{c}{\Gamma(\alpha)} n^{\alpha - 1} \rho^{-n} + O(n^{\beta - 1} \rho^{-n}). \]
Here, $\Gamma(x)$ is the gamma function $\Gamma(x) = \int_0^\infty z^{x-1} e^{-z} dz$, and we have $\Gamma(k) = (k-1)!$ for any integer $k>0$.

If we have instead
\[ f(t) = - c \ln(1-t/\rho) + O((1-t/\rho)^{-\beta}) \]
for $\beta \in \mathbb{R}_- \setminus \mathbb{Z}_{\leq 0}$, then we have
\[ [t^n] f(t) = c n^{-1} \rho^{-n} + O(n^{\beta - 1} \rho^{-n}).\]
\end{thm}

The $\Delta$-analytic functions discussed above are closed under differentiation and integration. We can thus also apply the transfer theorem to derivatives and primitives of a $\Delta$-analytic function with certain singularities.

\begin{thm}[Theorem~VI.8 in \cite{flajolet}, Lemma~2.3 in \cite{fang-graz}]
For a power series $f(t)$ that is $\Delta$-analytic with real dominant singularity $\rho$, suppose that for some $\alpha > \beta$,
\[ f(t) = c (1-t/\rho)^{-\alpha} + O((1-t/\rho)^{-\beta}) \quad \mathrm{when} \quad t \to \rho \quad \mathrm{in} \quad \Delta(\rho,R,\theta). \]
Then $f'(t)$ is also $\Delta$-analytic, with an expansion near $\rho$ coming from term-by-term differentiation as
\[ f'(t) = c \alpha \rho^{-1} (1-t/\rho)^{-\alpha-1} + O((1-t/\rho)^{-\beta-1}) \quad \mathrm{when} \quad t \to \rho \quad \mathrm{in} \quad \Delta(\rho,R,\theta).\]
Similarly, every primitive $F(t)$ of $f(t)$ is also $\Delta$-analytic, with an expansion near $\rho$ coming from term-by-term integration.
\end{thm}

All these theorems apply to $\Delta$-analytic functions with one dominant singularity. When there are multiple dominant singularities, for the asymptotic behavior of coefficients, we only need to add up the contributions of all dominant singularities. For differentiation and integration, we can also treat each singularity separately.

To further extend the applicability of singularity analysis to $\Delta$-analytic functions, we introduce the following partial order on formal power series. For two formal power series $f(t), g(t)$ with non-negative coefficients, we say that $f(t)$ is \mydef{coefficient-wise smaller} than $g(t)$ (denoted by $f(t) \preceq g(t)$) if for all $n \in \mathbb{N}$, we have $[t^n]f(t) \leq [t^n] g(t)$. It is easy to see that the order $\preceq$ is stable by addition, multiplication by any formal power series and differentiation. Given a formal power series $f(t)$ with non-negative coefficients, if there exist two $\Delta$-analytic power series with non-negative coefficients $f_-(t)$, $f_+(t)$ with the same dominant singularities and the same asymptotic behavior $g(t) + O(h(t))$ near the real dominant singularity, and that $f_-(t) \preceq f(t) \preceq f_+(t)$, then we say that $f(t)$ is \mydef{congruent} to $g(t) + O(h(t))$, denoted by $f(t) \cong g(t) + O(h(t))$, and $f(t)$ has the same asymptotic behavior $g(t)+O(h(t))$ as $f_-(t)$ and $f_+(t)$ in this case. The relation $\simeq$ is also stable by addition, multiplication and differentiation. As an example, we take 
\[ f_-(t) = (1-t)^{-2}, f_+(t) = (1-t)^{-2} + (1-t)^{-1}, f(t) = \sum_{n\geq 0} (n+1+|\sin(n)|)t^n. \]
In this case, we have $f_-(t) \preceq f(t) \preceq f_+(t)$, and both $f_-(t)$ and $f_+(t)$ has the asymptotic behavior $(1-t)^{-2} + O((1-t)^{-1})$ near $t=1$. Therefore, although the $\Delta$-analyticity of $f(t)$ is not clear from its strange definition, we still have $f(t) \simeq (1-t)^{-2} + O((1-t)^{-1})$, which is also the asymptotic behavior of $f(t)$ near $t=1$. Using this notion of congruence, we can sometimes obtain asymptotic results without proving $\Delta$-analyticity.

\subsection{Resolution of functional equations in examples} \label{sec:2:example}

We have now equipped ourselves with some essential elements of how to manipulate generating functions. It is time to see how they work in practice. We will now see three examples of application of generating functions to combinatorial enumeration: plane trees, Dyck paths and planar maps. In the course, we will see how to write functional equations, how to solve them and how to extract both exact and asymptotic enumeration information. We will also see two resolution methods that are widely used in the study of maps: the \emph{kernel method} and the \emph{quadratic method}.

\paragraph{Example 1: plane trees} ~\\

A \mydef{plane tree} is a planar map with only one face, which is the outer face. As a consequence, the underlying graph of a plane tree has no cycle. The size of a plane tree is the number of edges it contains. By Euler's formula, a plane tree of size $n$ has $n+1$ vertices. Here we allow the \emph{empty tree}, which has no edge but only a vertex. We often draw a plane tree with its vertices spread over layers, where the root vertex occupies the highest layer $0$, and a vertex at distance $k$ to the root vertex stays on layer $k$. For a vertex $u$ on layer $k$ that is adjacent to another vertex $v$ on layer $k+1$, we say that $u$ is the \mydef{parent} of $v$, and $v$ is a \mydef{child} of $u$. In this illustration, we always put the root edge as right-most edge among all edges between layers $0$ and $1$. Figure~\ref{fig:2:tree-decomp}(a) gives an example of a plane tree. Since the root is fixed by edge ordering, in figures we often omit the marking on the root edge. We denote by $\mathcal{T}$ the class of plane trees, and $T(t)$ its OGF. The first few terms of $T(t)$ are
\[ T(t) = 1 + t + 2t^2 + 5t^3 + 14 t^4 + 42 t^5 + \cdots. \]

\begin{figure}
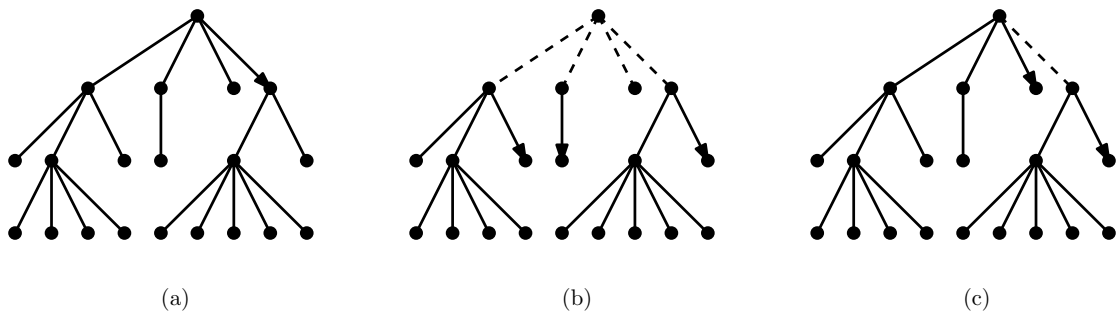

  \centering
  \insertfigure[0.85]{ch2-fig.pdf}{4}
  \caption{A plane tree and its two decompositions}
  \label{fig:2:tree-decomp}
\end{figure}

Plane trees can be decomposed into smaller plane trees in more than one way. Each decomposition gives a combinatorial specification of plane trees. The first way to decompose a plane tree is to delete the root vertex and its adjacent edges, then for each connected component, we choose its only vertex on layer $1$ to be the root vertex and the right-most edge between layer $1$ and layer $2$ (if any) to be the root edge. If the original plane tree is not empty, we obtain a sequence of plane trees. Conversely, given a sequence of plane trees, we can reverse the procedure to construct a non-empty plane tree. Figure~\ref{fig:2:tree-decomp}(b) gives an example of such a decomposition. We thus have the following specification:
\[ \mathcal{T} = \mathcal{E} + \mathcal{Z} \textsc{Seq}(\mathcal{T}). \]
The second decomposition can be seen as a Tutte decomposition. For a non-empty plane tree with root vertex $u$ and root edge $e = \{u,v\}$, we delete $e$ and obtain two connected components. For the component containing $u$, we choose the right-most edge adjacent to $u$ (if any) to be the new root edge. For the component containing $v$, we choose $v$ to be the root vertex and the right-most edge adjacent to $v$ (if any) as root edge. We thus obtain a pair of plane trees, and this is clearly a bijection between a non-empty plane tree and a pair of plane trees. Figure~\ref{fig:2:tree-decomp}(c) illustrates an example. We now have another specification:
\[ \mathcal{T} = \mathcal{E} + \mathcal{Z} \mathcal{T}^2. \]
The two specifications give rise to the same functional equation for $T$:
\begin{equation}\label{eq:2:tree} 
tT^2 - T + 1 = 0.
\end{equation}
Although this quadratic equation has two solutions, only one of them is a power series in $t$:
\[ T(t) = \frac1{2t}\left( 1 - (1-4t)^{1/2} \right). \]
Using the generalized binomial theorem
\[ (1+t)^\alpha = \sum_{k \geq 0} t^k \binom{\alpha}{k} = \sum_{k \geq 0} \frac{t^k}{k!} \prod_{i=0}^{k-1} (\alpha - i),\]
we obtain the following expression of coefficients in $T(t)$:
\[ [t^n]T(t) = \frac1{2n+1} \binom{2n+1}{n} = \frac1{n+1}\binom{2n}{n}. \]
These coefficients are also called the \mydef{Catalan numbers}, and they count many combinatorial objects, such as non-crossing partitions, binary trees and stack-sortable permutations. 

We can also obtain the asymptotic behavior of coefficients by applying the transfer theorem. It seems that $T(t)$ has a singularity at $t=0$, but it is in fact removable, and $T(t)$ is indeed a power series. The dominant singularity of $T(t)$ occurs at $t=1/4$, and $T(t) = 2 - 2(1-4t)^{1/2} + 2(1-4t) + O((1-4t)^{3/2})$ near $t=1/4$. By applying Theorem~\ref{thm:2:transfer} to $T-2-2(1-4t)$, we obtain
\[ [t^n]T(t) = \frac1{\pi^{1/2}} n^{-3/2} 4^n + O(n^{-5/2} 4^n). \]
This is because $\Gamma(-1/2) = -2\pi^{1/2}$. We also observe that, although poles of a function are always singularities, not all singularities are poles. For instance, $t=1/4$ is a singularity of $T(t)$ but not a pole, since $T(1/4) = 2$ is a finite value. This observation will play a crucial role in later chapters when we need to find the singularity of formal power series of the form $f(g(t))$, which are compositions of other formal power series.

\paragraph{Example 2: Dyck paths, catalytic variables and the kernel method} ~\\

We now consider walks on $\mathbb{Z}^2$, starting from $(0,0)$ and consisting of \mydef{up-steps} $u = (1,1)$ and \mydef{down-steps} $d = (1,-1)$. A \mydef{positive path} of size $n$ is a walk $(p_i)_{1 \leq i \leq n}$ of length $n$ with $u,d$ steps that always stays in the upper plane ($y \geq 0$). A \mydef{Dyck path} is a positive path that ends on the $x$-axis ($y=0$). Since the height on which a positive path ends has the same parity as its length, we know that a Dyck path is always of even size. We denote by $\mathcal{P}$ the class of positive paths and by $\mathcal{D}$ the class of Dyck paths. In $\mathcal{P}$, we define the \mydef{finishing height} $h(P)$ of a positive path $P$ as the $y$-coordinate at the end of the path. Figure~\ref{fig:2:dyck} gives an example of a positive path with finishing height $2$ and a Dyck path, both of length $10$. We observe that Dyck paths are exactly positive paths with finishing height $0$. We denote by $P(t,x)$ the OGF of positive path with variable scheme $(h,x)$, and $D(t)$ the OGF of Dyck path. We have $D(t) = P(t,0)$. 

\begin{figure}[!htbp]
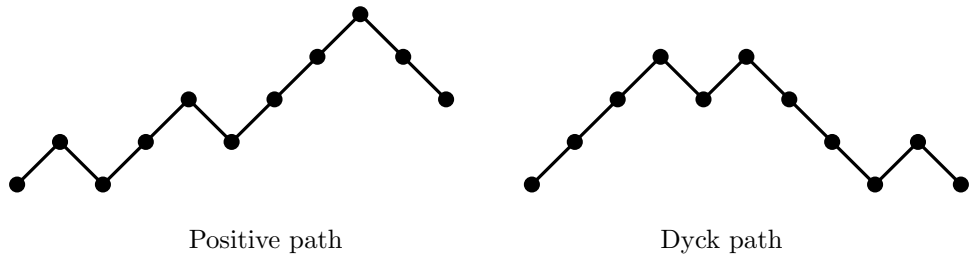

  \centering
  \insertfigure{ch2-fig.pdf}{5}
  \caption{Examples of a positive path and a Dyck path}
  \label{fig:2:dyck}
\end{figure}

We now want to count the number of Dyck paths with given size. Although it is possible to decompose a Dyck path into smaller Dyck paths, which will give the same functional equations as those in the previous example, here we choose another type of decomposition. We will now try to write a functional equation for $P(t,x)$ counting positive paths, which are more general than Dyck paths, then solve for $D(t)$ as a special case.

Given a positive path $P$, we consider the positive path $P'$ obtained by adding one step to the end of $P$. There are two possibilities: adding an up-step or a down-step. Adding an up-step always results in a positive path. The same works for adding a down-step to a positive path, except when it is a Dyck path. Every non-empty positive path can be obtained in this way. We thus have the following specification for positive paths:
\[ \mathcal{P} + d \mathcal{D} = \mathcal{E} + u \mathcal{P} + d \mathcal{P}. \]
Here, $u$ (resp. $d$) corresponds to a class containing only the up-step (resp. down-step), with OGF $tx$ (resp. $tx^{-1}$). We thus have the following functional equation:
\[
P(t,x) + tx^{-1}D(t) = 1 + t(x+x^{-1})P(t,x).
\]
A rewriting gives
\begin{equation} \label{eq:2:dyck}
(tx^2 - x + t)P(t,x) + x = tD(t).
\end{equation}
We observe that, if $x(t)$ is a formal power series that cancels the factor $(tx^2-x+t)$ by substitution, we will have immediately $tD(t) = x(t)$, which gives a solution of $D(t)$. We thus only need to search for such a function $x(t)$ satisfying that $t^{-1}x(t)$ is also a formal power series. Indeed, the equation $tx^2+t-1=0$ has two solutions in $x$, but only one of them satisfies our condition:
\[ x(t) = \frac{1 - (1 - 4t^2)^{1/2}}{2t}. \]
We thus find the OGF of Dyck paths:
\[ D(t) = \frac{1-(1-4t^2)^{1/2}}{2t^2}.\]
A brief examination shows that $D(t)$ is indeed a power series. As a bonus, by substituting back to \eqref{eq:2:dyck}, we get an expression of the OGF for the more general class of positive paths:
\[ P(t,x) = \frac{1-(1-4t^2)^{1/2}-2tx}{2t(tx^2-x+t)}. \]

We observe that, in the expression of $D(t)$, there are only terms involving $t^2$, and $D(t)$ is in fact also a power series in $t^2$, \textit{i.e.} only terms $t^{2n}$ with even power have non-zero coefficients, which agrees with the fact that Dyck paths are of even length. We also observe that the expression of $D(t)$ is similar to that of plane trees in the previous example. Indeed, Dyck paths are also counted by Catalan numbers, and in a later chapter we will describe a bijection between plane trees and Dyck paths.

Although we are primarily interested in the OGF $D(t)$ of Dyck paths, to write a functional equation, we choose to extend our sight to a more general class of objects with an extra statistic, namely the class of positive paths with finishing height statistic. In the functional equation, we thus need an extra variable $x$ in order to control the finishing height, which is then dropped. This extra variable acts like a catalyst in a chemical reaction, which is needed for the reaction to happen at a desirable speed, but ultimately absent from the product. This is perhaps the image that Zeilberger tried to convey when he coined the term \emph{catalytic variable} in \cite{zeilberger}. The equation \eqref{eq:2:dyck} can be called an \emph{functional equation with one catalytic variable} of the OGF of positive paths.

Our strategy of resolution of \eqref{eq:2:dyck} comes from a more general principle called the \emph{kernel method} for resolution of linear functional equations with one catalytic variable. Suppose that we have an equation for the generating function $F(t,x)$, alongside with some unknown function $G(t)$ that does not depend on $x$ (which is usually the catalytic variable), of the form
\[ Y(t,x)F(t,x) + H(t,x) = G(t). \]
The function $Y(t,x)$ is called the \mydef{kernel}. Let $x(t)$ be a solution to the equation $Y(t,x(t))=0$ which can be substituted into both $Y(t,x)$ and $H(t,x)$ legitimately. By substitution, we have immediately $G(t) = H(x(t),t)$ the value of the unknown function $G(t)$, which can be used to obtain $F(t,x)$ by substitution back to the equation.

Since there is a clear correspondence between combinatorial constructions and operations on generating functions, to write down a functional equation for the generating function of a certain class, we usually just state a decomposition of elements in the class, then convert it directly to a functional equation without passing by the exact form of combinatorial specification.

\paragraph{Example 3: planar maps, Tutte equation and the quadratic method} ~\\

We now consider general planar maps that we have introduced in Section~1.1.1 and briefly analyzed in Section~\ref{sec:1:gen}. Let $\mathcal{M}$ be the class of planar maps, with the statistics $\fdeg$ of the degree of the outer face. Recall that we allow the \emph{empty map} consisting of one single vertex and no edge. We denote by $M(t,x)$ the OGF of $\mathcal{M}$ with variable scheme $(\fdeg,x)$. Since we are primarily interested in the number of planar maps regardless of the degree of their outer faces, we only need to solve for $M(t,1)$.

We will now write the Tutte equation for non-empty planar maps by considering the effect of deleting their roots. As we have mentioned previously, there are two cases.
\begin{itemize}
\item The root borders only the outer face, which implies that it is a bridge, and its removal breaks the original map into two planar maps of smaller sizes.
\item The root borders the outer face and another face $f$, which implies that its removal will ``merge'' $f$ with the outer face, increasing the degree of the outer face.
\end{itemize}
Figure~\ref{fig:2:planar} illustrates the two cases and the remaining maps after root removal. To reconstruct a planar map in the first case, we pick a pair of planar maps, link their root vertices by a new edge, then orient it from the first map to the second as the new root. We thus add $1$ to the size, and $2$ to the degree of the outer face, since the new added root borders the outer face twice. To reconstruct a planar map in the second case, we first pick a planar map, and denote its root vertex by $v$. We then add a new edge from $v$ to one of the corners of the outer face. For a planar map whose outer face is of degree $k$, there are $k+1$ possible corners, since the new edge split the root corner into two. Figure~\ref{fig:2:planar} also illustrates the possible reconstructions, and we see that all non-empty maps can be reconstructed in this way.

\begin{figure}
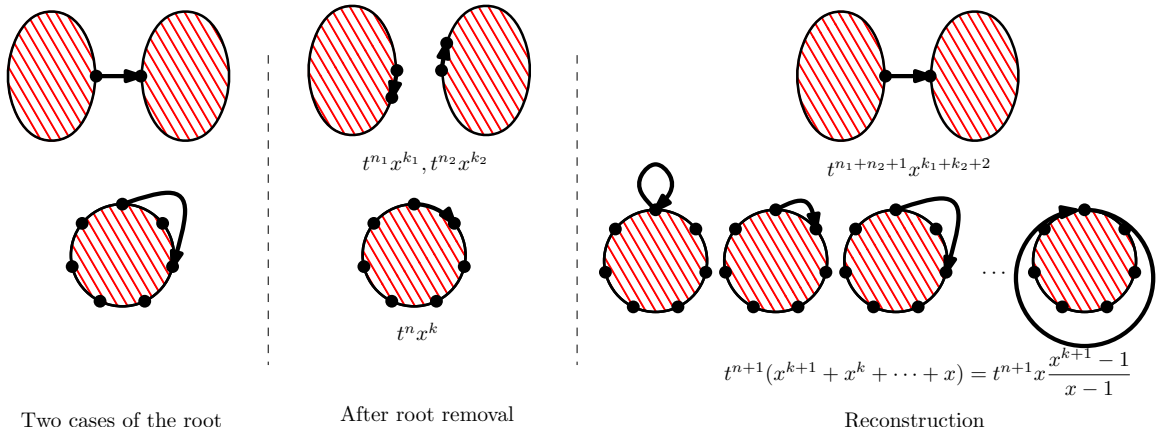

  \centering
  \insertfigure[0.8]{ch2-fig.pdf}{7}
  \caption{Cases in the Tutte equation for planar maps}
  \label{fig:2:planar}
\end{figure}

Root removal and reconstruction of planar maps gives the following Tutte equation for $M(t,x)$:
\[ M(t,x) = 1 + tx^2M(t,x)^2 + tx\frac{xM(t,x) - M(t,1)}{x-1}. \]
Some simplification and rewriting give:
\begin{equation} \label{eq:2:planar}
tx^2(x-1)M(t,x)^2 + (tx^2-x+1) M(t,x) + (x-1 - txM(t,1)) = 0.
\end{equation}
Let $a=tx^2(x-1)$, $b=tx^2-x+1$ and $c = x - 1 - txM(t,1)$, we have
\[ aM(t,x)^2 + bM(t,x) + c = 0,\]
whose left-hand side can be made up to a square:
\[ \left( 2aM(t,x) + b \right)^2 = b^2 - 4ac. \]
We observe that $P(x) = b^2-4ac$, as a polynomial in $x$, has a double root. Therefore, the discriminant of $P(x)$ must be zero, which gives a functional equation for $M(t,1)$. By taking the discriminant (the computation is not presented here, and is better to be done with a computer algebra software), we have:
\[ 27t^2M(t,1)^2-(18t-1)M(t,1)+(16t-1)=0.\]
There is only one solution that is also a power series, which is
\[ M(t,1) = \frac{18t-1+(1-12t)^{3/2}}{54t^2}.\]
We denote by $M_n = [t^n]M(t,1)$ the number of planar maps with $n$ edges. Again, using the generalized binomial theorem, we find
\[ M_n = \frac{2 \cdot 3^n}{(n+1)(n+2)} \binom{2n}{n}. \]
By substituting the expression of $M(t,1)$ back to \eqref{eq:2:planar}, we can also solve for an explicit expression of $M(t,x)$.

For asymptotic behavior, we observe that the unique dominant singularity of $M(t,1)$ is $\rho = 1/12$, and near $\rho$ we have
\[ M(t,1) = \frac{4}{3} - \frac{4}{3}(1-12t) - 4(1-12t)^2 + \frac{8}{3}(1-12t)^{3/2} + O((1-12t)^{5/2}). \]
By the transfer theorem (Theorem~\ref{thm:2:transfer}), we have
\[ M_n = \frac{2}{\pi^{1/2}} n^{-5/2} 12^n + O(n^{-7/2} 12^n). \]
This is because $\Gamma(-3/2) = 4\pi^{1/2}/3$.

Our resolution of \eqref{eq:2:planar} is an application of a more general principle called \emph{quadratic method} mainly for resolution of quadratic functional equations with one catalytic variable. The idea is to make squares or higher powers on the left-hand side. Then any series $x(t)$ that cancels the left-hand side when substituted for the catalytic variable $x$ will be a multiple root of the right-hand side. The presence of a multiple root will then give several equations containing $x(t)$ and unknown functions at once that can be used to solve for the wanted unknown functions, then the original equation. In \cite{BMJ} there is a simultaneous generalization of the kernel method and the quadratic method.

\horizrule

We have seen some essential tools we will need in the rest of this thesis: definitions of various maps, the symmetric group and its representations, then finally the most important tool: generating functions. For generating functions, we have seen how to write a functional equation from a recursive decomposition, how to solve these equations and how to extract exact and asymptotic enumeration results from the solution we obtained, all illustrated with examples. We are now properly equipped to visit various enumeration problems concerning maps.

\chapter{Generalized quadrangulation relation}

In Section~\ref{sec:2:map-const}, we have seen that bipartite maps, which are maps with a proper 2-coloring, can be generalized to $m$-constellations, which are maps with $m$ colors on its vertices that satisfy some extra conditions. This generalization works not only in the superficial sense of the number of colors, but also in a deep sense of their rotation systems, where factorizations into $3$ permutations for bipartite maps are generalized to $m+1$ permutations for $m$-constellations where $m\geq 2$ can be taken arbitrarily. It thus seems a natural idea to generalize results on bipartite maps to constellations, especially those proved using the character method. Since constellations can be used as a unified framework of various factorization models in the symmetric group (\textit{cf.} Section~\ref{sec:2:sym}), such generalized results may also be extended to these models.

This chapter will be a demonstration of the character method (\textit{cf.} Section~\ref{sec:1:char}) in the enumeration of maps. It is based on \cite{fang2014generalization}, in which an enumerative relation between constellations and hypermaps is given, generalizing the \emph{quadrangulation relation}. Definitions of constellations and hypermaps can be found in Section~\ref{sec:2:map-const}. This generalized relation is then proved using the character method. Definitions of characters of the symmetric group and their relation with maps can be found in Section~\ref{sec:2:sym}. Although the character method relies essentially on the algebraic structure of the symmetric group, our treatment in this chapter has a more combinatorial flavor, using the Murnaghan-Nakayama rule to deal with characters of the symmetric group. 

\section{Motivation} \label{sec:3:intro}
In \cite{jackson1999combinatorial}, the following strikingly simple enumerative relation was established:
\begin{displaymath}
E^{(g)}_{n,D} = \sum_{i=0}^{g} 4^{g-i} B^{(g-i, 2i)}_{n,D} = 4^g B^{(g,0)}_{n,D} + 4^{g-1} B^{(g-1, 2)}_{n,D} + \ldots + B^{(0,2g)}_{n,D} .
\end{displaymath}
Here, for $D \subseteq \naturals^{+}$, we define $B^{(g,k)}_{n,D}$ as the number of rooted bipartite maps of genus $g$ with every face degree of the form $2d$ with $d \in D$, whose vertices are colored black and white, rooted in a white vertex and with $n$ edges such that $k$ black vertices are marked. The number $E^{(g)}_{n,D}$ is the counterpart for rooted (not necessarily bipartite) maps with the same restriction on face degrees without marking. In the planar case, we have $E^{(0)}_{n,D} = B^{(0,0)}_{n,D}$, meaning that a planar map with all faces of even degree is always bipartite. The situation in higher genera is more complicated, where every map whose faces are all of even degree is not always bipartite. Figure~\ref{fig:3:grid-torus} gives such an example of a $5 \times 6$ rectangular grid on a torus, which is a quadrangulation but not bipartite.

\begin{figure}
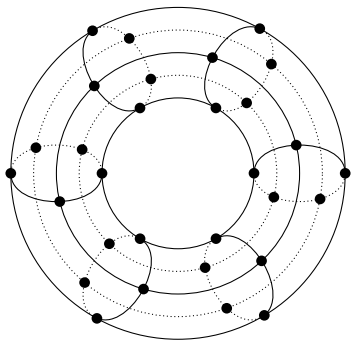

\begin{center}
\insertfigure{torus-grid.pdf}{1}
\end{center}
\caption{An example of a $5 \times 6$ grid on a torus} \label{fig:3:grid-torus}
\end{figure}

The special case $D=\{ 2 \}$ had been proved in \cite{JV1990a}, and the maps counted in this case are quadrangulations, which gives this special case the name \emph{quadrangulation relation}. It had been then extended to $D= \{ p \}$ in \cite{JV1990b}. Despite its nice form, the combinatorial meaning of the quadrangulation relation remains unclear, though some effort was done in \cite{jackson1999combinatorial} to explore properties of the possible hinted bijection.

In enumeration of maps, there is a recurrent phenomenon: results on bipartite maps can often be generalized to constellations (see \textit{e.g.} \cite{BMS, BDFG, poulalhon2002factorizations}). It is because constellations generalize bipartite maps. In the same spirit, we will generalize the quadrangulation relation to $m$-constellations and $m$-hypermaps. See Section~\ref{sec:2:map-const} for the definitions of these objects. As an example, our result in the case $m=3$ gives rise to the following relation (\textit{cf.} Corollary~\ref{coro:3:counting-relation-m-3-4}):
\begin{displaymath}
H^{(g)}_{n,3,D} = \sum_{i=0}^{g} 3^{2g-2i} \sum_{l=0}^{2i} \frac{2^{\ell+1} - (-1)^{\ell+1}}{3} C^{(g-i, \ell, 2i-l)}_{n,3,D}.
\end{displaymath}
Here, $C^{(g, a, b)}_{n,3,D}$ is the number of rooted 3-constellations with $n$ hyperedges, and hyperface degree restricted by the set $D$, with $a$ marked vertices of color 1 and $b$ marked vertices of color 2. The number $H^{(g)}_{n,3,D}$ is the counterpart for rooted 3-hypermaps without marking. We also give the same type of relation for general $m$. While our generalization of the quadrangulation relation still has a simple form, it involves extra coefficients in the weighted sum that do not appear in the quadrangulation relation. Explicit expressions of these coefficients are given in Corollary~\ref{coro:3:explicit-coeffs} using symmetries in $m$-constellations. We then establish Theorem~\ref{thm:3:positivity-of-differential-operator-coefficient} stating that these coefficients are all positive integers, revealing the possibility that a combinatorial interpretation exists for our relation. Finally, we recover a relation between the asymptotic behavior of $m$-constellations and $m$-hypermaps found in \cite{chapuy2009asymptotic}, which can be seen as an asymptotic version of our relation.

Given a partition $\mu \vdash n$, we denote by $m\mu$ the partition obtained by multiplying every part in $\mu$ by $m$. In \cite{JV1990a}, the quadrangulation relation was obtained using a factorization of irreducible characters of the symmetric group on partitions of the form $[(mk)^n]$ using a notion called \emph{$m$-balanced partition}, which is a special case of a more general result for characters evaluated on partitions of the form $m\mu$ in an article of Littlewood \cite{littlewood1951modular}. For the sake of self-containedness, a combinatorial proof of this result is given.

\section{Rotation systems and generating functions} \label{sec:3:defs}

If we recall the definition of constellations and hypermaps in Section~\ref{sec:2:map-const}, we will notice that constellations and hypermaps only differ in the possession of a vertex coloring. However, in the planar case, a hypermap is also a constellation, in the sense that every planar $m$-hypermap can be given a proper vertex coloring to become an $m$-constellation. Suppose that we orient edges in an $m$-hypermap in such a way that that the adjacent hyperedge of an edge is always on the right. With this orientation, when we travel in the $m$-hypermap following the direction of edges, we always see vertices with colors in cyclic order: $1, 2, \ldots, m, 1, 2,$ \textit{etc.} Now the only obstacle for an $m$-hypermap to be an $m$-constellation is the existence of a directed cycle whose length is not divisible by $m$, but in the planar case we can easily show that such a cycle does not exist using the Jordan curve theorem. However, this is not necessarily true for higher genera, in which an $m$-hypermap does not necessarily have a coloring that conforms with the additional condition to be an $m$-constellation. 

We now define the OGFs of $m$-hypermaps and $m$-constellations. We use $\varvec{x}$ to denote a sequence of variables $x_1, \ldots, x_m$, and $[x_i \gets f(i)]$ to denote the substitution of $\varvec{x}$ by $x_i = f(i)$. We also introduce an infinite set of variables $\varvec{y} = y_1, y_2, \ldots$. We define $H(x, \varvec{y}, t, u)$ to be the OGF of $m$-hypermaps, with $x$ marking the number of vertices, $y_i$ the number of hyperfaces of degree $mi$ for each $i$, $t$ the number of hyperedges and $u$ twice the genus. Similarly, we define $C(\varvec{x}, \varvec{y}, t, u)$ to be the OGF of $m$-constellations, except that with $x_i$ we mark the number of vertices with color~$i$.

We recall from Section~\ref{sec:2:sym} that rotation systems of an $m$-constellation with $n$ hyperedges are transitive factorizations of the identity in $S_n$ of the form $(\sigma_1, \ldots, \sigma_m, \phi)$ such that
\[ \sigma_1 \cdots \sigma_m \phi = id_n. \]
Furthermore, each $m$-constellation with $n$ hyperedges has $(n-1)!$ rotation systems. For rotation systems of $m$-hypermaps, we observe that their duals are bipartite maps with black vertices of degree $m$ and white vertices of degree divisible by $m$. We can thus adapt the definition of rotation systems of bipartite maps to $m$-hypermaps. For an $m$-hypermap with $n$ hyperedges, we first label hyperedges from $1$ to $n$, then we label edges adjacent to a hyperedge with label $k$ in counter-clockwise order with integers from $(k-1)m+1$ to $km$. With the convention that the root edge always receives label $1$, for each $m$-hypermap we have $(n-1)! m^{n-1}$ possible labelings. For a labeling obtained in this way, we define $\sigma_\bullet$ (resp. $\sigma_\circ$) to be the permutation whose cycles are cyclic counter-clockwise orders of edges adjacent to each hyperedge (resp. hyperface), and for each vertex, we consider its adjacent edges with a hyperedge on their right (seen from the vertex), whose label, in counter-clockwise order, form a cycle of the permutation $\phi$. Figure~\ref{fig:3:hypermap} shows how these bijections act on edges of a hypermap. By our labeling process, the permutation $\sigma_\bullet$ is fixed to $(1,2,\ldots,m)(m+1,m+2,\ldots,2m)\cdots(mn-m+1,\ldots,mn)$. A \mydef{factorization of $m$-hypermap type} of size $n$ is a pair of permutations $(\sigma_\circ, \phi)$ in $S_{mn}$ such that the cycle types of $\sigma_\circ$ is $m\lambda$ for some $\lambda \vdash n$, and that $\sigma_\circ \sigma_\bullet \phi = \id_{mn}$ with the $\sigma_\bullet$ given above. A \mydef{rotation system of an $m$-hypermap} is a factorization of $m$-hypermap type that is transitive. We take the convention that the root edge is always labeled by $1$, and now each $m$-hypermap with $n$ hyperedges has $(n-1)! m^{n-1}$ rotation systems.

\begin{figure}
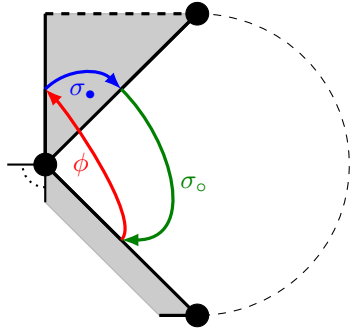

  \centering
  \insertfigure{hypermap-rot.pdf}{1}
  \caption{Actions of permutations in a rotation system of an $m$-hypermap}
  \label{fig:3:hypermap}
\end{figure}

For a partition $\mu$, we define $y_\mu=\prod_{i>0} y_{\mu_i}$. We recall that $Cl(\lambda)$ is the conjugacy class of permutations with cycle type $\lambda$. We can now define the following EGF $R_H$ of factorizations $(\sigma_\circ, \phi)$ of $m$-hypermap type, with $x$ marking the number of cycles in $\sigma_\circ$, the variables $\underline{\boldsymbol{y}}$ the cycle type of $\phi$ and $t$ the number of edges divided by $m$, with one of the permutations $\sigma_\bullet = (1,2,\ldots,m) \cdots (mn-m+1,mn-m+2,\ldots, mn)$ fixed:
\begin{displaymath}
R_H(x,\varvec{y},t) \eqdef \sum_{n \geq 0} \frac{t^n}{n!} \sum_{\mu \vdash n} \sum_{\substack{\sigma_\circ \sigma_\bullet \phi = id_{mn} \\ \sigma_\circ \in Cl(m\mu)}} x^{\ell(\sigma_\circ)} y_{\mu}.
\end{displaymath}
Similarly, we can also define the generating function $R_C$ of factorizations $(\sigma_1, \ldots, \sigma_m, \phi)$, with $x_i$ marking the number of cycles in $\sigma_i$, the variables $\underline{\boldsymbol{y}}$ the cycle type of $\phi$ and $t$ the size of the factorization:  
\begin{displaymath}
R_C(\varvec{x},\varvec{y},t) \eqdef \sum_{n \geq 0} \frac{t^n}{n!} \sum_{\mu \vdash n} \sum_{\substack{\sigma_1 \ldots \sigma_m \phi = id_n \\ \phi \in Cl(\mu)}} y_{\mu} \prod_{i=1}^{m} x_i^{\ell(\sigma_i)}.
\end{displaymath}

By taking the logarithm of the corresponding generating function, we can enforce transitivity (\textit{cf.} Section~\ref{sec:2:gen}). We notice that taking the logarithm of $R_H$ and $R_C$ is a legitimate operation, since they start with a constant term $1$ given by the case $n=0$. We now relate the OGF $H$ of hypermaps to the EGF $R_H$ of factorizations of $m$-hypermap type. It is clear that $\log R_H$ is EGF of rotation systems of $m$-hypermap type, with $t$ marking the number of cycles in $\sigma_\bullet$, which is equal to the number of edges divided by $m$. Now, for a given $m$-hypermap with $n$ hyperedges, it has $m^{n-1} (n-1)!$ different rotation systems, since by convention the root edge always receives label $1$, which fixes the labels of other edges sharing the same hyperedges, and there are $(n-1)!$ ways to distribute other cycles of $\sigma_\bullet$ to other hyperedges, each has $m$ possible ways to label its edges given a cycle. For the genus $g$, we observe that an $m$-hypermap with $n$ hyperedges, $v$ vertices and $f$ hyperfaces has $mn$ edges, $n+f$ faces, and by Euler's relation \eqref{eq:1:euler}, we have $2g = 2 + mn - n - f - v$. We thus have
\begin{align} \label{eq:3:H-to-RH}
\begin{split}
H(x,\varvec{y},t,u) &= \sum_{n \geq 1} \frac{t^n}{m^{n-1}(n-1)!} \sum_{\mu \vdash n} \sum_{\substack{(\sigma_\circ, \sigma_\bullet, \phi) \; \mathrm{rotation\;system} \\ \sigma_\circ \in Cl(m\mu)}} x^{\ell(\phi)} y_\mu u^{2+(m-1)n-\ell(\mu)-\ell(\phi)} \\
&= m u^2 \sum_{n \geq 1} \frac{(tu^{m-1})^n}{m^{n}(n-1)!} \sum_{\mu \vdash n} \sum_{\substack{(\sigma_\circ, \sigma_\bullet, \phi) \; \mathrm{rotation\;system} \\ \sigma_\circ \in Cl(m\mu)}} (xu^{-1})^{\ell(\phi)} y_\mu u^{-\ell(\mu)} \\
&= mu^2 \left( \frac{t\partial}{\partial t} (\log R_H) \right) (xu^{-1}, [y_i \gets y_i u^{-1}], \frac{1}{m}tu^{m-1}),
\end{split}
\end{align}
where the factor $1/n!$ in the EGF $\log R_H$ is turned into $1/(n-1)!$ by the operator $t\partial / \partial t$. Similarly, we can relate the OGF $C$ of constellations and the EGF $R_C$ of factorizations $\sigma_1 \cdots \sigma_m \phi = \id$ as follows:
\begin{equation} \label{eq:3:C-to-RC}
C(\varvec{x},\varvec{y},t,u) = u^2 \left( \frac{t\partial}{\partial t} (\log R_C) \right) ([x_i \gets x_i u^{-1}], [y_i \gets y_i u^{-1}], tu^{m-1}).
\end{equation}

From an algebraic point of view, the series $R_H$ and $R_C$ are much easier to manipulate than $H$ and $C$. We can use characters to express $R_H$ and $R_C$. We should notice that, in a factorization of $m$-hypermap type $(\sigma_\circ, \phi)$, we have $\sigma_\circ \sigma_\bullet \phi = \id$ with $\sigma_\bullet$ fixed instead of an arbitrary permutation in $Cl([m^n])$. Therefore, suppose that $\sigma_\circ$ and $\phi$ are of cycle types $\lambda \vdash mn$ and $m\mu$ with $\mu \vdash n$ respectively, the number of such factorizations is simply $[K_{[m^n]}]K_\lambda K_{m\mu}$. We can thus compute $R_H$ using the change of basis \eqref{eq:2:basis-change} in Section~\ref{sec:2:sym} between $(K_\theta)_{\theta \vdash n}$ and $(F_\theta)_{\theta \vdash n}$, which leads to
\begin{align}
\begin{split}
R_H(x,\varvec{y},t) &= \sum_{n\geq 0} \frac{t^n}{n!} \sum_{\mu \vdash n, \lambda \vdash mn} x^{\ell(\lambda)} y_\mu [K_{[m^n]}]K_\lambda K_{m\mu} \\
&= \sum_{n\geq 0} \frac{t^n}{n!} \sum_{\mu \vdash n, \lambda \vdash mn} x^{\ell(\lambda)} y_\mu \frac{((mn)!)^2}{z_\lambda z_{m\mu}} \sum_{\theta \vdash mn} \frac{\chi^\theta_\lambda \chi^\theta_{m\mu}}{(f^\theta)^2} [K_{[m^n]}]F_\theta \\
&= \sum_{n\geq 0} \frac{t^n}{n!} \sum_{\mu \vdash n, \lambda \vdash mn} x^{\ell(\lambda)} y_\mu \frac{((mn)!)^2}{z_\lambda z_{m\mu}} \sum_{\theta \vdash mn} \frac{\chi^\theta_\lambda \chi^\theta_{m\mu}}{(f^\theta)^2} \frac{f^\theta \chi^\theta_{[m^n]}}{(mn)!} \\
&=\sum_{n \geq 0} \frac{t^n}{n!} \sum_{\lambda \vdash mn, \mu \vdash n} (mn)! z_{\lambda}^{-1} z_{m\mu}^{-1} x^{\ell(\lambda)} y_{\mu} \sum_{\theta \vdash mn} \frac{1}{f^{\theta}} \chi_{\lambda}^{\theta} \chi_{[m^n]}^{\theta} \chi_{m\mu}^{\theta}
\end{split}\label{eq:3:RH-in-all-characters}
\end{align}
For $R_C$, we can directly use \eqref{eq:2:const} in Section~\ref{sec:2:sym} to obtain
\begin{equation} \label{eq:3:RC-in-all-characters}
R_C(\varvec{x},\varvec{y},t) = \sum_{n \geq 0} \frac{t^n}{n!} \sum_{\lambda^{(1)}, \ldots, \lambda^{(m)}, \mu \vdash n} \left( \prod_{i=1}^{m} x_i^{\ell(\lambda^{(i)})} \right) y_{\mu} \sum_{\theta \vdash n} (f^{\theta})^{(1-m)} z_{\mu}^{-1} \chi^{\theta}_{\mu} \prod_{i=1}^{m} n! z_{\lambda^{(i)}}^{-1} \chi^{\theta}_{\lambda^{(i)}}.
\end{equation}

To further simplify the expressions above, we define the \emph{rising factorial function} $x^{(n)} = x(x+1) \cdots (x+n-1)$ for $n \in \naturals$. For a partition $\theta$, we define the polynomial $H_\theta(x)$ as $\prod_{i=1}^{\ell(\theta)} (x-i+1)^{(\theta_i)}$. With this notation, we give the following expressions of $R_H$ and $R_C$. We recall that the quantity $z_\mu$ related to a partition $\mu$ is defined as $z_\mu = \prod_i m_i ! i^{m_i}$, where $m_i$ is the number of parts of size $i$ in $\mu$. We then observe that $z_{m\mu} = \prod_{i\geq 1} m_i! (mi)^{m_i} = m^{\ell(\mu)} z_\mu$.

\begin{prop} \label{prop:3:series-in-big-characters-simplified}
We have
\[ R_H(x,\varvec{y},t) = \sum_{n \geq 0} \frac{t^n}{n!} \sum_{\mu \vdash n} m^{-\ell(\mu)} z_{\mu}^{-1} y_{\mu} \sum_{\theta \vdash mn} \chi_{[m^n]}^{\theta} \chi_{m\mu}^{\theta} H_{\theta}(x), \]
\[ R_C(\varvec{x},\varvec{y},t) = \sum_{n \geq 0} \frac{t^n}{n!} \sum_{\mu \vdash n} y_{\mu}  z_{\mu}^{-1} \sum_{\theta \vdash n} f^{\theta}  \chi^{\theta}_{\mu} \left( \prod_{i=1}^{m} H_\theta(x_i) \right). \]
\end{prop}

This proposition comes from direct application of the following lemma (Lemma 3.4 in \cite{JV1990a}) to \eqref{eq:3:RH-in-all-characters} and \eqref{eq:3:RC-in-all-characters}.

\begin{lem} \label{lem:3:polynomial-H}
We have the following equality:
\[ n! \sum_{\alpha \vdash n} z_{\alpha}^{-1} \chi_{\alpha}^{\theta} x^{\ell(\alpha)} = f^{\theta}H_{\theta}(x). \]
\end{lem}

Since this lemma relies crucially on deep algebraic results in the representation theory of the symmetric group, we choose to omit its proof here. Curious readers can find a complete proof in \cite{fang2014generalization}, which relies on the elegant construction of irreducible representations of $S_n$ in \cite{vershik2004new}.

We can see that characters in $R_H$ in Proposition~\ref{prop:3:series-in-big-characters-simplified} are evaluated at partitions of the form $m\mu$ with $\mu \vdash n$. In \cite{JV1990a}, $\chi_{[m^n]}^{\theta}$ is proved to have an expression as a product of smaller characters, which is a crucial step towards the quadrangulation relation. This factorization is also presented in Section~2.7 of \cite{james1981representation} under the framework of $p$-core and abacus display of a partition. With a generalization due to Littlewood \cite{littlewood1951modular} that applies to all partitions of the form $m\mu$, we will give a similar relation between $m$-hypermaps and $m$-constellations in Section~\ref{sec:3:app}.

\section{Factorization of characters} \label{sec:3:fact}

In this section we present the following result on factorizing $\chi_{m\lambda}^{\theta}$ into smaller characters. The notion of \emph{$m$-splittable} partition will be defined later.

\begin{thm}[Littlewood 1951 \cite{littlewood1951modular}] \label{thm:3:character-factorization}
Let $m,n$ be two natural numbers, and $\lambda \vdash n$, $\theta \vdash mn$ two partitions. We consider partitions as multisets and we denote the multiset union by $\uplus$. If $\theta$ is $m$-splittable, we have
\[ \chi_{m\lambda}^{\theta} = z_{\lambda} \sgn_\theta \sum_{\lambda^{(1)} \uplus \cdots \uplus \lambda^{(m)} = \lambda} \prod_{i=1}^m \chi_{\lambda^{(i)}}^{\theta^{(i)}} z_{\lambda^{(i)}}^{-1}, \]
with $\sgn_\theta$ and all $\theta^{(i)}$ depending only on $\theta$ and $m$.

If $\theta$ is not $m$-splittable, $\chi_{m\lambda}^{\theta} = 0$.
\end{thm}

An algebraic proof was given in \cite{littlewood1951modular}. For the sake of self-containedness, we will present a combinatorial proof here. We will first give a natural combinatorial interpretation of $m$-splittable partitions using the \emph{infinite wedge space}. A brief introduction to the infinite wedge space can be found in the appendix of \cite{okounkov2001infinite}, after which some of our notations here follow. With this combinatorial interpretation, we will give a straightforward, purely combinatorial proof of Theorem~\ref{thm:3:character-factorization}.

\subsection{Infinite wedge space and boson-fermion correspondence}
We recall some definitions about the infinite wedge space taken from \cite{okounkov2001infinite}. Let $(\underline{k})_{k \in \integers}$ be a set of variables indexed by integers, and $\wedge$ be an associative and anti-commutative binary relation acting as an exterior product. For $S \subset \integers$, we define $S_+ = S \cap \naturals$ and $S_- = \integers_{< 0} \setminus S$. We denote by $\Lambda^{\infty/2}$ the vector space spanned by vectors of the form $v_{S} = \underline{s_1} \wedge \underline{s_2} \wedge \ldots$ with $S = \{ s_1 > s_2 > \ldots \}$ such that both $S_+$ and $S_-$ are finite. The vector space $\Lambda^{\infty/2}$ with the exterior product $\wedge$ is called the \tdef{infinite wedge space}.

We now define the creation operator $\phi_k$ and the annihilation operator $\phi^*_k$, in a combinatorial way. By the definition of $\Lambda^{\infty/2}$, we only need to specify the action of $\phi_k$ and $\phi^*_k$ on $v_S$ for any set $S$ such that both $S_+$ and $S_-$ are finite:
\[
\phi_k v_S =
\begin{cases}
0 &\text{if $k \in S$} \\
(-1)^{\# \{ i \in S \mid i > k \}} v_{S \cup \{k\}} &\text{if $k \notin S$}
\end{cases}
, \quad
\phi^*_k v_S =
\begin{cases}
(-1)^{\# \{ i \in S \mid i > k \}} v_{S \setminus \{k\}} &\text{if $k \in S$} \\
0 &\text{if $k \notin S$}
\end{cases}.
\]
Briefly speaking, the creation operator $\phi_k$ tries to add the element $k$ into the set $S$. If $k$ is already in $S$, then it fails and gives $0$. Otherwise, it adds $k$ into $S$, and gives the vector a sign depending on the parity of the number of elements in $S$ that are greater than $k$. The annihilation operator $\phi^*_k$ does a similar operation to remove $k$ from $S$. Readers familiar with properties of exterior products will recognize that $\phi_k$ is the left multiplication by the formal variable $\underline{k}$, \textit{i.e.} $\phi_k(v) = \underline{k} \wedge v$ for all $v \in \Lambda^{\infty/2}$, and $\phi^{*}_k$ is the adjoint operator of $\phi_k$ with respect to the canonical scalar product. Let $\Lambda_0$ be the subspace spanned by vectors of the form $v_{S}$ satisfying that $S_+$ and $S_-$ are finite and $|S_+|=|S_-|$. A set $S$ such that $v_S \in \Lambda_0$ is called \tdef{well-charged}.

\begin{figure}
\begin{center}
\insertfigure{partition-russian.pdf}{1}
\end{center}
\caption{Rotated diagram of the partition $\lambda=(4,2,2,2,1,1)$ corresponding to the vector $v_S$ with $S=\{ 3, 0, -1, -2, -4, -5, -7, -8 \ldots \}$} \label{fig:3:russian-notation}
\end{figure}

Partitions are in bijection with well-charged sets. Given a partition $\lambda = (\lambda_1, \lambda_2, \ldots, \lambda_l)$, we draw its diagram in French convention, then rotate the diagram by 45 degrees in counter-clockwise direction. The diagram is framed by a lattice path (thick line in Figure~\ref{fig:3:russian-notation}), called \mydef{the framing path}, consisting of two types of steps, one parallel to the line $y=x$ and the other to $y=-x$. This framing path eventually coincides with the line $y = -x$ to the left and $y=x$ to the right. We can also consider $\lambda$ as an infinite sequence by taking $\lambda_k = 0$ for $k>l$. We define the set $S_\lambda = \{ \lambda_i - i \mid i \in \naturals \} $. We can see from Figure~\ref{fig:3:russian-notation} that $S_\lambda$ is exactly the set of starting abscissas of down-going steps (parallel to $y=-x$).

This map from $\lambda$ to $S_\lambda$ is a classical bijection between partitions and well-charged sets. We illustrate the set $S_\lambda$ as a diagram of $\mathbb{Z}$ where each position indexed by an element of $S_\lambda$ is occupied by a particle. This bijection is closely related to the \emph{boson-fermion correspondence} in the literature. More information about this presentation of partitions can be found in \cite{okounkov2001infinite}.


For a partition $\lambda$, we denote by $v_\lambda$ the vector $v_{S_\lambda}$ corresponding to $S_\lambda$. We now define a new operator $\sigma_{k,p} = \phi_{p} \phi^{*}_{p+k}$ for a positive integer $k$ and an integer $p$. The effect of $\sigma_{k,p}$ on $v_\lambda$ is trying to remove a ribbon of size $k$ from the appropriate position of the Ferrers diagram of $\lambda$. When it is possible, we have $\sigma_{k,p} v_\lambda = (-1)^{ht(\lambda / \mu)} v_\mu$, where $\mu$ the partition after removal of the ribbon. We notice that the induced sign is exactly the same as the contribution of a ribbon to the sign of any ribbon tableau it belongs to. See Figure~\ref{fig:3:border-strip-removal} for an example. When the removal is impossible, \textit{i.e.} there is no particle at position $p+k$ or there is already a particle at position $p$, we have $\sigma_{k,p} v_\lambda = 0$. Combinatorially, the operator $\sigma_{k,p}$ can be viewed as a particle jumping from position $p+k$ back to position $p$, and the sign it induces corresponds to the number of particles underneath this jump, that is, the number of jump-overs. For example, in Figure~\ref{fig:3:border-strip-removal}, three jump-overs occur. We can also see that $ht(\lambda / \mu)$ is exactly the number of jump-overs in the removal of the ribbon $\lambda / \mu$. Therefore, a ribbon tableau $T$ of shape $\lambda$ and type $\mu = (\mu_1, \ldots, \mu_k)$ can be identified as a sequence of operators $\sigma_{\mu_k,p_k} \cdots \sigma_{\mu_2, p_2} \sigma_{\mu_1, p_1}$ applied to $v_\lambda$ such that the result is not zero. In this case, we have
\[
\sigma_{\mu_k,p_k} \cdots \sigma_{\mu_2, p_2} \sigma_{\mu_1, p_1} v_\lambda = \sgn(T) v_\epsilon.
\]
We can thus use the Murnaghan-Nakayama rule (Theorem~\ref{thm:2:mn-rule}) in the context of operators $\sigma_{k,p}$ to express characters of the symmetric group.

\begin{figure}
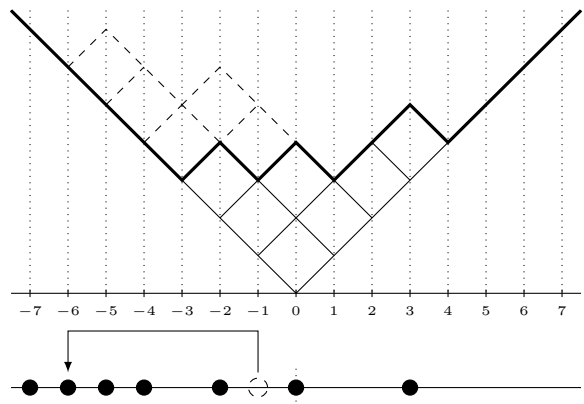

\begin{center}
\insertfigure{particle-jump.pdf}{1}
\end{center}
\caption{Effect of $\sigma_{5,-6}$ on $v_\lambda$ with $\lambda = (4,2,2,2,1,1)$. We have $\sigma_{5,-6} v_\lambda = - v_\mu$ with $\mu = (4,2,1)$} \label{fig:3:border-strip-removal}
\end{figure}


\subsection[$m$-splittable partitions]{$\boldsymbol{m}$-splittable partitions}
We now define $m$-splittable partitions. Let $S$ be a well-charged set. We define its \mydef{$m$-split} to be the $m$-tuple of sets $(S_0, S_1, \ldots, S_{m-1})$ such that $S_i = \{ a \mid ma+i \in S \}$ for $i$ from $0$ to $m-1$. A well-charged set $S$ is called \mydef{$m$-splittable} if every set in its $m$-split is well-charged. A partition $\lambda$ is called \mydef{$m$-splittable} if $S_\lambda$ is $m$-splittable. In this case, we define the \mydef{$m$-split} $(\lambda^{(0)}, \ldots, \lambda^{(m-1)})$ of $\lambda$ to be the tuple of partitions corresponding component-wise to the $m$-split of $S_\lambda$.

Here is an example in Figure~\ref{fig:3:m-split} of the $m$-split of an $m$-splittable partition. We take $m=3$ and we consider the partition $\theta = (6,6,4,4,4,3,3)$. We can verify easily that $\theta$ is 3-splittable. To obtain the 3-split of $\theta$, we split the set $S_\theta$ according to congruence classes modulo 3, then rescale to obtain 3 smaller sets, and finally we reconstruct partitions corresponding to the smaller sets.

\begin{figure}
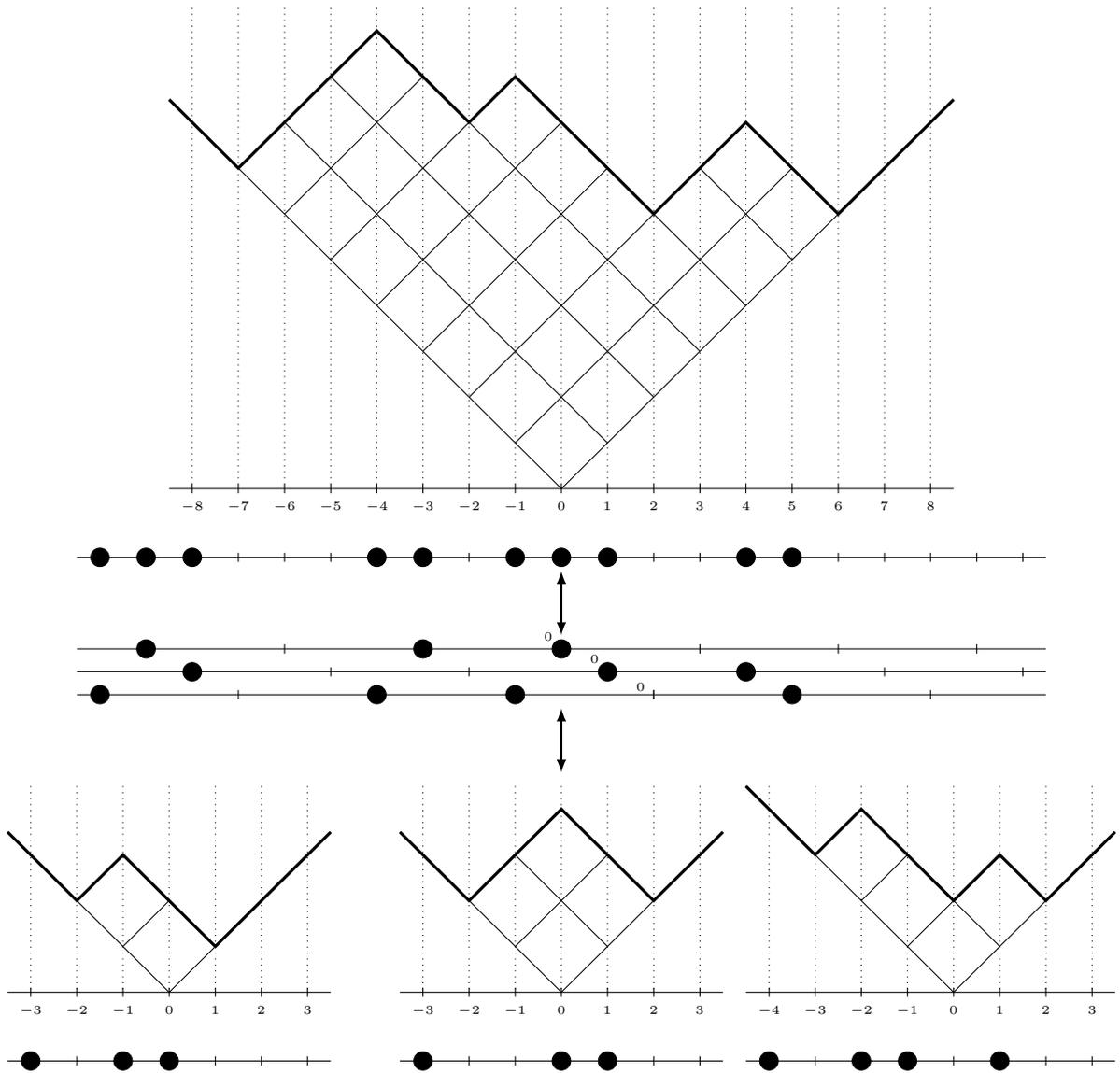

\begin{center}
\insertfigure{littlewood-decomp.pdf}{1}
\end{center}
\caption{Example of a 3-splittable partition, alongside with its 3-split} \label{fig:3:m-split}
\end{figure}

As a remark, comparing our terminology with the one in \cite{littlewood1951modular}, it is easy to see that an $m$-splittable partition is exactly a partition with an empty $m$-core, and in this case, its $m$-split coincides with its $m$-quotient. Moreover, we can easily show that the notion of ``$m$-balanced partitions'' used in \cite{JV1990a} and \cite{jackson1999combinatorial} is exactly the notion of $m$-splittable partitions. An advantage of the point of view here is that it is much more intuitive and avoids technical lemmas when dealing with these objects as in \cite{JV1990a}.

\subsection{Combinatorial proof of Theorem~\ref{thm:3:character-factorization}}

We are now ready to give a combinatorial proof to Theorem~\ref{thm:3:character-factorization}, alongside with explicit expression of $\sgn_\theta$ and all $\theta^{(i)}$. Essentially, using the Murnaghan-Nakayama rule, we establish a bijection between ribbon tableaux of shape $\theta$ and content $m\lambda$ and sequences of $m$ ribbon tableaux $T_0, \ldots, T_{m-1}$ of shape $\theta^{(0)}, \ldots, \theta^{(m-1)}$ respectively and total content $\lambda$. Readers are referred to \cite{Stanley:EC2} for more on ribbon tableaux and the Murnaghan-Nakayama rule. The idea is that the removal of a ribbon $s$ of length multiple of $m$ from the partition $\lambda$ only affects elements in $S_\lambda$ in one congruence class modulo $m$. Since one congruence class corresponds to one component in the $m$-split, the removal of $s$ can be reduced to the removal of a smaller strip $s'$ in the corresponding component in the $m$-split. We then deal with the sign issue.

\begin{proof}[Combinatorial proof of Theorem~\ref{thm:3:character-factorization}] $ $ \\
We try to evaluate $\chi^\theta_{m\lambda}$ with the Murnaghan-Nakayama rule (Theorem~\ref{thm:2:mn-rule}).

For any integer $p, k$ with $k > 0$, the operator $\sigma_{mk, p}$ only affects particles occupying the $n$-th position with $n \equiv p \mod m$. For any $S$, $\sigma_{mk, mp+i}$ only changes the component $S_{i}$ in the $m$-split, and its effect is equivalent to $\sigma_{k, p}$ applied on $S_i$. Therefore, the operator $\alpha_{mk}$ preserves the $m$-splittable property of a partition. We then have $\chi^\theta_{m\lambda} = 0$ for $\theta$ not $m$-splittable, since the empty partition is $m$-splittable.

Now we suppose that $\theta$ is $m$-splittable. Let $( \theta^{(0)}, \ldots, \theta^{(m-1)} )$ be its $m$-split. We recall that the operator $\sigma_{mk, mp+i}$ on $v_\theta$ acts only on $\theta^{(i)}$, and acts as the operator $\sigma_{k,p}$. A ribbon tableau $T_\theta$ of shape $\theta$ and of type $m\lambda$ can be considered as a sequence of operators of the form $\sigma_{mk,p}$ on $v_\theta$, which can be separated into $m$ sequences of operators of the form $\sigma_{k,p}$ on all $\theta^{(i)}$'s. This induces a surjective function that sends a ribbon tableau of shape $\theta$ and content $m\lambda$ to $m$ ribbon tableaux $T_0, \ldots, T_{m-1}$ of shapes $\theta^{(0)}, \ldots, \theta^{(m-1)}$, and the union of their types is $\lambda$. We denote by $\lambda^{(0)}, \ldots, \lambda^{(m-1)}$ the content of $T_0, \ldots, T_{m-1}$ respectively. Different ribbon tableaux of the same shape and content can be mapped to the same $m$-tuple of smaller ribbon tableaux, because when separating the original sequence of operators, we lost some information on their ordering. The multiplicity will be considered later in the proof.

We now consider the signs of ribbon tableaux. We recall that the sign $\sgn(T)$ of a ribbon tableau $T$. We denote by $j(T)$ the number of jump-overs in the corresponding operator sequence, and we have $\sgn(T) = (-1)^{j(T)}$. In $T_\theta$, there are two types of jump-overs: those involving particles in the same congruence class, and those involving particles in different congruence classes. The number of jump-overs of the first type is denoted by $j_{\mathrm{endo}}(T)$, and that of the second type $j_{\mathrm{inter}}(T)$. We have $j(T) = j_{\mathrm{endo}}(T) + j_{\mathrm{inter}}(T)$. By definition, $j_{\mathrm{endo}}(T) = \sum_{i=0}^{m-1} j(T_i)$. For $j_{\mathrm{inter}}(T)$, we can check that the parity of $j_{\mathrm{inter}}(T)$ is preserved when commuting any two operators $\sigma_{mk,p_1}$ and $\sigma_{mk,p_2}$, and when replacing any operator $\sigma_{m(k+l), p}$ by $\sigma_{ml, p-mk} \sigma_{mk,p}$. With these two transformations of operators, we can break operators of the form $\sigma_{mk,p}$ into those of the form $\sigma_{m,p}$ and reorder them by $p$, and the sign remains unchanged. Therefore, the parity of $j_{\mathrm{inter}}(T)$ depends only on $\theta$. We thus define 
\[
\sgn_\theta \eqdef (-1)^{j_{\mathrm{inter}}(T)},
\]
and we have $\sgn(T) =\sgn_\theta \prod_{i=0}^{m-1} \sgn(T_i) $.

To evaluate $\chi^\theta_{m\lambda}$ with the Murnaghan-Nakayama rule, we consider the sum over the sign of all ribbon tableau $T_\theta$ of shape $\theta$ and of content $m\lambda$. We denote by $t_k$ the multiplicity of $k$ as parts in $\lambda$, and $t_{k,i}$ the multiplicity of $k$ in $\lambda^{(i)}$. We have $t_k = \sum_{j=0}^{m-1}t_{k,j}$. By the surjective function and the sign relation between $T$ and $T_0, \ldots, T_{m-1}$ mentioned above, we have the following formula:
\begin{displaymath}
\sum_{T_\theta} \sgn(T_\theta) = \sum_{\lambda^{(1)} \uplus \cdots \uplus \lambda^{(m)} = \lambda} \sum_{T_0, \ldots, T_{m-1}} \sgn_\theta \left( \prod_{i=0}^{m-1} \sgn(T_i) \right) \prod_{k > 0}  \binom{t_k}{t_{k,0}, \ldots, t_{k,m-1}}.
\end{displaymath}
This is essentially a double-counting formula. The multinomial factor in the final summand is in fact the multiplicity of the surjective function, which stands for the number of ways to arrange a sequence of $t_k$ operators of the form $\sigma_{mk,p}$ with the same $k$ to achieve the same sequence of $\lambda^{(0)}, \ldots, \lambda^{(m-1)}$. Given such $t_k$ operators, the only case forbidding two operators $\sigma_{mk,p_1}$ and $\sigma_{mk,p_2}$ from commuting is that they are applied to the same component in the $m$-split of $\lambda$, \textit{i.e.} $p_1$ and $p_2$ come from the same congruence class modulo $m$. For the congruence class $i$, there are $t_{k,i}$ such operators whose relative positions cannot be changed. We thus have the multinomial.


We conclude the proof by the Murnaghan-Nakayama rule while noticing the following equality coming from $z_\lambda=\prod_{k>0}k^{t_k}t_{k}!$:
\begin{displaymath}
z_\lambda \prod_{i=0}^{m-1} z^{-1}_{\lambda^{(i)}} = \prod_{k>0} k^{t_k - \sum_{j=0}^{m-1}t_{k,j}} \frac{t_{k}!}{t_{k,0}! t_{k,1}! \cdots t_{k,m-1}!} = \prod_{k>0} \binom{t_k}{t_{k,0}, \ldots, t_{k,m-1}}. \qedhere
\end{displaymath}
\end{proof}

As an application of this combinatorial point of view, we also have a factorization result on the polynomial $H_\theta$ for $m$-splittable partitions $\theta$.

\begin{lem} \label{lem:3:from-Hm-back-to-H}
For an $m$-splittable partition $\theta \vdash n$, we have
\[ H_\theta(x) = m^{mn} \prod_{i=1}^{m} \prod_{j=0}^{m-1} H_{\theta^{(i)}} \left( \frac{x-i+j+1}{m} \right). \]
\end{lem}

\begin{proof}
Let $\theta \vdash mn$ be an $m$-splittable partition and $(\theta^{(0)}, \ldots, \theta^{(m-1)})$ be its $m$-split. Consider an arbitrary ribbon tableau $T$ of form $\theta$ and content $[m^n]$. We denote by $T_1, \ldots, T_{m}$ the corresponding ribbon tableaux of form $\theta^{(1)}, \ldots ,\theta^{(m)}$ respectively, as in the proof of Theorem~\ref{thm:3:character-factorization}. Each cell $w$ of $\theta^{(i)}$ corresponds to a strip $s$ of length $m$ in $T$. Moreover, we can see that the contents of cells in $s$ are exactly $mc(w)-i+1, mc(w)-i+2, \ldots, mc(w)-i+m$. These facts are independent of the choice of $T$. We thus have
\begin{align*}
H_\theta(x) &= \prod_{w \in \theta} (x+c(w)) = \prod_{i=1}^{m} \prod_{w \in \theta^{(i)}} \prod_{j=0}^{m-1} (x+mc(w)-i+j+1) \\
&= m^{mn} \prod_{i=1}^{m} \prod_{j=0}^{m-1} \prod_{w \in \theta^{(i)}} \left( \frac{x-i+j+1}{m}+c(w) \right) \\
&= m^{mn} \prod_{i=}^{m} \prod_{j=0}^{m-1} H_{\theta^{(i)}} \left( \frac{x-i+j+1}{m} \right). \qedhere
\end{align*}
\end{proof}

As a final remark, the operation in Figure~\ref{fig:3:m-split} is sometimes called drawing the \emph{abacus display} in some literature in algebra. Readers are referred to, for example, Section 2.7 of \cite{james1981representation} for more details on $p$-core, $p$-quotient and abacus display, even though there is only a specialized version of Theorem~\ref{thm:3:character-factorization}.

\section{Generalization of the quadrangulation relation} \label{sec:3:app}

In this section, using Theorem~\ref{thm:3:character-factorization}, we establish a relation between $m$-hypermaps and $m$-constellations that generalizes the quadrangulation relation. We then recover a result in \cite{chapuy2009asymptotic} on the asymptotic number of $m$-hypermaps related to that of $m$-constellations. Finally, by exploiting symmetries of the generating function, we are able to arrange our generalized relation in a form where all coefficients are positive integers, for which there might be a combinatorial explanation.

\subsection{From series to numbers}

We start by the following proposition that relates the series $R_H$ and $R_C$ using Theorem~\ref{thm:3:character-factorization}.

\begin{prop} \label{prop:3:link-of-hypermaps-and-constellation-series}
We have the following relation:
\[ R_H(x,\varvec{y},t) = \prod_{j=1}^{m} R_C \left( \left[ x_i \gets \frac{x-j+i}{m} \right], \left[ y_i \gets \frac{y_i}{m} \right], m^m t \right). \]
\end{prop}
\begin{proof}
We take the expressions of $R_H$ and $R_C$ from Proposition~\ref{prop:3:series-in-big-characters-simplified}. We observe that, in the expression of $R_H$, we need to evaluate characters of the form $\chi^\theta_{[m^n]}$ and $\chi^\theta_{m\lambda}$. Since $[m^n] = m[1^n]$, all these characters are evaluated at partitions of the form $m\lambda$. We thus only need to consider those $\theta$ that are $m$-splittable according to Theorem~\ref{thm:3:character-factorization}. For such $\theta$, let $(\theta^{(1)}, \ldots, \theta^{(m)})$ be its $m$-split. We have the following equality derived from Theorem~\ref{thm:3:character-factorization}:
\begin{displaymath}
\chi_{m\mu}^{\theta} = z_{\mu} \sgn_\theta \sum_{\mu^{(1)} \uplus \cdots \uplus \mu^{(m)} = \mu} \prod_{i=1}^m \frac{\chi_{\mu^{(i)}}^{\theta^{(i)}}}{z_{\mu^{(i)}}}, \quad \chi_{[m^n]}^{\theta} = n! \sgn_\theta \prod_{i=1}^m \frac{f^{\theta^{(i)}}}{|\theta^{(i)}|!}.
\end{displaymath}
The formula for $[m^n]$ is drastically simplified, because in this case $\lambda = [1^n]$. In the sum over multiset partitions $\lambda = \lambda^{(1)} \uplus \cdots \uplus \lambda^{(m)}$, the only way to have a non-zero contribution is that $|\lambda^{(i)}| = |\theta^{(i)}|$. Since $\lambda=[1^n]$ is formed only by parts of size $1$, we thus only need to consider $\lambda^{(i)}$ formed by parts of size $1$. All the $z_\lambda^{(i)}$'s are now just factorials, and we know that $\chi^{\mu}_{[1^n]} = f^\mu$ for any partition $\mu \vdash n$. We thus have the formula for $[m^n]$.

We observe that the only unknown factor in the equalities above is $\sgn_\theta$, which is never made explicit. We only know that it depends only on $\theta$. However, in the expression of $R_H$, the two characters come in the product $\chi_{m\mu}^{\theta} \chi_{[m^n]}^{\theta}$, which cancels out the sign $\sgn_\theta$.

We now use the equalities for the two characters and Lemma~\ref{lem:3:from-Hm-back-to-H} to rewrite the expression of $R_H$ in Corollary~\ref{prop:3:series-in-big-characters-simplified} in order to factorize $R_H$ into a product of $R_C$ evaluated on different points as follows:
\begin{align*}
&\quad R_H(x,\varvec{y},t)\\
&= \sum_{n \geq 0} \frac{t^n}{n!} \sum_{\mu \vdash n} m^{-\ell(\mu)} z_{\mu}^{-1} y_{\mu} \sum_{\theta \vdash mn} \chi_{[m^n]}^{\theta} \chi_{m\mu}^{\theta} H_{\theta}(x) \\
&= \sum_{n \geq 0} \frac{t^n}{n!} \sum_{\substack{\mu \vdash n \\ \mu^{(1)} \uplus \cdots \uplus \mu^{(m)} = \mu}} m^{-\ell(\mu)} z_{\mu}^{-1} y_{\mu} \\
&\quad \quad \quad \sum_{\theta \vdash mn} z_\mu \left( \prod_{i=1}^m \frac{\chi^{\theta^{(i)}}_{\mu^{(i)}}}{z^{\mu^{(i)}}} \right) n! \left( \prod_{i=1}^m \frac{f^{\theta^{(i)}}}{|\theta^{(i)}|} \right) m^{mn} \left( \prod_{i=1}^m \prod_{j=0}^{m-1} H_{\theta^{(i)}} \left( \frac{x-i+j+1}{m} \right) \right) \\
&= \sum_{n \geq 0} (m^m t)^n \sum_{\mu \vdash n} m^{-\ell(\mu)} y_\mu \sum_{\substack{\theta \vdash mn \\ \mu^{(1)} \uplus \cdots \uplus \mu^{(m)} = \mu}} \prod_{i=1}^m \left( \frac{f^{\theta^{(i)}} \chi_{\mu^{(i)}}^{\theta^{(i)}}}{(|\theta^{(i)}|)! z_{\mu^{(i)}}}  \prod_{j=0}^{m-1} H_{\theta^{(i)}} \left( \frac{x-i+j+1}{m} \right) \right) \\
&= \sum_{\mu^{(1)} \vdash n_1, \ldots, \mu^{(m)} \vdash n_m} \prod_{i=1}^m  \sum_{n_i \geq 0} \frac{(m^m t)^{n_i}}{n_i!} m^{-\ell(\mu^{(i)})} y_{\mu^{(i)}} z_{\mu^{(i)}}^{-1} \sum_{\theta^{(i)} \vdash n_i} f^{\theta^{(i)}} \chi_{\mu^{(i)}}^{\theta^{(i)}} \prod_{j=1}^{m} H_{\theta^{(i)}} \left( \frac{x-i+j}{m} \right) \\
&= \prod_{i=1}^{m} R_C \left( \left[ x_j \gets \frac{x-i+j}{m} \right], \left[ y_j \gets \frac{y_j}{m} \right], m^m t \right) \\
&= \prod_{j=1}^{m} R_C \left( \left[ x_i \gets \frac{x-j+i}{m} \right], \left[ y_i \gets \frac{y_i}{m} \right], m^m t \right). \qedhere
\end{align*}
\end{proof}

This relation between $R_H$ and $R_C$ can be translated directly into a relation between the series $H(x, \varvec{y}, t, u)$ of $m$-hypermaps and the series $C(\varvec{x}, \varvec{y}, t, u)$ of $m$-constellations, resulting in our main result as follows.

\begin{thm}[Relations of series of constellations and hypermaps] \label{thm:3:link-between-H-and-C}
The generating functions of $m$-constellations and $m$-hypermaps are related by the following formula:
\[ H(x,\varvec{y},t,u) = m \sum_{j=1}^{m} C \left( \left[ x_i \gets \frac{x+(i-j)u}{m} \right], \left[ y_i \gets \frac{y_i}{m} \right], m^{m-1} t, u \right). \]
\end{thm}
\begin{proof}
This comes directly from a substitution of \eqref{eq:3:H-to-RH} and \eqref{eq:3:C-to-RC} into Proposition~\ref{prop:3:link-of-hypermaps-and-constellation-series}. Note that the product of $R_C$ is turned into a sum of $C$ by taking the logarithm.
\end{proof}

We define $H^{(g)}(x,\varvec{y},t) = [u^{2g}]H(x,\varvec{y},t,u)$ and $C^{(g)}(\varvec{x},\varvec{y},t) = [u^{2g}]C(\varvec{x},\varvec{y},t,u)$ to be respectively the OGFs of $m$-hypermaps and $m$-constellations of genus $g$. We have the following corollary concerning $m$-hypermaps and $m$-constellations with given genus.

\begin{coro} \label{coro:3:link-hypermap-constellation-with-genus}
We have the following relation between the generating functions $H^{(g)}$ and $C^{(g)}$:
\[ H^{(g)}(x,\varvec{y},t) = \sum_{k=0}^{g} \frac{m^{2g-2k}}{m(2k)!} \Bigg( \sum_{j=1}^{m} \bigg(\sum_{i=1}^{m} (i-j)\frac{\partial}{\partial x_i} \bigg)^{2k} C^{(g - k)} \Bigg)([x_i \gets x], \varvec{y}, t). \]
\end{coro}
\begin{proof}
We compute $H^{(g)}(x,\varvec{y},t) = [u^{2g}]H(x,\varvec{y},t,u)$ directly as follows:
\begin{align*}
&\quad [u^{2g}]H(x,\varvec{y},t,u)  \\
&= m \sum_{j=1}^{m} \sum_{k=0}^{g} [u^{2k}]C^{(g - k)}\left( \left[ x_i \gets \frac{x+(i-j)u}{m} \right], \left[ y_i \gets \frac{y_i}{m} \right], m^{m-1} t \right) \\
&= m \sum_{j=1}^{m} \sum_{k=0}^{g} \frac{1}{(2k)!} \left( \frac{\partial}{\partial u} \right)^{2k} C^{(g - k)}\left( \left[ x_i \gets \frac{x+(i-j)u}{m} \right], \left[ y_i \gets \frac{y_i}{m} \right], m^{m-1} t \right) \bigg|_{u=0} \\
&= m \sum_{j=1}^{m} \sum_{k=0}^{g} \frac{1}{(2k)!} \Bigg( \bigg(\sum_{i=1}^{m} \frac{i-j}{m}\frac{\partial}{\partial x_i} \bigg)^{2k} C^{(g - k)} \Bigg) \left( \left[ x_i \gets \frac{x}{m} \right], \left[ y_i \gets \frac{y_i}{m} \right], m^{m-1} t \right). \\
\end{align*}
To obtain the final form, we then simplify the formula above with the fact that each term in $C^{(g)}$ has the form $x_1^{v_1} \cdots x_m^{v_m} y_{\phi} t^{f_2}$ with $v_1 + \cdots +v_m - mf_2 + |\phi| + f_2 = 2 - 2g$, according to the Euler formula.
\end{proof}

These results can be further generalized. Let $D$ be a subset of $\naturals^{*}$. We define $(m,D)$-hypermaps and $(m,D)$-constellations as $m$-hypermaps and $m$-constellations with the restriction that every hyperface has degree $md$ for some $d \in D$. We denote by $H_D(x,\varvec{y},t,u)$ and $C_D(\varvec{x},\varvec{y},t,u)$ their generating functions respectively. We have the following corollary.

\begin{coro}[Main result in the form of series] \label{coro:3:link-hypermap-constellation-with-genus-extended}
We have the following equations:
\[ H_D(x,\varvec{y},t,u) = m \sum_{j=1}^{m} C_D\left( \left[ x_i \gets \frac{x+(i-j)u}{m} \right], \left[ y_i \gets \frac{y_i}{m} \right], m^{m-1} t, u \right) \]
\[ H_D^{(g)}(x,\varvec{y},t) = \sum_{k=0}^{g} \frac{m^{2g-2k}}{m(2k)!} \Bigg( \sum_{j=1}^{m} \bigg(\sum_{i=1}^{m} (i-j)\frac{\partial}{\partial x_i} \bigg)^{2k} C_D^{(g - k)} \Bigg)([x_i \gets x], \varvec{y}, t). \]
\end{coro}

\begin{proof}
By specifying $y_i=0$ for $i \notin D$ in Corollary~\ref{coro:3:link-hypermap-constellation-with-genus}, we obtain our result.
\end{proof}

By taking $m=2$ and $D=\{ 2 \}, \{ p \}$ or $D$ arbitrary, we recover the quadrangulation relation in \cite{JV1990a} and its extensions in \cite{JV1990b} and \cite{jackson1999combinatorial} respectively. We define $C^{(g, a_1, \ldots, a_{m-1})}_{n,m,D}$ to be the number of rooted $m$-constellations with $n$ hyperedges, and hyperface degree restricted by the set $D$, with $a_i$ marked vertices of color $i$ for $i$ from $1$ to $m-1$. The number $H^{(g)}_{n,m,D}$ is the counterpart for rooted $m$-hypermaps without markings. These numbers can be obtained from corresponding generating functions by extracting appropriate coefficients evaluated at $y_i=1$.

According to Theorem 3.1 in \cite{chapuy2009asymptotic}, the number $C^{(g)}_{n,m,D} = C^{(g,0,\ldots,0)}_{n,m,D}$ of $(m,D)$-constellations with $n$ hyperedges without marking grows asymptotically in $ \Theta(n^{\frac{5}{2}(g-1)} \rho_{m,D}^{n})$ when $n$ tends to infinity in multiples of $\operatorname{gcd}(D)$ for some $\rho_{m,D}>0$. Using Corollary~\ref{coro:3:link-hypermap-constellation-with-genus-extended}, we now give a new proof of Theorem 3.2 of \cite{chapuy2009asymptotic} about the asymptotic behavior of the number of $(m,D)$-hypermaps.

\begin{coro}[Asymptotic behavior of $(m,D)$-hypermaps]
For a fixed $g$, when $n$ tends to infinity, we have the following asymptotic behavior of $(m,D)$-hypermaps:
\[ H^{(g)}_{n,m,D} \sim m^{2g} C^{(g)}_{n,m,D}. \]
\end{coro}
\begin{proof}
We observe that, in the second part of Corollary~\ref{coro:3:link-hypermap-constellation-with-genus-extended}, for a fixed $k$, the number of differential operators applied to $C_D^{(g-k)}$ does not depend on $n$, and they are all of order $2k$. Since in an $m$-hypermap, the number of vertices with a fixed color $i$ is bounded by the number of hyperedges $n$, the contribution of the term with $k = t$ is $O(n^{\frac{5}{2}(g-t-1)+2t} \rho_{m,D}^{n})=O(n^{\frac{5}{2}(g-1)-\frac{1}{2}t} \rho_{m,D}^{n})$. The dominant term is therefore given by the case $k=0$, with $C^{(g)}_{n,m,D} = \Theta(n^{\frac{5}{2}(g-1)} \rho_{m,D}^{n})$, and we can easily verify the multiplicative constant.
\end{proof}

Our generalized relation, alongside with its proof, is a refinement of the asymptotic enumerative results established in \cite{chapuy2009asymptotic} on the link between $m$-hypermaps and $m$-constellations.

\subsection{Positivity of coefficients in the expression of $H^{(g)}$} \label{apdx:3:pos}

In Corollary~\ref{coro:3:link-hypermap-constellation-with-genus}, the OGF $H^{(g)}$ of $m$-hypermaps of genus $g$ is expressed as a sum of OGFs $C^{(g-k)}$ of $m$-constellations with smaller genus applied to various differential operators. It is not obvious that this sum can be arranged into a sum with positive coefficients of all terms, but we will show that it is indeed the case. For $m=3$ and $m=4$, we have the following relations.

\begin{coro}[Generalization of the quadrangulation relation, special case $m=3,4$] \label{coro:3:counting-relation-m-3-4}
For $m=3,4$, we have
\begin{displaymath}
H^{(g)}_{n,3,D} = \sum_{i=0}^{g} 3^{2g-2i} \sum_{\ell=0}^{2i} \frac{2 \cdot 2^{\ell} + (-1)^{\ell}}{3} C^{(g-i, \ell, 2i-\ell)}_{n,3,D},
\end{displaymath}

\begin{displaymath}
H^{(g)}_{n,4,D} = \sum_{i=0}^{g} 4^{2g-2i} \sum_{\ell_1, \ell_2 \geq 0, \ell_1 + \ell_2 \leq 2i} \frac{2 (3^{\ell_1}2^{\ell_2} + (-1)^{\ell_1}2^{\ell_2})}{4} C^{(g-i,\ell_1,\ell_2,2i-\ell_1-\ell_2)}_{n,4,D}.
\end{displaymath}
\end{coro}

We notice that the coefficients are always positive integers. This is not a coincidence. In fact, by carefully rearranging terms, we can obtain the following relation, whose proof is the subject of this section.

\begin{coro}[Generalization of the quadrangulation relation, for arbitrary $m$] \label{coro:3:general-counting-relation-in-numbers}
With certain coefficients $c^{(m)}_{k_1, \ldots, k_{m-1}} \in \naturals$, we have
\begin{displaymath}
H^{(g)}_{n,m,D} = \sum_{i=0}^{g} m^{2g-2i} \sum_{\substack{k_1, \ldots, k_{m-1} \geq 0 \\ k_1 + \cdots + k_{m-1} = 2i}} c^{(m)}_{k_1, \ldots, k_{m-1}} C^{(g-i, k_1, \ldots, k_{m-1})}_{n,m,D}.
\end{displaymath}
\end{coro}

In the following, we prove these two corollaries. We start from the observation that the series $C^{(g)}(\varvec{x},\varvec{y},t)$ is symmetric in $x_i$. This can be seen algebraically from the expression of $R_C$ in Proposition~\ref{prop:3:series-in-big-characters-simplified}, or bijectively with a ``topological surgery'' that permutes the order of two consecutive colors (details are left to readers). We can thus deduce the following property of $C^{(g)}$.

\begin{prop} \label{prop:3:constellation-symmetric}
Let $k_1, k_2, \ldots, k_m$ be natural numbers and $\sigma \in S_m$ an arbitrary permutation. We have the following equality.
\[ \left( \frac{\partial^{k_1 + \cdots + k_m}}{\partial x_1^{k_1} \cdots \partial x_m^{k_m}} C^{(g)} \right)([x_i \gets x],\varvec{y},t) = \left( \frac{\partial^{k_1 + \cdots + k_m}}{\partial x_1^{k_{\sigma(1)}} \cdots \partial x_m^{k_{\sigma(m)}}} C^{(g)} \right)([x_i \gets x],\varvec{y},t) \]
\end{prop}
\begin{proof}
Since in the evaluation all $x_i$ are given the same value $x$, any permutation of variables $x_i$ in the series has no effect on the evaluation.
\end{proof}

We now define a new sequence of differential operators $D^{(2k)}$. For $m = 2p$ even, we define
\[D^{(2k)} \eqdef 2 \sum_{j=1}^{p} \left( \sum_{i=1}^{m} (i-j)\frac{\partial}{\partial x_i} \right)^{2k}.\]
For $m = 2p + 1$ odd, we define
\[D^{(2k)} \eqdef  \left( \sum_{i=1}^{m} (i-p-1)\frac{\partial}{\partial x_i} \right)^{2k} + 2\sum_{j=1}^{p} \left( \sum_{i=1}^{m} (i-j)\frac{\partial}{\partial x_i} \right)^{2k}. \]
Using these differential operators, we can rewrite the equations in Corollary~\ref{coro:3:link-hypermap-constellation-with-genus} as follows.

\begin{prop} \label{prop:3:sum-in-big-D}
We have the following equation.
\[ \left( \sum_{j=1}^{m} \left( \sum_{i=1}^{m} (i-j)\frac{\partial}{\partial x_i} \right)^{2k} C^{(g - k)} \right)([x_i \gets x], \varvec{y}, t) = \left( D^{(2k)} C^{(g - k)} \right) ([x_i \gets x], \varvec{y}, t) \]
\end{prop}
\begin{proof}
We observe that, according to Proposition~\ref{prop:3:constellation-symmetric}, for any $g$ and $j$,
\begin{align*}
&\quad \left( \left( \sum_{i=1}^{m} (i-j)\frac{\partial}{\partial x_i} \right)^{2k} C^{(g)} \right)([x_i \gets x], \varvec{y}, t) \\
&=  \left( \left( \sum_{i=1}^{m} (-(m+1-i) + j)\frac{\partial}{\partial x_i} \right)^{2k} C^{(g)} \right) ([x_i \gets x], \varvec{y}, t) \\
&=  \left( \left( \sum_{i=1}^{m} (i - (m+1-j))\frac{\partial}{\partial x_i} \right)^{2k} C^{(g)} \right) ([x_i \gets x], \varvec{y}, t).
\end{align*}
With this equality, the proposition is easily verified.
\end{proof}

We define two kinds of coefficients $e^{(m),j}_{k_1, \ldots, k_{m-1}}$ and $d^{(m)}_{k_1, \ldots, k_{m-1}}$ for $k_i \geq 0$ as follows:
\begin{align*}
e^{(m),j}_{k_1, \ldots, k_{m-1}} &\eqdef \prod_{1 \leq i \leq m, i \neq j} (i-j)^{k_{i-j \mod m}}, \\
d^{(m)}_{k_1, \ldots, k_{m-1}} &\eqdef \begin{cases} 2\sum_{j=1}^{p} e^{(m),j}_{k_1, \ldots, k_{m-1}}, &(m=2p) \\ e^{(m),p+1}_{k_1, \ldots, k_{m-1}} + 2\sum_{j=1}^{p} e^{(m),j}_{k_1, \ldots, k_{m-1}}, &(m=2p+1) \end{cases}.
\end{align*}

We can now rewrite the equation in Proposition~\ref{prop:3:sum-in-big-D} with these coefficients.

\begin{prop} \label{prop:3:rewrite-with-new-coeff}
For any $m \geq 2$, we have
\begin{align*}
&\quad (D^{(2k)} C^{(g)}) ([x_i \gets x],\varvec{y},t) \\
&= \left( \sum_{k_1 + \ldots + k_{m-1} = 2k} \binom{2k}{k_1, \ldots, k_{m-1}} d^{(m)}_{k_1, \ldots, k_{m-1}} \frac{\partial^{2k}}{\partial x_1^{k_1} \cdots \partial x_{m-1}^{k_{m-1}}} C^{(g)} \right) ([x_i \gets x],\varvec{y},t).
\end{align*}
\end{prop}
\begin{proof}
For $m=2p$,
\begin{align*}
&\quad (D^{(2k)} C^{(g)}) ([x_i \gets x],\varvec{y},t) \\
&= 2 \left( \sum_{j=1}^{p} \left( \sum_{i=1}^{m} (i-j)\frac{\partial}{\partial x_i} \right)^{2k} C^{(g)} \right) ([x_i \gets x],\varvec{y},t) \\
&= 2 \left( \sum_{j=1}^{p} \sum_{k_1 + \ldots + k_m = 2k} \binom{2k}{k_1, \ldots, k_m} \left( \prod_{i=1}^{m} (i-j)^{k_{i}} \right) \frac{\partial^{2k}}{\partial x_1^{k_1} \cdots \partial x_m^{k_m}} C^{(g)} \right) ([x_i \gets x],\varvec{y},t) \\
&= 2 \left( \sum_{k_1 + \ldots + k_m = 2k} \binom{2k}{k_1, \ldots, k_m} \left( \sum_{j=1}^{p} \prod_{i=1}^{m} (i-j)^{k_{i-j\mod m}} \right) \frac{\partial^{2k}}{\partial x_1^{k_1} \cdots \partial x_m^{k_m}} C^{(g)} \right) ([x_i \gets x],\varvec{y},t) \\
&= \left( \sum_{k_1 + \ldots + k_{m-1} = 2k} \binom{2k}{k_1, \ldots, k_{m-1}} d^{(m)}_{k_1, \ldots, k_{m-1}} \frac{\partial^{2k}}{\partial x_1^{k_1} \cdots \partial x_{m-1}^{k_{m-1}}} C^{(g)} \right) ([x_i \gets x],\varvec{y},t)
\end{align*}

The computation for $m=2p+1$ is similar.
\end{proof}

We will now show that the coefficients $d^{(m)}_{k_1, \ldots, k_{m-1}}$ are all positive integers divisible by $m$. We start with a lemma concerning $e^{(m),j}_{k_1, \ldots, k_{m-1}}$, which we will use for telescoping.

\begin{lem} \label{lem:3:telescoping-e}
For $j \leq m/2$, we have $\left| e^{(m),j+1}_{k_1, \ldots, k_{m-1}} \right| \leq \left| e^{(m),j}_{k_1, \ldots, k_{m-1}}\right| $.

Moreover, if $k_{m-j} \geq 1$, we have $\left| e^{(m),j+1}_{k_1, \ldots, k_{m-1}} \right| \leq \frac{j}{m-j} \left| e^{(m),j}_{k_1, \ldots, k_{m-1}} \right| $.
\end{lem}
\begin{proof}
The result is a direct consequence of the following formula coming from the definition of $e^{(m),j}_{k_1, \ldots, k_{m-1}}$:
\[ e^{(m),j+1}_{k_1, \ldots, k_{m-1}} = \left( \frac{-j}{m-j} \right)^{k_{m-j}} e^{(m),j}_{k_1, \ldots, k_{m-1}}. \qedhere \]
\end{proof}

We can now prove the positivity of $d^{(m)}_{k_1, \ldots, k_{m-1}}$.

\begin{thm} \label{thm:3:positivity-of-differential-operator-coefficient}
For all $m \geq 2$ and $k_1, \ldots, k_{m-1}$ natural numbers, $d^{(m)}_{k_1, \ldots, k_{m-1}}$ is a positive integer. Moreover, it is always divisible by $m$.
\end{thm}
\begin{proof}
By definition, $d^{(m)}_{k_1, \ldots, k_{m-1}}$ is an integer.

For the case $k_1 = \cdots = k_{m-1} = 0$, we have $d^{(m)}_{0, \ldots, 0} > 0$ by definition. We now suppose that all $k_i$ are not zero. Let $t$ be the largest index such that $k_t > 0$, we then have $k_{t+1} = \cdots = k_{m-1} = 0$. If $t \leq m/2 $, we have $d^{(m)}_{k_1, \ldots, k_{m-1}} > 0$, since all terms of the sum in the definition are positive. We now suppose that $t > p$ for $m=2p$ and $m=2p+1$. We always have $e^{(m),m-t}_{k_1, \ldots, k_{m-1}} \geq t^{k_t} > 0$ in this case.

We start from the case $m=2p$. According to Lemma~\ref{lem:3:telescoping-e}, we have
\begin{align*}
d^{(m)}_{k_1, \ldots, k_{m-1}} &= 2\sum_{j=1}^{p} e^{(m),j}_{k_1, \ldots, k_{m-1}} \\
&\geq 2(m-t)e^{(m),m-t}_{k_1, \ldots, k_{m-1}} - 2(t-p) \left| e^{(m),m-t+1}_{k_1, \ldots, k_{m-1}} \right| \\
&\geq 2(m-t) \left( e^{(m),m-t}_{k_1, \ldots, k_{m-1}} - \left( 1-\frac{p}{t} \right) e^{(m),m-t}_{k_1, \ldots, k_{m-1}} \right) \\
&= \frac{2p(m-t)}{t}e^{(m),m-t}_{k_1, \ldots, k_{m-1}} > 0.
\end{align*}

The computation for the case $m=2p+1$ is similar.
\begin{align*}
d^{(m)}_{k_1, \ldots, k_{m-1}} &= e^{(m),p+1}_{k_1, \ldots, k_{m-1}} + 2\sum_{j=1}^{p} e^{(m),j}_{k_1, \ldots, k_{m-1}} \\
&\geq 2(m-t)e^{(m),m-t}_{k_1, \ldots, k_{m-1}} - (2t-2p+1) \left| e^{(m),m-t+1}_{k_1, \ldots, k_{m-1}} \right| \\
&\geq (m-t) \left( 2e^{(m),m-t}_{k_1, \ldots, k_{m-1}} - \left( 2-\frac{2p-1}{t} \right) e^{(m),m-t}_{k_1, \ldots, k_{m-1}} \right) \\
&= \frac{(2p-1)(m-t)}{t} e^{(m),m-t}_{k_1, \ldots, k_{m-1}} > 0
\end{align*}

The positivity to be proved follows from the computation above.

For divisibility by $m$, we only need to observe that the value of $e^{(m),j}_{k_1, \ldots, k_{m-1}}$ modulo $m$ does not depend on $j$, and that $d^{(m)}_{k_1, \ldots, k_{m-1}}$ is the sum of $m$ such coefficients.
\end{proof}

We denote by $C^{(g, k_1, \ldots, k_{m-1})}_{n,m,D}(\varvec{y})$ the OGF of $m$-constellations of genus $g$ with $n$ hyperedges, $k_i$ marked vertices of each color $i$ and $D$ as restriction on degree of hyperfaces. We define similarly $H^{(g)}_{n,m,D}(\varvec{y})$. They are the series versions of numbers $C^{(g, k_1, \ldots, k_{m-1})}_{n,m,D}$ and $H^{(g)}_{n,m,D}$, with $y_i$ marking the degree of each hyperfaces.

We can now rewrite Corollary~\ref{coro:3:link-hypermap-constellation-with-genus-extended} into a more agreeable form.

\begin{coro} \label{coro:3:explicit-coeffs}
We have the following relation, with all coefficients positive integers.
\begin{displaymath}
H^{(g)}_{n,m,D}(\varvec{y}) = \sum_{i=0}^{g} m^{2g-2i} \sum_{\substack{k_1, \ldots, k_{m-1} \geq 0 \\ k_1 + \cdots + k_{m-1} = 2i}} c^{(m)}_{k_1, \ldots, k_{m-1}} C^{(g-i, k_1, \ldots, k_{m-1})}_{n,m,D}(\varvec{y})
\end{displaymath}
Here, $c^{(m)}_{k_1, \ldots, k_{m-1}} = m^{-1} d^{(m)}_{k_1, \ldots, k_{m-1}}$.
\end{coro}

\begin{proof}
Since according to Theorem~\ref{thm:3:positivity-of-differential-operator-coefficient}, $d^{(m)}_{k_1, \ldots, k_{m-1}}$ is a positive integer divisible by $m$, $c^{(m)}_{k_1, \ldots, k_{m-1}}$ is always a positive integer.

The corollary now follows directly from Corollary~\ref{coro:3:link-hypermap-constellation-with-genus} and Proposition~\ref{prop:3:sum-in-big-D}, \ref{prop:3:rewrite-with-new-coeff}, with the observation that successive derivation of $x_i$ means marking vertices of color $i$ with order, and the restriction imposed by $D$ can be established by specifying $y_i=0$ for any $i \notin D$.
\end{proof}

We have thus proved a generalization of Corollary~\ref{coro:3:general-counting-relation-in-numbers}. The corollary above actually tells us that Corollary~\ref{coro:3:general-counting-relation-in-numbers} holds even if we refine by the degree profile of hyperfaces. We can obtain Corollary~\ref{coro:3:general-counting-relation-in-numbers} by specifying all $y_i$ to $1$. 

\horizrule

Our result could hint at a combinatorial bijection between $m$-hypermaps and some families of $m$-constellations with markings that preserves the degree sequence of hyperfaces, although this might be hindered by the fact that the combinatorial meaning of the coefficient $c^{(m)}_{k_1, \ldots, k_{m-1}}$ itself is not clear. A better understanding of these coefficients is indispensable in the quest of the hinted bijection.

\chapter{Enumerating constellations}

In Section~\ref{sec:2:sym}, we have mentioned the importance of $m$-constellations as a model of factorizations of the identity in the symmetric group, since it can serve as a unified framework of maps, classical and monotone Hurwitz numbers. We are thus interested in their enumeration, especially in higher genera, which can lead to a unified formulation of similar enumeration results such as those in \cite{classical-hurwitz, GGPN}. 

Although we are mostly interested by $m$-constellations in higher genera as a unified framework, the planar case is also of particular interest. The enumeration of constellations was first started for the planar case, and an enumeration formula of planar $m$-constellations was obtained long ago in \cite{BMS} using bijective methods. However, the same result was never obtained in full generality by solving functional equations. We would like to revisit this case by resolving a functional equation using a method dubbed the ``differential-catalytic'' method, originally developed in \cite{BMCPR2013representation} for enumeration of intervals in the $m$-Tamari lattice. Although $m$-constellations and intervals in the $m$-Tamari lattice seem to be of very different nature, their generating functions are in fact both governed by functional equations with unlimited repeated iterations of an operator, which are difficult to solve in all generality due to the arbitrary number of iterations. By this resolution, we want to demonstrate the potential versatility of the differential-catalytic method and fathom its potential to be generalized. It is also worth noting that there are curious bijective links from intervals in the Tamari lattice and its generalizations to different classes of planar maps, which will be studied in the next chapter. We thus have an intriguing question: Why these intervals are so deeply connected with planar maps? A deeper understanding from the comparison of functional equations for the two classes in the lens of the differential-catalytic method may help us to answer that question.

This chapter concerns the enumeration of $m$-constellations of a fixed genus $g$ by solving functional equations satisfied by their OGFs. It consists of three parts. In the first part, we will explain a Tutte decomposition of $m$-constellations and extract a functional equation satisfied by their generating functions from this decomposition, for both the planar case and the higher genus case. In the second part, we will solve the planar case $g=0$ for arbitrary $m$, using the differential-catalytic method from \cite{BMCPR2013representation}. In the third part, we will solve the functional equation for arbitrary genus $g$ in the case $m=2$ (bipartite maps) using an inductive method inspired by the \emph{topological recursion} method. 

This chapter is partially based on an unpublished work and the article \cite{cf-bipartite}, both in collaboration with Guillaume Chapuy. Related definitions about constellations and generating functions can be found in Chapter~2. Since we will deal with generating functions of all sorts in this chapter, readers are referred to Table~\ref{tab:2:series} on page~\pageref{tab:2:series} for notation, which will be used in a nested way.

\section{Functional equations}

For $f$ a hyperface in an $m$-constellation, we say that $f$ is an \mydef{internal hyperface} if $f$ is not the outer hyperface. We denote by $F_{m,g} = F_{m,g}(t,x;p_1, p_2, \ldots)$ the OGF of $m$-constellations of genus $g$ with $t$ marking the number of hyperedges, $x$ marking the degree of the outer hyperface divided by $m$ and $p_k$ marking the number of internal hyperfaces of degree $mk$. In this section, we aim to write a functional equation for $F_{m,g}$. Every $F_{m,g}$ is in the ring $\mathbb{Q}[x,p_1,p_2,\ldots][[t]]$ of formal power series with coefficients in $\mathbb{Q}[x,p_1,p_2,\ldots]$, since the total degree of hyperfaces is bounded by $m$ times the number of hyperedges.

We start with the planar case $g=0$. As always, we follow the philosophy of Tutte: study how a constellation breaks into smaller constellations by the removal of a building block. However, simply removing the root edge does not work due to the strict structural constraint in the definition. For constellations, we will instead remove the \emph{root hyperedge}. Such removal will break a planar $m$-constellation into several connected components, and it is easy to verify by definition that every connected component is an $m$-constellation. Figure~\ref{fig:4:planar-removal} gives an example scheme of how removing the root hyperedge of a planar $6$-constellation breaks it into smaller constellations, which we call \emph{components}. We allow the empty constellation as a component. Then, for each component, we construct its set of ``touching colors'', which are the colors of vertices by which the component is connected to the original root hyperedge. By collecting the set of touching colors of all components, we thus obtain a set family of colors from $1$ to $m$, which is clearly a non-crossing partition, \textit{i.e.}, a set partition $\mathcal{P}$ of the set $\{ 1, 2, \ldots, m \}$ such that, for any $P_1, P_2 \in \mathcal{P}$, and for any $i,k \in P_1$ and $j,\ell \in P_2$ such that $i < k$ and $j < \ell$, we can never have $i<j<k<\ell$. The reason is that every such quadruple $(i,j,k,\ell)$ will introduce a crossing of two components in the original constellation, which violates the planarity. Since there are finitely many non-crossing partitions with given $m$, in principle we can write the functional equation as a sum of terms indexed by these non-crossing partitions.

\begin{figure}
  \centering
  \insertfigure[0.85]{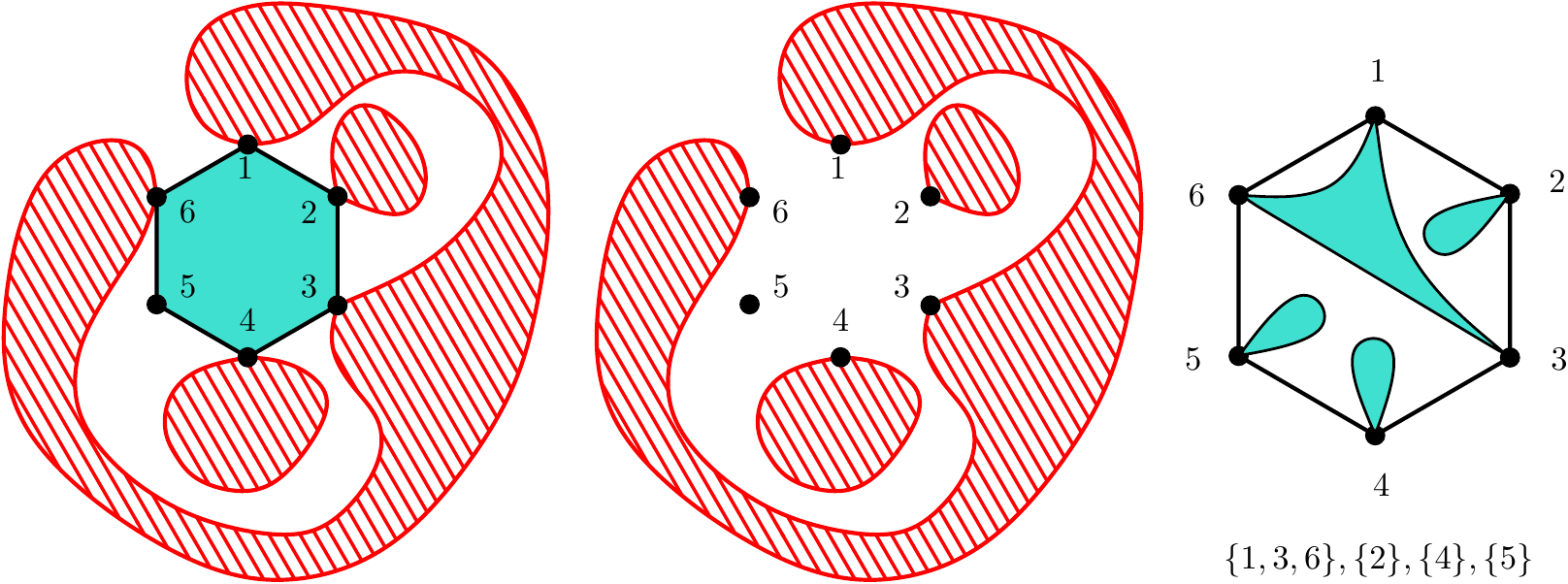}{1}
  \caption{Example of root hyperedge removal on a $6$-constellation}
  \label{fig:4:planar-removal}
\end{figure}

This is exactly the approach taken by Bousquet-M\'elou and Jehanne in \cite{BMJ}. They showed that it is possible to write a functional equation with one catalytic variable for planar $m$-constellations with arbitrary $m$, and with the powerful machinery they established in the same paper, they proved that the OGF $F_{m,0}(t,x;1,1,\ldots)$ of planar $m$-constellations (without refinement on hyperface degrees) is algebraic, and their approach can be easily generalized to prove the algebraicity of $F_{m,0}(t,x;p_1,\ldots,p_K)$ which restricts the degree of hyperfaces to be at most $mK$. However, functional equations obtained in this way are too cumbersome for exact solution. As an evidence, in \cite{BMJ}, only equations of $2$- and $3$-constellations are written and solved explicitly. If we want to solve for the generating function for all $m$, it seems that we need a simpler equation.

We now look at the root hyperedge removal in greater detail. Instead of removing the root hyperedge as a whole, we try to detach its vertices one by one from other parts of the constellation, starting with the vertex of color $m$ and ending with that of color $1$. In the inverse direction, to construct a constellation, we first construct the root hyperedge, then attach its vertices from color $1$ to $m$ one by one to either a new component or an existing component when vertex coloring and planarity permit. Figure~\ref{fig:4:planar-constr} illustrates an example of such a construction. Suppose that we are about to attach the vertex $v_i$ of color $i$ of the root hyperedge. We can see that there are only two basic operations: 
\begin{itemize}
\item \textbf{Planar join case}: take a new planar $m$-constellation as a new component, and attach $v_i$ on the root hyperedge to the next corner on the outer hyperface of the new component of color $i$ in clockwise order, starting from the old root corner of the new component, in the unique way that preserves the orientation;
\item \textbf{Cut case}: attaching $v_i$ to another corner of the same color of the outer hyperface. 
\end{itemize}
In the planar join case, we join the outer hyperface of the new component with the existing outer hyperface, while in the cut case, we cut the outer hyperface into two to create a new hyperface. Under the chosen order of attaching vertices, to guarantee planarity in the cut case, we must identify the current vertex with a vertex of the same color adjacent to the outer hyperface. The planar join case, on the other hand, always works.

\begin{figure}
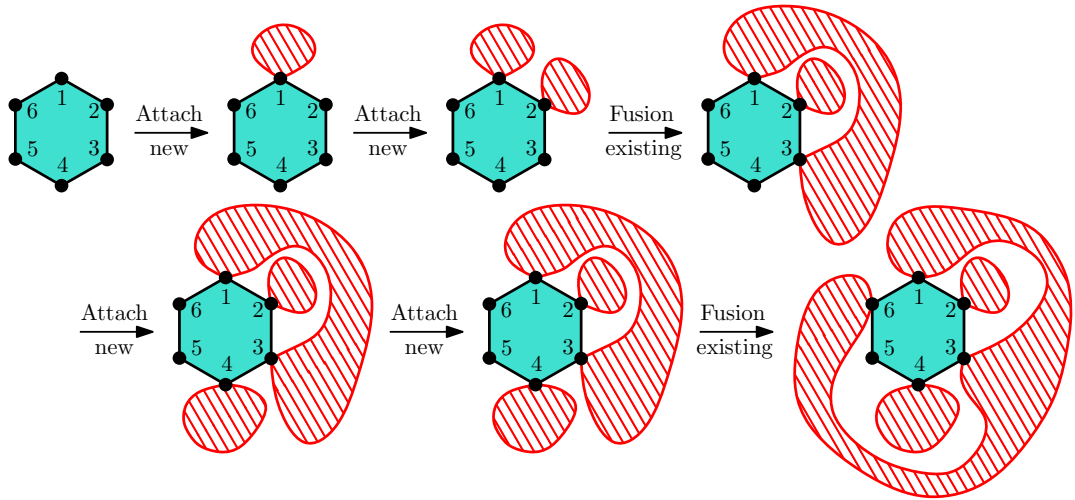

  \centering
  \insertfigure[0.85]{ch4-fig.pdf}{2}
  \caption{Construction of a constellation, step by step}
  \label{fig:4:planar-constr}
\end{figure}

We now write a functional equation for $F_{m,0}$ by translating this procedure of vertex attaching into operations on OGFs. For intermediate objects during vertex attaching, we use the same variable scheme as $m$-constellations for their OGF, except that $x$ now marks the degree of the outer hyperface divided by $m$ then \textbf{minus 1}, and $t$ marks the number of hyperedges \textbf{minus 1}. This modification simplifies the treatment of the cut case. Suppose that we are about to attach the vertex of color $i$, and the current outer hyperface is of degree $md$. Then there are exactly $d-1$ possible existing corners to attach, since we cannot attach the new vertex to itself. This fact is accounted by the ``minus 1''. Similarly, until all vertices are properly attached, we do not account for the new root hyperedge. It is after all attachings that we make up the missing factor $xt$. The initial state is thus represented by $1$. We now translate the two basic operations into operators on formal power series. 
\begin{itemize}
\item \textbf{Planar join case}: In this case, all statistics are additive. Therefore, the operator for this case is simply the multiplication by $F_{m,0}$. 
\item \textbf{Cut case}: Suppose that we have an intermediate object with weight $x^k p_\lambda$, whose outer hyperface has degree $d(k+1)$. We have $k$ choices to attach the new vertex, and in all choices the sum of the degrees of the new hyperface and the new outer hyperface will be $d(k+1)$. Each choice gives a different distribution of degrees among the two hyperfaces. Therefore, the weights of the $k$ resulting objects are $x^{k-1} p_1 p_\lambda, x^{k-2} p_2 p_\lambda, \ldots, x^0 p_k p_\lambda$. Here $x^0$ is allowed because it stands for an outer hyperface of degree $m$, which is a perfectly valid situation. We thus define the following linear operator $\Omega$ on power series of $x$ by
\[ \forall k \geq 1, \Omega x^k = \sum_{i=1}^k x^{k-i} p_i. \]
\end{itemize}
Figure~\ref{fig:4:const-planar-case} illustrates the two cases. The only planar $m$-constellation that cannot be constructed in this way is the ``empty'' constellation with weight $1$. We thus have the following functional equation for $F_{m,0}$:

\begin{thm}
The generating function $F_{m,0}$ of planar $m$-constellations defined previously satisfies the following functional equation:
\begin{equation} \label{eq:4:planar-const}
F_{m,0} = 1 + xt(F_{m,0} + \Omega)^m (1).
\end{equation}
\end{thm}
\begin{proof}
A planar $m$-constellation is either empty or not. The contribution of the only empty constellation is $1$. For the contribution of other planar constellations, we consider its construction starting from the root hyperedge by either attaching a new component or cutting out a new face using each of the $m$ vertices adjacent to the root hyperedge. The operation on each vertex can be described as the operator $F_{m,0}+\Omega$. Starting from $1$ that stands for the root hyperedge before all the operations, we apply $m$ times the operator $F_{m,0}+\Omega$, each standing for the operation on one of the $m$ vertices, and finally we multiply by $xt$ to rectify variable markings of statistics to obtain the contribution of non-empty planar constellations. We thus have \eqref{eq:4:planar-const}.
\end{proof}

\begin{figure}
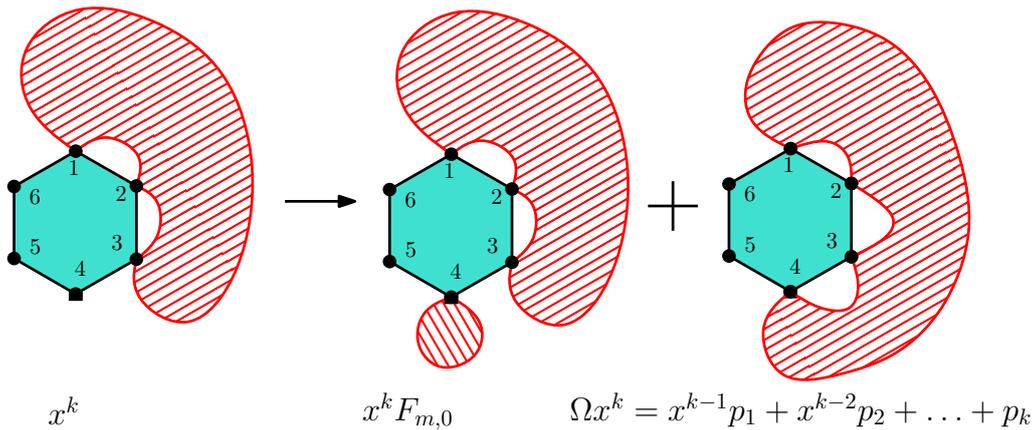

  \centering
  \insertfigure[0.85]{ch4-fig.pdf}{3}
  \caption{Two ways of attaching a new vertex}
  \label{fig:4:const-planar-case}
\end{figure}

We now consider the higher genus case $g > 0$. We proceed by the same vertex attaching procedure. Suppose that we are about to attach the vertex $v_i$ of color $i$ of the new root hyperedge. Let $C^{(i)}$ be the object we obtain after attachment. The detachment of $v_i$ in $C^{(i)}$ either disconnects the object into two parts, or it doesn't. When it doesn't, according to whether the two edges adjacent to both $v_i$ and the new root hyperedge are adjacent to the same hyperface or not, the detachment can split a hyperface into two or fusion two into one. In the reverse direction, we have the following cases for vertex attaching.
\begin{itemize}
\item \textbf{Separating join case}: We take an $m$-constellation $C$ of arbitrary genus as a new component, and attach $v_i$ to the next corner of color $i$ of the outer hyperface of $C$, starting from the root corner of $C$ in clockwise order.
\item \textbf{Cut case}: We attach $v_i$ to another corner of the outer hyperface with the same color, which splits the current outer hyperface into a new internal hyperface and the new outer hyperface.
\item \textbf{Non-separating join case}: We attach $v_i$ to another corner with the same color of an internal hyperface $f$. This attachment will merge $f$ with the outer hyperface.
\end{itemize}

We now translate the three cases into operators. The separating join case is similar to the planar join case for planar constellations, and we only need to pay attention to the fact that the genus of the new component now adds to the total genus of the $m$-constellation. The cut case is exactly the same as in that for planar constellations. The only new case is the non-separating join case illustrated in Figure~\ref{fig:4:const-handle-case}, whose attaching procedure can be separated into several steps: first choose the degree $mk$ of the internal hyperface to merge, then choose an internal hyperface $f$ of such degree, finally attach the vertex to one of the $k$ possible corners of $f$. The choice of $f$ can be translated into the pointing construction related to variable $p_k$, which gives the operator $\frac{p_k\partial}{\partial p_k}$, and the merging of $f$ with the outer hyperface translates to the multiplication by $x^k p_k^{-1}$. The whole procedure in the non-separating join case is thus translated into the operator
\begin{equation} \label{eq:4:Gamma-def}
\Gamma \eqdef \sum_{k \geq 1} kx^k \frac{\partial}{\partial p_k}.
\end{equation}
We also notice that the non-separating join case adds $1$ to the genus, since it requires an extra ``handle'' for the attachment.

\begin{figure}
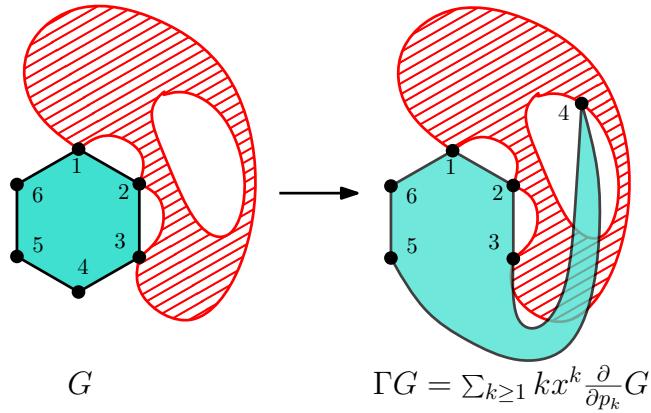

  \centering
  \insertfigure[0.85]{ch4-fig.pdf}{4}  
  \caption{Non-separating join case in the construction of constellations of higher genus}
  \label{fig:4:const-handle-case}
\end{figure}

We now write a functional equation for $F_{g,m}$ with $g > 0$ that depends on OGFs $F_{g',m}$ in lower genus $g' < g$. We use an extra variable $w$ to control the genus, and we do not count the ``empty'' constellation, since we consider it to be planar. Under the previous case study, we thus have the following theorem:

\begin{thm}
The generating function $F_{g,m}$ of $m$-constellations of genus $g$ satisfies the functional equation
\begin{equation} \label{eq:4:genus-const}
F_{g,m} = xt [w^{g}] \left( \sum_{0 \leq g' \leq g} w^{g'} F_{g',m} + \Omega + w\Gamma \right)^m (1).
\end{equation}
\end{thm}

We notice that the variable $w$ is only for simplifying the expression, and in principle we can write a functional equation without using $w$ by exhausting all possible cases of making up the correct genus. For smaller $m$, such as $m=2,3$, this can be done by hand:
\begin{align*}
F_{g,2} &= xt \left( \sum_{g_1+g_2=g} F_{g_1,2} F_{g_2,2} + \Omega F_{g,2} + \Gamma F_{g-1,2} \right), \\
F_{g,3} &= xt \Bigg( \sum_{g_1+g_2+g_3=g} F_{g_1,3} F_{g_2,3} F_{g_3,3} + \sum_{g_1+g_2=g} \left( \Omega(F_{g_1,3} F_{g_2,3}) + F_{g_1,3} (\Omega F_{g_2,3}) \right) \\
&+ \sum_{g_1+g_2=g-1} \left( \Gamma(F_{g_1,3} F_{g_2,3}) + F_{g_1,3} (\Gamma F_{g_2,3}) \right) + \Omega \Omega F_{g,3} + (\Omega \Gamma + \Gamma \Omega) F_{g-1,3} + \Gamma \Gamma F_{g-2,3} \Bigg).
\end{align*}

As a remark, we notice that we can also keep track of the number of vertices of each color in the functional equations with slight modifications. Let $F_{g,m}^*$ be the generating function of $m$-constellations with $c_i$ marking the number of vertices of color $i$. We observe that a new vertex is created only in the case where we attach the empty constellation. Therefore, by separating the contribution empty constellation from the generating function of planar constellations and replacing it with $c_i$, we obtain the following functional equations for $F_{g,m}^*$:
\begin{align*}
F_{0,m}^* &= 1 + xt \left( \prod_{i=1}^m \left( F_{0,m}^* - 1 + c_i + \Omega \right) \right) (1), \\
F_{g,m}^* &= xt [w^{g}] \left( \prod_{i=1}^m \left( F_{0,m}^* - 1 + c_i + \sum_{1 \leq g' \leq g} w^{g'} F_{g',m} + \Omega + w\Gamma \right) \right) (1).
\end{align*}

For the resolution, we first observe that $F_{m,g}$ contains an infinite number of variables $p_1, p_2, \ldots$, which complicates the resolution. Therefore, we consider a restriction on degrees of internal hyperfaces. Let $K$ be an arbitrary positive integer. In the rest of this chapter, without additional indication, we will only consider $m$-constellations whose internal hyperfaces are all of degree \textbf{at most} $mK$. The OGF of these constellations is $F_{m,g}(t,x;p_1, p_2, \ldots, p_K, 0, 0, \ldots) = F_{m,g}|_{p_k = 0 \; \mathrm{for} \; k>K}$. We denote this restricted OGF by $F_{m,g,K}$. We notice that $F_{m,g}$ is the projective limit for $K \to +\infty$ of $(F_{m,g,K})_{K \geq 1}$. Therefore, we only need to resolve for $F_{m,g,K}$ for all $K$ to obtain $F_{m,g}$ as projective limit.

In the following, we fix the value of $K$, and to simplify the notation, we still denote by $F_{m,g}$ the OGF of degree-restricted $m$-constellations of genus $g$ instead of $F_{m,g,K}$. We denote by $\varvec{p}$ the sequence of variables $p_1,\ldots,p_K$. Since $p_k=0$ for $k>K$, we can rewrite the operator $\Omega$ as
\begin{equation} \label{eq:4:Omega-restricted}
\Omega: F \mapsto [x^{\geq 0}](\theta F) \; \textrm{where} \; \theta = \sum_{k=1}^K p_k x^{-k}.
\end{equation}

We now introduce a change of variables $(t,x) \leftrightarrow (z,u)$, implicitly defined by
\begin{equation} \label{eq:4:tx-zu}
t = z\left( 1+ \sum_{k = 1}^K \binom{mk-1}{k} p_k z^k \right)^{1-m}, \quad x = u(1+uz)^{-m}.
\end{equation}
These equations are closely related to equations satisfied by the family of mobiles related to $m$-constellations (\textit{cf.} \cite{BDFG}). The first equation uniquely determines a formal power series $z \equiv z(t) \in \mathbb{Q}[\varvec{p}][[t]]$. Similarly, the second uniquely determines $u \equiv u(t,x) \in \mathbb{Q}[x,\varvec{p}][[t]]$. This change of variable is clearly reversible by $t \equiv t(z)$ and $x \equiv x(z,u)$ as in Equation~\eqref{eq:4:tx-zu}. We also note that, for any ring $\mathbb{B}$ containing $\varvec{p}$, if $H \equiv H(t,x) \in \mathbb{B}[x][[t]]$ is a formal power series in $t$ with polynomial coefficients in $x$, then $H(t(z),x(z,u))$ is a member of $\mathbb{B}[u][[z]]$, and \textit{vice versa}. In the rest of this chapter, we will \emph{abuse notation} by switching back and forth between a member $H(t,x)$ of $\mathbb{B}[x][[t]]$ and its image in $\mathbb{B}[u][[t]]$ by the change of variable \emph{without further warning}. We will use the letter $H$ for both objects, and rely on the context (including the names of arguments) that should lead to no ambiguity. We will also omit $\varvec{p}$ in the arguments of a series. 

We further introduce another quantity $\gamma$:
\begin{equation} \label{eq:4:gamma}
\gamma \eqdef \sum_{k=1}^K \binom{mk-1}{k} p_k z^k.
\end{equation}
We notice that $t = z(1+\gamma)^{1-m}$.

We are now prepared to resolve the Tutte equations \eqref{eq:4:planar-const} in all generality and \eqref{eq:4:genus-const} for the bipartite case.

\section{Resolution of the planar case} \label{sec:4:planar}

Although the number of $m$-constellations with given degree profile of hyperfaces was already given in \cite{BMS} using a bijective method, it was never obtained by resolving functional equations for general $m$. In the framework of resolution of functional equations, the case $m=2$ was solved in \cite{BC:planar}, and the case $m=3$ in \cite{BMJ}. In \cite{BMJ} it was also proved that the OGF $F_{m,0}(t,x;1,1,\ldots)$ without refinement on the degree profile of hyperfaces is algebraic, \textit{i.e.} it is a solution of a polynomial equation. In this section, we will briefly show how to solve Equation~\eqref{eq:4:planar-const} for planar $m$-constellations, using a method that we call the \emph{differential-catalytic method}, first devised in \cite{BMCPR2013representation} for the enumeration of labeled intervals in the $m$-Tamari lattice. The reason why we choose to apply this method here is that it seems to be a powerful method for a large class of functional equations with unlimited repeated iterations of the same operator, and we want to see how general it is.

The differential-catalytic method is a ``guess-and-check'' method, \textit{i.e.}, instead of solving the equation directly, we verify that a given solution candidate indeed satisfies the equation, and prove that it is the only possible solution. The candidate can come from any method, including guessing from numerical result. We thus need the following expression of $F_{m,0}$.
\begin{thm} \label{thm:4:planar-expr}
We define the generating function $A$ by
\begin{equation} \label{eq:4:planar-expr}
A = (1+uz) \left( 1 - \sum_{k \geq 1} p_k z^k \sum_{\ell=1}^{(m-1)k-1} (uz)^\ell \binom{mk-1}{k+\ell} \right) - 1.
\end{equation}
We have $F_{m,0}=1+A$.
\end{thm}
This expression was already implicit in \cite{BMS, BDFG}, in the form of the following enumeration formula of $m$-constellations containing $n$ hyperedges with $d_k$ hyperfaces of degree $mk$:
\[
\frac{m(m-1)^f [(m-1)n]!}{[(m-1)n-f+2]!} \prod_{k \geq 1} \frac1{d_k!} \binom{mk-1}{k-1}^{d_k},
\]
where $f = \sum_k d_k$ is the number of hyperfaces. The detailed computation from this exact formula to \eqref{eq:4:planar-expr} is omitted here due to its volume. 

We now comment briefly and vaguely on how the differential-catalytic method works. It is a guess-and-check method with two stages:
\begin{itemize}
\item \textbf{Transformation}: A major difficulty in checking a solution against general $m$ is the unlimited applications of the operator $(F_{0,m} + \Omega)$ in Equation~\eqref{eq:4:planar-const}. We thus introduce a new catalytic variable $y$ into the repeated operator, then differentiate $y$ to ``linearize'' the equation into a linear differential equation in $y$, with an unknown polynomial function in $x^{-1}$. We then introduce a ``harmonizing operator'' , which turns a series $F(u,z)$ into a linear combination of $F(U_i,z)$ for some solutions $(U_i)_i$ of a properly chosen equation such that the operator leaves the differential equation invariant while eliminating the unknown function. The solution of the new equation is thus a ``harmonized'' version of the original solution.
\item \textbf{Validity and uniqueness}: We first compute the solution of the new equation, and verify that it coincides with the ``harmonized'' version of our conjectured expression. We then prove that, under certain assumptions that hold in our case, the ``harmonizing operator'' is injective when restricted by our assumptions. We thus conclude that our conjectured expression is indeed a solution of the original equation.
\end{itemize}

As we mentioned before, the differential-catalytic method was first devised to solve the functional equation for labeled intervals in the $m$-Tamari lattices in \cite{BMCPR2013representation} (Proposition~5, with $y=1$), which looks very different from our Tutte equation of planar $m$-constellations:
\[
F(x) = \exp \left( \sum_{k \geq 1} \frac{p_k}{k} \left( tx(F(x)\Delta)^m \right)^k \right) (x),
\]
where $\Delta$ is the divided difference operator defined by
\[
\Delta S(x) = \frac{S(x) - S(1)}{x-1}.
\]
However, we still manage to adapt essentially the same resolution method to our case, which means that this method may be more than a one-shot \textit{ad hoc} method for one single functional equation. Of course, our equation bears some similarities to that in the original paper \cite{BMCPR2013representation}, such as the presence of catalytic variables and operators of similar flavor (divided differences and $\Omega$), and perhaps more importantly, a pattern of repeated application of the same operator that can be ``linearized'' by introducing and differentiating a new catalytic variable. These similarities are part of the reasons why we choose to apply the differential-catalytic method to our equation. 

Now, there are two ways to explain the success of this method to our equation. The first one is that there is a combinatorial model that interpolates between labeled intervals in the $m$-Tamari lattices and planar constellations, and this model is governed by a functional equation that can be solved by this method. The second one is that this differential-catalytic method is applicable to a vast domain of functional equations that are similar to our equation and the equation in \cite{BMCPR2013representation} in some sense. If the first statement is true, then we will be able to explain why the enumeration formulae for intervals in the $m$-Tamari lattices are so similar to those for planar maps. If the second one is true, then we have in our hands a powerful generic method to solve very complicated functional equations, which may allow us to enumerate some other complex combinatorial objects. Either way is exciting, but we don't know yet which is correct. To find the correct answer, we need to find other combinatorial models that give similar functional equations, and then to try solving them with the differential-catalytic method to see its boundary. A wild guess is that lattice paths with some sort of decorations may be good candidates of such combinatorial models, because intervals in the $m$-Tamari lattices are formed by lattice paths, and planar constellations are in bijection with a class of trees with decorations (\textit{cf.} \cite{BMS}), which can probably be transformed into lattice paths with some kind of decorations.

In the following, we will apply the differential-catalytic method to prove that Theorem~\ref{thm:4:planar-expr}.

\subsection{Transformation}

We recall the generating function $A$ in \eqref{eq:4:planar-expr}. We introduce yet another catalytic variable $y$ to define a series $G(t,x,y) \in \mathbb{Q}[x,y,\varvec{p}][[t]]$:
\begin{equation} \label{eq:4:planar-const-ext}
G(t,x,y) \eqdef tx(1+y(A + \Omega))^m(1).
\end{equation}
We notice that, by taking $y=1$, we have
\[ G(t,x,1) = tx(1+A+\Omega)^m(1).\]
Therefore, to prove Theorem~\eqref{thm:4:planar-expr}, we only need to verify that the series $G$ defined in \eqref{eq:4:planar-const-ext} satisfies $G(t,x,1) = A(x,t)$. As a first step, we will now show that $G(t,x,y)$ is a solution of a certain differential equation.

\begin{lem} \label{lem:4:G-eq}
We define $N(t,x) \eqdef A+\theta$. For the series $G(t,x,y)$ defined in \eqref{eq:4:planar-const-ext}, we have
\begin{equation} \label{eq:4:planar-lin}
(1+yN)\frac{\partial G}{\partial y} = mNG + xtS.
\end{equation}
Here, $xtS$ is a polynomial in $x^{-1}$ of degree at most $K-1$.
\end{lem}
\begin{proof}
We start by an algebraic observation. Since both $A$ and $\Omega$ does not depend on $y$, we observe that
\begin{align*}
\frac{\partial}{\partial y} (1+y(A+\Omega)) F &= \frac{\partial}{\partial y} F + (A+\Omega)F + y \frac{\partial}{\partial y}((A+\Omega)F) \\ 
&= (A+\Omega)F + (1+y(A+\Omega))\frac{\partial}{\partial y}F.
\end{align*}
We now prove by induction on $k$ that for any integer $k \geq 1$, we have
\begin{equation} \label{eq:4:tele}
\frac{\partial}{\partial y} (1+y(A+\Omega))^k (1) = k (A+\Omega)(1+y(A+\Omega))^{k-1}(1).
\end{equation}
The base case $k=1$ is easily verified. Suppose that \eqref{eq:4:tele} is correct for $k=a-1\geq 1$, we want to prove it for $k=a$. Using the observation above with $F=(1+y(A+\Omega))^{a-1}$, and the fact that $(A+\Omega)$ commutes with $(1+y(A+\Omega))$, we have
\begin{align*}
&\quad \frac{\partial}{\partial y} (1+y(A+\Omega))^a (1) \\ 
&= (A+\Omega)(1+y(A+\Omega))^{a-1}(1) + (1+y(A+\Omega)) \frac{\partial}{\partial y} (1+y(A+\Omega))^{a-1}(1) \\
&= (A+\Omega)(1+y(A+\Omega))^{a-1}(1) + (a-1) (A+\Omega) (1+y(A+\Omega))^{a-1}(1) \\ 
&= a(A+\Omega) (1+y(A+\Omega))^{a-1}(1).
\end{align*}
By induction, we establish \eqref{eq:4:tele} for arbitrary $a$.

We now define $H = t^{-1} x^{-1} G$. From \eqref{eq:4:planar-const-ext} we see that $H$ is also an element in $\mathbb{Q}[x,y,\varvec{p}][[t]]$. The differentiation by $y$ of \eqref{eq:4:planar-const-ext} with both sides divided by $xt$ gives the following special case of \eqref{eq:4:tele}:
\begin{align*}
\frac{\partial H}{\partial y} = \frac{\partial}{\partial y}(1+y(A+\Omega)^m)(1) = m(A+\Omega)(1+y(A+\Omega))^{m-1}(1).
\end{align*}
Applying $(1+y(A+\Omega))$ to both side, with \eqref{eq:4:planar-const-ext} and the observation that $(A+\Omega)$ commutes with $(1+y(A+\Omega))$, we have
\[ (1+y(A+\Omega)) \frac{\partial H}{\partial y} = m(A+\Omega)H. \]
We recall that we only consider constellations with degree restriction, using the parameter $K$. In this case, from \eqref{eq:4:Omega-restricted} we know that the operator $\Omega$ is very close to the multiplication by $\theta$, except that $\Omega$ drops the part with negative powers in $x$. We thus have
\[
(1+y(A+\theta))\frac{\partial H}{\partial y} - y[x^{<0}]\left( \theta \frac{\partial H}{\partial y} \right) = m(A+\theta)H - m[x^{<0}](\theta H),
\]
We now define $S = S_1 - S_2$, with $S_1 = y[x^{<0}]\theta \frac{\partial H}{\partial y}$ and $S_2 = m[x^{<0}](\theta H)$. The term $S$ thus contains all the residual with negative power in $x$ from the replacement of $\Omega$ by $\theta$. Multiplying both sides by $xt$, and we have \eqref{eq:4:planar-lin}.

Since $\theta$ is a polynomial of degree $K$ in $x^{-1}$ and both $H$ and $\frac{\partial H}{\partial y}$ have no negative powers in $x$, both $S_1$ and $S_2$ are polynomials in $x^{-1}$ of degree at most $K$ and divisible by $x^{-1}$, thus their sum $S$ too. Therefore, $xtS$ is a polynomial in $x^{-1}$ of degree at most $K-1$.
\end{proof}

We would like to get rid of the unknown function $S$ in \eqref{eq:4:planar-lin}. To this end, we want to find several series $U_i \equiv U_i(z,u)$ such that $N$ is stable under the change of variable $u \to U_i$, \textit{i.e.} $N(z,U_i(z,u)) = N(z,u)$. We thus need the explicit expression of $N$ given below.

\begin{prop} \label{prop:4:N-expr}
For the series $N(t,x)=A+\theta$ defined in \ref{lem:4:G-eq}, we have
\[
N = uz + (1+uz)\sum_{k=1}^{K} p_k z^k \sum_{\ell=-k}^{0} (uz)^\ell \binom{mk-1}{k+\ell}.
\]
\end{prop}
\begin{proof}
By the change of variable $(t,x) \leftrightarrow (z,u)$, we have
\begin{align*}
\theta &= \sum_{k=1}^{K} p_k x^{-k} = \sum_{k=1}^{K} p_k u^{-k} (1+uz)^{mk} = (1+uz) \sum_{k=1}^{K} p_k z^k \sum_{\ell=0}^{mk-1} (uz)^{\ell-k} \binom{mk-1}{\ell} \\
&= (1+uz) \sum_{k=1}^{K} p_k z^k \sum_{\ell=-k}^{(m-1)k-1} (uz)^\ell \binom{mk-1}{k+\ell}.
\end{align*}
Comparing to \eqref{eq:4:planar-expr}, we observe that the part of the sum over $\ell$ with strictly positive values of $\ell$ is presented both in $\theta$ and in $A$ with opposite signs. In $N$ they cancel out and we have the wanted expression.
\end{proof}

We observe that the power of $u$ in the terms of $N$ varies from $-K$ to $1$, thus $u^K N(u)$ is a polynomial in $u$ of degree $K+1$. Let $U$ be a new variable, and we consider the equation $N(U) = N(u)$, which is equivalent to $u^K U^K N(U) - U^K u^K N(u) = 0$, a polynomial equation of degree $K+1$, in $\mathbb{Q}[z,\varvec{p}][u,U]$. By Newton-Puiseux Theorem (Theorem~\ref{thm:2:newton-puiseux}), it has $K+1$ solutions $U_0 = u, U_1, \ldots, U_K$, all in $\overline{\mathbb{Q}[z,\varvec{p}]}((u^*))$, where $\overline{\mathbb{Q}[z,\varvec{p}]}$ is the algebraic closure of $\mathbb{Q}[z,\varvec{p}]$. Using again the change of variable $x=u(1+uz)^{-m}$, we define $X_0, X_1, \ldots, X_K$ that corresponds to each $U_i$ by $X_i = U_i(1+U_i z)^{-m}$, again in $\overline{\mathbb{Q}[z,\varvec{p}]}((u^*))$. Although the space in which $U_i$ and $X_i$ live looks highly complicated, this is not a problem, since we will only use symmetric functions in all $U_i$'s later, whose values are in $\mathbb{Q}(z,u,\varvec{p})$ (actually also in a much smaller ring $\mathbb{Q}[\varvec{p},N][[z]]$, see Proposition~\ref{prop:4:sympoly-u} later), or symmetric functions in all $U_i$'s with $i>0$, whose values are in $\mathbb{Q}(z,u,\varvec{p})$ (actually also in a much smaller ring $u^{-1}\mathbb{Q}[\varvec{p},z,u^{-1}]$, see Proposition~\ref{prop:4:sympoly-u-other} later).

\begin{prop} \label{prop:4:no-multiple-root}
All $U_i$ are distinct.
\end{prop}
\begin{proof}
Let $P(U) = u^K U^K (N(U) - N(u))$. Since $U=0$ is not a root of $P(U)$, if $P(U)$ has a multiple root $U_*$, we must have $N'(U_*) = 0$, thus $U_*$ is also a multiple root of $N(U) - N(u)$. We observe that the coefficients of $N'(U)$ do not involve $u$, therefore $U_*$ does not depend on $u$. However, $N(U_*) = N(u)$ depends on $u$, which is impossible. Therefore, $P(U)$ has no multiple root.
\end{proof}

We now introduce the following ``harmonizing'' operator $\widetilde{\cdot}$ in $\mathbb{Q}[u,\varvec{p},y][[z]]$ and in $\mathbb{A}[x,x^{-1}]$ for any ring $\mathbb{A}$ of characteristic 0:
\[
\widetilde{\cdot}: F(u) \mapsto \widetilde{F}(u) \eqdef \sum_{i=0}^K \frac{F(U_i)}{\prod_{j \neq i} (X_i^{-1} - X_j^{-1})}.
\]
It is worth mentioning that the definition of the operator $\widetilde{\cdot}$ depends on the $X_i$'s, which are related to the $U_i$'s that are solutions of the equation $N(U)=N(u)$. We will use the following property of the operator $\widetilde{\cdot}$, which has already appeared in \cite{BMCPR2013representation} with its proof, to cancel out the unknown function $S$ in \eqref{eq:4:planar-lin}. The statement and the proof of this property are essentially the same as in \cite{BMCPR2013representation}, only slightly adapted to our context.
\begin{prop}[Lemma~13 in \cite{BMCPR2013representation}] \label{prop:4:cancel-S}
For any ring $\mathbb{A}$ of characteristic 0 and $P(x)$ in $\mathbb{A}[x^{-1}]$ of degree at most $K-1$ in $x^{-1}$, we have $\widetilde{P} = 0$. Furthermore, we have $\widetilde{x^{-K}} = 1$ and $\widetilde{x} = (-1)^K \prod_{k=0}^K X_k$.
\end{prop}
\begin{proof}
The substitution $P(u=U_i) = P(X_i)$ is clearly well-defined in $\overline{\mathbb{Q}[z,\underline{p}]}((u^*))$. By Proposition~\ref{prop:4:no-multiple-root}, all the $U_i$'s are distinct, thus all the $X_i$'s are also distinct. Therefore, by Lagrange interpolation, for any function $Q(x)$ that is a polynomial in $x^{-1}$ of degree at most $K$, we have
\begin{equation} \label{eq:4:lagrange}
Q(x) = \sum_{i=0}^K Q(X_i) \prod_{j \neq i} \frac{x^{-1}-X_j^{-1}}{X_i^{-1} - X_j^{-1}}.
\end{equation}
We notice that both sides of the equality are polynomials in $x^{-1}$.

Since $P(x)$ is of degree at most $K-1$ in $x^{-1}$, we have $[x^{-K}]P(x)=0$. By taking the coefficient of $x^{-K}$ in \eqref{eq:4:lagrange} applied on $P(x)$, we have
\[
0 = [x^{-K}]P(x) = [x^{-K}]\sum_{i=0}^K P(X_i) \prod_{j \neq i} \frac{x^{-1}-X_j^{-1}}{X_i^{-1} - X_j^{-1}} = \sum_{i=0}^K \prod_{j \neq i} \frac{P(X_i)}{X_i^{-1} - X_j^{-1}} = \widetilde{P}.
\]
Therefore, $\widetilde{P}=0$.

Similarly, by taking $Q(x)=x^{-K}$ in \eqref{eq:4:lagrange} and taking the coefficient of $x^{-K}$, we have
\[
1 = [x^{-K}]x^{-K} = [x^{-K}]\sum_{i=0}^K X_i^{-K} \prod_{j \neq i} \frac{x^{-1}-X_j^{-1}}{X_i^{-1} - X_j^{-1}} = \sum_{i=0}^K \prod_{j \neq i} \frac{X_i^{-K}}{X_i^{-1} - X_j^{-1}} = \widetilde{x^{-K}}.
\]
Therefore, $\widetilde{x^{-K}}=1$.

For the last relation, we take $Q(x)=1$ in \eqref{eq:4:lagrange} and then take the constant coefficient, which leads to
\begin{align*}
1 = [x^{-0}]1 &= [x^{-0}]\sum_{i=0}^K \prod_{j \neq i} \frac{x^{-1}-X_j^{-1}}{X_i^{-1} - X_j^{-1}} \\
&= \sum_{i=0}^K \prod_{j \neq i} \frac{X_j}{X_i^{-1} - X_j^{-1}} (-1)^K \prod_{k=0}^K X_k^{-1} = \widetilde{x} (-1)^K \prod_{k=0}^K X_k^{-1}.
\end{align*}
Therefore, $\widetilde{x} = (-1)^K \prod_{k=0}^K X_k$.
\end{proof}

We can now use Proposition~\ref{prop:4:cancel-S} to compute the transformation $\widetilde{G}$ of $G$, by removing the unknown residue $xtS$ in (\ref{eq:4:planar-lin}) using the ``harmonizing'' operator.

\begin{prop} \label{prop:4:harmonized-H}
The transformation $\widetilde{G}$ of $G$ has the following expressions:
\begin{equation} \label{eq:4:harmonized-H}
\widetilde{G} = (-1)^K t \left( \prod_{i=0}^K X_i \right) (1+yN)^m, \quad \widetilde{G}(t,x,1) = (-1)^K t \left( \prod_{i=0}^K X_i \right) (1+N)^m.
\end{equation}
\end{prop}
\begin{proof}
Since $G$ is in $\mathbb{Q}[\varvec{p},y,x][[t]] = \mathbb{Q}[\varvec{p},y,u][[z]]$, we can apply the operator $\widetilde{\cdot}$ to both sides of \eqref{eq:4:planar-lin}. By the definition of all the $U_i$'s, we have $N(U_i)=N(u)$ for all $U_i$, which leads to
\begin{align*}
&\quad \sum_{i=0}^K (1+yN) \frac{\partial G}{\partial y}(U_i)\frac1{\prod_{j \neq i} (X_i^{-1} - X_j^{-1})} \\
&= \sum_{i=0}^K mN G(U_i) \frac1{\prod_{j \neq i} (X_i^{-1} - X_j^{-1})} + \sum_{i=0}^K (xtS)|_{u=U_i} \frac1{\prod_{j \neq i} (X_i^{-1} - X_j^{-1})}.
\end{align*}
Recalling the definition of the operator $\widetilde{\cdot}$ and the fact that $y$ is independent of all $U_i$, we have
\[ (1+yN)\frac{\partial \widetilde{G}}{\partial y} = mN\widetilde{G} + \widetilde{xtS}. \]
By Lemma~\ref{lem:4:G-eq}, $xtS$ is a polynomial in $x^{-1}$ of degree at most $K-1$. Therefore, by Proposition~\ref{prop:4:cancel-S}, we have $\widetilde{xtS}=0$, which leads to
\begin{equation} \label{eq:4:planar-lin-pure}
(1+yN)\frac{\partial \widetilde{G}}{\partial y} = mN \widetilde{G}.
\end{equation}
Since $G(t,x,0) = tx$, by Proposition~\ref{prop:4:cancel-S}, we have 
\[
\widetilde{G}(y=0) = \widetilde{tx} = (-1)^K t \prod_{i=0}^K X_i.
\]
With this initial condition, it is straight-forward to solve for $\widetilde{G}$, which gives the expression of $\widetilde{G}$ in \eqref{eq:4:harmonized-H}. The expression of $\widetilde{G}(t,x,1)$ is obtained by specifying $y=1$ in the expression of $\widetilde{G}$.
\end{proof}

\subsection{Validity and uniqueness}

We now prove that $G(t,x,1) = A(t,x)$. The proof takes two steps.
\begin{itemize}
\item \textbf{Validity}: verify that $\widetilde{G}(t,x,1) = \widetilde{A}(t,x)$;
\item \textbf{Uniqueness}: prove that for any series $G, A \in xt\mathbb{Q}[x,\varvec{p}][[t]]$, if $\widetilde{G} = \widetilde{A}$, then $G = A$.
\end{itemize}

To verify that $\widetilde{G}(t,x,1) = \widetilde{A}(t,x)$, we only need to compute $\widetilde{A}$ and compare the result with \eqref{eq:4:harmonized-H}.

\begin{prop} \label{prop:4:sympoly-u}
A polynomial of $U_0, \ldots, U_K$ of total degree $d$ that is symmetric in all variables belongs to the domain $z^{-d}\mathbb{Q}[\varvec{p},N][[z]]$, where $N=N(u,z)=A+\theta$ is as defined in Lemma~\ref{lem:4:G-eq}. In particular,
\[ \prod_{i=0}^K U_i = (-1)^{K+1} p_K z^{-1} (1 + \gamma)^{-1}, \]
where $\gamma$ is as defined in \eqref{eq:4:gamma}.
Furthermore, a symmetric polynomial in all $U_i$'s of degree $d$ has total valuation in $\varvec{p}$ and $N$ of at least $\lceil d/K \rceil$, \textit{i.e.}, when viewed as power series in all $\varvec{p}$ and $N$, each term has at least $\lceil d/K \rceil$ factors of $N$ or $p_k$'s, counted with multiplicity.
\end{prop}
\begin{proof}
We observe that all coefficients of the polynomial $P(U) = U^K (N(U) - N(u))$ are in $\mathbb{Q}[\varvec{p},z]$, except for $[U^K]P(U)$, which is in $\mathbb{Q}[\varvec{p},z,N]$. Furthermore, we have $[U^{K+1}]P(U) = [U^1]N(U) = z(1+\gamma)$. Since $\gamma$ is in $z\mathbb{Q}[\varvec{p}][[z]]$, the series $1+\gamma$ has its inverse $(1+\gamma)^{-1}$ in $\mathbb{Q}[\varvec{p}][[z]]$. Since all symmetric polynomials of total degree $d$ evaluated at the set of roots of $P(U)$ are polynomials of total degree at most $d$ in quotients of the form $[U^d]P(U) / [U^{K+1}]P(U)$, they are all in $z^{-d}\mathbb{Q}[\varvec{p},N][[z]]$. For the product, we have
\[
\prod_{i=0}^K U_i = (-1)^{K+1} \frac{[U^0]P(U)}{[U^{K+1}]P(U)} = (-1)^{K+1} p_K z^{-1} (1+\gamma)^{-1}
\]
by observing that $[U^{-K}]N(U) = p_K$. For the valuation in $\varvec{p}$ of a symmetric polynomial in all $U_i$'s, we only need to observe that the only term in $N(u)$ that does not contain any $p_k$ is $uz$. Therefore, $[U^d]P(U)$ is of valuation $1$ in $\varvec{p}$ for any $0 \leq d \leq K-1$. For $[U^K]P(U)$, the only term that does not contain any $p_k$ is $N$. Therefore, $[U^d]P(U)/[U^{K+1}]P(U)$ of valuation $1$ in $\varvec{p}$ and $N$ for all $0 \leq d \leq K$.
\end{proof}

\begin{prop} \label{prop:4:validity}
We have $\widetilde{G}(t,x,1) = \widetilde{A}(t,x)$.
 \end{prop}
\begin{proof}
The following direct computation using Proposition~\ref{prop:4:cancel-S} gives the value of $\widetilde{A}$:
\begin{align*}
\widetilde{A} &= \widetilde{N-\theta} = \sum_{i=0}^K \frac{N(U_i) - \theta(X_i)}{\prod_{j \neq i} (X_i^{-1} - X_j^{-1})} \\
&= \sum_{i=0}^K \frac{N(u) - \sum_{j=1}^K p_j X_i^{-j}}{\prod_{j \neq i} (X_i^{-1} - X_j^{-1})} = \sum_{i=0}^K \frac{- p_K X_i^{-K}}{\prod_{j \neq i} (X_i^{-1} - X_j^{-1})} = - p_K.
\end{align*}
We now turn to the computation of $\widetilde{G}(t,x,1)$. According to \eqref{eq:4:harmonized-H}, we can deduce the value of $\widetilde{G}(t,x,1)$ from that of $\prod_{i=0}^K X_i$. We now compute $\prod_{i=0}^K X_i$ via expressions of $\prod_{i=0}^K U_i$ and $\prod_{i=0}^K (1+U_i z)$. Proposition~\ref{prop:4:sympoly-u} gives the value of $\prod_{i=0}^K U_i$. To obtain $\prod_{i=0}^K (1+U_i z)$, we define a new polynomial $Q(V) = P((V-1)/z)$ with $P(U)=U^K(N(U)-N(u))$ as in Proposition~\ref{prop:4:sympoly-u}, which is a polynomial in $V$ of degree $K+1$, whose roots are exactly $V_i = 1+U_iz$. We have
\[
Q(V) = z^{-K} \left( (V-1)^{K+1} - (V-1)^K N(u) + V \sum_{k=1}^K p_k z^k \sum_{\ell=-k}^0 (V-1)^{K+\ell} \binom{mk-1}{k+\ell} \right).
\]
We can thus read off the following coefficients of $Q(V)$:
\begin{align*}
[V^0]Q(V) &= (-1)^{K+1} z^{-K} (1 + N(u)), \\ 
[V^{K+1}]Q(V) &= z^{-K} \left( 1 + \sum_{k=1}^K p_k z^k \binom{mk-1}{k} \right) = z^{-K} (1+\gamma).
\end{align*}
We thus have
\[
\prod_{i=0}^K (1+U_iz) = (-1)^{K+1} \frac{[V^0]Q(V)}{[V^{K+1}]Q(V)} = (1+\gamma)^{-1} (1+N).
\]
Therefore, by the observation that $z=t(1+\gamma)^{m-1}$,
\begin{align*}
\widetilde{G}(t,x,1) &= (-1)^K t \left( \prod_{i=0}^K U_i \right) \left( \prod_{i=0}^K (1+U_iz) \right)^{-m} (1+N)^m \\
&= - t p_K (1+\gamma)^{m-1} z^{-1} = - p_K.
\end{align*}
We thus have $\widetilde{G}(t,x,1) = \widetilde{A}(t,x)$.
\end{proof}

We will now show that, given the ``harmonized'' version $\widetilde{A}$ of a series $A \in xt\mathbb{Q}[x,\varvec{p}][[t]] = uz\mathbb{Q}[u,\varvec{p}][[z]]$ such that $[z^i]A$ is a polynomial in $u$ of degree at most $i$, we can uniquely ``reconstruct'' $A$ from $\widetilde{A}$.

Our reconstruction is done via a step-by-step ``interpolation'' between $\widetilde{A}$ and $A$, following the process introduced in \cite{BMCPR2013representation}. We first define the following series in free variables $x_0, \ldots, x_k$ that we will use as intermediate steps:
\begin{equation} \label{eq:4:def-ak}
A_k(x_0, x_1, \ldots, x_k) = \sum_{i=0}^k \frac{A(u_i)}{\prod_{j \neq i} x_i^{-1} - x_j^{-1}}.
\end{equation}
Here, $u_i$ is the unique series in $x_i$ and $z$ that satisfies $x_i = u_i(1+u_i z)^{-m}$ (the usual change of variable from $x$ to $u$). Later $x_i$ will be specialized to $X_i$. We have $A_0(x) = A(x)$ and $A_K(X_0, \ldots, X_K) = \widetilde{A}(x)$, which means that the $A_k$'s indeed interpolate between $A$ and $\widetilde{A}$. We observe the following recurrence on $A_k$ for $k \geq 1$:
\begin{equation} \label{eq:4:ak-rec}
(x_{k-1}^{-1} - x_{k}^{-1})A_k(x_0, \ldots, x_k) = A_{k-1}(x_0, \ldots, x_{k-2}, x_{k-1}) - A_{k-1}(x_0, \ldots, x_{k-2}, x_k).
\end{equation}
We have the following proposition on the form of $A_k$.

\begin{prop} \label{prop:4:ak-divisible}
For variables $u_0, \ldots, u_k$ and $x_0, \ldots, x_k$ such that $x_i = u_i(1+u_i z)^{-m}$, given a series $A(u) \in uz\mathbb{Q}[\varvec{p},u][[z]]$ such that $[z^d]A$ is a polynomial in $u$ of degree at most $d$ for all $d$, the series $A_k(x_0, \ldots, x_k)$ is in $\mathbb{Q}[u_0,\ldots,u_k,\varvec{p}][[z]]$. Furthermore, all coefficients of $A_k$ are divisible by $u_0u_1\cdots u_k$ and symmetric in all $u_i$'s, and $[z^d]A_k$ has total degree in all the $u_i$'s at most $d+k$. 
\end{prop}
\begin{proof}
Using the change of variable $x=u(1+uz)^{-m}$, we have
\begin{equation} \label{eq:4:diff-x-inv}
\frac1{x_i^{-1} - x_j^{-1}} = \frac{u_i u_j C(u_i,u_j)}{u_i - u_j},
\end{equation}
where $C(u_i,u_j)$ is a member of $\mathbb{Q}[u_i z,u_j z][[z]]$, therefore $[z^d]C(u_i,u_j)$ is a polynomial in $u_i,u_j$ of total degree $d$ for all $d$. For $A$ in $uz\mathbb{Q}[\varvec{p},u][[z]]$, let $A(u)=uzD(u)$, we have
\begin{equation} \label{eq:4:Ak-in-u}
A_k(x_0, \ldots, x_k) = \sum_{i=0}^k zD(u_i)u_i^{k+1} \prod_{j \neq i} \frac{u_j C(u_i,u_j)}{u_i-u_j}.
\end{equation}
We thus observe that
\[ B_k \eqdef A_k(x_0, \ldots, x_k) \prod_{0 \leq i < j \leq k} (u_i-u_j) \]
is in $\mathbb{Q}[u_0, \ldots, u_k,\varvec{p}][[z]]$. Since $A_k$ is symmetric in all $u_i$'s, $B_k$ is antisymmetric in all $u_i$'s, and so are its polynomial coefficients. Therefore, each coefficient of $B_k(z)$ must be a multiple of the Vandermonde polynomial $\prod_{i<j} (u_i - u_j)$, which implies that $A_k$ is in $\mathbb{Q}[u_0, \ldots, u_k][[z]]$. Moreover, we can see from \eqref{eq:4:Ak-in-u} that $A_k$ is a multiple of all $u_i$'s. 

For the total degree of $[z^d]A_k$ in all the $u_i$'s, we observe from \eqref{eq:4:diff-x-inv} that $(x_i^{-1} - x_j^{-1})^{-1}$ is an element of $\mathbb{Q}(u_i, u_j)[[z]]$ with $[z^d](x_i^{-1} - x_j^{-1})^{-1}$ a rational function in $u_i, u_j$ of total degree $d+1$ in $u_i, u_j$. Therefore, by the definition of $A_k$, we can see that $[x^d]A_k$ has total degree $d+k$ in all $u_i$'s as a rational function.
\end{proof}

We now define the \mydef{complete homogeneous symmetric polynomial} $h_d(x_0, x_1, \ldots, x_k)$ of degree $d$ in $k+1$ variables to be
\begin{equation} \label{eq:4:hd-rec}
h_d(x_0, x_1, \ldots, x_k) = \sum_{1 \leq i_1 \leq \cdots \leq i_d \leq k} x_{i_1} x_{i_2} \cdots x_{i_d}.
\end{equation}
For $d=0$, $h_0$ takes the value $1$. Any symmetric polynomial can be written as a polynomial in $h_d$'s with the same set of variables (see, \textit{e.g.} \cite[Section~7.5]{Stanley:EC2}. We should notice that, although variables of the form $p_k$ in $\varvec{p}$ have the same notation as another type of symmetric polynomials called \emph{powersum symmetric polynomials}, in this section they only serve as formal variables. The polynomials $h_d$ satisfy the following recurrence:
\[
(x_{k-1}^{-1} - x_{k}^{-1})h_d(x_0^{-1}, \ldots, x_k^{-1}) = h_{d+1}(x_0^{-1}, \ldots, x_{k-2}^{-1}, x_{k-1}^{-1}) - h_{d+1}(x_0^{-1}, \ldots, x_{k-2}^{-1}, x_{k}^{-1}).
\]
Comparing to \eqref{eq:4:ak-rec}, it seems a good idea to relate the unknown series $(A_k)_{0\leq k \leq K}$ to the symmetric polynomials $(h_d)_{0\leq d \leq K}$.

For a set $S = \{ i_0 < i_1 < \cdots < i_k \}$ formed by $k+1$ integers between $0$ and $K$, we define 
\begin{equation} \label{eq:4:def-aks}
A_{k}[S] = A_k(X_{i_0}, \ldots, X_{i_k}), \quad h_d[S] = h_d(X_{i_0}^{-1}, \ldots, X_{i_k}^{-1}).
\end{equation} 
We have $A_0[\{ 0\}] = A$ and $A_K[\{0, 1, \ldots, K\}] = \widetilde{A}$. We have the following proposition about relating $A_k[S]$ to $h_d[S]$.

\begin{prop} \label{prop:4:interp-coeff}
Let $A$ be a series in $uz\mathbb{Q}[u,\varvec{p}][[z]]$ such that $[z^i]A$ is of degree at most $i$ in $u$, and that $\widetilde{A} \in \mathbb{Q}[z^{-1}N][[\varvec{p},z]]$. We construct $A_k[S]$ using \eqref{eq:4:def-ak} and \eqref{eq:4:def-aks}. For any set $S$ formed by integers from $0$ to $K$, let $k+1$ be the size of $S$, we have
\[
A_k[S] = \sum_{j=k}^K \Phi_j h_{j-k}[S]
\]
with some coefficients $\Phi_j$ in $\mathbb{Q}[z^{-1}N][[\varvec{p},z]]$ that do not depend on $S$ nor on $k$, where $N = N(u)$ is as defined in Lemma~\ref{lem:4:G-eq}, with an explicit expression in Proposition~\ref{prop:4:N-expr}.
\end{prop}
\begin{proof}
We proceed by downward induction for $k$ from $K$ to $0$. 

For the base case $k=K$, the set $S$ can only be $\{0, 1, \ldots, K\}$, and we have $A_K[S] = \widetilde{A} \in \mathbb{Q}[z^{-1}N][[\varvec{p},z]]$ by assumption.

For induction, we fix $k$, and the induction hypothesis is that $A_{k+1}[S'] = \sum_{j=k+1}^K \Phi_j h_{j-k-1}[S']$ with $\Phi_j \in \mathbb{Q}[\varvec{p},N]((z))$ for any $S'$ with $k+2$ elements. Let $S$ be a set of $k$ integers from $0$ to $K$ that does not contain two distinct integers $p,q$. By \eqref{eq:4:ak-rec} and \eqref{eq:4:hd-rec} we have
\begin{align*}
(X_p^{-1} - X_q^{-1})A_{k+1}[S \cup \{p,q\}] &= A_k[S \cup \{ p \}] - A_k[S \cup \{ q \}], \\
(X_p^{-1} - X_q^{-1}) \sum_{j=k+1}^K \Phi_j h_{j-k-1}[S \cup \{p,q\}] &= \sum_{j=k+1}^K \Phi_j (h_{j-k}[S \cup \{p\}] - h_{j-k}[S \cup \{q\}].
\end{align*}
Let $\Phi_k[T] = A_k[T] - \sum_{j=k+1}^K \Phi_j h_{j-k}[T]$ be the $\Phi_k$ we search for, which seems to depend on the set $T$. Combining the two equalities with the induction hypothesis, we have
\[ \Phi_k[S \cup \{p\}] = \Phi_k[S \cup \{q\}]. \]
Therefore, $\Phi_k[T]$ does not depend on $T$, thus can be denoted simply by $\Phi_k$. We have the following equality:
\begin{equation} \label{eq:4:superpose-K+1}
\binom{K+1}{k+1} \Phi_k = \sum_{S \subset \{0, 1, \ldots, K\}, \#S = k+1} A_k[S] - \sum_{j=k+1}^K \Phi_j \sum_{S \subset \{0, 1, \ldots, K\}, \#S = k+1} h_{j-k}[S].
\end{equation}

Since $A \in \mathbb{Q}[\varvec{p},u][[z]]$, the first sum is a formal power series in $z$ with symmetric polynomials of all $U_i$'s as coefficients. By Proposition~\ref{prop:4:sympoly-u}, a symmetric polynomial of total degree $d$ in all $U_i$'s is in $z^{-d}\mathbb{Q}[\varvec{p},N][[z]]$, and with valuation in $\varvec{p}$ and $N$ at least $d/K$. By Proposition~\ref{prop:4:ak-divisible}, $[z^i]A_k$ has total degree at most $i+k$ and total valuation at least $K$ in all $U_i$'s. Let $P_d$ be the component of degree $K \leq d \leq i+k$ in the symmetric polynomial $[z^i]A_k$ in all $U_i$'s. We know from Proposition~\ref{prop:4:sympoly-u} that $z^i P_d$ is in $z^{i-d}\mathbb{Q}[\varvec{p},N][[z]]$, but of valuation at least $d/K$ in $\varvec{p}$ and $N$. If this polynomial contributes to a term $N^a z^j p_\mu$, we must have $d \leq K (a + \ell(\mu))$ and $i \leq d + j$, and there are only finitely many such possibilities of contribution in the first sum. Therefore, the first sum is in $\mathbb{Q}[[N,\varvec{p},z]]$. We then observe that the term $N^a z^j p_\mu$ can be expressed as $(z^{-1}N)^a z^{j+a} p_\mu$, which means that the coefficient of $z^k$ in the first sum is a polynomial in $z^{-1}N$ of degree at most $k$. Therefore, we know that the first sum is in $\mathbb{Q}[z^{-1}N][[\varvec{p},z]]$. For the second sum, we observe that it is a symmetric polynomial in all $X_i = U_i^{-1}(1+U_iz)^m$. Again by Proposition~\ref{prop:4:sympoly-u}, the second sum is in $\mathbb{Q}[z^{-1}N][[\varvec{p},z]]$. Combining the two sums, we prove that $\Phi_k \in \mathbb{Q}[z^{-1}N][[\varvec{p},z]]$.
\end{proof}

We now investigate the symmetric polynomials in $U_1, \ldots, U_K$.

\begin{prop} \label{prop:4:sympoly-u-other}
All symmetric polynomials of $U_1, \ldots, U_K$ without constant term belong to $u^{-1}\mathbb{Q}[\varvec{p},z,u^{-1}]$.
\end{prop}
\begin{proof}
The series $U_1, \ldots, U_K$ are all the roots of the following polynomial $R(U)$ in $N$:
\begin{align*}
R(U) &= \frac{U^K(N(U)-N(u))}{U-u} \\ &= U^K \left( z - \sum_{k=1}^K p_k z^k \sum_{\ell = -k}^{-1} \binom{mk-1}{k+\ell} \left( \sum_{i=-1}^{\ell} U^i u^{\ell-i-2} + \sum_{i=-1}^{\ell+1} U^i u^{\ell-i-1} \right) \right) \\
&= U^K z - \sum_{k=1}^K p_k z^k \sum_{\ell = -k}^{-1} \binom{mk-1}{k+\ell} \left( \sum_{i=\ell}^{-1} U^{K+i} u^{\ell-i-2} + \sum_{i=\ell+1}^{-1} U^{K+i} u^{\ell-i-1} \right).
\end{align*}
We see that $R(U)$ has top coefficient $z$ and all other coefficients in $zu^{-1}\mathbb{Q}[\varvec{p},z,u^{-1}]$. Therefore, all symmetric polynomials without constant term of the roots $U_1, \ldots, U_K$ must be in $u^{-1}\mathbb{Q}[\varvec{p},z,u^{-1}]$.
\end{proof}

For $W(u)$ a polynomial in $u$, we see that $W(z^{-1}N)$ is a Laurent polynomial in $u$. We have the following proposition for a reconstruction of $W(z^{-1}N)$ from $[u^{\geq 0}]W(z^{-1}N)$.

\begin{prop}\label{prop:4:reconstruct}
For a polynomial $W(z^{-1}N) \in \mathbb{Q}[z^{-1}N]$ of $z^{-1}N$, we have
\[
W(z^{-1}N) = \left. \left([z^0 p_\epsilon]W(z^{-1}N) \right) \right|_{u=z^{-1}N}.
\]
\end{prop}
\begin{proof}
From the expression of $N$ in Proposition~\ref{prop:4:N-expr}, we have $[z^0 p_\epsilon](z^{-1}N)^k = u^k$. The proposition thus follows by linearity of the operators involved.
\end{proof}

\begin{prop} \label{prop:4:uniqueness}
Given the transform $\widetilde{A}$ of a series $A \in uz\mathbb{Q}[u,\varvec{p}][[z]]$ such that $[z^i]A$ is of degree at most $i$ in $u$, there is a unique sequence $(\Phi_K, \Phi_{K-1}, \ldots, \Phi_{0})$ of series in $\mathbb{Q}[z^{-1}N][[\varvec{p},z]]$ such that, for all $0 \leq k \leq K$, we have
\[ A_k[S] = \sum_{j=k}^K \Phi_j h_{j-k}[S]. \]
Here, $A_k[S]$ is as defined by \eqref{eq:4:def-ak} and \eqref{eq:4:def-aks}.

Furthermore, each $\Phi_i$ can be effectively computed from $\widetilde{A}$ in a way that does not depend on $A$.
\end{prop}
\begin{proof}
The existence and the uniqueness of such an expression is guaranteed by Proposition~\ref{prop:4:interp-coeff}. We now prove that each $\Phi_i$ can be effectively computed, using a downward induction for $k$ from $K$ down to $0$. The case $k = K$ is already established as the base case in Proposition~\ref{prop:4:interp-coeff} by $\Phi_K = \widetilde{A}$. For the induction, we only need to prove that we can compute $\Phi_k$ explicitly from all $\Phi_i$ with $i > k$.

Since $\Phi_k$ does not depend on the index set $S$ of variables, we have the following variant of \eqref{eq:4:superpose-K+1}:
\begin{equation} \label{eq:4:uniqueness-1}
\binom{K}{k+1} \Phi_k = \sum_{S \subseteq \{ 1, 2, \ldots, K \}, \#S = k+1} A_k[S] - \sum_{j=k+1}^K \Phi_j \sum_{S \subseteq \{ 1, 2, \ldots, K \}, \#S = k+1} h_{j-k}[S].
\end{equation}
If we consider $\Phi_k$ as a series in $z$ and $\varvec{p}$, its coefficients are all polynomials in $z^{-1}N$. Proposition~\ref{prop:4:reconstruct} suggests that we can determine $\Phi_k$ itself by computing $[u^{\geq 0}]\Phi_k$. However, we observe that the sum over $A_k[S]$ is symmetric in all $U_i$ with $1 \leq i \leq K$, and by Proposition~\ref{prop:4:ak-divisible} and Proposition~\ref{prop:4:sympoly-u-other}, this sum has no positive power in $u$. By taking the positive part of \eqref{eq:4:uniqueness-1}, we have
\[
\binom{K}{k+1} [u^{\geq 0}]\Phi_k = - [u^{\geq 0}] \sum_{j=k+1}^K \Phi_j \sum_{S \subseteq \{ 1, 2, \ldots, K \}, \#S = k+1} h_{j-k}[S].
\]
The sum of $h_{j-k}[S]$ is a symmetric polynomial in $X_i^{-1} = U_i^{-1} (1+U_i z)^m$ for $1 \leq i \leq K$, which can be effectively computed. By induction hypothesis, we can effectively compute all the coefficients of $[u^{\geq 0}]\Phi_k$. 

We now recover $\Phi_k$ from $[u^{\geq 0}]\Phi_k$. Let $\Phi_*$ be an element in $\mathbb{Q}[z^{-1}N][[z,\varvec{p}]]$. Let $i$ be a natural number and $\mu$ a partition such that $[z^i p_\mu]([u^{\geq 0}])\Phi_* \neq 0$ and for any monomial $z^j p_\lambda$ with $j\leq i$ and $|\lambda| < |\mu|$ we have $[z^j p_\lambda]([u^{\geq 0}]\Phi_* = 0$. It is clear that for any monomial $z^j p_\lambda$ with $j \leq i$ and $|\lambda| < |\mu|$, we have $[z^j p_\lambda]\Phi_* = 0$, since $[u^1]z^{-1}N=1$, which means that any non-zero term will also lead to a non-zero $[z^j p_\lambda]([u^{\geq 0}]\Phi_*$. We can then write $\Phi_* = P_{i,\mu}(z^{-1}N) z^i p_\mu + \Phi_\circ$, where $\Phi_\circ$ does not contain the term $z^i p_\mu$. We have $[u^{\geq 0}]\Phi_* = [u^{\geq 0}]P_{i,\mu}(z^{-1}N) z^i p_\mu + [u^{\geq 0}]\Phi_\circ$. To compute $P_{i,\mu}(z^{-1}N)$, we first observe that $[u^{\geq 0}]\Phi_\circ$ does not contain the term $z^i p_\mu$. Therefore, according to Proposition~\ref{prop:4:reconstruct}, we can recover $P_{i,\mu}(z^{-1}N)$ from $[p_\epsilon]P_{i,\mu}(z^{-1}N) = [z^i p_\mu][u^{\geq 0}]\Phi_*$, which can be effectively computed from $[u^{\geq 0}]\Phi_*$. After recovering $P_{i,\mu}(z^{-1}N)$, we can use it to compute $[u^{\geq 0}]\Phi_\circ = [u^{\geq 0}]\Phi_* - [u^{\geq 0}]P_{i,\mu}(z^{-1}N) z^i p_\mu$, and we can pursuit other terms in $\Phi_*$ by studying $\Phi_\circ$. We now apply the previous procedure to $[u^{\geq 0}]\Phi_k$ iteratively on all the terms $z^i p_\mu$, in an order that is increasing by $i+\ell(\mu)$, with ties broken by increasing order in $\ell(\mu)$, and any further ties resolved arbitrarily. We first initialize $[u^{\geq 0}]\Phi_*$ by $[u^{\geq 0}]\Phi_k$, then after computing the coefficient of a term, we compute the corresponding $[u^{\geq 0}]\Phi_\circ$ and use it as the $[u^{\geq 0}\Phi_*$ for the computation of the coefficient of the next term. In our prescribed order of terms, the assumption is always valid during the whole iterated procedure, which means that we can effectively compute all the coefficients of $\Phi_k$.

As we can see, the previous computation procedure is uniform for all $A$. Therefore, the quantity $\Phi_k$ is uniquely determined by all the $\Phi_j$ in an effectively computable way, thus determined by $\widetilde{A}$.
\end{proof}

\begin{coro} \label{coro:4:uniqueness}
For two series $F,G$ in $uz\mathbb{Q}[u,\varvec{p}][[z]]$ such that $\widetilde{F} = \widetilde{G}$, we have $F=G$.
\end{coro}
\begin{proof}
We recall that $X_0=x$, and we observe that $h_j(x^{-1})=x^{-j}$. By taking $k=0$ and $S=\{ 0 \}$ in Proposition~\ref{prop:4:interp-coeff}, we have
\[
F = \sum_{j=0}^K \Phi^F_j \left( u^{-1} (1+uz)^m \right)^j, \quad G = \sum_{j=0}^K \Phi^G_j \left( u^{-1} (1+uz)^m \right)^j.
\]
Furthermore, we have $\Phi^F_K = \widetilde{F} = \widetilde{G} = \Phi^G_K$. By Proposition~\ref{prop:4:uniqueness}, we have $\Phi^F_k = \Phi^F_k$ for all $k$. We thus have $F=G$.
\end{proof}

Since $A$ in \eqref{eq:4:planar-expr} satisfies $A \in xt\mathbb{Q}[\varvec{p},x][[t]] = uz\mathbb{Q}[\varvec{p},u][[z]]$, combining Proposition~\ref{prop:4:validity} and Corollary~\ref{coro:4:uniqueness}, we know that the expression of $F_{0,m}$ in Theorem~\ref{thm:4:planar-expr} is the unique series solution of \eqref{eq:4:planar-const} in the restricted case $p_k=0$ for all $k>K$. By taking the projective limit, we know that it is the OGF of planar $m$-constellations in full generality.

\section{The higher genus case for bipartite maps}
\label{sec:4:in}

In Chapter~2 we have mentioned that constellations are closely related to other factorization models, such as transposition factorizations counted by classical and monotone Hurwitz numbers. It is then natural to expect that they share some common properties. In fact, the generating functions of classical and monotone Hurwitz numbers of a given genus have been shown to be rational functions in some simple series. The case for classical Hurwitz numbers was proved by Goulden, Jackson and Vakil in \cite{classical-hurwitz} using deep algebraic results~\cite{ELSV}. The case for monotone Hurwitz numbers was done more recently by Goulden, Guay-Paquet, and Novak in \cite{GGPN}. It is worth remarking that the expressions in both cases are extremely similar, which invites us to investigate the closely related case of constellations. In the context of map enumeration, Bender and Canfield obtained in \cite{BC} explicit expressions of OGFs of maps of higher genera without face-degree statistics, and Gao determined in \cite{Gao} that the OGFs of maps of higher genera with given set of face degree are algebraic series with a special form. In this chapter, we will start by the simplest case of bipartite maps ($2$-constellations), and prove a rationality result both on bipartite maps themselves and on the associated rotation systems, which can be seen as an ``unrooted'' version of bipartite maps.

Bipartite maps have been considered before in the literature, and we only mention some recent results here. In \cite{chapuy2009asymptotic}, Chapuy used \emph{bijective} methods to obtained a rationality result, which is weaker than the one proved here, but applies to all $m$-constellations. Kazarian and Zograf~\cite{KZ} proved a \emph{polynomiality} statement for generating functions of bipartite maps with \emph{a fixed number of} faces using a variant of topological recursion. It is worth mentioning that these authors deal with \textit{dessins d'enfants} rather than bipartite maps, but the two models are equivalent (\textit{cf.} \cite[Chap. 1]{LandoZvonkine}). In contrast, our main result covers the case of arbitrarily many faces, which is more general.

Since we will be talking only about bipartite maps in the rest of this chapter, by a slight abuse of notation, we will denote by $F_g \eqdef F_{2,g}$ the OGF of bipartite maps, \textit{i.e.} $2$-constellations. Basically, we are concerned by the resolution of \eqref{eq:4:genus-const} in the case $m=2$ for $g \geq 1$, where it takes the form
\begin{equation} \label{eq:4:bipartite}
F_{g} = xt\left( \Omega F_g + \Gamma F_{g-1} + \sum_{0 \leq g' \leq g} F_{g'} F_{g-g'} \right).
\end{equation}
To give a rough idea of our proof, we essentially solve \eqref{eq:4:bipartite} using an induction on the genus $g$. In the proof, we recycle two ideas of the topological recursion (see \cite{EO,Eynard:book}), but adapt them to an algebraic viewpoint so that we only need to manipulate formal power series. More precisely, the two crucial steps that are directly inspired from the topological recursion, and that differ from the kernel method are Proposition~\ref{prop:4:noPoles!} and Theorem~\ref{thm:4:toprec}. Once these two results are proved, an important part of the work deals then with making explicit the auxiliary variables that underlie the rationality statements (the ``Greek'' variables in Theorem~\ref{thm:4:mainUnrooted} below). This is done in Section~\ref{sec:4:Gamma} and~\ref{sec:4:Y}. Finally, we lift the rationality of bipartite maps (Theorem~\ref{thm:4:mainRooted}) to that of their rotation systems (Theorem~\ref{thm:4:mainUnrooted}) using an \textit{ad hoc} proof, partly relying on a bijective insight from~\cite{chapuy2009asymptotic}.

Our resolution is organized as follows. In the rest of this section, we present our main results (Theorems~\ref{thm:4:mainUnrooted} and~\ref{thm:4:mainRooted}) after a short introduction of extra definitions and notation. In Section~\ref{sec:4:Tutte} we will describe the precise road map to our main result for enumeration of bipartite maps (Theorem~\ref{thm:4:mainRooted}) while stating a list of propositions and lemmas, without proof. The proofs of these propositions and lemmas are given in Section~\ref{sec:4:Gamma} and Section~\ref{sec:4:Y}. Finally, Section~\ref{sec:4:unrooting} gives the proof of Theorem~\ref{thm:4:mainUnrooted} for rotation systems, and Section~\ref{sec:4:comments} collects some final comments.

\horizrule

Our approach to solve \eqref{eq:4:bipartite} also relies on the ``change of variables'' $(t,x)\leftrightarrow (z,u)$ in \eqref{eq:4:tx-zu} that we used in the planar case. When $m=2$, it is specialized to the following equations:
\begin{align}
\label{eq:4:z}
\displaystyle
z=t\left(1+\sum_{k\geq 1} {2k-1 \choose k} p_k z^k \right),
\end{align}
\begin{align}\label{eq:4:u}
u=x(1+zu)^2.
\end{align}
In the following, we will continue to \emph{abuse notation} in the same way that we did for planar constellations, such as switching without warning between a series $H\in \mathbb{B}[x][[t]]$ and its image in $\mathbb{B}[u][[z]]$ via the change of variables. We will also continue to use the single letter $H$ for both objects, relying on the context that should prevent any misunderstanding. 

Other than the generating function $F_g$ for bipartite maps of genus $g$, we will also consider the generating function for rotation systems of bipartite maps, which are defined in Section~\ref{sec:2:map-const}. For a bipartite map $M$ of $n$ edges, all of its rotation systems  $(\sigma_\bullet, \sigma_\circ, \phi)$ are formed by permutations in $S_n$, while the number of black (resp. white) vertices in $M$ is equal to the number of cycles in $\sigma_\bullet$ (resp. $\sigma_\circ$), and the number of faces in $M$ is equal to the number of cycles in $\phi$. We say that a rotation system of a bipartite map $M$ has genus $g$ if $M$ itself is of genus $g$. By Euler's relation, we have
\[
\ell(\sigma_\bullet) + \ell(\sigma_\circ) + \ell(\phi) - n = 2 - 2g
\]
for a rotation system $(\sigma_\bullet, \sigma_\circ, \phi)$ of genus $g$ formed by permutations in $S_n$. We denote by $L_g$ the EGF of such rotation systems of genus $g \geq 1$, with $t$ marking the degree $n$ of the symmetric group $S_n$ and $p_k$ marking the number of cycles of length $k$ in $\phi$. We use EGF here because rotation systems form a labeled combinatorial class.

Later we will often omit in the notation the dependency of generating functions on the variables, for example we will write $L_g$ for $L_g(t; p_1,p_2,\dots)$ and $F_g$ for $F_g(t,x;p_1,p_2,\dots)$. As a final remark, all fields in the following have characteristic~$0$, and for a field $\mathbb{F}$, we denote by $\overline{\mathbb{F}}$ its algebraic closure.

\subsection{Main results} \label{sec:4:mainres}

To express our main results, we first define the following variables, which are series in $z$ with coefficients in $\mathbb{Q}[\varvec{p}]$. They are very similar to the ``variable'' $\gamma$ that we defined in \eqref{eq:4:gamma}. We define the ``variables'' $\eta$ and $\zeta$ as the following formal power series:
\begin{equation} \label{eq:4:eta-zeta}
\eta \eqdef \sum_{k\geq 1} (k-1) {2k-1 \choose k} p_k z^k,
\quad
\zeta \eqdef \sum_{k\geq 1} \frac{k-1}{2k-1} {2k-1 \choose k} p_k z^k.
\end{equation}
We define further two infinite families of ``variables'' $(\eta_i)_{i\geq 1}$ and $(\zeta_i)_{i\geq 1}$ by
\begin{equation} \label{eq:4:eta-zeta-i}
\eta_i\eqdef\sum_{k\geq 1} (k-1) k^i {2k-1 \choose k} p_k z^k ,
\quad
\zeta_i \eqdef \sum_{k\geq 1} \frac{(-2)^{i+1} k(k-1) \cdots (k-i)}{(2k-1)(2k-3) \cdots (2k-2i-1)}{2k-1 \choose k} p_k z^k .
\end{equation}
All these variables are called \mydef{Greek variables}, since they are represented by Greek letters. We recall that, for a partition $\lambda = (\lambda_1, \lambda_2, \ldots, \lambda_k)$, we denote by $p_\lambda$ the product $\prod_{i=1}^k p_{\lambda_i}$.

Our first main result is the following theorem:
\begin{thm}[Main result -- unrooted labeled case ($g \geq 2$)]\label{thm:4:mainUnrooted}
For $g \geq 2$, the EGF $L_g$ of rotation systems of bipartite maps of genus $g$ is given by a finite sum:
\begin{equation}\label{eq:4:mainUnrooted}
L_g = \sum_{\alpha, \beta, a, b} c_{a,b}^{\alpha, \beta} \frac{\eta_\alpha \zeta_\beta}{(1-\eta)^a (1+\zeta)^b}, 
\end{equation}
for rational numbers $c_{a,b}^{\alpha, \beta}$, where the (finite) sum is taken over integer partitions $\alpha, \beta$ and non-negative integers $a,b,$ such that 
$|\alpha| + |\beta| \leq 3(g-1)$ and $a+b = \ell(\alpha)+\ell(\beta)+2g-2$.
\end{thm}

\begin{exa}[unrooted labeled generating function for genus $2$]\label{ex:4:unrootedGenus2}
\begin{multline*}
\footnotesize
L_2=
{\frac {1}{120}}
-{\frac {1}{23040}}{\frac {\eta_{{1}} \left( 185
\eta_{{1}}-58\eta_{{2}} \right) }{ \left( 1-\eta \right) ^{4}
}}
-{\frac {1}{46080}}{\frac {20\eta_{{3}}-168\eta_{{2}}+415
\eta_{{1}}}{ \left(1 - \eta \right) ^{3}}}
-\frac{53/15360}{(1-\eta)^2}
\\
-{\frac {7}{2880}}{\frac {{\eta_{
{1}}}^{3}}{ \left(1- \eta \right) ^{5}}}
-{\frac {1/512}{ \left(
1-\eta \right)  \left( 1+\zeta  \right) }}
+{\frac {
\eta_{{1}}/1536}{ \left( 1-\eta \right) ^{2} \left( 1+
\zeta  \right) }}
-{\frac {3}{1024}} \frac{1}{(1+\zeta )^{2}}
+{\frac {3}{8192}}{\frac {\zeta _{1}}{
 \left( 1+\zeta  \right) ^{3}}}.
\end{multline*}
Here, for instance, the term $-\frac{7}{2880} \frac{\eta_{1}^{3}}{(1-\eta)^5}$ corresponds to $\alpha = (1,1,1)$, $\beta = \epsilon$, $a=5$ and $b=0$, with the coefficient $c^{(1,1,1), \epsilon}_{5,0} = -\frac{7}{2880}$.
\end{exa}

The case of genus $1$ will be stated separately later since it involves logarithms.

\begin{thm}[Unrooted labeled case for genus $1$]\label{thm:4:unrootedGenus1}
The EGF $L_1$ of rotation systems of bipartite maps on the torus is given by the following expression, with the notation of (\ref{eq:4:eta-zeta}):
\[
L_1 = \frac{1}{24} \log \frac{1}{1-\eta} + \frac{1}{8}  \log \frac{1}{1+\zeta}.
\]
\end{thm}
We invite the reader to compare Theorem~\ref{thm:4:mainUnrooted} with \cite[Theorem~3.2]{classical-hurwitz}, \cite[Theorem~1.4]{GGPN} (see also~\cite[Section~1.5]{GGPN}) and \cite[Chapter~3]{Eynard:book} to see the strong similarities of these results.

In order to establish Theorem~\ref{thm:4:mainUnrooted}, we will first prove the following rationality result for the OGF $F_g$ of bipartite maps.

\begin{thm}[Main result -- rooted case]\label{thm:4:mainRooted}
For $g\geq1$, the OGF $F_g$ of rooted bipartite maps of genus $g$ is equal to
\begin{equation}\label{eq:4:mainRooted}
F_g = \sum_{c=1}^{6g-1} \sum_{\alpha, \beta, a \geq 0, b \geq 0} \frac{\eta_\alpha \zeta_\beta}{(1-\eta)^a (1+\zeta)^b} \left( \frac{d_{a,b,c,+}^{\alpha, \beta}}{(1-uz)^c} + \frac{d_{a,b,c,-}^{\alpha, \beta}}{(1+uz)^c} \right) , 
\end{equation}
for $d_{a,b,c,\pm}^{\alpha, \beta} \in \rationals$, with the same notation as in Theorem~\ref{thm:4:mainUnrooted}. Furthermore, $d_{a,b,c,\pm}^{\alpha, \beta} \neq 0$ implies $b \geq \ell(\beta)$, $(2 \pm 1)g \geq \lceil \frac{1+c}{2} \rceil + |\alpha| + |\beta|$ and $a+b = \ell(\alpha)+\ell(\beta)+2g-1$ for the two signs, and the sum above is finite.
\end{thm}

\begin{exa}[rooted generating function for genus $1$]
\begin{multline} \label{eq:4:F1-expr}
F_1 = \frac{(\eta - 2\eta_1 - 1)/16}{(1-uz)^2(1-\eta)^2}+\frac{4(1+ \zeta)\eta_1 + 3\eta^2 - 6\zeta(1-\eta)  + 3}{96(1-uz) (1+\zeta) (1-\eta)^2}-\frac{1/2}{(1-uz)^5 (1-\eta)}\\
 -\frac{5/4}{(1-uz)^4 (1-\eta)}-\frac{1/32}{(1+uz)(1+\zeta)}-\frac{(21\eta - 2\eta_1 - 21)/24}{(1-uz)^3 (1-\eta)^2}.
\end{multline}
Here, for instance, the term $-\frac{5/4}{(1-uz)^4 (1-\eta)}$ corresponds to $\alpha=\beta=\epsilon$, $a=1$, $b=0$, $c=4$, with $d^{\epsilon,\epsilon}_{1,0,4,+} = -5/4$.
\end{exa}

\begin{rmk}
Note that the Greek variables $\eta, \zeta, \eta_i, \zeta_i$ are all infinite \emph{linear} combinations of the $p_kz^k$ with explicit coefficients. Moreover, for fixed $g$ the sums~\eqref{eq:4:mainRooted}, \eqref{eq:4:mainUnrooted} depend only of \emph{finitely many} Greek variables, see \textit{e.g.} Example~\ref{ex:4:unrootedGenus2}. Note also that if only finitely many $p_i$'s are non-zero, then all the Greek letters are polynomials in $z$ and $\varvec{p}$. For example, if $p_i = \mathbf{1}_{i=2}$, \textit{i.e.} if we enumerate bipartite quadrangulations, all Greek variables are linear in $z$. In particular, and since bipartite quadrangulations are in bijection with general rooted maps (see \textit{e.g.}~\cite{Schaeffer:survey}), the rationality results of~\cite{BC} are a (very) special case of our results.
\end{rmk}

We conclude this section by recalling the notation in (\ref{eq:4:gamma}) in the context of bipartite maps ($m=2$). In addition to the Greek variables $\eta, \zeta, (\eta_i)_{i\geq 1}, (\zeta_i)_{i\geq 1}$ already defined, we recall the following expression of $\gamma$ in the bipartite case $m=2$:
\begin{equation} \label{eq:4:gamma-2}
\gamma\eqdef\sum_{k\geq 1} {2k-1 \choose k} p_k z^k. 
\end{equation}
Note that the change of variables~\eqref{eq:4:z} relating $z$ to $t$ is given by $z=t(1+\gamma)$.

\subsection{Proof strategy of Theorem~\ref{thm:4:mainRooted}}
\label{sec:4:Tutte}

The strategy we will use to prove Theorem~\ref{thm:4:mainRooted} is to solve \eqref{eq:4:bipartite} recursively on the genus $g$. Note that, for $g\geq 1$, and assuming that all the series $F_h, \Gamma F_h$ are known for $h<g$, the Tutte equation \eqref{eq:4:bipartite} is \emph{linear} in the unknown series $F_g(x)$. More precisely it is a linear equation for the unknown series $F_g$ involving one catalytic variable (the variable $x$), see \textit{e.g.}~\cite{MBM:icm}. Therefore, it is tempting to solve this equation by means of the kernel method, or a variant of it.

In the following, we will set a threshold $K$ for the degree of faces in bipartite maps that we count, as we did for planar constellations in Section~\ref{sec:4:planar}. Namely, we will only consider bipartite maps whose faces are of degree at most $K$, which translates to the specialization $F_g(t,x;p_1,p_2,\ldots,p_K,0,0\ldots)$ in the OGF $F_g$, and we will also abuse the notation to denote this specialization by $F_g$. After resolution of the equation for this restricted version, by taking a projective limit, we can recover the OGF $F_g$ in full generality. Therefore, we can focus on the restricted version from now on. The same applies to the EGF $L_g$ of rotation systems.

But there is a catch: the substitution of all $p_k$'s by $0$ \textbf{does not commute} with $\Gamma$ defined in (\ref{eq:4:Gamma-def}). For instance, $\Gamma(p_i)|_{\varvec{p}=0}=i x^i$, but $\Gamma\left(p_i|_{\varvec{p}=0}\right)=\Gamma(0)=0$. Therefore,
\[ \left. \left(\Gamma F_g \right)\right|_{p_i=0} ~\neq~ \Gamma \left( \left. F_g \right|_{p_i=0}\right).\]
The reason is that $\Gamma$ taps into faces of every possible degree and turns the selected face to a part of the outer face. Even when we restrict ourselves to bipartite maps with restricted degree of faces, some maps can be the result of joining the outer face with an internal face with arbitrarily large degree via $\Gamma$. Therefore, to count these bipartite maps correctly, we have to use the version of $F_g$ with all $p_k$ in $\Gamma F_g$, and then specialize to $p_k=0$ for all $k > K$.

There is another important notion in our induction. For $\mathbb{K}$ a field containing $z$, a rational function $A(u) \in \mathbb{K}(u)$ of $u$ is called \mydef{$uz$-symmetric} if $A(z^{-2} u^{-1}) = A(u)$, and \mydef{$uz$-antisymmetric} if $A(z^{-2} u^{-1}) = -A(u)$. These notions can be seen as symmetries with respect to the involutive transformation $u z \leftrightarrow u^{-1} z^{-1}$. We now give a more precise description of our induction. The base case $g=1$ of the induction in \eqref{eq:4:F1-expr} will be proved separately without induction hypothesis.

\bigskip

\textbf{Global Induction Hypothesis}: For genus $g\geq 2$, we assume that Theorem~\ref{thm:4:mainRooted} holds for genus $g'$ with $1 \leq g' < g$. In particular, $F_{g'}$ is a rational function of $u$. We further assume that $F_{g'}$ is $uz$-antisymmetric.

\bigskip

We now start examining the induction step. We recall that we use the degree restriction $K$ here, that is, we will solve for $F_g$ with $p_k=0$ for any $k >K$. Our first observation is the following proposition.
\begin{prop}[Kernel form of the Tutte equation]\label{prop:4:TutteY}
Equation~\eqref{eq:4:bipartite} can be rewritten as
\begin{equation} \label{eq:4:TutteY}
Y F_g = x t \Gamma F_{g-1} + xt \sum_{\substack{g_1+g_2 = g \\ g_1,g_2>0}} F_{g_1} F_{g_2} - xtS, 
\end{equation}
where
\[
Y \eqdef 1- 2tx F_0 - tx \theta, \quad \mathrm{with} \quad \theta \eqdef \sum_{k=1}^K \frac{p_k}{x^k},
\]
and $S \equiv S(t,x;\varvec{p})$ is an element of $\mathbb{Q}[\varvec{p}][[t]][x^{-1}]$ of degree at most $K-1$ in $x^{-1}$ without constant term, which depends on $F_g$ among other series.
\end{prop}
\begin{proof}
Consider the Tutte equation \eqref{eq:4:bipartite}. With the degree restriction, by \eqref{eq:4:Omega-restricted} we have
\[ \Omega F_g = [x^{\geq 0}] F_g \theta. \]
Now let $S$ be the negative part of $F_g\theta$, \textit{i.e}: 
\[ S \eqdef [x^{<0}] F_g \theta = F_g \theta - [x^{\geq 0}] F_g \theta. \]
Since $\theta$ is in $\mathbb{Q}[\varvec{p},x^{-1}]$ and of degree $K$ in $x^{-1}$, and since $F_g$ has no constant term in $x$, $S$ is also in $\mathbb{Q}[\varvec{p},x^{-1}]$ and has degree at most $K-1$ in $x^{-1}$. Since $\theta F_g = \Omega F_g + S$, we can now rewrite the equation as
\[F_g = x t \theta F_{g} - xtS +x t F^{(2)}_{g-1} +x t \sum_{\substack{g_1+g_2 = g \\ g_1,g_2 > 0}} F_{g_1} F_{g_2} + 2xtF_0F_g. \]
We now move all terms involving $F_g$ to the left and factor out $F_g$ to obtain~\eqref{eq:4:TutteY}.
\end{proof}

\subsubsection{Rational structure of $F_g$ and the topological recursion}

In this section we describe in detail the structure of the ``kernel'' $Y$ and of the generating function $F_g$, in order to establish our main recurrence equation (Theorem~\ref{thm:4:toprec}). We leave the proofs of the most technical statements to Section~\ref{sec:4:Gamma} and Section~\ref{sec:4:Y}.

\newcommand{\rA}{\mathbb{A}}
\newcommand{\PP}{\mathbb{P}}
In order to analyze the Tutte equation in its kernel form \ref{eq:4:TutteY}, it is natural to study the properties of the kernel $Y$. In the following, we will consider polynomials in $\rA[z][u]$ or $\rA[[z]][u]$ where $\rA \eqdef \mathbb{Q}(\varvec{p})$. Note that any such polynomial, viewed as a polynomial in $u$, is split over $\PP \eqdef \overline{\rA}((z^*))$ the Puiseux series field of variable $z$, defined as in the beginning of Section~\ref{sec:2:gen} (see also Table~\ref{tab:2:series}).

An element $u_0\in\PP$ is \mydef{large} if it starts with a strictly negative power in $z$, and is \mydef{small} if it starts with a strictly positive power in $z$.  The following result is a consequence of Theorem~\ref{thm:4:planar-expr} and some computations that we delay to Section~\ref{sec:4:Y}. As explained in Section~\ref{sec:4:in}, it is implicit in the following that generating functions are considered under the change of variables $(t,x)\leftrightarrow (z,u)$:

\begin{prop}[Rational structure of the kernel]\label{prop:4:structY}
$Y$ is an element of $\mathbb{Q}(z,u;\varvec{p})$ of the form
\[ Y = \frac{N(u)(1-uz)}{u^{K-1}(1+\gamma) (1+uz)}, \]
where $N(u)\in\rA[z][u]$ is a polynomial of degree $2(K-1)$ in $u$.
\end{prop} 
\begin{proof}
See Section~\ref{sec:4:Y}.
\end{proof}

\textbf{Caveat}: The polynomial $N(u)$ here in this section is different from that in Proposition~\ref{prop:4:N-expr} in the previous section on enumeration of planar constellations.

\begin{prop}[Structure of zeros of the kernel] \label{prop:4:structY-zero}
~
\begin{itemize}
\item[(1)] $Y$ is $uz$-antisymmetric.
\item[(2)] Among the $2(K-1)$ zeros of $N(u)$ in $\PP$, $(K-1)$ are small and $(K-1)$ are large. Moreover, large and small zeros are permuted by the transformation $u \leftrightarrow z^{-2}u^{-1}$.
\end{itemize}
\end{prop}
\begin{proof}
See Section~\ref{sec:4:Y}.
\end{proof}

Before solving \eqref{eq:4:bipartite}, we still need to examine more closely the structure of $F_g$. In the following, each rational function $R(u) \in \mathbb{B}(u)$ for some field $\mathbb{B}$ is implicitly considered as an element of $\overline{\mathbb{B}}(u)$. In particular, its denominator is split, and its roots are called \mydef{poles}, which are elements of $\overline{\mathbb{B}}$. Moreover, $R(u)$ has a partial fraction expansion with coefficients in $\overline{\mathbb{B}}$, and the \mydef{residue} of $R(u)$ at a pole $u_* \in \overline{\mathbb{B}}$ is defined as the coefficient of $(u-u_*)^{-1}$ in this expansion, which is an element in $\overline{\mathbb{B}}$. 

The following result is perhaps the most crucial conceptual step that we borrowed from the topological recursion method (\textit{cf.} \cite[Chapter~3]{Eynard:book}) in our proof of Theorem~\ref{thm:4:mainRooted}.
\begin{prop}[Structure and poles of $F_g$]\label{prop:4:noPoles!}
For $g=1$ and for $g\geq 2$ while assuming the Global Induction Hypothesis, $F_g$ is an element of $\rA[[z]](u)$ that is $uz$-antisymmetric. Its poles, which are elements of $\PP$, are contained in $\{z^{-1},-z^{-1}\}$. Moreover, $F_g$ has negative degree in $u$, \textit{i.e.} when written as a rational fraction in $u$, the degree of its numerator in $u$ is strictly smaller than that of its denominator.
\end{prop}

The proof of Proposition~\ref{prop:4:noPoles!} uses the two following lemmas.
\begin{lem}\label{lem:4:stabGamma}
If $A$ is an element of $\mathbb{Q}(u,z,\gamma,\eta,\zeta,(\eta_i)_{i \geq 1}, (\zeta_i)_{i \geq 1})$ with negative degree in $u$ whose poles in $u$ are among $\{\pm z^{-1}\}$, then so is $\Gamma A(u)$. Moreover, if $A(u)$ is $uz$-antisymmetric, then $\Gamma A(x)$ is $uz$-symmetric.
\end{lem}
\begin{proof}
See Section~\ref{sec:4:actionGamma}.
\end{proof}

\begin{lem}\label{lem:4:noSmall}
Let $A(u) \in \mathbb{B}[[z]](u)\cap\mathbb{B}[u]((z)) \subset \mathbb{B}(u)((z))$ be a rational function in $u$ whose coefficients are formal power series in $z$ over some field $\mathbb{B}$, and that as a Laurent series in $z$ has coefficients that are polynomials in $u$. Then $A(u)$, seen as a rational function in $u$, has no small pole. 
\end{lem}
\begin{proof}
By the Newton-Puiseux theorem, we can write $A(u)=\frac{P(u)}{Q_1(u)Q_2(u)}$ as an irreducible fraction with $P(u) \in \overline{\mathbb{B}}((z^*))[u]$, $Q_1(u) = \prod_i(1-u u_i)$ and $Q_2(u)=\prod_j (u-v_j)$, where the $u_i$, $v_j$ are \emph{small} Puiseux series over an algebraic closure $\overline{\mathbb{B}}$ of $\mathbb{B}$ and $v_j$ without constant term. Since $P(u)/Q_2(u) = c A(u)Q_1(u)$, and since $\overline{\mathbb{B}}[u]((z^{*}))$ is a ring, we see that $P(u)/Q_2(u)\in \overline{\mathbb{B}}[u]((z^{*}))$ . But since  $1/Q_2(u) = \prod_j \sum_{k\geq0}u^{-1-k}v_j^k$ is in $\overline{\mathbb{B}}[u^{-1}]((z^{*}))$, this is impossible unless $Q_2$ divides $P$ in $\overline{\mathbb{B}}((z^*))[u]$, which concludes the proof.
\end{proof}

We can now give the proof of Proposition~\ref{prop:4:noPoles!}:

\begin{proof}[Proof of Proposition~\ref{prop:4:noPoles!}]
We recall that in the Global Induction Hypothesis, for a fixed genus $g \geq 2$, Theorem~\ref{thm:4:mainRooted} holds for genus $g'$ with $1 \leq g' < g$, which means that Lemma~\ref{lem:4:stabGamma} applies to $F_{g'}$. Furthermore, the Global Induction Hypothesis also state that $F_{g'}$ is $uz$-antisymmetric.

We first claim that the right-hand side of~\eqref{eq:4:TutteY} is $uz$-symmetric. In the case $g \geq 2$ this follows from the Global Induction Hypothesis, since each term $F_{g_1}F_{g_2}$ is $uz$-symmetric as a product of two $uz$-antisymmetric factors, the term $\Gamma F_{g-1}$ is $uz$-symmetric by Lemma~\ref{lem:4:stabGamma}, and $S$, as any rational fraction in $x$, is $uz$-symmetric since $x(u)=\frac{u}{(1+zu)^2}$ is $uz$-symmetric.

In the case $g=1$, the right-hand side of~\eqref{eq:4:TutteY} is equal to $xt\Gamma F_0 + xtS$, it thus suffices to see that $\Gamma F_0$ is $uz$-symmetric. Now, the series $\Gamma F_0$ is given by
\begin{equation}\label{eq:4:F02}
\Gamma F_{0} = \frac{u^2 z^2}{(1-uz)^4}.
\end{equation}
This expression can be found in \cite[Chapter~3, cylindric amplitude]{Eynard:book} by interpreting $\Gamma F_0$ as the OGF of bipartite maps with an \emph{extra root} on a \emph{special face} other than the outer face, marked by $x^k$ instead of $p_k$. We should note that what \cite{Eynard:book} calls \textit{bipartite} maps are in fact maps with all faces of even degree, and as mentioned in Chapter~3, they do not coincide with bipartite maps in general, except in genus $0$. Therefore, we can still use this result here. It can also be obtained from direct computations from the explicit expression of $F_{m,0}$ given in Theorem~\ref{thm:4:planar-expr} while taking $m=2$, and it is also easily derived from \cite[Theorem~1]{ColletFusy} (in the case $p=r=2$, with the notation of this reference). Since \eqref{eq:4:F02} is clearly $uz$-symmetric, the claim is proved in all cases.

Hence, by Proposition~\ref{prop:4:structY-zero}, $F_g$ is $uz$-antisymmetric, being the quotient of the $uz$-symmetric right-hand side of \eqref{eq:4:TutteY} by $Y$. Now, by the Global Induction Hypothesis and Lemma~\ref{lem:4:stabGamma} (or by a direct check on~\eqref{eq:4:F02} in the case $g=1$), the right-hand side of \eqref{eq:4:TutteY} is in $\rA[[z]](u)$, and its poles are contained in $\{\pm z^{-1},0\}$. Hence, by solving \eqref{eq:4:TutteY} for $F_g$ and by using Proposition~\ref{prop:4:structY}, we deduce that $F_g$ belongs to $\rA[[z]](u)$ and that its only possible poles are $\pm z^{-1}, 0$ and the zeros of $N(u)$.

Now, viewed as a series in $z$, $F_g$ is an element of $\rA[u][[z]]$. Indeed, in the variables $(t,x)$, $F_g$ belongs to $\mathbb{Q}[\varvec{p}][x][[t]]$ for clear combinatorial reasons, and as explained after \eqref{eq:4:tx-zu}, the change of variables $(t,x) \leftrightarrow (z,u)$ preserves the polynomiality of coefficients. Therefore, by Lemma~\ref{lem:4:noSmall}, $F_g$ has no small poles. This excludes $0$ and all small zeros of $N(u)$.
Since $F_g$ is $uz$-antisymmetric and since by Proposition~\ref{prop:4:structY-zero} the transformation $u \leftrightarrow z^{-2}u^{-1}$ exchanges small and large zeros of $N(u)$, this also implies that $F_g$ has no pole at the large zeros of $N(u)$. By Proposition~\ref{prop:4:structY-zero}, all zeros of $N(u)$ are excluded.

The last thing to do is to examine the degree of $F_g$ in $u$. We know that $S$ is a polynomial in $x^{-1}$ of degree at most $K-1$, thus has degree at most $K-1$ in $u$. Therefore, by induction and Lemma~\ref{lem:4:stabGamma} (or by a direct check on \eqref{eq:4:F02} in the case $g=1$) the degree in $u$ of the right-hand side of \eqref{eq:4:TutteY} is at most $K-2$. Since the degree of $Y$ in $u$ is $K-1$, the degree of $F_g$ in $u$ is at most $-1$.
\end{proof}

\begin{rem}\label{rem:4:oldKernel}
Analogues of the previous proposition, stated in similar contexts~\cite[Chap. 3]{Eynard:book} play a crucial role in Eynard's ``topological recursion'' framework.

To understand the importance of Proposition~\ref{prop:4:noPoles!}, let us make a historical comparison. As explained in Example 2 of Section~\ref{sec:2:example}, the ``traditional'' way of solving~\eqref{eq:4:TutteY} using the kernel method would be to substitute in~\eqref{eq:4:TutteY} \emph{all} the small roots of $N(u)$, and use the $(K-1)$ equations thus obtained to eliminate the ``unknown'' polynomial $S$. This approach was historically the first one to be considered, see \textit{e.g.}~\cite{Gao}. It leads to much weaker rationality statements than the kind of methods we use here, since the cancellations that appear between those $(K-1)$ equations are formidable and very hard to track. As we will see, Proposition~\ref{prop:4:noPoles!} circumvents this problem by showing that we just need to study \eqref{eq:4:TutteY} at the \emph{two} points $u=\pm z^{-1}$ rather than at the $(K-1)$ small roots of $N$.
\end{rem}

With Proposition~\ref{prop:4:noPoles!}, we can now apply one of the ideas of the topological recursion (\textit{cf.} \cite[Chapter~3]{Eynard:book}), namely that the whole series $F_g$ can be recovered from the expansion of~\eqref{eq:4:TutteY} at the critical points $u=\pm z^{-1}$. In what follows, we will only consider rational functions of the variable $u$ over $\rA[[z]]$. In particular, we will use the notation $F_g(u)$ to emphasize the fact that $F_g$ is an element of $\rA[[z]](u)=\mathbb{Q}(\varvec{p})[[z]](u)$, or even $\mathbb{Q}[\varvec{p}][[z]](u)$. We denote by $P(u)=\frac{1-uz}{1+uz}$ the \mydef{prefactor} that we will use later in the resolution of \eqref{eq:4:TutteY}. By Proposition~\ref{prop:4:noPoles!} the rational function $P(u)F_g(u)$ has only finite poles at $u=\pm z^{-1}$ and has negative degree in~$u$, which denies any pole at the infinity $\infty$. Therefore, if $u_0$ is some new indeterminate, we can write $P(u_0)F(u_0)$ as the sum of two residues:
\begin{equation}\label{eq:4:residu}
P(u_0)F(u_0) = \mathrm{Res}_{u=\pm z^{-1}} \frac{1}{u_0-u} P(u)F(u).
\end{equation}
This equality only relies on the algebraic fact that the sum of the residues of a rational function at all poles (including $\infty$) is equal to zero, no complex analysis is required. Now, multiplying~\eqref{eq:4:TutteY} by $P(u)$, we have:
\[
P(u)F_g(u) = \frac{xtP(u)H_g(u)}{Y(u)} - \frac{xtP(u)S(x)}{Y(u)}.
\]
where
\begin{equation} \label{eq:4:H-def}
H_g(u)\eqdef \Gamma F_{g-1}(u) + \sum_{\substack{g_1+g_2=g \\ g_1,g_2>0}} F_{g_1}(u)F_{g_2}(u).
\end{equation}
Now observe that the second term in the right-hand side has \emph{no} pole at $u=\pm z^{-1}$: indeed the factor $(1-uz)$ in $Y(u)$ is canceled out by the prefactor $P(u)$, and $xS(x)$ is a polynomial in $x^{-1}=u^{-1}(1+uz)^2$. Returning to \eqref{eq:4:residu} we have proved the following theorem.
\begin{thm}[Topological recursion for bipartite maps]\label{thm:4:toprec}
For $g \geq 2$, the series $F_g(u_0)$ can be expressed as follows:
\begin{equation}\label{eq:4:toprec}
F_g(u_0) = \frac{1}{P(u_0)} \mathrm{Res}_{u=\pm z^{-1}}\frac{P(u)}{u_0-u} \frac{xt}{Y(u)}\left(
\Gamma F_{g-1}(u)+\sum_{g_1+g_2=g\atop g_1,g_2>0} F_{g_1}(u)F_{g_2}(u)
\right).
\end{equation}
\end{thm}
Note that the right-hand side of \eqref{eq:4:toprec} involves only series $F_h$ for $h<g$ and the series $\Gamma F_{g-1}$, which are covered by the induction hypothesis. This contrasts with~\eqref{eq:4:TutteY}, where the term $S(x)$ involves small coefficients of $F_g$ expanded in $z$, which are unknown.

\subsubsection{Proof of Theorem~\ref{thm:4:mainRooted}.}
\label{subsec:4:proofMainRooted}

In order to compute $F_g(u_0)$ from Theorem~\ref{thm:4:toprec}, it suffices to compute the expansion of the rational fraction $H_g(u) Y(u)^{-1}$ at $u=\pm z^{-1}$, with $H_g(u)$ as defined in \eqref{eq:4:H-def}. The expansion of the product $F_{g_1}(u)F_{g_2}(u)$ is well covered by the induction hypothesis, so the focus will be the structure of the term $\Gamma F_{g-1}(u)$ and the derivatives of $Y(u)$ at $u=\pm z^{-1}$. The first aspect requires a closer look at the action of the operator $\Gamma$ on Greek variables, and the second requires specific computation. Note also that, in order to close the induction step, we will need to take the projective limit $K \to \infty$. Therefore, we need to prove not only that the derivatives of $H_g(u) Y(u)^{-1}$ at $u=\pm z^{-1}$ are rational functions in the Greek variables, but also that these functions do not depend on $K$.

In the rest of this section, we apply this program and prove Theorem~\ref{thm:4:mainRooted}, using two intermediate results (Proposition~\ref{prop:4:diffY} and~\ref{prop:4:Gamma-on-Greek-expr} below), whose proofs are delayed to Section~\ref{sec:4:Y} and~\ref{sec:4:Gamma}.

The derivatives of $Y(u)$ at the critical points $u = \pm z^{-1}$ can be computed explicitly with some algebraic work. It is there that the Greek variables appear. In Section~\ref{sec:4:Y} we will prove the following proposition.

\begin{prop}\label{prop:4:diffY}
The rational function $xt P(u) Y(u)^{-1}$ in $u$ has the following formal expansions at $u=\pm z^{-1}$:
\begin{align*}
\frac{xt P(u)}{Y(u)}&= \frac1{4(1-\eta)} + \sum_{\alpha, a \geq 2|\alpha|} c'''_{\alpha, a} \frac{\eta_\alpha}{(1-\eta)^{\ell(\alpha) + 1}} (1-uz)^{a} , \\
\frac{xt P(u)}{Y(u)}&= -\frac1{(1+\zeta)(1+uz)^2} + \sum_{\alpha, a \geq 2|\alpha|} c''_{\alpha, a} \frac{\zeta_\alpha}{(1+\zeta)^{\ell(\alpha) + 1}} (1+uz)^{a-2},
\end{align*}
where $c''_{\alpha,a}, c'''_{\alpha, a}$ are computable rational numbers independent of $K$.
\end{prop}
Note that the proposition above is just a formal way of collecting all the derivatives of $xt P(u) Y(u)^{-1}$ at $u=\pm z^{-1}$, we are not interested in convergence at all here.

The next result, to be proved in Section~\ref{sec:4:actionGamma}, details the action of the operator $\Gamma$ on Greek variables and $uz$.
\begin{prop} \label{prop:4:Gamma-on-Greek-expr}
The operator $\Gamma$ is a derivation on $\mathbb{Q}[p_1,p_2,\dots][x][[t]]$, i.e. it satisfies $\Gamma (AB) = A \Gamma B + B \Gamma A$. Moreover, its action on Greek variables is given by the following expressions:
\begin{align*}
\Gamma uz &= \frac{s^{-2}(s^{-1}-1)^2}{4(1-\eta)}, \\
\Gamma \zeta_i &= \frac{s^{-1} - s}{8(1-\eta)s^2} \left( (2i+1)\zeta_i + \sum_{j=1}^{i-1} (-1)^{j-1} \zeta_{i-j} + 4(-1)^i (1+\zeta) \right), \\
&\quad \quad + \frac1{2}(s^{-1} - s)\left( (2i+1)(s^2-1)^i +\sum_{j=1}^{i-1} (-1)^{j-1} (s^2 - 1)^{i-j} + (-1)^i \right), \\
\Gamma \zeta &= - \frac{s^{-1} - s}{8s^2} + \frac{(s^{-1} - s)(1+\zeta)}{8(1-\eta)s^2} + \frac1{8}(s^{-3} - s^{-1} - 2 + 2s), \\
\Gamma \gamma &= \frac{s^{-1} - s}{4(1-\eta)s^2} (\eta + \gamma) + \frac1{4}(s^{-3} - s^{-1}), \\
\Gamma \eta_i &= \frac{s^{-1} - s}{4(1-\eta)s^2} \eta_{i+1} + \frac1{2^{i+3}} \left( (s-s^{-1}) \frac{d}{ds} \right)^{i+1} (s^{-3}-3s^{-1}+2), 
\end{align*}
where $\displaystyle{s=\frac{1-uz}{1+uz}}$, which is equal to the prefactor $P(u)$ by coincidence. In the expression of $\Gamma \eta_i$, the differential operator $\frac{d}{ds}$ is simply a differentiation by $s$ of a Laurent polynomial in $s$.
\end{prop}

Before proceeding to the proof of Theorem~\ref{thm:4:mainRooted}, we first introduce two notions of degrees that will be very helpful in our proof: the Greek degree and the pole degree. 

Let $\mathbb{G}$ be the sub-ring of $\mathbb{Q}(\eta, \zeta, (\eta_i)_{i \geq 1}, (\zeta_i)_{i \geq 1}, uz)$ formed by polynomials in the variables $(1-\eta)^{-1}$, $(1+\zeta)^{-1}$, $(\eta_i)_{i\geq 1}$, $(\zeta_i)_{i\geq 1}$, $(1-uz)^{-1}$, $(1+uz)^{-1}$. Equivalently, we have $\displaystyle \mathbb{G}\eqdef \mathbb{Q}\left[\frac{1}{1-\eta},\frac{1}{1+\zeta}, (\eta_i)_{i \geq 1}, (\zeta_i)_{i \geq 1},s,s^{-1}\right]$, where $s=(1-uz)(1+uz)^{-1}$. We define another ring $\mathbb{G}_+$ by adding $(1+\zeta)$ as a variable in $\mathbb{G}$, or simply written as $\mathbb{G}_+ = \mathbb{G}[(1+\zeta)]$. It is clear that $s$ is algebraically independent from all Greek variables, since it depends on $u$, thus on $x$, and all Greek variables do not depend on $x$. We now prove that the Greek variables are algebraically independent.

\begin{prop} \label{prop:4:greek-alg-indep}
For a natural number $d>0$, the Greek variables $\eta, \eta_1, \ldots, \eta_d, \zeta, \zeta_1, \ldots, \zeta_d$ are algebraically independent.
\end{prop}
\begin{proof}
The monomials $p_2 z^2, p_3 z^3 \ldots, p_{2d+3} z^{2d+3}$ are clearly algebraically independent, since $z$ depends on $t$, which is independent of all the $p_k$'s. On the other hand, all the Greek variables are of the form $\sum_{k \geq 1} R(k)\binom{2k-1}{k} p_k z^k$, for different rational fractions $R(k)$. We thus only need to prove that the Jacobian matrix $M$ of the chosen Greek variables $\eta, \eta_1, \ldots, \eta_d, \zeta, \zeta_1, \ldots, \zeta_d$ restricted to variables $p_2 z^2, p_3 z^3 \ldots, p_{2d+3} z^{2d+3}$ is of full rank (\textit{cf.} \cite{jacobian}). According to the definition of Greek variables in (\ref{eq:4:eta-zeta}) and (\ref{eq:4:eta-zeta-i}), we have
\[
M = (-2)^{d(d+3)/2} \left( \prod_{k=1}^{2d+2} (k-1) \binom{2k-1}{k} \right) M_1,
\]
with the matrix $M_1$ defined as
\[
M_1 \eqdef 
\begin{bmatrix}
1 & 1 & \cdots & 1 & \cdots & 1 \\
2 & 3 & \cdots & k & \cdots & 2d+3 \\
\vdots & \vdots & \ddots & \vdots & \ddots & \vdots \\
2^d & 3^d & \cdots & k^d & \cdots & (2d+3)^d \\
1/3 & 1/5 &  \cdots & \frac{1}{2k-1} & \cdots & \frac{1}{4d+5} \\
2/3 & 1/5 & \cdots & \frac{k}{(2k-1)(2k-3)} & \cdots & \frac{2d+3}{(4d+5)(4d+3)} \\
\vdots & \vdots & \ddots & \vdots & \ddots & \vdots \\
0 & \cdots & \cdots & \frac{k(k-2)(k-3)\cdots(k-d)}{(2k-1)(2k-3)\cdots(2k-2d-1)} & \cdots & \frac{(2d+3)(2d+1)(2d)\cdots(d+3)}{(4d+5)(4d+3)\cdots(2d+5)}
\end{bmatrix}.
\]
Here, we extracted the common binomial factors $(k-1)\binom{2k-1}{k}$ from each column, and the common factors $(-2)^{i+1}$ from rows corresponding to all $\zeta_i$ (the last $d$ rows). We now only need to prove that $\det(M_1) \neq 0$. To simplify the notation, we define
\[ R(k,i) \eqdef (2k-1)(2k-3)\cdots(2k-2i+1).\]
We observe that $R(k,i)$ is a polynomial in $k$ of degree $i$. Furthermore, we have $R(k,j) = R(k,i)R(k-i,j-i)$.

We will first sort out the last $d+1$ rows. We denote by $Z_1, \ldots, Z_d$ the row vectors of the last $d$ rows. We now introduce another set of row vectors $Z_1', Z_2' ,\ldots, Z_d'$, defined as
\[
Z_i' \eqdef [kR(k,i+1)^{-1}]_{1\leq k-1 \leq 2d+2} = \left[\frac{k}{(2k-1)(2k-3)\cdots(2k-2i-1)}\right]_{1\leq k-1 \leq 2d+2}.
\]
We now prove by induction on $i$ that $Z_1, \ldots, Z_i$ spans the same vector space as $Z_1', \ldots, Z_i'$. The base case is clear since $Z_1 = Z_1'$. We now perform the induction step. Suppose that $\operatorname{Span}(Z_1, \ldots, Z_i) = \operatorname{Span}(Z_1',\ldots,Z_i')$, we now want to study $\operatorname{Span}(Z_1, \ldots, Z_i, Z_{i+1})$, which is equal to $\operatorname{Span}(Z_1', \ldots, Z_i', Z_{i+1})$. We now only need to prove that $Z_{i+1}$ is in $\operatorname{Span}(Z_1', \ldots, Z_i', Z_{i+1}')$ but not $\operatorname{Span}(Z_1', \ldots, Z_i')$, with a non-zero coefficient for $Z_{i+1}'$. We observe that, under the common denominator $R(k,i+1)=(2k-1)(2k-3)\cdots(2k-2i-3)$, the denominator of the $k^{\rm th}$ term of $Z_a'$ for $1 \leq a \leq i+1$ is a polynomial in $k\mathbb{Q}[k]$ of degree $i+2-a$ in $k$, with coefficients independent of $k$. Therefore, the $Z_a'$ for $1 \leq a \leq i+1$ spans the same space as $([k^c/R(k,i+1)]_{1\leq k-1 \leq 2d+2})_{1 \leq c \leq i+1}$, which also contains $Z_{i+1}$. Since $k$ takes $2d+2$ values from $2$ to $2d+3$, which is larger than the degree of numerators of all $Z_a'$, we can reason about linear combinations of all the $Z_a'$ in polynomials in $k$ instead of the $2d+2$ different specializations. Under the common denominator $R(k,i+1)$, the numerator of the $k^{\rm th}$ term of $Z_a$ for $1 \leq a \leq i$ are all divisible by $(2k-2i-3)$, while the numerator of the $k^{\rm th}$ term of $Z_{i+1}$ is $k(k-2)\cdots(k-d)$, not divisible by $(2k-2i-3)$. Therefore, $Z_{i+1}$ cannot be in $\operatorname{Span}(Z_1', \ldots, Z_i')$. We thus finish the induction step.

Since $Z_1, \ldots, Z_i$ spans the same vector space as $Z_1',\ldots,Z_i'$, we can replace the last $d$ rows by $Z_1', \ldots, Z_d'$ without changing the rank of the matrix $M_1$. Similarly, using $\frac{1}{2k-1}+1=\frac{2k}{2k-1}$, we can replace the $(d+2)^{\rm th}$ row (corresponding to $\zeta$) by the row vector $[\frac{k}{2k-1}]_{1 \leq k-1 \leq 2d+2}$, which is a linear combination of the first row and the $(d+2)^{\rm th}$ row of $M_1$. We can thus define a new matrix $M_2$ by
\[
M_2 \eqdef 
\begin{bmatrix}
1 & 1 & \cdots & 1 & \cdots & 1 \\
2 & 3 & \cdots & k & \cdots & 2d+3 \\
\vdots & \vdots & \ddots & \vdots & \ddots & \vdots \\
2^d & 3^d & \cdots & k^d & \cdots & (2d+3)^d \\
2/3 & 3/5 &  \cdots & kR(k,1)^{-1} & \cdots & (2d+3)R(2d+3,1)^{-1} \\
2/3 & 1/5 & \cdots & kR(k,2)^{-1} & \cdots & (2d+3)R(2d+3,2)^{-1} \\
\vdots & \vdots & \ddots & \vdots & \ddots & \vdots \\
2 & \cdots & \cdots & kR(k,d+1)^{-1} & \cdots & (2d+3)R(2d+3,d+1)^{-1}
\end{bmatrix}.
\]
We know that $M_1$ is of full rank if and only if $M_2$ is of full rank.

We now multiply the $(k-1)^{\rm th}$ column of $M_2$ by $R(k,d+1)$ for all $k$ from $2$ to $2d+3$. By the relation $R(k,j) = R(k,i)R(k-i,j-i)$, we have
\begin{align*}
&\prod_{k=2}^{2d+3} R(k,d+1) M_2 = \\
&\begin{bmatrix}
R(2,d+1) & \cdots & R(k,d+1) & \cdots & R(2d+3,d+1) \\
2R(2,d+1) & \cdots & kR(k,d+1) & \cdots & (2d+3)R(2d+3,d+1) \\
\vdots & \ddots & \vdots & \ddots & \vdots \\
2^dR(2,d+1) & \cdots & k^dR(k,d+1) & \cdots & (2d+3)^dR(2d+3,d+1) \\
2R(1,d) &  \cdots & kR(k-1,d) & \cdots & (2d+3)R(2d+2,d) \\
2R(0,d-1) & \cdots & kR(k-2,d-1) & \cdots & (2d+3)R(2d+1,d-1) \\
\vdots & \ddots & \vdots & \ddots & \vdots \\
2 & \cdots & kR(k-d-1,0) & \cdots & (2d+3)R(d+2,0)
\end{bmatrix}.
\end{align*}
Since $R(k,i)$ is a polynomial in $k$ of degree $i$, by looking at the degree of terms in each row, we know that all row vectors except the first are linearly independent, and is triangular with respect to the basis $([k^i]_{1\leq k-1 \leq 2d+2})_{1\leq i \leq 2d+1}$ which spans the same vector space. Since $R(k,d+1)$ has a constant term, the first row is also linearly independent of the other rows. Therefore, $M_2$ is of full rank, which implies that the selected Greek variables are linearly independent.
\end{proof}

The Greek degree, the pole degrees and the $\zeta$-degree are defined for elements of $\mathbb{G}$ as generalized degrees, where each variable is assigned a weight. The degree of a monomial is thus the weighted sum of the powers of each variable, and the degree of a polynomial is the highest degree of its monomials. We take the convention that $0$ is of degree $-\infty$ for all the notions of degree we are going to define.

The \emph{Greek degree}, denoted by $\deg_\gamma$, depends only on Greek variables, \textit{i.e.} $\deg_\gamma(s)=0$, and is defined by: 
\[ \deg_\gamma(s) = 0, \;\; \deg_\gamma((1 - \eta)^{-1}) = \deg_\gamma((1 + \zeta)^{-1}) = -1, \;\; \deg_\gamma(\eta_i) = \deg_\gamma(\zeta_i) = 1 \, \mbox{for} \, i \geq 1. \]
As examples, we have $\deg_\gamma\left(\frac{s^{-1}\eta_1 \zeta_{42}}{1-\eta}\right)=1$, and $\deg_\gamma\left(\frac{3s^2(5\eta_3+7\eta_2\zeta_2)}{(1-\eta)^2(1+\zeta)}\right)=-1$.

The \emph{pole degrees} are defined for each of the two poles $u = \pm 1/z$, and are denoted by $\deg_+$ and $\deg_-$. They depend on both Greek variables and $(1 \pm uz)$ as follows:
\[ \deg_+(s^{-1}) = 1, \deg_+((1-\eta)^{-1}) = \deg_+((1+\zeta)^{-1})=0, \deg_+(\eta_i) = \deg_+(\zeta_i) = 2i \, \mbox{for} \, i \geq 1, \]
\[ \deg_-(s) = 1, \deg_+((1-\eta)^{-1}) = \deg_+((1+\zeta)^{-1})=0, \deg_-(\eta_i) = \deg_-(\zeta_i) = 2i \, \mbox{for} \, i \geq 1. \]
As examples, we have $\deg_+(s) = -1$, $\deg_-(s^{-1}) = -1$, $\deg_+\left(\frac{s^{-3}(1+\zeta_3)}{(1-\eta)^{65535}}\right) = 3$, and $\deg_-\left(\frac{3s^2(5\eta_3+7\eta_2\zeta_2)}{(1-\eta)^2(1+\zeta)}\right)=6$.

The \emph{$\zeta$-degree}, denoted by $\deg_\zeta$, only depends on $(1+\zeta)^{-1}$ and $\zeta_i$ for $i \geq 1$ as follows:
\begin{align*}
&\deg_\zeta((1-\zeta)^{-1})=-1, \deg_\zeta(\zeta_i)=1 \; \mbox{for} \; i \geq 1 \\
&\deg_\zeta(s)=\deg_\zeta(s^{-1})=\deg_\zeta((1-\eta)^{-1})=\deg_\zeta(\eta_i)=0\; \mbox{for} \; i \geq 1.
\end{align*}


Since the variable $s$ and the Greek variables are algebraically independent according to Proposition~\ref{prop:4:greek-alg-indep}, the Greek degree, the pole degrees and the $\zeta$-degree are well-defined. By allowing negative powers of $(1+\zeta)^{-1}$, we can extend the definition of these degrees to $\mathbb{G}_+$. We have the following proposition.

\begin{prop} \label{prop:4:Gamma-degrees}
If $T$ is a monomial in $(1+\zeta)^{-1}$, $(1-\eta)^{-1}$, $s$, $s^{-1}$, $\eta_i$ and $\zeta_i$ for $i \geq 1$, then $\Gamma T$ is in $\mathbb{G}_+$ and is homogeneous in Greek degree. Furthermore,
\begin{align*}
&\deg_\gamma(\Gamma T) = \deg_\gamma(T) - 1, \\ 
&\deg_+(\Gamma T) \leq \deg_+(T) + 5, \; \deg_-(\Gamma T) \leq \deg_-(T) +1, \\
&\deg_\zeta(\Gamma T) = \deg_\zeta(T).
\end{align*}
Moreover, if $T$ satisfies $\deg_\zeta(T) \leq 0$, then $\Gamma T$ is in $\mathbb{G}$.
\end{prop}
\begin{proof}
Since $\Gamma$ is a weighted sum of partial differentiations, we have the following expression for $\Gamma T$:
\begin{equation} \label{eq:4:Gamma-terms}
\Gamma T = (\Gamma s)\frac{\partial}{\partial s} T + (\Gamma \zeta) \frac{\partial}{\partial \zeta} T + (\Gamma \eta) \frac{\partial}{\partial \eta} T + \sum_{i \geq 1} (\Gamma \eta_i) \frac{\partial}{\partial \eta_i} T + \sum_{i \geq 1} (\Gamma \zeta_i) \frac{\partial}{\partial \zeta_i} T.
\end{equation}

With Proposition~\ref{prop:4:Gamma-on-Greek-expr}, we verify that $\Gamma T \in \mathbb{G}_+$. For the rest of the proposition, it suffices to analyze each term for each type of degree using expressions in Proposition~\ref{prop:4:Gamma-on-Greek-expr}.

We start by the Greek degree. According to Proposition~\ref{prop:4:Gamma-on-Greek-expr} and the fact that $s=(1-uz)/(1+uz)$, we have
\[
\Gamma s = - \frac{s^{-2}(s^{-1}-s)^2}{8(1-\eta)}.
\]
We can thus compute the Greek degree of each term, supposing that they are not zero. We have
\begin{align*}
&\deg_\gamma(\Gamma s) = -1, \quad \deg_\gamma\left(\frac{\partial T}{\partial s}\right) = \deg_\gamma(T); \\
&\deg_\gamma(\Gamma \eta) = \deg_\gamma(\Gamma \zeta) = \deg_\gamma(\Gamma \eta_i) = \deg_\gamma(\Gamma \zeta_i) = 0 \;\; \mbox{for} \;\; i \geq 1, \\
&\deg_\gamma\left(\frac{\partial T}{\partial \eta}\right) = \deg_\gamma\left(\frac{\partial T}{\partial \zeta}\right) = \deg_\gamma\left(\frac{\partial T}{\partial \eta_i}\right) = \deg_\gamma\left(\frac{\partial T}{\partial \zeta_9}\right) = \deg_\gamma(T) - 1\;\; \mbox{for} \;\; i \geq 1.
\end{align*}
Therefore, each non-zero term in \eqref{eq:4:Gamma-terms} has Greek degree $\deg_\gamma(T)-1$.


We proceed similarly to pole degrees $\deg_+$ and $\deg_-$. For $\deg_+$, supposing that none of the partial differentiations is zero, we have
\begin{align*}
&\deg_+(\Gamma s) = 4, \quad \deg_+\left(\frac{\partial T}{\partial s}\right) = \deg_+(T)+1; \\
&\deg_+(\Gamma \eta) = 5, \quad \deg_+\left(\frac{\partial T}{\partial \eta}\right) = \deg_+(T); \\
&\deg_+(\Gamma \zeta) = 3, \quad \deg_+\left(\frac{\partial T}{\partial \zeta}\right) = \deg_+(T); \\
&\deg_+(\Gamma \eta_i) = 2i+5, \quad \deg_+\left(\frac{\partial T}{\partial \eta_i}\right) = \deg_+(T) - 2i\;\;\mbox{for}\;\;i\geq 1; \\
&\deg_+(\Gamma \zeta_i) = 2i+3, \quad \deg_+\left(\frac{\partial T}{\partial \zeta_i}\right) = \deg_+(T) - 2i\;\;\mbox{for}\;\;i\geq 1. \\
\end{align*}
Therefore, each non-zero term $T'$ in \eqref{eq:4:Gamma-terms} has $\deg_+(T') \leq \deg_+(T) + 5$.



For $\deg_-$, again by supposing that none of the partial differentiations is zero, we have
\begin{align*}
&\deg_-(\Gamma s) = 0, \quad \deg_-\left(\frac{\partial T}{\partial s}\right) = \deg_-(T)-1; \\
&\deg_-(\Gamma \eta) = -1, \quad \deg_-\left(\frac{\partial T}{\partial \eta}\right) = \deg_-(T); \\
&\deg_-(\Gamma \zeta) = 1, \quad \deg_-\left(\frac{\partial T}{\partial \zeta}\right) = \deg_-(T); \\
&\deg_-(\Gamma \eta_i) = 2i+1, \quad \deg_-\left(\frac{\partial T}{\partial \eta_i}\right) = \deg_-(T) - 2i\;\;\mbox{for}\;\;i\geq 1; \\
&\deg_-(\Gamma \zeta_i) = 2i+1, \quad \deg_-\left(\frac{\partial T}{\partial \zeta_i}\right) = \deg_-(T) - 2i\;\;\mbox{for}\;\;i\geq 1. \\
\end{align*}
Therefore, each non-zero term $T'$ in \eqref{eq:4:Gamma-terms} has $\deg_-(T') \leq \deg_-(T) +1$.


Finally we deal with the $\zeta$-degree. Again, by supposing that the partial differentiations are all non-zero, we have
\begin{align*}
&\deg_\zeta(\Gamma s) = 0, \quad \deg_\zeta\left(\frac{\partial T}{\partial s}\right) = \deg_\zeta(T); \\
&\deg_\zeta(\Gamma \eta) = 0, \quad \deg_\zeta\left(\frac{\partial T}{\partial \eta}\right) = \deg_\zeta(T); \\
&\deg_\zeta(\Gamma \zeta) = 1, \quad \deg_\zeta\left(\frac{\partial T}{\partial \zeta}\right) = \deg_\zeta(T) - 1; \\
&\deg_\zeta(\Gamma \eta_i) = 0, \quad \deg_\zeta\left(\frac{\partial T}{\partial \eta_i}\right) = \deg_\zeta(T)\;\;\mbox{for}\;\;i\geq 1; \\
&\deg_\zeta(\Gamma \zeta_i) = 1, \quad \deg_\zeta\left(\frac{\partial T}{\partial \zeta_i}\right) = \deg_\zeta(T) - 1\;\;\mbox{for}\;\;i\geq 1. \\
\end{align*}
Therefore, for each non-zero term in \eqref{eq:4:Gamma-terms} has $\deg_\zeta(T') = \deg_-(T)$.


For the last statement, we clearly have $\deg_\zeta(\Gamma T) = \deg_\zeta(T) \leq 0$. We then observe that a monomial in $\mathbb{G}_+$ that has strictly positive powers in $(1+\zeta)$ must have strictly positive $\zeta$-degree, which cannot occur in $\Gamma T$ of negative $\zeta$-degree. Therefore, $\Gamma T$ is in $\mathbb{G}$.
\end{proof}

We can now prove our first main result (up to the proofs that have been omitted in what precedes, which will be addressed in the next sections).

\begin{proof}[Proof of Theorem~\ref{thm:4:mainRooted}]
We prove the theorem by induction on the genus $g \geq 1$, with the Global Induction Hypothesis plus the hypothesis about various degrees of $F_g$, namely $\deg_\gamma(F_g)=1-2g$, $\deg_+(F_g)\leq 6g-1$, $\deg_-(F_g) \leq 2g-1$ and $\deg_\zeta(F_g)\leq 0$.

We consider~\eqref{eq:4:toprec} in Theorem~\ref{thm:4:toprec}. Proposition~\ref{prop:4:diffY} implies that all terms in the expansion of $xtP(u)/Y(u)$ at $u=\pm z^{-1}$ are rational fractions in the Greek variables, with denominator of the form $(1-\eta)^a (1+\zeta)^b$ for $a, b \geq 0$. Moreover, the expressions of these terms in Greek variables do not depend on $K$, or we can say that all the dependence on $K$ lies in Greek variables.

When $g\geq 2$, from the induction hypothesis and the case of non-positive $\zeta$-degree in Proposition~\ref{prop:4:Gamma-on-Greek-expr}, the quantity $H_g$ defined in (\ref{eq:4:H-def}) is a rational fraction in $u,z$ and the Greek variables, with denominator of the form $(1-\eta)^a (1+\zeta)^b (1 \pm uz)^c$ for $a,b,c\geq 0$. This rational function does not depend on $K$ (when written in the Greek variables). The same is true for $g=1$ using the explicit expression of $\Gamma F_0$ given by~\eqref{eq:4:F02}.

Therefore, the evaluation of each residue in~\eqref{eq:4:toprec} is a rational function of Greek variables, independent of $K$, and with denominator of the form $(1-\eta)^a (1+\zeta)^b (1 \pm uz)^c$, with $a,b,c\geq 0$.

\smallskip


We now prove the degree conditions for $F_g$ using the induction hypothesis for degrees.

We first look at $H_g$, in the case $g\geq 2$. It has two parts: the sum part, which is $\sum_{g' = 1}^{g-1} F_{g'} F_{g-g'}$, and the operator part, which is $\Gamma F_{g-1}$. We analyze the degree for both parts. For the sum part, it is easy to see that any term $T$ in the sum is homogeneous of Greek degree $\deg_\gamma(T) = 2 - 2g$, the pole degrees satisfy $\deg_+(T) \leq 6g-2$ and $\deg_-(T) \leq 2g-2$, and the $\zeta$-degree is at most $0$. For the operator part, it results from Proposition~\ref{prop:4:Gamma-degrees} that $\Gamma F_{g-1}$ is a sum of terms $T$ homogeneous of Greek degree $2-2g$, and $\deg_+(\Gamma F_{g-1}) \leq 6g-2$, $\deg_-(\Gamma F_{g-1}) \leq 2g-2$, and lastly $\deg_\zeta(\Gamma F_{g-1}) \leq 0$. Therefore, the result from the sum part and the operator part agrees, and $H_g$ thus satisfies the same conditions as its two parts. For $g=1$, the same bound holds, as one can check from the explicit expression of $H_1=xt\Gamma F_0$ following from~\eqref{eq:4:F02}.

We now observe from Proposition~\ref{prop:4:diffY} that all terms appearing in the expansion of $xtP/Y$ at $u = \pm z^{-1}$ are homogeneous of Greek degree $-1$. Therefore, all the terms in the expansion of $xtPH_g/Y$ at $u=\pm z^{-1}$ have Greek degree $\deg_\gamma(H_g) + \deg_\gamma(xtP/Y) = 1-2g$. For the pole degrees, we notice from Proposition~\ref{prop:4:diffY} that $\deg_+(xtP/Y) \leq 0$ and $\deg_-(xtP/Y) \leq 2$. Similar to the Greek degree, counting also the contribution from $P$, we have $\deg_+(F_g) = \deg_+(H_g) + \deg_+(xtP/Y) + 1 \leq 6g-1$ and $\deg_-(F_g) = \deg_-(H_g) + \deg_-(xtP/Y) - 1 \leq 2g-1$. For the $\zeta$-degree, we check that $\deg_\zeta(xtP/Y)=0$, which makes $\deg_\zeta(F_g) = \deg_\zeta(H_g) + \deg_\zeta(xtP/Y) \leq 0$. We thus complete the induction step.

We have proved that, under the specialization $p_i=0$ for $i>K$, the series $F_g$ has the form stated in Theorem~\ref{thm:4:mainRooted}. But, since the numbers $d_{a,b,c,\pm}^{\alpha,\beta}$ do not depend on $K$, we can let $K\rightarrow \infty$ in \eqref{eq:4:mainRooted} and conclude that this equality holds without considering this specialization. This concludes the proof of Theorem~\ref{thm:4:mainRooted}.
\end{proof}

\paragraph{Overview of omitted proofs} \quad \\
We have just proved Theorem~\ref{thm:4:mainRooted}, however using several intermediate statements without proof (yet) in order to (we hope) clarify the global structure of the proof. All these statements will be proved in Section~\ref{sec:4:Gamma} and \ref{sec:4:Y}. In order to help the reader check that we do not forget any proof(!), we list here the statements stated without proof so far, and indicate where their proofs locate.
\begin{itemize}
\item Proposition~\ref{prop:4:Gamma-on-Greek-expr} and Lemma~\ref{lem:4:stabGamma}, which deal with the action of the operator $\Gamma$, are proved at the end of Section~\ref{sec:4:Gamma}. The rest of Section~\ref{sec:4:Gamma} contains other propositions and lemmas that prepare these proofs.

\item Proposition~\ref{prop:4:structY} is proved in Section~\ref{sec:4:Y}, where we also prove Proposition~\ref{prop:4:structY-zero}. Proposition~\ref{prop:4:diffY} is also proved in Section~\ref{sec:4:Y}. This proof is rather long, due to the sheer volume of explicit computations using the explicit expression of the series $F_0$.
\end{itemize}
Therefore, at the end of Section~\ref{sec:4:Gamma} and~\ref{sec:4:Y}, the proof of Theorem~\ref{thm:4:mainRooted} will be completed (without omissions). The two remaining statements (Theorem~\ref{thm:4:mainUnrooted} and \ref{thm:4:unrootedGenus1}) about rotation systems will be deduced from Theorem~\ref{thm:4:mainRooted} in Section~\ref{sec:4:unrooting}.

\paragraph{Reference of notation} ~\\
Since the following proofs are based on heavy computations, we now offer a summary of notation that we will define and use later in the form of Tables. Table~\ref{tab:4:series} contains a list of formal power series that are crucial to our proofs, while Table~\ref{tab:4:domains} contains a list of domains in which we work, and Table~\ref{tab:4:operators} is a list of operators that we will use. Entries include pointers to related propositions. These tables can be used as a reference of notation.

\begin{table}[!hptb]
  \centering
  \begin{tabular}{ccc}
    \toprule
    \quad & Definition & Related Propositions \\
    \midrule
    $z$ & $\displaystyle z=t\left(1+\binom{2k-1}{k}p_kz^k\right)$ & Prop.~\ref{prop:4:diff-of-vars}, \ref{prop:4:Gamma-on-uz} \\
    \midrule
    $u$ & $\displaystyle u=x(1+uz)^2$ & Prop.~\ref{prop:4:diff-of-vars}, \ref{prop:4:Gamma-on-uz} \\
    \midrule
    $s$ & $\displaystyle \frac{1-uz}{1+uz}$ & Prop.~\ref{prop:4:Gamma-on-uz}, \ref{prop:4:Theta-base} \\
    \midrule
    $\gamma, \eta, \zeta, (\eta_i)_{i\geq 1}, (\zeta_i)_{i\geq 1}$ & Eq.~\eqref{eq:4:gamma-2}, \eqref{eq:4:eta-zeta}, \eqref{eq:4:eta-zeta-i} & Prop.~\ref{prop:4:Gamma-on-Greek-expr}, \ref{prop:4:Theta-base} \\
    \midrule
    $\theta$ & $\displaystyle \sum_{k=1}^K p_k x^{-k}$ & Prop.~\ref{prop:4:TutteY}, \ref{prop:4:structY} \\
    \midrule
    $Y$ & $1-tx(2F_0+\theta)$ & Prop.~\ref{prop:4:TutteY}, \ref{prop:4:structY} \\
    \bottomrule
  \end{tabular}
  \caption{List of series}
  \label{tab:4:series}
\end{table}

\begin{table}[!hptb]
  \centering
  \begin{tabular}{cc}
    \toprule
    \quad & Definition \\
    \midrule
    $\greeks$ & $\{ \gamma, \eta, \zeta, (\eta_i)_{i\geq 1}, (\zeta_i)_{i \geq 1}\}$ \\
    $\mathbb{G}$ & $\displaystyle \mathbb{Q}\left[ \frac1{1-\eta}, \frac1{1+\zeta}, (\eta_i)_{i \geq 1}, (\zeta_i)_{i \geq 1}, s, s^{-1}\right]$ \\
    $\mathbb{G_+}$ & $\displaystyle \mathbb{Q}\left[ \frac1{1-\eta}, \frac1{1+\zeta}, (1+\zeta), (\eta_i)_{i \geq 1}, (\zeta_i)_{i \geq 1}, s, s^{-1}\right]$ \\
    \bottomrule
  \end{tabular}
  \caption{List of sets and domains}
  \label{tab:4:domains}
\end{table}

\begin{table}[!hptb]
  \centering
  \begin{tabular}{ccc}
    \toprule
    \quad & Definition & Related Propositions \\
    \midrule
    $\Gamma$ & $\displaystyle \sum_{k \geq 1} k x^k \frac{\partial}{\partial p_k}$ & Prop.~\ref{prop:4:Gamma-on-Greek-expr}, \ref{prop:4:Gamma-derivative}, \ref{prop:4:Gamma-on-uz}, \ref{prop:4:Gamma-on-Greek} \\
    \midrule
    $\Theta$ & $\forall k \geq 1, p_k z^k \mapsto x^k z^k$ & Prop.~\ref{prop:4:Theta-base}, \ref{prop:4:Gamma-on-Greek} \\
    \midrule
    $D$ & $\forall k \geq 1, p_k z^k \mapsto k p_k z^k$ & Prop.~\ref{prop:4:Theta-base}, \ref{prop:4:Gamma-on-Greek}, \ref{prop:4:decomposeGamma} \\
    \midrule
    $\partial_{p_k}$ & \begin{minipage}[t]{0.4\columnwidth}Formal differentiation by $p_k$ that ignores the dependence of $z$ on $p_k$\end{minipage} & -- \\
    \midrule
    $\Diamond$ & $\displaystyle \sum_{k \geq 1} p_k \partial_{p_k}$ & Prop.~\ref{prop:4:decomposeGamma}, \ref{prop:4:diamondLg}, \ref{LgAsIntegral} \\
    \midrule
    $\Box$ & $\displaystyle \forall \lambda, \forall k \geq 1, \forall a, x^k p_\lambda z^a \mapsto \left( \frac1{k} - \frac{\gamma}{1+\gamma}\right)p_k p_\lambda z^a$ & Prop.~\ref{prop:4:decomposeGamma}, \ref{prop:4:diamondLg} \\
    \midrule
    $\Pi$ & $\forall k \geq 1, x^k \mapsto p_k$ & Prop.~\ref{prop:4:decomposeGamma}, \ref{lem:4:faceedge} \\
    \midrule
    $\Xi$ & $\displaystyle \forall k \geq 1, x^k \mapsto \frac{p_k}{k}$ & Prop.~\ref{prop:4:decomposeGamma}, \ref{lem:4:faceedge} \\
    \bottomrule
  \end{tabular}
  \caption{List of operators}
  \label{tab:4:operators}
\end{table}

\subsection{Structure of the Greek variables}
\label{sec:4:Gamma}

In this section we establish several properties of the Greek variables defined in Section~\ref{sec:4:mainres}. In particular we will prove Proposition~\ref{prop:4:Gamma-on-Greek-expr} and Lemma~\ref{lem:4:stabGamma}. We also fix some notation that will be used in the rest of this chapter. In this section, we will exceptionally consider all $p_k$'s, without considering the restriction by $K$.

\subsubsection{A projected version of the Greek variables}
\label{subsec:4:defRingsOperators}

We start by fixing some notation and by defining some spaces and operators that will be used throughout the rest of the chapter.
First we let $\greeks \eqdef \{ \gamma, \eta, \zeta, (\eta_i)_{i\geq 1}, (\zeta_i)_{i \geq 1}\}$ be the set of all Greek variables defined in (\ref{eq:4:eta-zeta}) and (\ref{eq:4:eta-zeta-i}).
Elements of $\greeks$ are infinite linear combinations of $p_k z^k$. Acting on such objects, we define the linear operators:
\begin{align}\label{eq:4:defTheta}
\Theta \eqdef p_k z^k\mapsto x^k z^k, \quad \quad D \eqdef p_k z^k \mapsto k p_k z^k.
\end{align}
We notice that $D$ coincides with the operator $z\partial/\partial z$ in the domain at which we are looking. Nevertheless, we will still use the current definition of $D$ to underline the fact that it is defined on (potentially infinite) linear combinations of $p_k z^k$. Recall that the variable $z \equiv z(t;p_1,p_2,\dots)$ defined by~\eqref{eq:4:z} is an element of $\mathbb{Q}[p_1,p_2,\dots][[t]]$ without constant term. Therefore, each formal power series $A\in\mathbb{Q}[x,p_1,p_2,\dots][[z]]$ is an element of $\mathbb{Q}[x,p_1,p_2,\dots][[t]]$. Recall that, in this ring, the operator $\Gamma$ is defined by:
\[
\Gamma = \sum_{k\geq 1}k x^k \frac{\partial}{\partial p_k}, 
\]
where $\frac{\partial}{\partial p_k}$ is the partial differentiation with respect to $p_k$ in $\mathbb{Q}[x,p_1,p_2,\dots][[t]]$. We now introduce another operator $\partial_{p_k}$, given by the partial differentiation with respect to $p_k$ in $\mathbb{Q}[x,p_1,p_2,\dots][[z]]$ \textit{omitting the dependency of $z$ in $p_k$}. Equivalently, $\partial_{p_k}$ is defined in  $\mathbb{Q}[x,p_1,p_2,\dots][[z]]$ by the formula:
\begin{equation}\label{eq:4:defpartialpk}
\frac{\partial}{\partial p_k} = \frac{\partial z}{\partial p_k} \frac{\partial}{\partial z} +  \partial_{p_k}.
\end{equation}

Our first statement deals with the action of $\Theta$ on elements of $\greeks$. The operator $\Theta$ can be seen as a projection of infinitely many variables to polynomials in the single variable $x$. Here and later it will be convenient to work with the variable $s$ defined by
\begin{equation}\label{def:4:s}
s \eqdef \frac{1-uz}{1+uz}.
\end{equation}
\begin{prop} \label{prop:4:Theta-base}
The action of the operator $\Theta$ on elements of $\greeks$ is given by:
\begin{align*}
 \Theta \gamma &= \frac1{2}(s^{-1} - 1) , &
 \Theta \zeta &= \frac1{4}s^{-1}(s-1)^2, \\
 \Theta \eta &=  \frac1{4}(s^{-1}-1)(s^{-2} + s^{-1} - 2), & 
 \Theta \zeta_i &= (s^{-1} - s) (s^2 - 1)^i, ~ i\geq 1, \\
 \Theta \eta_i &=  \frac1{2^{i+2}} \left((s - s^{-1}) \frac{\partial}{\partial s}\right)^i (s^{-3} - 3s^{-1} + 2), ~ i\geq 1. 
\end{align*}
In particular, the images $\Theta (\eta + \gamma)$, $\Theta (\zeta - \gamma)$, $\Theta \eta_i$, $\Theta \zeta_i$ for $i \geq 1$ is a basis of the vector space $(s^{-1}-s)\mathbb{Q}[s^2, s^{-2}]$.
\end{prop}
\begin{proof}
The proof is elementary but let us sketch the computations that are not totally obvious if not performed in the right way. We observe, and will use several times, that by the Lagrange inversion formula, one has $[x^\ell]s = -\frac{2}{\ell} {2\ell-2 \choose \ell -1}z^\ell$ for any $\ell \geq 1$. Our proof will take two parts, the first to prove the expressions of the images of Greek variables under the projection $\Theta$, the second to study the vector space they span.

\smallskip

\noindent \textbf{Expression of Greek variables projected by $\boldsymbol{\Theta}$} ~\\
By definition we have $\Theta \gamma = \sum_{k \geq 1} \binom{2k-1}{k} x^k z^k$, so to prove the first equality we need to show that $[x^k]\frac1{2}s^{-1}={2k-1\choose k}z^k$ for all $k \geq 1$.
To this end, we first observe by a direct computation that $s^2=1-4xz$, which implies that $2x (\partial s/\partial x)=s-s^{-1}$. It follows that 
\[
[x^k]\frac1{2}s^{-1} = [x^k]\frac1{2}\left( s- 2x\frac{\partial}{\partial x}s\right)=
(1-2k)[x^k] s,
\]
which is equal to ${2k-1 \choose k}z^k$ from the observation above.
The value of $\Theta \zeta$ is easily checked similarly, namely \[ [x^k]\frac1{4}(s+s^{-1}) = [x^k]\frac1{2}\left(s-x\frac{\partial}{\partial x}s\right) = \frac{1-k}{2}[x^k]s=\frac{k-1}{2k-1}{2k-1 \choose k} z^k. \]

To check the value of $\Theta \zeta_i$, we observe again that $s^2-1=-4xz$ and $2x\frac{\partial}{\partial x}s=s-s^{-1}$, therefore 
\begin{align*}
&\quad [x^k](s^{-1}-s)(s^2-1)^i = (-4z)^i[x^{k-i}](s^{-1}-s) = (-4z)^i \cdot 2(i-k) [x^{k-i}]s \\
&= (-1)^{i+1} 2^{2i+1} {2k-2i-2 \choose k-i-1} z^k = \frac{(-2)^{i+1}k(k-1)\dots(k-i)}{(2k-1)(2k-3)\dots(2k-2i-1)}{2k-1\choose k} z^k,
\end{align*} 
which agrees with what we expect from the definition of $\zeta_i$.

To compute $\Theta \eta$ and $\Theta \eta_i$, we first notice that $\Theta D = x(\partial / \partial x)\Theta$, and we observe that 
\[\eta = D \gamma - \gamma,\; \eta_1 = D \eta, \; \eta_i = D \eta_{i-1}. \]
We can then compute the action of $\Theta$ on these variables:
\[ \Theta \eta = \left( x\frac{\partial}{\partial x} - 1 \right) \Theta \gamma = \frac1{4}(s^{-3} - 3s^{-1} + 2), \]
\[ \Theta \eta_i = \left( x\frac{\partial}{\partial x} \right)^i \Theta \eta = \frac1{2^{i+2}} \left( (s - s^{-1}) \frac{\partial}{\partial s} \right)^i (s^{-3} - 3s^{-1} + 2). \]

\smallskip

\noindent \textbf{Spanned vector space} ~\\
We now prove the last statement of the proposition. We have $\Theta (\zeta - \gamma) = (s-s^{-1})/4$ and $\Theta \zeta_i = (s^{-1} - s) (s^2 - 1)^i$ of degree $2i+1$ in $s$, and they form a triangular basis of $(s^{-1}-s)\mathbb{Q}[s^{2}]$. We observe that $\Theta (\eta + \gamma) = (s-s^{-1}) s^{-2}/4$ and $\Theta \eta_i$ is in $(s^{-1}-s)s^{-2}\mathbb{Q}[s^{-2}]$ of degree $2i+1$ in $s^{-1}$, and also that they form a triangular basis for $(s^{-1}-s)s^{-2}\mathbb{Q}[s^{-2}]$. This proves that altogether these variables span the whole desired space.
\end{proof}

The next proposition collects some partial derivatives of our main variables that will be useful afterwards.
\begin{prop} \label{prop:4:diff-of-vars}
We have the following expressions of partial derivatives of the variable sets $(t,x)$ and $(z,u)$:
\begin{align*}
\frac{\partial u}{\partial x} = \frac{(1+uz)^{3}}{1-uz}, &\quad \frac{\partial z}{\partial t} = \frac{(1+\gamma)^2}{1-\eta}, \\
\frac{\partial u}{\partial t} = \frac{2(1+\gamma)^{2}u^{2}}{(1-\eta)(1-uz)}, &\quad \frac{\partial z}{\partial x} = 0 \\
\frac{\partial z}{\partial p_k} = \frac{\binom{2k-1}{k} z^{k+1}}{1-\eta}, &\quad \frac{\partial u}{\partial p_k} = \frac{2 u^2 \binom{2k-1}{k}z^{k+1}}{(1-uz)(1-\eta)} \\
\end{align*}
\end{prop}
\begin{proof}
The proof is a simple check from the definitions of $u$, $z$ and the Greek variables in (\ref{eq:4:u}), (\ref{eq:4:z}), (\ref{eq:4:eta-zeta}) and \eqref{eq:4:eta-zeta-i}, via implicit differentiation. 
\end{proof}

\subsubsection{Action of $\Gamma$ and proofs of Proposition~\ref{prop:4:Gamma-on-Greek-expr} and Lemma~\ref{lem:4:stabGamma}}
\label{sec:4:actionGamma}

We are now ready to study more explicitly the action of $\Gamma$. The next statement is obvious from the definition of $\Gamma$.
\begin{prop} \label{prop:4:Gamma-derivative}
The operator $\Gamma$ is a derivation, \textit{i.e.} $\Gamma (AB) = A \Gamma B + B \Gamma A$.
\end{prop}
\begin{proof}
  Since $\Gamma = \sum_{k \geq 1} kx^k (\partial / \partial p_k)$ is a weighted sum of partial differential operators, the result thus follows.
\end{proof}

The action of $\Gamma$ on variables $u,z,s$ can be examined by direct computation.
\begin{prop} \label{prop:4:Gamma-on-uz}
We have
\[ \Gamma z = \frac{z s^{-2} (s^{-1} - s)}{4(1-\eta)}, \quad \Gamma u = \frac{u s^{-2} (s^{-1} - 1) (s^{-1} - s)}{4(1-\eta)}, \quad \Gamma s = - \frac{(s^{-1} - s)^2}{8(1-\eta)s^2}.  \]
\end{prop}
\begin{proof}
We proceed by direct computation by recalling the differentials computed in Proposition~\ref{prop:4:diff-of-vars}:
\begin{align*}
\Gamma z &= \sum_{k \geq 1} k x^k \frac{\partial}{\partial p_k} t \left( 1 + \sum_{m \geq 1} \binom{2m-1}{m} p_m z^m \right) \\
&= t\sum_{k \geq 1} k \binom{2k-1}{k} x^k z^k + \frac1{1+\gamma} (\Gamma z) \sum_{k \geq 1} k \binom{2k-1}{k} p_k z^k \\
&= \frac{z}{1+\gamma} \Theta (\gamma + \eta) + \frac1{1+\gamma} (\Gamma z) (\gamma + \eta).
\end{align*}
By solving this linear equation, we obtain $\Gamma z$.

To obtain $\Gamma u$, we notice that $\Gamma$ is a derivation and apply it to $u=x(1+uz)^{2}$ to obtain
\[
\Gamma u = u(1+uz)^{-2} \cdot 2(1+uz)\left( (\Gamma u)z + (\Gamma z)u \right),
\]
which leads to the expression of $\Gamma u$.

Finally, using the fact that $\Gamma$ is a derivation and the expressions of $\Gamma z$ and $\Gamma u$, we easily verify the expression of $\Gamma s$.
\end{proof}

\begin{prop} \label{prop:4:Gamma-on-Greek}
For $G$ a linear combination of Greek variables, we have
\[
\Gamma G = \left( \frac{s^{-1} - s}{4(1-\eta)s^2} + \Theta \right) D G
\]
\end{prop}
\begin{proof}
Since $G$ is a linear combination of Greek variables, it is an infinite linear combination of $p_k z^k$. Recalling the definition \eqref{eq:4:defpartialpk} of the operator $\partial_{p_k}$, we have:
\begin{align*}
\Gamma G = \sum_{k \geq 1} k x^k \frac{\partial}{\partial p_k} G &=
 \sum_{k \geq 1} k x^k \frac{\partial z }{\partial p_k} \frac{\partial}{\partial z} G
+ 
 \sum_{k \geq 1}k x^k \partial_{p_k} G \\ &= \sum_{k \geq 1} k x^k \frac{\partial z}{\partial p_k} z^{-1} DG + \Theta D G \\
&= \left( z^{-1} (\Gamma z) + \Theta \right) DG = \left( \frac{s^{-1} - s}{4(1-\eta)s^2} + \Theta \right) DG,
\end{align*}
where the last equality uses the value of $\Gamma z$ given by Proposition~\ref{prop:4:Gamma-on-uz}.
\end{proof}

We are now prepared to prove Proposition~\ref{prop:4:Gamma-on-Greek-expr} and Lemma~\ref{lem:4:stabGamma}.

\begin{proof}[Proof of Proposition~\ref{prop:4:Gamma-on-Greek-expr}]
The fact that $\Gamma$ is a derivation was proved in Proposition~\ref{prop:4:Gamma-derivative}. The action of $\Gamma$ on $uz$ can be deduced easily from Proposition~\ref{prop:4:Gamma-on-uz}. To obtain explicit formulas giving the action of $\Gamma$, we use Proposition~\ref{prop:4:Gamma-on-Greek}. For $G \in \greeks$, the value of $DG$ is given by the following list, which is straight-forward to prove from the definitions:
\begin{align*}
D \gamma &= \eta + \gamma, \, \, D \eta = \eta_1, \, \, D \zeta = \frac{\eta}{2} + \frac{\zeta}{2}, \\
D \eta_i &= \eta_{i+1}, \\
D \zeta_i &= \frac{1}{2}\left( (2i+1)\zeta_i + \sum_{j=1}^{i-1} (-1)^{j-1} \zeta_{i-j} + 4(-1)^i (\zeta + \eta) \right).
\end{align*} 

Since all the quantities appearing on the right-hand side of these equalities are linear combinations of elements of $\greeks$, their images by $\Theta$ can be computed thanks to Proposition~\ref{prop:4:Theta-base}. Therefore, using Proposition~\ref{prop:4:Gamma-on-Greek}, we can compute explicitly the value of $\Gamma G$ for $G\in \greeks$, which leads to the values given in Proposition~\ref{prop:4:Gamma-on-Greek-expr}.
\end{proof}

\begin{proof}[Proof of Lemma~\ref{lem:4:stabGamma}]
For $A \in \mathbb{Q}(u,z,\greeks)$, since the operator $\Gamma$ is a derivation, we have the following equality.
\[ \Gamma A = (\Gamma u)\frac{\partial}{\partial u} A + (\Gamma z) \frac{\partial}{\partial z} A + (\Gamma \zeta) \frac{\partial}{\partial \zeta} A + (\Gamma \gamma)\frac{\partial}{\partial \gamma} A + (\Gamma \eta) \frac{\partial}{\partial \eta} A + \sum_{i \geq 1} (\Gamma \eta_i) \frac{\partial}{\partial \eta_i} A + \sum_{i \geq 1} (\Gamma \zeta_i) \frac{\partial}{\partial \zeta_i} A. \]
By Proposition~\ref{prop:4:Gamma-on-uz} and Proposition~\ref{prop:4:Gamma-on-Greek-expr}, with the fact that $s=(1-uz)(1+uz)^{-1}$, we easily verify that $\Gamma A$ is also an element of $\mathbb{Q}(u,z,\greeks)$. Moreover, if the poles of $A$ in $u$ are among $\pm z^{-1}$, then so are the poles of $\Gamma A$. Note also that since $s$ has degree $0$ in $u$, the quantity $\Gamma G$ for $G \in \{z\} \cup \greeks$ has degree $0$. Since $\Gamma u$ has degree $1$, and since differentiations decrease the degree by 1, we conclude that the degree of $\Gamma A$ is at most the degree of $A$, both as rational fractions in $u$.

We now assume that $A$ is $uz$-symmetric. For $G \in \{z\} \cup \greeks$, the operator $\frac{\partial}{\partial G}$ preserves the $uz$-antisymmetry, and according to Proposition~\ref{prop:4:Gamma-on-uz} and Proposition~\ref{prop:4:Gamma-on-Greek-expr}, $\Gamma G$ is $uz$-antisymmetric. Therefore, $(\Gamma G)\frac{\partial}{\partial G}A$ is $uz$-symmetric, being the product of two $uz$-antisymmetric factors. For $u$, according to Proposition~\ref{prop:4:Gamma-on-uz}, $u^{-1} \Gamma u$ is $uz$-symmetric. We now inspect $\frac{u\partial}{\partial u} A$. By $uz$-antisymmetry, $A(u) = - A(u^{-1} z^{-2})$, and we have
\[ \frac{u\partial}{\partial u} A(u) = - \frac{u\partial}{\partial u} A(u^{-1} z^{-2}) = u^{-1} z^{-2} \frac{\partial A}{\partial u} (u^{-1} z^{-2}), \]
so $\frac{u\partial}{\partial u} A$ is $uz$-symmetric. Therefore, all terms in the expression of $\Gamma A$ above are $uz$-symmetric, thus also $\Gamma A$.
\end{proof}

\subsection{Structure of the kernel and its expansions at critical points}
\label{sec:4:Y}

In this section we study the kernel $Y(u)$ at the points $u=\pm z^{-1}$ via explicit computations. This is the place where the Greek variables emerge. The purpose of this section is to give the proofs of the propositions concerning $Y$, namely Proposition~\ref{prop:4:structY}, Proposition~\ref{prop:4:structY-zero} and Proposition~\ref{prop:4:diffY}, which will conclude the proof of all auxiliary results stated in proof of Theorem~\ref{thm:4:mainRooted}. We recall that $Y(u)$ is defined as
\[
Y = 1 -xt (2 F_0+\theta).
\]

\bigskip

\begin{proof}[Proof of Proposition~\ref{prop:4:structY}]
We can rewrite $\theta$ in the following form:
\begin{align*}
\theta &= \sum_{k=1}^{K} \frac{p_k}{x^k} = \sum_{k=1}^{K} \frac{p_k (1+uz)^{2k}}{u^k} = (1+zu) \sum_{k=1}^{K} p_k z^k \frac{(1+zu)^{2k-1}}{u^k z^k} \\
&= (1+uz) \sum_{k=1}^{K} p_k z^k \sum_{\ell=0}^{2k-1} (uz)^{\ell-k} \binom{2k-1}{\ell} \\
&= (1+uz) \sum_{k=1}^{K} p_k z^k \sum_{\ell=-k}^{k-1} u^{\ell} z^{\ell} \binom{2k-1}{k+\ell}.
\end{align*}
We recall the expression of $F_0$ in Theorem~\eqref{thm:4:planar-expr} to compute $2F_0 + \theta$ directly:
\begin{align*}
2F_0 + \theta &= (1+uz) \left( 2 - \sum_{k=1}^{K} p_k z^k \left( 2\sum_{\ell=1}^{k-1} u^\ell z^\ell \binom{2k-1}{k+\ell} - \sum_{\ell=-k}^{k-1} u^\ell z^\ell \binom{2k-1}{k+\ell} \right) \right) \\
&= (1+uz) \left( 2 - \sum_{k=1}^{K} p_k z^k \left( \sum_{\ell=1}^{k-1} u^\ell z^\ell \binom{2k-1}{k+\ell} - \sum_{\ell=-k}^{0} u^\ell z^\ell \binom{2k-1}{k+\ell} \right) \right).
\end{align*}
We observe that $u^K (2F_0 + \theta) = (1+uz) Q(u)$ with $Q(u)$ polynomial in $u$ of degree $2K-1$. The polynomial $Q(u)$ has the additional property that $[u^k] Q(u)$ is a polynomial in $z$, and for $k \geq K-1$, $[u^k] Q(u)$ is divisible by $z^{2(k-K)+1}$. 

We now evaluate $2F_0 + \theta$ at the point $u=z^{-1}$:
\begin{align*}
(2F_0 + \theta)\Big|_{u=z^{-1}} &= 2 \left( 2 - \sum_{k=1}^{K} p_k z^k \left( \sum_{\ell=1}^{k-1} \binom{2k-1}{k+\ell} - \sum_{\ell=-k}^{0} \binom{2k-1}{k+\ell} \right) \right) \\
&= 2 \left( 2 + 2 \sum_{k=1}^{K} p_k z^k \binom{2k-1}{k} \right) \\
&= 4 + 4\gamma.
\end{align*}
Therefore $\left. Q\right|_{u=z^{-1}} = (2 + 2\gamma) z^{-K}$. We now write
\[ Y = 1 - xt(2F_0 + \theta) = \frac{(1+uz)^2 (1+\gamma) - uz(2F_0 + \theta)}{(1+uz)^2 (1+\gamma)}, \]
so that
\[ (1+uz) (1+\gamma) u^{K-1} Y = (1+uz) (1+\gamma) u^{K-1} - z Q(u). \]
When evaluated at $u=1/z$, the right-hand side vanishes, which means that the left-hand side, which is a polynomial in $u$ of degree $2K-1$, has $(1-uz)$ as factor. We can thus write:
\[ Y = \frac{N(u) (1-uz)}{u^{K-1}(1+uz)(1+\gamma)} \]
with $N(u)$ polynomial in $u$ of degree $2(K-1)$.
\end{proof}

\bigskip

\begin{proof}[Proof of Proposition~\ref{prop:4:structY-zero}]
We first observe that $Y^2$ is $uz$-symmetric. Indeed, using the \emph{quadratic method} (see \textit{e.g.}~\cite{BC:planar, MBM:icm} and Example 3 in Chapter~\ref{sec:2:example}), we can rewrite the Tutte equation~\eqref{eq:4:planar-const} of planar $m$-constellations for $m=2$ as follows:
\[ (1 - xt(2F_0 + \theta))^2 = x^2 t^2 \theta^2 - 4xt - 2xt \theta + 1 - 4xt(\Omega F_0 - \theta F_0). \]
The right-hand side is a Laurent polynomial in $x$, therefore it is $uz$-symmetric. Since $Y=1-xt(2F_0+\theta)$, we conclude that $Y^2$ is $uz$-symmetric. Now, since $Y$ is a Laurent polynomial in $uz$, it follows that $Y$ is either $uz$-symmetric or $uz$-antisymmetric (indeed $Y(u)^2-Y^2(z^{-2}u^{-1})$ is equal to zero and factors into $(Y(u)-Y(z^{-2}u^{-1}))(Y(u)+Y(z^{-2}u^{-1}))$, so one of the two factors must be equal to zero, as a Laurent polynomial).

To determine whether $Y$ is $uz$-symmetric or $uz$-antisymmetric, we examine its poles at $uz=0$ and $uz=\infty$. Using the expression $Y=1-xt(2F_0+\theta)$, the definition of $\theta$ and the explicit expression of $F_0$ given in Theorem~\eqref{thm:4:planar-expr}, it is straightforward to check that:
\[
Y(u) \sim - t p_k /(uz)^{k-1} \mbox{ when } uz \to 0 \; , \;
Y(u) \sim  t p_k (uz)^{k-1} \mbox{ when } uz \to \infty.
\]
We conclude that $Y$ is $uz$-antisymmetric.

Now we study the zeros of $N(u)$ by studying the \mydef{Newton polygon} of $N(u)$, defined as the convex hull of the points $(i,j)\in\mathbb{R}^2$ such that the monomial $u^iz^j$ occurs in $N(u)$. 

We will rely on the computations done in the previous proof. We first observe that $[u^{K-1}]((1+uz) (1+\gamma) u^{K-1} - z Q(u))$ is a polynomial in $z$ with a constant term $1$, therefore the same holds for $[u^{K-1}] N(u)$, which implies that the point $B = (K-1, 0)$ occurs in the Newton polygon of $N(u)$. Moreover, we observe that $[u^{0}]((1+uz) (1+\gamma) u^{K-1} - z Q(u)) = - [u^0] z Q(u)$. But $[u^0] Q(u) = p_K$, therefore the point $A = (0,1)$ is present in the Newton polygon of $N(u)$. For any $k < K-1$, since $[u^k] Q(u)$ is a polynomial in $z$, the point $(k,0)$ is never in the Newton polygon of $N(u)$. Therefore, the segment $AB$ is a side of the Newton polygon of $N(u)$, and accounts for the $(K-1)$ small roots of $N(u)$, whose series expansions start with the power $z^{1/(K-1)}$.

We then observe that $[u^{2K-1}]((1+uz) (1+\gamma) u^{K-1} - z Q(u)) = -z[u^{2K-1}]Q(u) = p_K z^{2K}$. Therefore, the point $C = (2(K-1), 2K-1)$ occurs in the Newton polygon of $N(u)$. Furthermore, for any $ k > K-1$, $[u^{k}]((1+uz) (1+\gamma) u^{K-1} - z Q(u)) = -z[u^{k}]Q(u)$, and $[u^k]Q(u)$ is divisible by $z^{k-K+1}$, thus $[u^k]N(u)$ is divisible by $z^{2(k-K)+2}$. The point corresponding to this term is $(k, 2(k-K)+2)$, and is always above the segment $BC$. We conclude that $BC$ is a side of the Newton polygon of $N(u)$, which accounts for the $(K-1)$ large roots of $N(u)$, whose series expansions start with the power $z^{-(2K-1)/(K-1)}$.

It remains to prove that the transformation $u \to u^{-1}z^{-2}$ exchanges large and small zeros of $N(u)$. Let $u_0$ be a small zero of $N(u)$, it is also a zero of $Y(u)$. But $Y$ is $uz$-antisymmetric, therefore $Y(u_0) = Y(u_0^{-1} z^{-2})$, thus $u_0^{-1} z^{-2}$ is also a zero of $Y(u)$, and it is clearly not $1/z$. The only possibility is that $u_0^{-1} z^{-2}$ is a zero of $N(u)$ that is also large. Since the transformation $u \leftrightarrow u^{-1} z^{-2}$ is involutive, we conclude that it exchanges small and large zeros of $N(u)$. 
\end{proof}

\bigskip

We now study the expansion of $Y(u)$ at critical points. This is where (finally!) Greek variables appear, and what explains their presence in Theorem~\ref{thm:4:mainRooted}. 

We will start by the Taylor expansion of $2F_0 + \theta$. Since we are computing the Taylor expansion by successive differentiation by $u$, for simplicity, we will use the shorthand $\partial_u$ for $\partial / \partial u$. For integers $\ell$ and $a$, we define the \mydef{falling factorial} $(\ell)_a$ to be $(\ell)_a = \ell (\ell - 1) \dots (\ell -a+1)$.

\begin{prop} \label{prop:4:taylor-pos-pre}
At $u=1/z$, we have the following Taylor expansion of $2F_0 + \theta$:
\[ 2F_0+\theta = 4+4\gamma - 2(1-\eta)(1-uz) + \sum_{a \geq 2} (1-uz)^a \left( (\eta + \gamma) + \sum_{i=1}^{\lfloor \frac{a-1}{2} \rfloor} c^+_{i,a} \eta_i \right). \]
Here $c_{i,a}^+$ are rational numbers depending only on $i, a$.
\end{prop}
\begin{proof}
We proceed by computing successive derivatives evaluated at $u=1/z$. In the proof of Proposition~\ref{prop:4:structY}, we already showed that $\left. (2F_0 + \theta)\right|_{u=1/z} = 4 + 4\gamma$, which accounts for the first term. For other terms, we rewrite the expression of $2F_0 + \theta$ we used in the proof of Proposition~\ref{prop:4:structY} by grouping together powers of $(uz)$:
\begin{align*}
2F_0 + \theta &= (1+uz) \left( 2 - \sum_{k=1}^{K} p_k z^k \left( \sum_{\ell=1}^{k-1} u^\ell z^\ell \binom{2k-1}{k+\ell} - \sum_{\ell=-k}^{0} u^\ell z^\ell \binom{2k-1}{k+\ell} \right) \right) \\
&= (2+2uz) - \sum_{k=1}^{K} p_k z^k \left( \sum_{\ell=1}^{k-1} u^\ell z^\ell \binom{2k-1}{k+\ell} + \sum_{\ell=2}^{k} u^\ell z^\ell \binom{2k-1}{k+\ell-1} \right) \\
&\quad + \sum_{k=1}^{K} p_k z^k \left( \sum_{\ell=-k}^{0} u^\ell z^\ell \binom{2k-1}{k+\ell} + \sum_{\ell=-k+1}^{1} u^\ell z^\ell \binom{2k-1}{k+\ell-1} \right) \\
&=  (2+2uz) + \sum_{k=1}^{K} p_k z^k \left(  \sum_{\ell=-k}^{0} u^\ell z^\ell \binom{2k}{k+\ell} - \sum_{\ell=2}^{k} u^\ell z^\ell \binom{2k}{k+\ell} + \frac{2}{k+1} \binom{2k-1}{k} uz \right).
\end{align*}
Now we compute the differentiation of $2F_0+\theta$ by $u$ evaluated at $u=1/z$:
\begin{align*}
\left. \partial_u (2F_0 + \theta) \right|_{u=1/z} &= 2z + z \sum_{k=1}^{K} p_k z^k \left(  \sum_{\ell=-k}^{0} \ell \binom{2k}{k+\ell} - \sum_{\ell=2}^{k} \ell \binom{2k}{k+\ell} + \frac{2}{k+1} \binom{2k-1}{k} \right) \\
&= 2z - z \sum_{k=1}^{K} p_k z^k \left(  \sum_{\ell=0}^{k} \ell \binom{2k}{k+\ell} + \sum_{\ell=2}^{k} \ell \binom{2k}{k+\ell} + \frac{2}{k+1} \binom{2k-1}{k} \right) \\
&= 2z - z \sum_{k=1}^{K} p_k z^k (2k-2) \binom{2k-1}{k} = 2z(1-\eta).
\end{align*}

For any $a \geq 2$, the $a$-th differentiation of $2F_0 + \theta$ by $u$ evaluated at $u=1/z$ is
\begin{align*}
\left. \partial_u^a (2F_0 + \theta)\right|_{u=1/z} &= z^a \sum_{k=1}^{K} p_k z^k \left(  \sum_{\ell=-k}^{0} (\ell)_a \binom{2k}{k+\ell} - \sum_{\ell=2}^{k} (\ell)_a \binom{2k}{k+\ell} \right) \\
&= z^a \sum_{k=1}^{K} p_k z^{k} \left( \sum_{\ell=1}^{k} (-1)^a \binom{2k}{k+\ell} (\ell+a-1)_a - \sum_{\ell=1}^{k} \binom{2k}{k+\ell} (\ell)_a \right).
\end{align*}

We first compute the quantity $\sum_{\ell=1}^{k} (\ell)_a \binom{2k}{k+\ell}$ given $a \geq 2$ fixed for any $k$. It is natural to consider the following formal power series:
\[ D_a(y) = \sum_{k \geq 0} y^k \sum_{\ell=1}^{k} (\ell)_a \binom{2k}{k+\ell} = a! \sum_{k \geq 0} y^k \sum_{\ell=1}^{k} \binom{\ell}{a} \binom{2k}{k+\ell}. \]
We choose to compute $D_a$ via a combinatorial interpretation in terms of lattice paths.

Note that the number $[y^k]D_a/a!$ counts paths of length $2k$ with up steps $(1,1)$ and down steps $(-1,1)$, starting from the origin and ending at height $2\ell$ (with $k+\ell$ up steps and $k-\ell$ down steps), alongside with a set of $a$ elements chosen among integers from $0$ to $2\ell$ called the \mydef{set of heights}. By decomposing the whole path at the last passage for each height in the set of heights, we have the following equality:
\[ D_a(y) = a! E(y)(1+C(y))C(y)^a. \]
Here, $E(y)$ is the OGF of paths of even length ending at $0$, and $C(y)$ is the OGF of paths of even length ending at a strictly positive height. In both OGFs, the variable $y$ marks the half-length of paths. All these series are classically expressed in terms of the series of Dyck paths as follows. Let $B(y)$ be the series of Dyck paths, we have by classical decompositions $E(y) = \frac{2}{1-(B(y)-1)} - 1$ and $C(y) = \frac{y B(y)^2}{1-y B(y)^2}$. But we know that $B(y)$ satisfies the equation $B(y) = 1 + y B(y)^2$ as showed in Example 2 of Section~\ref{sec:2:example}, so we finally obtain the expression of $D_a$:
\[ B(y) = \frac{1-\sqrt{1-4y}}{2y}, \quad E(y) = \frac1{\sqrt{1-4y}}, \quad C(y) = \frac1{2}\left( \frac1{\sqrt{1-4y}} - 1 \right), \]
\[ D_a(y) = \frac{a!}{2^{a+1}} \frac1{\sqrt{1-4y}} \left( \frac1{1-4y} - 1 \right) \left( \frac1{\sqrt{1-4y}} - 1 \right)^{a-1}.  \]

We now compute the quantity $\sum_{\ell=1}^{k} (\ell+a-1)_a \binom{2k}{k+\ell}$ given $a \geq 2$ fixed for any $k$. We consider the following OGF:
\[ T_a(y) = \sum_{k \geq 0} y^k \sum_{\ell=1}^{k} (\ell+a-1)_a \binom{2k}{k+\ell} = a! \sum_{k \geq 0} y^k \sum_{\ell=1}^{k} \binom{\ell+a-1}{a} \binom{2k}{k+\ell}. \]
The combinatorial interpretation of $T_a(y)$ is essentially the same as $D_a(y)$, but in this case the $a$ heights are not necessarily distinct (or to say that the set of heights becomes the \emph{multiset} of height). Therefore, we have
\[ T_a(y) = a! E(y)(1+C(y))^{a}C(y) = \frac{a!}{2^{a+1}} \frac1{\sqrt{1-4y}} \left( \frac1{1-4y} - 1 \right) \left( \frac1{\sqrt{1-4y}} + 1 \right)^{a-1}. \]

Since $[p_k z^{k+a}]\left. \partial_u^a (2F_0 + \theta)\right|_{u=1/z} = [y^k]((-1)^a T_a(y) - D_a(y))$, we now consider $(-1)^a T_a(y) - D_a(y)$, which gives
\[ (-1)^a T_a(y) - D_a(y) = \frac{a!(-1)^a}{2^{a+1}} \frac1{\sqrt{1-4y}} \frac{4y}{1-4y} \left( \left( \frac1{\sqrt{1-4y}} + 1 \right)^{a-1} + \left( 1 - \frac1{\sqrt{1-4y}} \right)^{a-1} \right). \]

We observe that, no matter $a$ is even or odd, when viewed as a polynomial in $\frac1{\sqrt{1-4y}}$, $(-1)^a T_a(y) - D_a(y)$ is always a linear combination of terms of the form $4y (1-4y)^{t-3/2}$ for $t \in \naturals$, and we also observe that $[y^k]4y (1-4y)^{t-3/2} = [x^k z^k] 4xz (1-4xz)^{t-3/2}$. We thus have the following expression of $\left. \partial_u^a (2F_0 + \theta)\right|_{u=1/z}$:
\begin{align*}
&\quad \left. \partial_u^a (2F_0 + \theta)\right|_{u=1/z}  \\
&= z^a \sum_{k = 1}^{K} p_k z^k [y^k] \left( \frac{a!(-1)^a}{2^{a+1}} \frac1{\sqrt{1-4y}} \frac{4y}{1-4y} \left( \left( \frac1{\sqrt{1-4y}} + 1 \right)^{a-1} + \left( 1 - \frac1{\sqrt{1-4y}} \right)^{a-1} \right) \right) \\
&= z^a \Theta^{-1} \left( \frac{a!(-1)^a}{2^{a}} s^{-2} (s^{-1} - s) \sum_{i=0}^{\lfloor \frac{a-1}{2} \rfloor} \binom{a-1}{2i} s^{-2i} \right).
\end{align*}

We observe that $\Theta \eta_1 = \frac{8}{3} s^{-3} (s^{-1} - s)^2$, and since $\Theta \eta_{i+1} = (s - s^{-1}) \partial_s \Theta \eta_{i}$, by induction on $i$ we know that $\Theta \eta_i$, as a Laurent polynomial in $s$, has a factor $(s-s^{-1})^2$ for $i \geq 1$. Therefore, from Proposition~\ref{prop:4:Theta-base} we know that, for any polynomial $P(s^{-2})$ in $s^{-2}$, $\Theta^{-1} \left( s^{-2}(s - s^{-1})P(s^{-2}) \right)$ is a linear combination of $(\eta+\gamma)$ and $\eta_i$ for $i \geq 0$, and we have $[\eta+\gamma]\Theta^{-1} \left( s^{-2}(s - s^{-1})P(s^{-2}) \right) = 4P(1)$ by the fact that $\Theta(\eta+\gamma)=s^{-2}(s - s^{-1})/4$. We thus have
\begin{align*}
\left. \partial_u^a (2F_0 + \theta)\right|_{u=1/z} &= a! z^a (-1)^a \left( (\eta + \gamma) + \sum_{i=1}^{\lfloor \frac{a-1}{2} \rfloor} c^+_{i,a} \eta_i \right),
\end{align*}
for some rational numbers $c^+_{i,a}$.
\end{proof}

We now perform a very similar computation for the other pole $u=-1/z$.

\begin{prop} \label{prop:4:taylor-neg-pre}
At $u=-1/z$, we have the following Taylor expansion of $2F_0 + \theta$:
\[ 2F_0+\theta = 2(1+\zeta)(1+uz) +  \sum_{a \geq 2} (1+uz)^a \left( (\zeta - \gamma) + \sum_{i=1}^{\lfloor \frac{a-1}{2} \rfloor} c^-_{i,a} \zeta_i \right).  \]
Here $c_{i,a}^-$ are rational numbers depending only on $i, a$.
\end{prop}
\begin{proof}
For the constant term, we first recall the following expression in the proof of Proposition~\ref{prop:4:structY}:
\[
2F_0 + \theta = (1+uz) \left( 2 - \sum_{k=1}^{K} p_k z^k \left( \sum_{\ell=1}^{k-1} u^\ell z^\ell \binom{2k-1}{k+\ell} - \sum_{\ell=-k}^{0} u^\ell z^\ell \binom{2k-1}{k+\ell} \right) \right).
\]
It is obvious that $\left. (2F_0 + \theta) \right|_{u=-1/z}$ vanishes.

For other terms, we will recycle the following expression of $2F_0 + \theta$ in the proof of Proposition~\ref{prop:4:taylor-pos-pre}:
\[
2F_0 + \theta = (2+2uz) + \sum_{k=1}^{K} p_k z^k \left(  \sum_{\ell=-k}^{0} u^\ell z^\ell \binom{2k}{k+\ell} - \sum_{\ell=2}^{k} u^\ell z^\ell \binom{2k}{k+\ell} + \frac{2}{k+1} \binom{2k-1}{k} uz \right).
\]
The first-order differentiation becomes
\begin{align*}
&\quad \left. \partial_u (2F_0 + \theta) \right|_{u=-1/z} \\
&= 2z - z \sum_{k=1}^{K} p_k z^k \left(  \sum_{\ell=-k}^{0} (-1)^{\ell} \ell \binom{2k}{k+\ell} - \sum_{\ell=2}^{k} (-1)^{\ell} \ell \binom{2k}{k+\ell} - \frac{2}{k+1} \binom{2k-1}{k} \right) \\
&= 2z - z\sum_{k=1}^{K} p_k z^k \frac{2-2k}{2k-1} \binom{2k-1}{k} = 2z(1+\zeta).
\end{align*}

For any $a \geq 2$, the $a$-th differentiation of $2F_0 + \theta$ evaluated at $u=-1/z$ is
\begin{align*}
\left. \partial_u^a (2F_0 + \theta)\right|_{u=-1/z} &= z^a \sum_{k=1}^{K} p_k z^{k} \left( \sum_{\ell=1}^{k} (-1)^\ell \binom{2k}{k+\ell} (\ell+a-1)_a - \sum_{\ell=1}^{k} (-1)^{\ell - a} \binom{2k}{k+\ell} (\ell)_a \right).
\end{align*}

We now borrow the combinatorial interpretation presented in the proof of Proposition~\ref{prop:4:taylor-pos-pre}. Consider the following OGFs:

\[ \tilde{D}_a(y) = \sum_{k \geq 0} y^k \sum_{\ell=1}^{k} (-1)^\ell (\ell)_a \binom{2k}{k+\ell} = a! \sum_{k \geq 0} y^k \sum_{\ell=1}^{k} (-1)^\ell \binom{\ell}{a} \binom{2k}{k+l}, \]

\[ \tilde{T}_a(y) = \sum_{k \geq 0} y^k \sum_{\ell=1}^{k} (-1)^\ell (\ell+a-1)_a \binom{2k}{k+\ell} = a! \sum_{k \geq 0} y^k \sum_{\ell=1}^{k} (-1)^\ell \binom{\ell+a-1}{a} \binom{2k}{k+\ell}. \]

We can see that $[p_k z^{k+a}] \left. \partial_u^a (2F_0 + \theta) \right|_{u=-1/z} = [y^k](\tilde{T}_a(y) - (-1)^a \tilde{D}_a(y))$. Furthermore, these series have combinatorial interpretation similar to $D_a(y)$ and $T_a(y)$ in the proof of Proposition~\ref{prop:4:taylor-pos-pre}, with the only difference that the parity of the height at the end also contributes as a sign. We define $\tilde{C}(y) = \frac{-y B(y)^2}{1+y B(y)^2}$. We have the following equalities, with $C(y)$ and $E(y)$ borrowed from the proof of Proposition~\ref{prop:4:taylor-pos-pre}:

\[ \tilde{C}(y) = \frac1{2}\left( \sqrt{1-4y} - 1 \right), \]
\[ \tilde{D}_a(y) = a! E(y) (1+\tilde{C}(y)) C(y)^a = \frac{a!}{2^{a+1}} \frac{-4y}{\sqrt{1-4y}} \left( \sqrt{1-4y} - 1 \right)^{a-1}, \]
\[ \tilde{T}_a(y) = a! E(y) (1+\tilde{C}(y))^a C(y) = \frac{a!}{2^{a+1}} \frac{-4y}{\sqrt{1-4y}} \left( \sqrt{1-4y} + 1 \right)^{a-1}. \]

Therefore, we have
\[ \tilde{T}_a(y) - (-1)^a \tilde{D}_a(y) = \frac{a!}{2^{a+1}} \frac{-4y}{\sqrt{1-4y}} \left( \left( 1 + \sqrt{1-4y} \right)^{a-1} + \left( 1 - \sqrt{1-4y} \right)^{a-1} \right). \]

We observe that for any value of $a \geq 2$, $\tilde{T}_a(y) - (-1)^a \tilde{D}_a(y)$ is a linear combination of terms of the form $\frac{-4y}{\sqrt{1-4y}} (1-4y)^t$, and we also observe that $[y^k]\frac{-4y}{\sqrt{1-4y}} (1-4y)^t = [x^k z^k] \frac{-4xz}{\sqrt{1-4xz}} (1-4xz)^t$. We thus have the following expression of $\left. \partial_u^a (2F_0 + \theta)\right|_{u=-1/z}$:
\begin{align*}
&\quad \left. \partial_u^a (2F_0 + \theta)\right|_{u=-1/z} \\
&= z^a \sum_{k=1}^{K} p_k z^{k} [y^k] \left( \frac{a!}{2^{a+1}} \frac{-4y}{\sqrt{1-4y}} \left( \left( 1 + \sqrt{1-4y} \right)^{a-1} + \left( 1 - \sqrt{1-4y} \right)^{a-1} \right) \right) \\
&= \frac{z^a a!}{2^{a}} \Theta^{-1} \left( (s-s^{-1}) \sum_{i=0}^{\lfloor \frac{a-1}{2} \rfloor} \binom{a-1}{2i} s^{2i} \right). \\
\end{align*}

We observe that $\Theta (\zeta - \gamma) = (s - s^{-1})/4$, and $\Theta \zeta_i = (s^{-1} - s)(s^2 - 1)^i$. Therefore, for any polynomial $P$, we have that $\Theta^{-1} ((s - s^{-1}) P(s^2))$ is a linear combination of $\zeta - \gamma$ and $\zeta_i$, and $[\zeta - \gamma]\Theta^{-1} ((s - s^{-1}) P(s^2)) = 4P(1)$. Hence, we have
\[ 
\left. \partial_u^a (2F_0 + \theta)\right|_{u=-1/z} = z^a a! \left( (\zeta - \gamma) + \sum_{i=1}^{\lfloor \frac{a-1}{2} \rfloor} c^-_{i,a} \zeta_i \right)
\]
for some rational numbers $c^-_{i,a}$.
\end{proof}

With Proposition~\ref{prop:4:taylor-pos-pre} and Proposition~\ref{prop:4:taylor-neg-pre}, we can prove Proposition~\ref{prop:4:diffY}.

\begin{proof}[Proof of Proposition~\ref{prop:4:diffY}]
We will first rewrite $xtP/Y$ as
\[ \frac{xt P}{Y} = \frac{1-uz}{1+uz} \frac1{(1+\gamma)\frac{(1+uz)^2}{uz} - (2F_0 + \theta)}. \]
And now we substitute the Taylor expansion of $2F_0 + \theta$ at $u = \pm z^{-1}$ into the above formula to obtain the Taylor expansion of $xtP/Y$ at $u=\pm z^{-1}$. 

We first treat the point $u=1/z$:
\begin{align*}
\frac{xtP}{Y} &= \frac{1-uz}{(8 - 4(1-uz) + 2(1-uz)^2 + \sum_{i \geq 3} (1-uz)^i) (1+\gamma) - (2 - (1-uz)) (2F_0 + \theta)} \\
&=  \frac{1}{4(1-\eta) - \sum_{a \geq 2} (1-uz)^a \left( (\eta - 1) + \sum_{i = 1}^{\lfloor a/2 \rfloor} (2c_{i,a+1}^+ - c_{i,a}^+ ) \eta_i  \right)} \\
&= \frac1{4(1-\eta)} + \sum_{\alpha, a \geq 2|\alpha|} c'''_{\alpha, a} \frac{\eta_\alpha}{(1-\eta)^{\ell(\alpha) + 1}} (1-uz)^{a}.
\end{align*}

The treatment for $u=-1/z$ is similar:
\begin{align*}
\frac{xtP}{Y} &= \frac{2-(1+uz)}{ - \sum_{i \geq 3} (1+uz)^i (1+ \gamma) - (1+uz)(2F_0 + \theta)} \\
&= \frac{-2+(1-uz)}{(1+uz)^2 \left[ 2(1+\zeta) +  \sum_{a \geq 2} (1+uz)^{a-1} \left( (1 + \zeta) +  \sum_{i=1}^{\lfloor \frac{a-1}{2} \rfloor} c^-_{i,a} \zeta_i \right) \right]} \\
&= -\frac1{(1+\zeta)(1+uz)^2} + \sum_{\alpha, a \geq 2|\alpha|} c''_{\alpha, a} \frac{\zeta_\alpha}{(1+\zeta)^{\ell(\alpha) + 1}} (1+uz)^{a-2}. \qedhere 
\end{align*}
\end{proof}

At this point, we have finished the proof of the rooted case of our main results (Theorem~\ref{thm:4:mainRooted}), including all the statements that had been stated in Section~\ref{sec:4:Tutte}. It remains to prove the unrooted labeled case (Theorem~\ref{thm:4:mainUnrooted} and Theorem~\ref{thm:4:unrootedGenus1}), which will be the purpose of the next section.

\subsection{Unrooting step and proof of Theorems~\ref{thm:4:mainUnrooted} and~\ref{thm:4:unrootedGenus1}}
\label{sec:4:unrooting}

In this section, we deduce Theorem~\ref{thm:4:mainUnrooted} from Theorem~\ref{thm:4:mainRooted}, and we also check the exceptional case of genus~$1$ given by Theorem~\ref{thm:4:unrootedGenus1}.

When applied to the EGF $L_g$ of rotation systems $(\sigma_\bullet, \sigma_\circ, \phi)$ of bipartite maps (\textit{cf.} Section~\ref{sec:2:map-const}), the operator $\Gamma$ defined in (\ref{eq:4:Gamma-def}) can be seen as picking one element in a cycle of $\phi$, since a cycle of length $k$ in $\phi$ stands for a face of degree $2k$, which is tracked by the variable $p_k$. Equivalently, we can say that $\Gamma$ picks a label from $1$ to $n$ in a rotation system formed by permutations in $S_n$, which corresponds to a bipartite map with $n$ edges. In the perspective of bipartite maps, the operator $\Gamma$ distinguishes an edge in the map, which can then be considered as the root, since the orientation of the root is already fixed by vertex colors. Therefore, $\Gamma L_g$ is the EGF of rotation systems of bipartite maps with a distinguished edge, which are essentially bipartite maps with arbitrary labels on their edges, while the root is the distinguished edge. Since a bipartite map with $n$ edges has $n!$ possible labelings, which cancels out the factor $(n!)^{-1}$ in the EGF $L_g$, we thus have $F_g = \Gamma L_g$.

Since the series $L_g$ and $F_g$ are related by $F_g = \Gamma L_g$, studying $L_g$ from $F_g$ essentially amounts to inverting the differential operator $\Gamma$, \textit{i.e.}, heuristically, to perform some kind of \textbf{integration}. Since in our case the OGFs of rooted bipartite maps given in Theorem~\ref{thm:4:mainRooted} are rational in our given set of parameters, it is no surprise that an important part of the work will be to show that this integration gives rise to no logarithm. This section is divided in two steps. We first construct two operators that enable us to ``partially'' invert the operator $\Gamma$ (Proposition~\ref{prop:4:decomposeGamma}), and we reduce the inversion of the operator $\Gamma$ to the computation of a univariate integral. Then we conclude the proof of Theorem~\ref{thm:4:mainUnrooted} by proving that this integral contains no logarithms, making joint use of two combinatorial arguments: a disymmetry-type theorem and an algebraicity statement proved with bijective tools in~\cite{chapuy2009asymptotic}. As we have taken the projective limit in the proof of Theorem~\ref{thm:4:mainRooted} to remove the restriction $p_k = 0$ for $k > K$, all generating functions we consider here contain all possible $p_k$'s.

\paragraph{The operators $\boldsymbol{\Diamond}$ and $\boldsymbol{\Box}$} ~\\

The first idea of the proof is inspired by \cite{GGPN} and consists in inverting the operator $\Gamma$ in two steps. 

We define the ring $\rL$ formed by elements $f$ of $\mathbb{Q}[p_1,p_2,\dots][[z]]$ such that for all $k\geq 0$, the coefficient of $z^k$ in $f$ is a homogeneous polynomial in the $p_i$ of degree $k$ (where the degree of $p_i$ is defined to be $i$). Equivalently, $\rL=\mathbb{Q}[[zp_1,z^2p_2,z^3 p_3,\dots]]$.

Note that any formal power series in the Greek variables with coefficients in $\mathbb{Q}$, considered as an element of $\mathbb{Q}[p_1,p_2,\dots][[z]]$, is an element of $\rL$. Note also that $L_g$ is an element of $\rL$, since the factor $(1+\gamma)$ in the change of variable $z=t(1+\gamma)$ is also in $\rL$. Indeed, if we view $L_g$ as a series in $t$, the coefficient of $t^k$ for $k\geq 0$ is a homogeneous polynomial of degree $k$ in the $p_i$, since the sum of face degrees in a map is equal to twice the number of edges. Given the form of the change of variable $t \leftrightarrow z$ given by~\eqref{eq:4:z}, namely $t=z(1+\sum_k {2k-1 \choose k}p_k z^k)^{-1}$, this clearly implies that as a series in $z$, $L_g$ is in $\rL$.

We now introduce the linear operators $\Box$ and $\Diamond$ on $\mathbb{Q}[x,p_1,p_2,\dots][[z]]$ defined by 
\begin{align*}
\Box x^k p_\lambda z^a = \left(\frac{1}{k} - \frac{\gamma}{1+\gamma}\right)p_k p_\lambda z^a,  ~ ~ ~ \Diamond = \sum_{k} p_k \partial_{p_k},
\end{align*} 
where $\partial_{p_k}$ is the differential operator defined by~\eqref{eq:4:defpartialpk} in Section~\ref{subsec:4:defRingsOperators}, that is, differentiation of $p_k$ that ignores the dependence of $z$ on $p_k$. We have the following proposition.
\begin{prop}\label{prop:4:decomposeGamma}
For any $R \in \rL$, we have
\[ \Diamond R =\Box\Gamma R. \]
In particular, $\Diamond L_g= \Box F_g$.  
\end{prop}
\begin{proof}
The proof is mainly a careful application of the chain rule and of the computations already made in Section~\ref{sec:4:Gamma}. Let $R\in\rL$. Since $\Gamma$ is a weighted sum of partial differential operators, we have:
\begin{align*}
\Gamma R &= \left(\Gamma z\right) \frac{\partial}{\partial z} R + \sum_{k} \left(\Gamma p_k \right) \partial_{p_k} R
= \left(\Gamma z\right) \frac{\partial}{\partial z} R + \sum_{k\geq 1}  k x^k \partial_{p_k} R.
\end{align*}
Therefore,
\[ \sum_{k}  k x^k \partial_{p_k} R = \left(\Gamma - \left(\Gamma z \right) \frac{\partial}{\partial z}\right) R. \]
We now define the linear operators 
\[ \Pi\eqdef x^k \mapsto p_k, \,\,\, \Xi\eqdef x^k \mapsto \frac{p_k}{k}. \]

By applying $\Xi$ to the last equality, we get:
\begin{align*}
\Diamond R = \sum_{k}  p_k\partial_{p_k} R 
&= \Xi \left(\Gamma - \left(\Gamma z\right) \frac{\partial}{\partial z}\right) R\\
&= \Xi \Gamma  R - \Xi \left(  \left(\Gamma z\right) \frac{\partial R}{\partial z} \right).
\end{align*}
We thus need to study $\Xi \left((\Gamma z) \frac{\partial R}{\partial z}\right)$. We notice that, over the ring $\rL$, the operators $\Pi \sum_{k \geq 1}k x^k \partial_{p_k}$ and $ \frac{z\partial}{\partial z}$ are equal. We recall that $\Gamma z = \frac{(s^{-1} - s) z}{4s^2 (1-\eta)}$ with $s=\frac{1-uz}{1+uz}$ from Proposition~\ref{prop:4:Gamma-on-Greek-expr}. Since $R \in \rL$ does not depend on $x$, we have $\Pi(AR) = (\Pi A)R$ for any $A$. We thus have
\[ \Pi \Gamma R = \Pi \left((\Gamma z) \frac{\partial R}{\partial z}\right) + \Pi \sum_{k \geq 1} kx^k \partial_{p_k}R = \left( \Pi \left( \frac{\Gamma z}{z} \right) + 1\right) z \frac{\partial}{\partial z}R = \frac{1+\gamma}{1-\eta}  \frac{z\partial}{\partial z} R. \] 
This is the only point in the proof where we use the assumption that $R\in\rL$. Note that we have used that $\left( \Pi \left( \frac{\Gamma z}{z} \right) + 1\right) = \left(z^{-1}\sum_k kp_k \frac{\partial}{\partial p_k} z \right)+1 = \dfrac{1+\gamma}{1-\eta}$ where the first equality comes from the definition of $\Pi$ and~$\Gamma$, and the second follows from Proposition~\ref{prop:4:diff-of-vars} and the definitions of $\gamma$ and $\eta$. We thus have
\[  
\frac{z\partial}{\partial z} R= \frac{1-\eta}{1+\gamma}\Pi \Gamma R.
\]
Substituting this equality in the previous expression of $\sum_{k \geq 1} p_k \partial_{p_k} R$ and recalling $\Gamma z = \frac{(s^{-1} - s) z}{4s^2 (1-\eta)}$ we obtain
\begin{align*} 
\sum_{k \geq 1} p_k \partial_{p_k}R &=  \Xi \Gamma R -  \Xi\left(\frac{s^{-1} - s}{4 s^2 (1+\gamma)} \right) (\Pi \Gamma)R \\ 
&= \Xi \Gamma  R -  \frac{\gamma}{(1+\gamma)} (\Pi \Gamma)R \\ 
&= \Box\Gamma R.
\end{align*}
We recall from (\ref{eq:4:defTheta}) the operator $D$ that sends $p_k z^k$ to $kp_kz^k$. he second equality follows from the fact that $\Xi \frac{s^{-1}-s}{s^{2}}= D^{-1} \Theta^{-1} (s^{-3}-s^{-1})$, from Proposition~\ref{prop:4:Theta-base}, the list of images of $D$ in the proof of Proposition~\ref{prop:4:Gamma-on-Greek-expr} and a direct computation. The last equality is straight-forward from the definitions of $\Box, \Pi$, and $\Xi$. This concludes the proof that $\Diamond R = \Box \Gamma R$ for $R\in\mathbb{L}$.

Finally, since $F_g = \Gamma L_g$ and $L_g \in \mathbb{L}$, it follows that $\Diamond L_g = \Box F_g$.
\end{proof}

\begin{prop}\label{prop:4:diamondLg}
$\Diamond L_g$ is a rational function of the Greek variables, i.e.:
$\Diamond L_g = R$ with $R\in \mathbb{Q}(\greeks)$, whose denominator is of the form $(1-\eta)^a(1+\zeta)^b(1+\gamma)^c$ for $a,b\geq 0$ and $c\in \{0,1\}$.
\end{prop}
\begin{proof}
We are going to use Theorem~\ref{thm:4:mainRooted} and the fact that $\Diamond L_g = \Box F_g$. By Theorem~\ref{thm:4:mainRooted}, and since $F_g$ is $uz$-antisymmetric, we know that $F_g$ is an element of $(s^{-1}-s) \mathbb{Q}(\greeks)[s^2,s^{-2}]$, where $s=(1-uz)(1+uz)^{-1}$ as in Section~\ref{sec:4:Y}. Therefore, we can write
\[
F_g = \sum_{i\in I } (s^{-1} -s)s^{2i} R_i,
\] 
where $I\subset \mathbb{Z}$ a finite set of integers and $R_i \in \mathbb{Q}(\greeks)$ is a rational function in the Greek variables for each $i\in I$. Since $\Diamond L_g =\Box F_g$, we have:
\begin{equation}\label{eq:4:diamondLg}
\Diamond L_g = \sum_{i\in I } R_i  \, \Box \left((s^{-1} -s )s^{2i} \right).
\end{equation}
Now, by Proposition~\ref{prop:4:Theta-base}, the vector space $(s^{-1}-s)\mathbb{Q}[s^{-2},s^2]$ is spanned by the basis 
$B= \{\Theta \zeta_i, i\geq 1 \,;\, \Theta (\zeta - \gamma)  \,;\, \Theta (\eta + \gamma) \,;\, \Theta \eta_i, i \geq 1\}$. Moreover, the action of $\Box$ on the basis $B$ is given by the formulas:
\begin{align}
\Box \Theta \zeta_i\quad\quad 
 &\quad=\quad
X_i - \frac{\gamma \zeta_i}{1+\gamma} \;, 
\label{eq:4:list1}\\
\Box \Theta (\zeta - \gamma) 
&\quad=\quad
\zeta + \frac{\zeta - \gamma}{1+\gamma}\;, \\
\Box \Theta (\eta + \gamma) 
&\quad= \quad \frac{\gamma (1 - \eta)}{1+\gamma}\;, \\
\Box \Theta \eta_i \quad \quad
&\quad= \quad \eta_{i-1} - \frac{\gamma \eta_i}{1+\gamma}. \label{eq:4:list4}
\end{align}
Here, $X_i$ is a linear combination of $\zeta, \zeta_1, \zeta_2 \dots, \zeta_i$ with rational coefficients. These formulas follow from the fact that $\Box\Theta : p_k z^k \longmapsto \left(\frac{1}{k}-\frac{\gamma}{1+\gamma}\right) p_k z^k$
and from the definitions of Greek variables given in \eqref{eq:4:eta-zeta} and \eqref{eq:4:eta-zeta-i}. Returning to \eqref{eq:4:diamondLg}, it proves that $\Diamond L_g$ is a rational function of the Greek variables, $L_g \in \mathbb{Q}(\greeks)$.
\end{proof}


\paragraph{Inverting $\boldsymbol{\Diamond}$} ~\\

Let $S\in \mathbb{Q}(\greeks)$ be a rational function in the Greek variables, depending on a finite number of Greek variables. Since each Greek variable is a linear function of the $p_k$, it is clear that $\Diamond$ leaves each Greek variable invariant. Using the fact that $\Diamond$ is a sum of differential operators, we can extend the domain of $\Diamond$ to rational functions in the Greek variables, which implies that $\Diamond S$ is given by a simple \emph{univariate} derivation:
\begin{equation}\label{eq:4:DiamondAsDiff}
(\Diamond S)(v\eta, v\gamma, (v\eta_i)_{i\geq 1}, (v\zeta_i)_{i\geq 1}) = \frac{d}{dv} S(v \eta, v\gamma, v\zeta, (v\eta_i)_{i\geq 1}, (v\zeta_i)_{i\geq1}).
\end{equation}
This implies the following proposition.
\begin{prop}\label{LgAsIntegral}
The series $L_g$ is given by
\[
L_g = \int_0^1 dv R(v \eta, v\gamma, v\zeta, (v\eta_i)_{i\geq 1}, (v\zeta_i)_{i\geq1}),
\]
where $R$ is the rational function such that $\Diamond L_g = \Box F_g = R(\eta, \gamma, \zeta, (\eta_i)_{i\geq 1}, (\zeta_i)_{i\geq 1})$.
\end{prop}
\begin{proof}
We simply integrate \eqref{eq:4:DiamondAsDiff} with $S=L_g$ for $v$ from $0$ to $1$. The only thing to check is the initial condition, namely that $R=0$ when all Greek variables are equal to zero. This is clear, since this specialization is equivalent to substitute $z=0$, and for $g\geq 1$ there is no map with $0$ edge.
\end{proof}

\begin{cor}\label{cor:4:LgAlmost}
The series $L_g$ has the following form:
\[
L_g = R_1 + R_2 \log (1-\eta) + R_3 \log (1+\zeta) + R_4 \log(1+\gamma)
\]
where $R_1, R_2, R_3, R_4$ are rational functions in $(\eta, \gamma, \zeta, (\eta_i)_{i\geq 1}, (\zeta_i)_{i\geq 1})$ depending on finitely many Greek variables. Furthermore, the denominator of $R_1$ is of the form $(1-\eta)^a(1+\zeta)^b$ for $a,b\geq 0$. \end{cor}
\begin{proof}
This follows from the last two propositions. Note that $R_1$ has no pole at $\gamma=-1$ since $F_g$ has at most a simple pole at $\gamma=-1$ from Proposition~\ref{prop:4:diamondLg}.
\end{proof}

\paragraph{Algebraicity and proof of Theorem~\ref{thm:4:mainUnrooted}} ~\\

In order to prove Theorem~\ref{thm:4:mainUnrooted} from Corollary~\ref{cor:4:LgAlmost}, it suffices to show that $R_2=R_3=R_4=0$, \textit{i.e.} no logarithm appears in the integration procedure. To this end, it suffices to show that the series $L_g$ is algebraic under some specializations that leaves only a finite number of $p_k$'s to be non-zero. This algebraicity will be proved here using a detour via arguments with a stronger combinatorial flavor and an algebraicity statement proved with bijective methods in~\cite{chapuy2009asymptotic}. 

The following lemma is a variant for maps of genus $g$ of the ``disymmetry theorem'', which is classical in the enumeration of labeled trees (and much popularized in the book~\cite{BergeronLabelleLeroux}; see also~\cite{CFKS} for a use in the context of planar maps).

We will first define some EGFs of rotation systems of bipartite maps that we will be using in the following. For a rotation system $(\sigma_\bullet, \sigma_\circ,\phi)$ with permutations in $S_n$ of a bipartite map $M$ with $n$ edges, cycles in $\sigma_\circ$ and $\sigma_\bullet$ stands for vertices of $M$, while cycles in $\phi$ stands for faces in $M$ and integers from $1$ to $n$ stand for edges in $M$ (\textit{cf.} Section~\ref{sec:2:map-const}). We denote by $L_g^{vertex}$ (resp. $L_g^{face}$ and $L_g^{edge}$) the EGF of such rotation systems with a marked cycle in $\sigma_\bullet$ or $\sigma_\circ$ (resp. a marked cycle in $\phi$ for $L_g^{face}$ and an integer from $1$ to $n$ in a rotation system with permutations in $S_n$ for $L_g^{edge}$). We can see that the superscripts of these EGFs correspond to the counterpart in bipartite maps of their marked elements.

\begin{lem}[Disymmetry theorem for maps]\label{lem:4:disymmetry}
The EGFs $L_g^{vertex}, L_g^{face}, L_g^{edge}$ are related by
\[
(2-2g) L_g = L_g^{vertex} +L_g^{face} - L_g^{edge}.
\]
\end{lem}
\begin{proof}
This is a straightforward consequence of Euler's formula.
\end{proof}

Now we observe that, for clear combinatorial reasons, $L_g^{face}$ and $L_g^{edge}$ can be obtained from $F_g$ as follows:
\[
L_g^{face}=\Xi F_g \quad , \quad L_g^{edge} = \Pi F_g,
\]
where $\Xi: x^k \mapsto \frac{p_k}{k}$ and $\Pi: x^k \mapsto p_k$ as defined previously in the proof of Proposition~\ref{prop:4:decomposeGamma}. We thus have the following lemma.
\begin{lem}\label{lem:4:faceedge}
$L_g^{face}$ and $L_g^{edge}$ are rational functions of $\eta,\gamma, \zeta, (\eta_i)_{i \geq 1}, (\zeta_i)_{i\geq1}$.
\end{lem}
\begin{proof}
Given Theorem~\ref{thm:4:mainRooted}, it is enough to prove that $\Xi$ and $\Pi$ send $(s^{-1}-s)\mathbb{Q}[s^{-2},s^2]$ to rational functions of Greek variables. But by Proposition~\ref{prop:4:Theta-base}, for any element $F\in(s^{-1}-s)\mathbb{Q}[s^{-2},s^2]$, there is a (finite) linear combination $G$ of elements of the set $B=\{\eta+\gamma;\zeta-\gamma; \eta_i, i\geq 1; \zeta_i, i\geq 1\}$ such that $F = \Theta G$. Now it is clear from the definitions that we have 
\[\Pi\Theta: p_k \mapsto p_k 
  \quad, \quad 
\Xi \Theta: p_k \mapsto \frac{p_k}{k}.
\]
We have to check that each of these two operators sends an element of $B$ to a linear combination of Greek variables. For the first one it is obvious. For the second one, we first observe that, by a simple check using the definition of Greek variables, we have $\Xi \Theta (\eta+\gamma) = \gamma$, $\Xi\Theta(\zeta -\gamma)= 2\zeta-\gamma$, and $\Xi \Theta(\eta_i) = \eta_{i-1}$ for $i\geq 1$ (with $\eta_0=\eta$). Finally, for $i\geq 2$, one similarly checks that there exist rational numbers 
$\alpha_i, \beta_i$ such that $\Xi\Theta \zeta_i = \alpha_i \zeta_i + \beta_i \Xi\Theta \zeta_{i-1}$  which is enough to conclude by induction, together with the base case $\Xi\Theta \zeta_1 = 1/3 (2\zeta_1-2\gamma+4\zeta)$.
\end{proof}

We now need the following result.
\begin{prop}[\cite{chapuy2009asymptotic}]\label{prop:4:bijectionImport}
Fix $g\geq 1$ and $D\subset \mathbb{N}$ a finite subset of the integers of maximum at least $2$. Let $\mathbf{p}_D$ denote the substitution $p_i=\mathbf{1}_{i\in D}$ for $i\geq 1$. The series $L_g^{vertex}(\mathbf{p}_D)$ is algebraic, \textit{i.e} there exists a non-zero polynomial $Q\in\mathbb{Q}[t;f]$ such that
$Q\left(t; L^{vertex}_g(t; \dots p_i=\mathbf{1}_{i\in D}\dots) \right)=0$.
\end{prop}
\begin{proof}
Since this statement is not written in this form in \cite{chapuy2009asymptotic}, let us clarify where it comes from. Let $O_g\equiv O_g (t;p_1,p_2,\dots)$ be the \emph{ordinary} generating function of \emph{rooted} bipartite maps with one pointed vertex, by the number of edges (variable $t$) and the faces (variable $p_i$ for faces of half-degree $i$, including the root face). Then it is easy to see that we have:
\[
O_g = \frac{t d}{dt} L_g^{vertex}.
\] 
Now in \cite[Equation~(8.2)]{chapuy2009asymptotic} (and more precisely in the case $m=2$ of that reference), it is proved that there exists an algebraic series $R_D=R_D(t;p_i,i\in D)$ such that $R_D(t=0)=0$ and
\[
O_g(\mathbf{p}_D)=\frac{t\partial}{\partial t} R_D,
\]
where $O_g(\mathbf{p}_D)$ is the series $O_g$ under the substitution $p_i=\mathbf{1}_{i \in D}$ as above. Since $L_g^{vertex}(t=0)=0$ for clear combinatorial reasons, we have $L_g^{vertex}(\mathbf{p}_D)=R_D$.
\end{proof}

We can now prove Theorem~\ref{thm:4:mainUnrooted}.
\begin{proof}[Proof of Theorem~\ref{thm:4:mainUnrooted}]
For any finite set $D$ of integers with maximum at least $2$, we denote by $\mathbf{p}_D$ the substitution of variables $p_i = \mathbf{1}_{i\in D}$. For $g\geq 2$, we can conclude from Lemma~\ref{lem:4:disymmetry}, Lemma~\ref{lem:4:faceedge} and Proposition~\ref{prop:4:bijectionImport} that the series $L_g(\mathbf{p}_D)$ after the substitution $\mathbf{p}_D$ for any valid set $D$ is algebraic. We recall that in Corollary~\ref{cor:4:LgAlmost}, $R_2, R_3, R_4$ comes with logarithmic factors $\log(1-\eta), \log(1+\zeta), \log(1+\gamma)$. Under the specialization $\mathbf{p}_D$, these logarithm factors becomes logarithms of polynomials in $z$, which are not algebraic. Therefore, the sum
\[
R_2(\mathbf{p}_D)\log(1-\eta(\mathbf{p}_D)) + R_3(\mathbf{p}_D)\log(1+\zeta(\mathbf{p}_D)) +
R_4(\mathbf{p}_D)\log(1+\gamma(\mathbf{p}_D))
\]
should be algebraic, which implies that when $z$ tends to infinity, the sum should behave as $\Theta(z^a)$ for an integer $a$.

We suppose that at least one of $R_2, R_3, R_4$ is non-zero. For a finite set of integers $D$, let $d$ be the maximum of $D$. When $z \to +\infty$, we must have $R_2(\mathbf{p}_D) = c_2(d) z^i + o(z^i), R_3(\mathbf{p}_D) = c_3(d) z^i + o(z^i), R_2(\mathbf{p}_D) = c_4(d) z^i + o(z^i)$ for the same integer $i$ such that at least one of the $c_i(d)$'s is non-zero. Here, all the $c_i(d)$'s are polynomials of $d$. By taking the coefficient of the dominating term of power $z^i \log(z^d)$ in each term of the relation above, we have
\[
(c_2(d)+c_4(d))\log(d-1) + c_3(d) - c_4(d)\log(2d-1) = 0.
\]
Since at least one of the $c_i(d)$'s is non-zero, this equation has no integral solution for $d>2$. Therefore, our hypothesis of at least one of $R_2,R_3,R_4$ being non-zero is false. Equivalently, for $D$ with maximum at least $3$, the rational functions $R_2,R_3,R_4$ defined in Corollary~\ref{cor:4:LgAlmost} must vanish under the specialization $\mathbf{p}_D$:
\[
R_2(\mathbf{p}_D)=0, \quad R_3(\mathbf{p}_D)=0, \quad R_4(\mathbf{p}_D)=0.
\]
Therefore, to conclude the proof that $R_2=R_3=R_4=0$ (hence the proof of Theorem~\ref{thm:4:mainUnrooted}) it suffices to show that if $Q$ is a polynomial in the Greek variables, $Q\in \mathbb{Q}[\greeks]$, such that $Q(\mathbf{p}_D)=0$ for all finite $D$, then $Q=0$.

We proceed by an infinite descent on the Greek degree $\deg_\gamma$, where each Greek variable is of degree 1. Let $Q$ be a non-zero element of $\mathbb{Q}[\greeks]$ such that $Q(\mathbf{p}_D)=0$ for all finite $D$. It is clear that $Q$ cannot be a constant, therefore $\deg_\gamma(Q)\geq 1$. We denote by $c$ the maximal index of all Greek variables $\eta_i$ and $\zeta_i$ that appears in $Q$, by $k$ the Greek degree $\deg_\gamma(Q)$ of $Q$.

Let $D$ be a finite subset of $\mathbb{N}_+$, and $d$ its maximal. Let $\ell > d$ be a large enough integer. We define $X=\binom{2\ell-1}{\ell}$. It is clear that $X$ is much larger than $\ell$, which is in turns much larger that any element in $D$ by definition. We denote by $D^* = D \cup \{\ell\}$. We observe that, for any Greek variable $G \in \greeks$, we have
\[ G(p_{D^*}) = G(\mathbf{p}_D) + R_G(\ell) X. \]
Here, $R_G(\ell)$ is a rational function of $\ell$ that depends on the Greek variable $G$. It is an element of the set 
\[
S_\gamma(\ell)=\left\{1,\ell-1,\frac{\ell-1}{2\ell-1}; \ell^i(\ell-1), 1 \leq i \leq c; \frac{\ell}{(2\ell-1)(2\ell-3)\dots(2\ell-2i-1)},1 \leq i \leq c\right\}.
\] 
This equality is due to the fact that all Greek variables are linear in $p_k$'s. We can thus express $Q(p_{D^*})$ as
\[
Q(p_{D^*}) = \sum_{i=0}^{k} Q_i(\ell,p_{D}) X^i.
\]
Here, the coefficients $Q_i(\ell,\mathbf{p}_D)$ are rational in $\ell$ and polynomial in $G(\mathbf{p}_D)$ for all Greek variables $G \in \greeks$. Furthermore, the total degree of $Q_i(\ell,\mathbf{p}_D)$ in all Greek variable specializations $G(\mathbf{p}_D)$ is at most $k-i$. When $\ell$ tends to infinity, $X$ grows exponentially with $\ell$, whereas the coefficients have at most a polynomial growth:
\[
Q_i(\ell,\mathbf{p}_D) = O\left( (\ell^c(\ell-1))^{k} \left( d^c(d-1) \binom{2d-1}{d} \right)^{k} \right) = O(\ell^{(c+1)k}).
\]
Since the coefficients $Q_i(\ell,\mathbf{p}_D)$ are negligible comparing to $X$, for an infinity of sufficiently large $\ell$ (thus $X$), the specialization $Q(p_{D^*})$ is dominated by the monomial $X^i$ such that $Q_i(\ell,\mathbf{p}_D)$ is non-zero. However, since $Q(p_{D^*})=0$, for all $i$ and $D$ we must have $Q_i(\ell,\mathbf{p}_D)=0$ for infinitely many values of $\ell$. 

We now take $Q_1$, which must exist since $k \geq 1$. From its definition, we have the following expression:
\begin{align*}
Q_1(\ell,\mathbf{p}_D) &= [X^{1}]Q(p_{D^*}) \\
&= \frac{\partial Q}{\partial \gamma}(\mathbf{p}_D) + (\ell-1)\frac{\partial Q}{\partial \eta}(\mathbf{p}_D) + \frac{\ell-1}{2\ell-1}\frac{\partial Q}{\partial \zeta}(\mathbf{p}_D) \\
&\quad \quad+ \sum_{i=1}^c \ell^i(\ell-1)\frac{\partial Q}{\partial \eta_i}(\mathbf{p}_D) + \sum_{i=1}^c \frac{(-2)^{i+1}\ell(\ell-1)\cdots(\ell-i)}{(2\ell-1)(2\ell-3)\cdots(2\ell-2i-1)}\frac{\partial Q}{\partial \zeta_i}(\mathbf{p}_D) .
\end{align*}
The partial differentiations come from the distinction of the Greek variable $G$ that takes the part $R_G(\ell)X$ instead of $G(\mathbf{p}_D)$. 

We notice that all the partial differentiations are specialized with $\mathbf{p}_D$. We fix the set $D$, and for an infinite number of values of $\ell$, we have $Q_1(\ell,\mathbf{p}_D)=0$. Therefore, we must have $Q_1(\ell,\mathbf{p}_D)=0$ as a rational fraction of $\ell$. However, since the elements in the set $S_\gamma(\ell)$ of rational fractions in $\ell$ that appears in the expression of $Q_1(\ell,\mathbf{p}_D)$ above are all linearly independent in the space of rational fractions of $\ell$, we must have $(\partial Q / \partial G)(\mathbf{p}_D)=0$ for all Greek variables $G \in \greeks$. The equality holds for all finite set $D$. Since $Q$ is not constant, it must depend on at least one Greek variable $G$. Therefore, $Q_G=\partial Q / \partial G$ is a non-zero polynomial in Greek variables with $\deg_\gamma(Q_G)<\deg_\gamma(Q)$ such that $Q_G(\mathbf{p}_D)=0$ for all finite set $D$. We can thus complete the infinite descent by finding a non-zero polynomial with strictly smaller degree. Since the infinite descent on degree is impossible, we thus conclude that such $Q$ cannot exist. Therefore, the only polynomial $Q$ in Greek variables such that $Q(\mathbf{p}_D)=0$ for all finite set $D$ is $Q=0$. We have especially $R_2=R_3=R_4=0$, which is what we want.

We thus have proved the rationality of $L_g$ for $g \geq 2$, and by Corollary~\ref{cor:4:LgAlmost}, the denominator of $L_g$ is of the form $(1-\eta)^a (1+\zeta)^b$ for $a,b \geq 0$.

We now prove that, expressed as a rational function in Greek variables, $L_g$ does not depend on $\gamma$. From the last paragraph we already know that $L_g$ is a polynomial in $\gamma$, \textit{i.e.} $L_g=\sum_{i=0}^k S_i \gamma^i$ where $k\geq 0$, the $S_i$ are rational function of $\greeks\setminus\{\gamma\}$, and $S_k\neq 0$. Recall that $F_g=\Gamma L_g$, so from the fact that $\Gamma$ is a derivation and from Proposition~\ref{prop:4:Gamma-on-Greek-expr}, $F_g$ is also polynomial in $\gamma$. Moreover, the only Greek variable $G\in \greeks$ such that $\Gamma G$ depends on $\gamma$ is $G=\gamma$, and we have more precisely $\Gamma \gamma = \frac{s^{-1} - s}{4(1-\eta)s^2} (\eta + \gamma) + \frac1{4}(s^{-3} - s^{-1})$. It follows that the coefficient of $\gamma^k$ in $F_g$ is equal to $k (\Gamma S_k) + \frac{s^{-3} - s^{-1}}{4(1-\eta)}S_k$. From the structure of $F_g$ provided by Theorem~\ref{thm:4:mainRooted} ($F_g$, as a rational function in Greek variables and $uz$, does not depend on $\gamma$), this coefficient is equal to zero. It is easy to see that this is impossible if $k>0$. Indeed, looking at Proposition~\ref{prop:4:Gamma-on-Greek-expr} again, $\Gamma S_k$ contains either a pole of order at least 5 at $u=1/z$, or a pole of order at least 1 at $u=-1/z$, which cannot be cancelled by the factor $(s^{-3} - s^{-1})$ in the second term. Therefore $k=0$, \textit{i.e.} $L_g$ does not depend of $\gamma$.

We now prove the bound conditions. Using the three notions of degree in Section~\ref{subsec:4:proofMainRooted}, we only need to check that $L_g$ is a homogeneous sum of Greek degree $\deg_\gamma(L_g) = 2-2g$ and $\deg_+(L_g) = \deg_-(L_g) \leq 6(g-1)$. We recall the following expression of $F_g$.
\[ F_g = \Gamma L_g = (\Gamma \zeta) \frac{\partial}{\partial \zeta} L_g + (\Gamma \eta) \frac{\partial}{\partial \eta} L_g + \sum_{i \geq 1} (\Gamma \eta_i) \frac{\partial}{\partial \eta_i} L_g + \sum_{i \geq 1} (\Gamma \zeta_i) \frac{\partial}{\partial \zeta_i} L_g \]

For the Greek degree, we observe that, by Proposition~\ref{prop:4:Gamma-degrees} and the fact that $L_g$ has no constant term, if $L_g$ is not homogeneous in Greek degree, then $F_g = \Gamma L_g$ cannot be homogeneous. Therefore, $L_g$ must be homogeneous in $\deg_\gamma$, with degree $\deg_\gamma(L_g) = \deg_\gamma(F_g) + 1 = 2 - 2g$.

For the pole degree $\deg_+$, let $T = c \eta_\alpha \zeta_\beta (1-\eta)^{-a} (1+\zeta)^{-b}$ for $c \in \rationals$, $a,b \geq 0$ and $\alpha, \beta$ two partitions be the largest term in $L_g$ such that $\deg_+(T) = \deg_+(L_g)$ when ordered first alphabetically by $\alpha$ then also alphabetically by $\beta$. We will now discuss by cases.

If $\alpha$ and $\beta$ are both empty, then $\deg_+(T) = 0$ and we are done. 

We now suppose that $\alpha$ is empty but not $\beta$. We observe that, for a term $S$ in the form $c \zeta_{\beta'} (1-\eta)^{-a} (1+\zeta)^{-b}$, if we order the terms in $\Gamma S$ first by the power of $(1+uz)$ in the denominator then alphabetically by $\nu$ in their factor of the form $\zeta_\nu$, then the largest term $S'$ comes from $(\Gamma \zeta_{\beta'_1}) \partial S / \partial \zeta_{\beta'_1}$, with pole degree $\deg_-(S') = 2|\beta'|+1$ and no possibility of cancellation. Therefore, in $F_g$ there is a term $T'$ coming from $(\Gamma \zeta_{\beta_1}) \partial T / \partial \zeta_{\beta_1}$ that can have no cancellation by the maximality of $\beta$ and by our previous observation, and $\deg_-(T') = 2|\beta|+1$. But since $\deg_-(F_g) \leq 2g-1$, we have $\deg_-(L_g) = 2|\beta| \leq 2g-2$, which concludes this case.

The final case is that $\alpha$ is non-empty. We observe that, for a term $S$ in the form $c \eta_{\alpha'} \zeta_{\beta'} (1-\eta)^{-a} (1+\zeta)^{-b}$, if we order the terms in $\Gamma S$ first by the power of $(1-uz)$ in the denominator then alphabetically by $\nu$ in their factor of the form $\eta_\nu$, then the largest term $S'$ comes from $(\Gamma \eta_{\alpha'_1}) \partial S / \partial \eta_{\alpha'_1}$, with pole degree $\deg_+(S') = 2|\alpha'|+2|\beta'|+5$ and no possibility of cancellation. Therefore, similarly to the previous case, by the fact that $\deg_+(F_g) \leq 6g-1$, we conclude that $\deg_+(L_g) = 2|\alpha|+2|\beta| \leq 6(g-1)$. We thus cover all cases and conclude the proof.
\end{proof}

We now address the only remaining point, which is the case of genus $1$.
\begin{proof}[Proof of Theorem~\ref{thm:4:unrootedGenus1}]
We first compute the OGF $F_1$ using Theorem~\ref{thm:4:toprec} for $g=1$. Recall that the value of $\Gamma F_0$ is explicitly given in~\eqref{eq:4:F02}, from which we observe that $\Gamma F_{0}$ has a pole of order $4$ at $u=1/z$ and no pole at $u=-1/z$. Therefore, in order to compute the residues in~\eqref{eq:4:toprec} in the case $g=1$, we need to compute explicitly the first $4$ terms at $u=1/z$ and the first $2$ terms at $u=-1/z$ in the expansions in Proposition~\ref{prop:4:diffY}. Now, since the proofs of Proposition~\ref{prop:4:taylor-pos-pre} and Proposition~\ref{prop:4:taylor-neg-pre} are computationally effective, we can follow these proofs to compute these quantities explicitly (it is easier, and more reliable, to use a computer algebra system here). We find the expression in \eqref{eq:4:F1-expr}.

Observe that another approach to prove the last equality is to use the structure given by Theorem~\ref{thm:4:mainRooted}, and compute sufficiently many terms of the expression of $F_1$ (for example by iterating the Tutte equation~\eqref{eq:4:bipartite}) to identify all undetermined coefficients appearing in the finite sum~\eqref{eq:4:mainRooted}.

We now note that all the steps performed to go from Theorem~\ref{thm:4:mainRooted} to Corollary~\ref{cor:4:LgAlmost} are valid when $g=1$, and are computationally effective. Therefore, using the explicit expression of $F_1$ given above, these steps can be followed to obtain an explicit expression of $L_1$. These computations are automatic (and better performed with a computer algebra system), thus not printed here.
\end{proof}

\subsection{Final comments} \label{sec:4:comments}

We conclude this section with several comments.

Firstly, as explained before, we have only used two basic ideas from the topological recursion of~\cite{EO}, but in a way that is accessible to combinatorialists. It remains to be seen if other ideas of the topological recursion can be applied to bipartite maps or to the more general $m$-constellations. In the case of bipartite maps, these ideas may provide a different way of performing the ``unrooting'' step in Section~\ref{sec:4:unrooting}, similar to~\cite[Sec. III-4.2]{Eynard:book}. However, our proof has the nice advantage of providing a partially combinatorial explanation of the absence of logarithms in genus $g>1$. More generally, it seems that understanding the link between the dissymmetry argument we used here and statements such as~\cite[Theorem~III~4.2]{Eynard:book} is an interesting question from the viewpoint of the topological recursion itself.

Our next comment is about computational efficiency. While it is tempting to use Theorem~\ref{thm:4:toprec} to compute the explicit expression of $F_g$ (and then $L_g$), it is much easier to simply compute the first few terms of $F_g$ (and $L_g$) using recursively the Tutte equation~\eqref{eq:4:bipartite}, and then determine the unknown coefficients in ~\eqref{eq:4:mainUnrooted} or~\eqref{eq:4:mainRooted} by solving a linear system (recall that~\eqref{eq:4:mainUnrooted} and~\eqref{eq:4:mainRooted} are \emph{finite} sums, so there are indeed finitely many coefficients to determine).

Thirdly, structural results similar to Theorem~\ref{thm:4:mainRooted} for the OGF $F_g^{(m)}(x_1,x_2,\dots, x_m)$ of bipartite maps of genus $g$ carrying $m\geq 1$ marked faces whose sizes are recorded in the exponents of variables $x_1,x_2,\dots,x_m$, are easily derived from our results. Indeed, this series is obtained by applying $m$ times to $L_g$ the rooting operator $\Gamma$, one time in each variable. More precisely:
\[
F_g^{(m)}(x_1,x_2,\dots, x_m)
=\Gamma_1\Gamma_2\dots\Gamma_m L_g,
\]
where $\Gamma_i = \sum_{k\geq 1} k x_i^k \frac{\partial}{\partial p_k}$. Since the action of $\Gamma_i$ is fully described by Proposition~\ref{prop:4:Gamma-on-Greek-expr} (up to replacing $s$ by $s_i=\frac{1-u_iz}{1+u_iz}$, where $u_i = x_i (1+zu_i)^2$), the series $F_g^{(m)}(x_1,x_2,\dots, x_m)$ are easily computable rational functions in the Greek variables and the $(1\pm u_iz)$. 

We observe as well that, by replacing all the $p_i$ by zero in the series 
$F_g^{(m)}(x_1,x_2,\dots, x_m)$, one obtains the generating function of bipartite maps with exactly $m$ faces, where the $x_i$ control the face degrees. Therefore, these functions have a nice structure as well, being polynomials in the $1/(1\pm u_iz)$ with rational coefficients. This special case also follows from the results of~\cite{KZ}. Note however that \cite{KZ} keeps track of one more variable (keeping control on the number of vertices of each color in their expressions). It is probably possible to extend our result to this case.

\horizrule

In this subsection, we have seen how to resolve the functional equation of constellations in higher genera, but only in the bipartite case $m=2$. It is thus natural to try to extend our proof to general $m$, which will also be a unified proof for similar results on classical and monotone Hurwitz numbers as in \cite{classical-hurwitz, GGPN}. Such a proof may be possible using the topological recursion, from which our proof of the bipartite case draws important ideas. This is a work in progress.

\chapter{Maps and generalized Tamari intervals}

In the two previous chapters, we are mostly concerned with the enumeration of maps themselves. Starting from this chapter, we will see how map enumeration results can be used to enumerate other related combinatorial objects. Although often studied with algebraic approaches, maps are firstly combinatorial objects with a simple and geometrically intuitive definition. It is thus most desirable to study them via combinatorial bijections, which can lead to more direct understanding of their structure and their relation with other combinatorial objects. This bijective approach to maps has been very successful, especially in the planar case. We have seen in Section~1 a brief review on how we can use bijections to relate maps to other combinatorial objects, either for enumeration or for the study of their structure. This chapter will be a demonstration of this bijective approach. More precisely, we will study the bijective link between non-separable planar maps and intervals in \emph{generalized Tamari lattices}.

In Section~\ref{sec:2:example}, we have seen Dyck paths and plane trees in the prism of generating functions, and we know that they are both enumerated by Catalan numbers. There are also many other combinatorial classes that are enumerated by Catalan numbers, such as binary trees and planar triangulations of a polygon. These objects are sometimes called \emph{Catalan objects} in general. More examples of Catalan objects can be found in \cite[Exercise~6.19]{Stanley:EC2}. These Catalan objects are often closely related by simple bijections. The versatility of Catalan numbers even extends to the realm of algebra (see, for example, \cite[Exercise~6.25]{Stanley:EC2}), which leads to interesting interactions between algebra and combinatorics.

It is thus not surprising to see algebraic structures defined on Catalan objects. The Tamari lattice is a partial order defined on a class of Catalan objects of the same size, for instance Dyck paths with $2n$ steps for a given natural number $n$. Its precise definition will be given later. While being deeply rooted in algebra, the Tamari lattice is also a well studied object in algebraic combinatorics, and it is related to many other combinatorial and algebraic objects. There are also many generalizations of the Tamari lattice, for example the $m$-Tamari lattice introduced by Bergeron (see \cite{bergeron-preville}) and the generalized Tamari lattice introduced by Préville-Ratelle and Viennot in \cite{PRV2014extension}. 

Albeit having been studied intensively for a long time, the Tamari lattice and its generalizations still have many mysterious enumerative aspects and bijective links yet to be unearthed. In this chapter, we will follow and develop a surprising enumerative link from intervals in these lattices to planar maps. In \cite{chapoton-tamari}, out of motivation from algebra, Chapoton counted the number of intervals in the Tamari lattice using a recursive decomposition approach, and he found that these intervals share the same enumeration formula with planar simple triangulations, which was rather unexpected. Later, Bernardi and Bonichon gave in \cite{BB2009intervals} a direct bijection between these objects. Similarly, the numbers of usual and labeled intervals in the $m$-Tamari lattice in \cite{bousquet-fusy-preville} and \cite{BMCPR2013representation} are also given by simple planar-map-like formulas, but a combinatorial explanation is still missing. All these links between intervals in Tamari-like lattices and planar maps motivate us to search for a bijection between intervals in the generalized Tamari lattice and some family of planar maps. And indeed, we discovered such a bijection, contributing to the combinatorial understanding of the Tamari lattice and its generalizations.

This chapter is based on a yet unpublished article in collaboration with Louis-François Préville-Ratelle at University of Talca. An extended abstract was accepted in the conference Formal Power Series and Algebraic Combinatorics 2016 (FPSAC 2016). In this chapter, by introducing a family of trees with labels on leaves called \emph{decorated trees}, we establish a natural bijection from non-separable planar maps to intervals in generalized Tamari lattices. As a consequence, we obtain a formula for the number for intervals in generalized Tamari lattices of a given size. 


\section{Tamari lattice and its generalizations}

We have already seen Dyck paths in Section~\ref{sec:2:example}. We recall that a Dyck path is a sequence of steps $u=(1,1)$ and $d=(1,-1)$, starting from the origin $(0,0)$, ending on the $x$-axis and staying in the upper plane with non-negative $y$ coordinate. A Dyck path can thus be viewed as a word formed by two letters $\{u,d\}$. Exceptionally in this section, we will represent Dyck paths in another way, by replacing steps $u=(1,1)$ and $d=(1,-1)$ by steps $N=(0,1)$ and $E=(1,0)$. In this setting, a Dyck path is thus a directed path on $\mathbb{Z}^2$, starting from the origin, going only by north or east unit steps, ending on the diagonal line $y=x$ and always staying above the diagonal. If we see Dyck paths as words, we simply replace the letters $\{u,d\}$ by $\{N,E\}$ in the corresponding words. We can also see this graphical presentation as a rotated and scaled version of the original one. Figure~\ref{fig:5:presentations} shows an example of the two ways to look at a Dyck path.

\begin{figure}
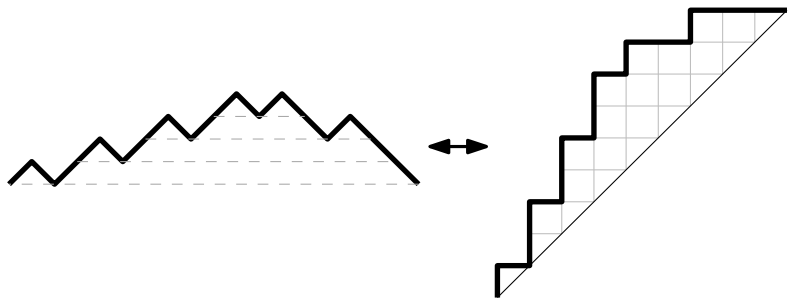

  \centering
  \insertfigure{ch5-fig.pdf}{34}
  \caption{The same Dyck path presented in up/down steps and in north/east steps}
  \label{fig:5:presentations}
\end{figure}

In this section, we change the presentation of the well-known and well-studied Dyck path for the sake of simplifying the introduction of its generalizations that we will study. We should note that, after the necessary notions are introduced in the section, we will revert to the usual up/down representation for the rest of this chapter.

By the definition of Dyck paths, it is clear that they are all of even length. We denote by $\mathcal{D}_n$ the set of Dyck paths with $2n$ steps, and we say that the size of a Dyck path is half of its length. Before introducing the Tamari lattice, which is a special partial order on the set $\mathcal{D}_n$, we first review some definitions concerning partial orders.

A \mydef{partial order} $\leq$ over a set $E$ is a binary relation that is:
\begin{enumerate}
\item Reflexive: $\forall x \in E, x \leq x$;
\item Antisymmetric: $\forall x,y \in E, x \leq y \land y \leq x \implies x = y$;
\item Transitive: $\forall x,y,z \in E, x\leq y \land y\leq z \implies x \leq z$.
\end{enumerate}
Its non-reflexive version is often denoted by $<$. A set $E$ endowed with a partial order is called a \mydef{partially ordered set} (or simply \mydef{poset}), and is sometimes denoted by $(E,\geq)$. In the following, we will only consider finite posets. Let $(E,\geq)$ be a finite poset. An element $a$ in $E$ is said to \mydef{cover} another element $b$ if $b < a$ and there is no element $c$ in between, \textit{i.e.} satisfying $b < c < a$. We can thus define another binary relation called the \mydef{covering relation} corresponding to the partial order $\leq$. A covering relation is sometimes called the \emph{transitive reduction} of its corresponding partial order, and by taking the transitive closure, we can obtain a partial order from a covering relation. Therefore, we can define a finite partial order by specifying its covering relation. We often represent a poset graphically by its \mydef{Hasse diagram}, in which the elements of the poset are represented by vertices. If $x$ covers $y$, then we place the vertex of $x$ above that of $y$, and draw an edge from $y$ to $x$. See Figure~\ref{fig:5:tamari-hasse} for an example of a Hasse diagram.

We can study extra structures on partial orders. A \emph{lattice} is a partial order that verifies extra properties, which will not be precised here since they are not needed in the following. However, we still mention here that a lattice has a minimal and a maximal element. We will also be interested in intervals in lattices. In a partial order $(E,\leq)$, given two comparable (but not necessary distinct) elements $x \leq y$, the \mydef{interval} $[x,y]$ between $x$ and $y$ is the subset $E_1$ of $E$ formed by elements of the form $z$ such that $x \leq z \leq y$, endowed with the restriction of $\leq$ on $E_1$. We can also identify an interval with its minimal and maximal element pair $(x,y)$. 

We can also construct new partial orders from existing partial orders. Given a partial order $(E,\leq)$, its \mydef{dual} is defined as the partial order $(E,\leq^*)$ on the same set of elements $E$ and the dual relation $\leq^*$ such that $x \leq^* y$ if $y \leq x$. To obtain the Hasse diagram of the dual partial order $(E,\leq^*)$, we simply flip that of the original partial order $(E,\leq)$ upside down. It is worth mentioning that the dual of a lattice is still a lattice. An \mydef{isomorphism} $f$ between two partial orders $(E_1,\leq_1)$ and $(E_2,\leq_2)$ is simply a bijection between $E_1$ and $E_2$ such that, for any two elements $x,y$ in $E_1$, we have $x\leq_1 y$ if and only if $f(x) \leq_2 f(y)$. In this case, the two partial orders are called \mydef{isomorphic}.

\begin{figure}
  \centering
  \insertfigure[0.75]{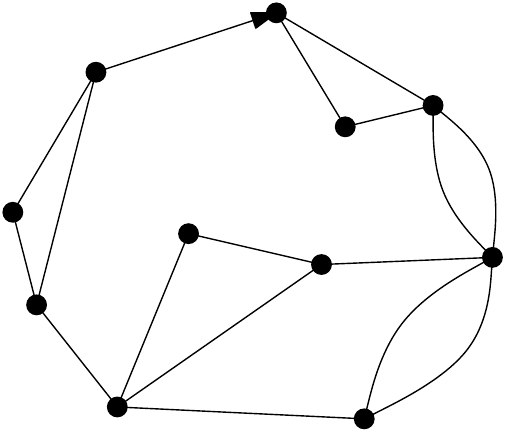}{35}
  \caption{Construction of a new Dyck path that covers a given Dyck path}
  \label{fig:5:tamari-cover}
\end{figure}

We now define a partial order $\preceq$ on the set $\mathcal{D}_n$ of Dyck paths by its covering relation. For a grid point $p$ above the diagonal, we define its \mydef{horizontal distance} as the number of east steps we can take before reaching the diagonal. For instance, in the left Dyck path of Figure~\ref{fig:5:tamari-cover}, the grid point $p$ has horizontal distance $2$. Let $D$ be a Dyck path in $\mathcal{D}_n$. We consider one of its valleys $p$, \textit{i.e.} a grid point on $D$ that is preceded by a step $E$ and followed by a step $N$. We then search for the next grid point $p'$ on $D$ with the same horizontal distance as $p$, and we denote by $[p,p']$ the segment of $D$ between $p$ and $p'$. By exchanging $[p,p']$ with the step $E$ that precedes $p$, we obtain a new Dyck path $D'$, and we say that $D'$ covers $D$, and we have $D \prec D'$. Figure~\ref{fig:5:tamari-cover} gives an example of this construction. By applying the same process to all Dyck paths and all their valleys, we obtain a covering relation, which defines a partial order $\preceq$ on $\mathcal{D}_n$. The poset $(\mathcal{D}_n,\preceq)$ is the \mydef{Tamari lattice of order $n$}. As an example, Figure~\ref{fig:5:tamari-hasse} is the Hasse diagram of the Tamari lattice of order 4.

\begin{figure}
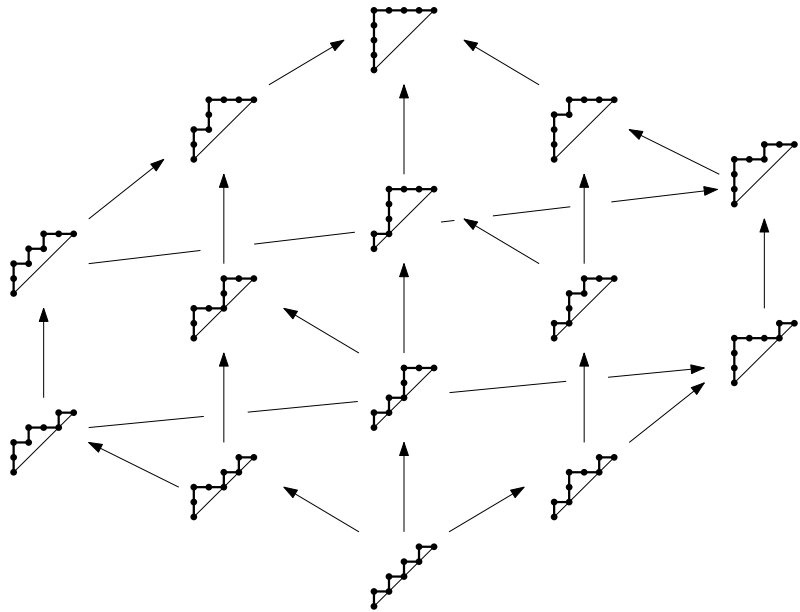

  \centering
  \insertfigure[0.7]{ch5-fig.pdf}{17}
  \caption{Hasse diagram of the Tamari lattice of order 4}
  \label{fig:5:tamari-hasse}
\end{figure}

Out of algebraic concern, Bergeron introduced the $m$-Tamari lattice, which is a natural generalization of the usual Tamari lattice. More precisely, the introduction of the $m$-Tamari lattice was inspired by many results and conjectures concerning the diagonal coinvariant spaces of the symmetric group, also called Garsia-Haiman spaces. Interested readers are referred to \cite{Bergeron-book-algcomb} and \cite{haglund-book} for more information about this algebraic link. 

The definition of the $m$-Tamari lattice itself is very similar to that of the usual Tamari lattice, and is rather straight-forward. We first define a \mydef{grid path} as a path on $\mathbb{Z}^2$ formed by north and east steps $N$ and $E$ that starts from the origin. Instead of taking the diagonal $y=x$, we now take the \mydef{$m$-diagonal} $y=x/m$, and consider the set $\mathcal{D}^{(m)}_n$ of grid paths with $n$ steps $N$ and $mn$ steps $E$ that start and end on the $m$-diagonal while staying above the $m$-diagonal. These paths are also called \mydef{$m$-ballot paths} of size $n$. We can similarly define the horizontal distance of a grid point with respect to the $m$-diagonal, and use it to define a covering relation on the set $\mathcal{D}^{(m)}_n$. The procedure is exactly the same: for an element $D$ of $\mathcal{D}^{(m)}_n$, we pick one of its valleys $p$ and find the next grid point $p'$ with the same horizontal distance, then by exchanging the segment $[p,p']$ with the step $E$ that precedes $p$, we obtain another element $D'$ of $\mathcal{D}^{(m)}_n$ which covers $D$. Figure~\ref{fig:5:m-tamari-cover} gives an example in the set $\mathcal{D}^{(2)}_5$. The partial order on $\mathcal{D}^{(m)}_n$ defined by the covering relation constructed in this way is a lattice, and is called the \mydef{$m$-Tamari lattice} of order $n$. When $m=1$, we obtain the usual Tamari lattice.

\begin{figure}
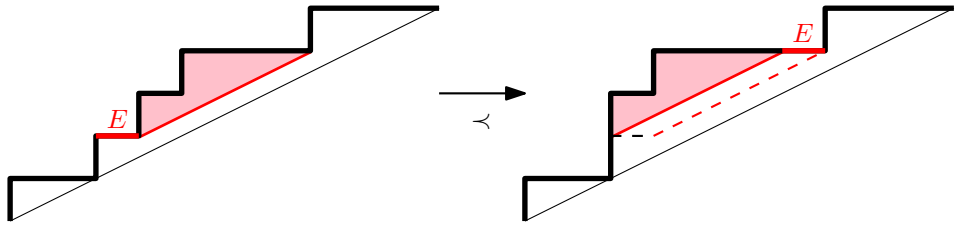

  \centering
  \insertfigure{ch5-fig.pdf}{14}
  \caption{Construction of a covering $m$-ballot path from a given $m$-ballot path in the $m$-Tamari lattice}
  \label{fig:5:m-tamari-cover}
\end{figure}

But why do we restrict ourselves to straight lines for the diagonal? In fact, when interpreted correctly, the same construction works for any grid path $v$ formed by north and east steps and gives a lattice. For a fixed grid path $v$ formed by $N$ and $E$, we consider the set of grid paths $\mathcal{D}_v$ formed by steps $N$ and $E$ that start and end at the same points as $v$ and that always stay weakly above $v$. The grid path $v$ is also called the \mydef{canopy} of the set $\mathcal{D}_v$, due to what it represents on binary trees (interested readers are referred to \cite{PRV2014extension} for more details). We can similarly define the horizontal distance of a grid point $p$ to the west of the canopy $v$, which is given by the maximal number of east steps that we can take starting from $p$ without crossing $v$. The left part of Figure~\ref{fig:5:tam-def} shows a grid path $v$ and an element $v_1$ in its corresponding set $\mathcal{D}_v$, with the horizontal distance of each grid point on $v_1$. We can now define a covering relation as we have done in previous cases: given a grid path $v_1$ in $\mathcal{D}_v$, we first pick a valley $p$ of $v_1$, then search for the next grid point $p'$ with the same horizontal distance as $p$, and by exchanging the segment $[p,p']$ with the step $E$ that precedes $p$, we obtain another element $v_2$ in $\mathcal{D}_v$ that covers $v_1$, and we have $v_1 \preceq v_2$. The right part of Figure~\ref{fig:5:tam-def} gives an example of this construction of covering elements in $\mathcal{D}_v$. We thus define a partial order on $\mathcal{D}_v$, which is also a lattice, and it is called the \mydef{generalized Tamari lattice} with canopy $v$, denoted by \tam{v}. Generalized Tamari lattices were first defined by Pr\'eville-Ratelle and Viennot in \cite{PRV2014extension}. If we take $v=(NE^m)^n$, the corresponding \tam{v} will be the $m$-Tamari lattice of order $n$, which means that \tam{v} is a legitimate generalization of the usual Tamari lattice and the $m$-Tamari lattice.

\begin{figure}
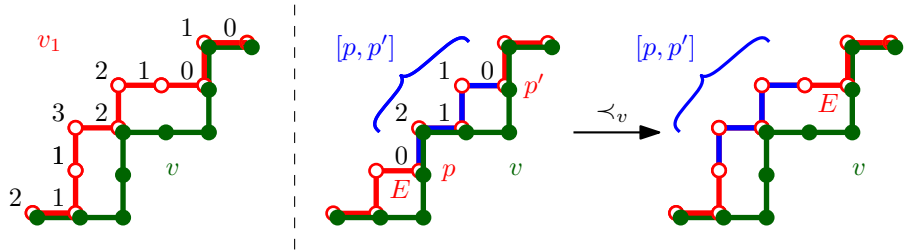

  \centering
  \insertfigure{ch5-fig.pdf}{9}
  \caption{Horizontal distance and the covering relation in \tam{v}}
  \label{fig:5:tam-def}
\end{figure}

The most interesting fact about generalized Tamari lattices is not that they generalize the usual Tamari lattice, but rather in the other way around: they reveal a fine structure of the usual Tamari lattice. To make sense of this assertion, we need a few more notions. For a Dyck path $P = (p_i)_{1 \leq i \leq 2n}$ where $p_i$ are steps in the set $\{N,E\}$, let $i_1, \ldots, i_n$ be the indices such that $p_{i_k} = N$. We define $Type(P)$ as the following word $w$ of length $n-1$: for $k \leq n-1$, if $p_{i_k} = p_{i_k+1} = N$, then $w_k = E$, otherwise $w_k = N$. In other words, for the $i^{\rm th}$ north step in $P$, if it is followed by a north step $N$ (resp. an east step $E$), then the $i^{\rm th}$ letter of the type $\mathit{Type}(P)$ will be $E$ (resp. $N$). This convention may look strange, but later we will see that there is a good reason to adopt it. Since the last north step of a Dyck path is always followed by an east step, it is not accounted in the type. The type of Dyck path of size $n$ is thus a word in $N,E$ of length $n-1$. See Figure~\ref{fig:5:dyck-type} for an illustration of how we determine the type of Dyck path. We observe that the type of Dyck path can also be reinterpreted as a grid path. We say that two Dyck paths are \mydef{synchronized} if they are of the same type. The term ``synchronized'' comes from the fact that they have east steps on the same horizontal levels. Figure~\ref{fig:5:dyck-type} shows an example of a pair of synchronized Dyck paths.

\begin{figure}
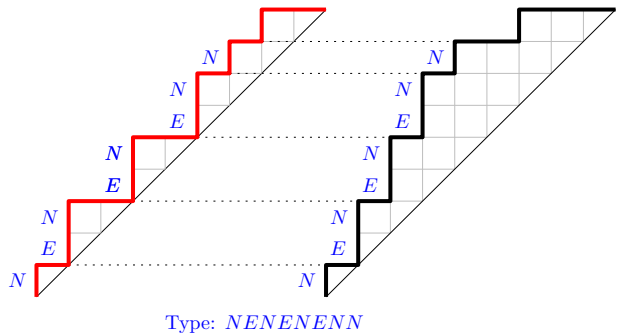

  \centering
  \insertfigure[0.75]{ch5-fig.pdf}{20}
  \caption{Two Dyck paths with the same type}
  \label{fig:5:dyck-type}
\end{figure}

We can now partition the elements of the usual Tamari lattice of order $n$ according to their $2^{n-1}$ possible types. The upper part of Figure~\ref{fig:5:tam-partition} is such a partition of the usual Tamari lattice of order $n$. We have three observations. The first one is that, given a type $v$, elements in $(\mathcal{D}_n, \preceq)$ with the same type $v$ form an interval (actually a sub-lattice), denoted by $I(v)$. The second one is that the interval $I(v)$ of Dyck paths with a given type $v$ is isomorphic to the generalized Tamari lattice \tam{v} with diagonal (or canopy) $v$, if we regard $v$ as a grid path. These isomorphisms can be seen on the lower part of Figure~\ref{fig:5:tam-partition}. The third one is that the set of intervals formed by elements of the same type possesses a ``central symmetry'', which is in fact a well-known isomorphism between the usual Tamari lattice and its dual. For readers familiar with the usual Tamari lattice, in its alternative definition using binary trees and tree rotation, this isomorphism is simply taking the mirror image of binary trees in the lattice. 

\begin{figure}
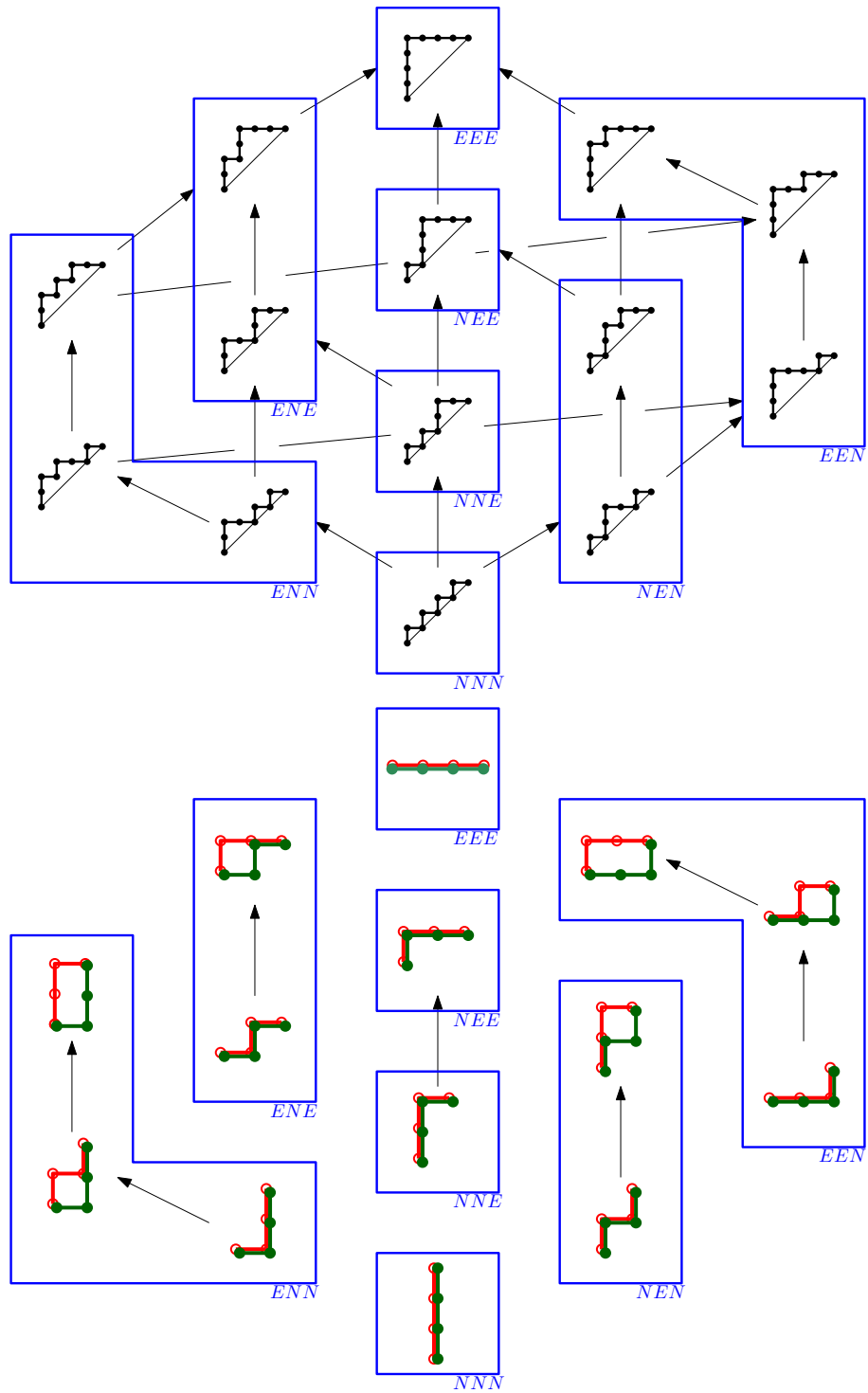

  \begin{center}
    \insertfigure[0.75]{ch5-fig.pdf}{18} 
    \insertfigure[0.75]{ch5-fig.pdf}{19}
  \end{center}
  \caption{Partition of the usual Tamari lattice, and isomorphic generalized Tamari lattices of each block}
  \label{fig:5:tam-partition}
\end{figure}

In fact, our observation is not a coincidence. They are in fact consequences of the main results in \cite{PRV2014extension} of Pr\'eville-Ratelle and Viennot, where they introduced the generalized Tamari lattice. 

\begin{thm}[Theorem~3 in \cite{PRV2014extension}] \label{thm:5:tamari-partition}
The usual Tamari lattice $(\mathcal{D}_n,\preceq)$ of order $n$ can be partitioned into disjoint intervals $I(v)$ according to the type of elements, where $v$ can be any word of length $n-1$ consisting of letters $N,E$, namely
\[
\mathcal{D}_n = \bigcup_{v \in \{N,E\}^{n-1}} I(v).
\]
Furthermore, $I(v)$ is isomorphic to \tam{v}, the generalized Tamari lattice with $v$ as canopy.
\end{thm}

\begin{thm}[Theorem~2 in \cite{PRV2014extension}] \label{thm:5:tamari-symmetry}
For a given grid path $v$, the lattice \tam{v} is isomorphic to the dual of \tam{\overleftarrow{v}}, where $\overleftarrow{v}$ is the word $v$ read from right to left, with letters $N$ and $E$ exchanged.
\end{thm}

An element in \tam{v} can be seen as a pair of non-crossing grid paths with the same endpoints where $v$ is the lower path. It is already known (\textit{cf.} \cite{XV_DeVi84, levine}) that such pairs of paths of length $n-1 \geq 0$ are counted by Catalan number $C_n = \frac{1}{n+1}\binom{2n}{n}$, which is also the number of Dyck paths of size $n$. On the other hand, the partition of the usual Tamari lattice $(\mathcal{D}_n,\preceq)$ into intervals $I(v)$ was already known to Loday and Ronco in \cite{loday1998hopf}, and has been used in the study of the associahedron. Nevertheless, in \cite{PRV2014extension} it was the first time that we have a good grasp on the exact structure of these $I(v)$'s, which was done through the prism of non-crossing grid paths.

We are now interested in the enumeration of intervals in the Tamari lattice and its generalizations. It is also here that we discover a mysterious link between these intervals and planar maps. Although the Tamari lattice is a relatively old object, the enumeration of its intervals came up only relatively recently. In 2005, Chapoton tackled this problem in \cite{chapoton-tamari} and obtained the following theorem.
\begin{thm}[Theorem~2.1 in \cite{chapoton-tamari}]
The number of intervals in the usual Tamari lattice of order $n$ is
\[ \frac{2}{n(n+1)} \binom{4n+1}{n-1}. \]
\end{thm}
Curiously, this is also the number of 3-connected planar triangulations with $n+3$ vertices, for which a formula was given by Tutte in \cite{tutte-0}. 

The next step is the $m$-Tamari lattice. The following formula for the number of intervals in the $m$-Tamari lattice was first conjectured by Bergeron and Pr\'eville-Ratelle in \cite{bergeron-preville}, then proved in \cite{bousquet-fusy-preville} by Bousquet-M\'elou, Fusy and Pr\'eville-Ratelle.
\begin{thm}[(26) in \cite{bergeron-preville}, Corollary~11 in \cite{bousquet-fusy-preville}]
The number of intervals in the $m$-Tamari lattice of order $n$ is
\[ \frac{m+1}{n(mn+1)} \binom{(m+1)^2 n + m}{n-1}.\]
\end{thm}
For readers familiar with the enumerative study of planar maps, we can see a strong resemblance of this formula to formulas for many classes of planar maps. This resemblance is an indirect evidence that these intervals should also be related to planar maps. Some, including me, believe that there must be a naturally-defined class of planar maps parametrized by $m$ that shares the same enumeration formula with intervals in the $m$-Tamari lattice. This belief is somehow reinforced by the bijective correspondence of intervals in the usual Tamari lattice (case $m=1$) and simple planar triangulations in \cite{BB2009intervals} by Bernardi and Bonichon. 

In \cite{BMCPR2013representation}, Bousquet-M\'elou, Chapuy and Pr\'eville-Ratelle also enumerated a labeled version of intervals in the $m$-Tamari lattice. The study of these intervals has a motivation from algebra, where it is conjectured that the number of intervals in the $m$-Tamari lattice of order $n$ is also the dimension of the alternating component of the trivariate Garsia-Haiman space of the same order, while labeled intervals correspond to the entire trivariate Garsia-Haiman space. This algebraic connection also motivated the introduction of the generalized Tamari lattice \tam{v} in \cite{PRV2014extension} by Pr\'eville-Ratelle and Viennot. We are thus interested in the enumeration of intervals in generalized Tamari lattices. Since intervals in both the usual Tamari lattice and the $m$-Tamari lattice are related to planar maps, we expect that we can relate intervals in generalized Tamari lattices to some class of planar maps.

In this chapter, we will give an expression for the total number of intervals in generalized Tamari lattices \tam{v} with canopy $v$ of a given length $n$. This task seems to be difficult, as we need to deal with many lattices. But by Theorem~\ref{thm:5:tamari-partition}, all such \tam{v} can be found exactly once in the usual Tamari lattice, and are isomorphic to intervals formed by Dyck paths of the same type. Therefore, intervals in \tam{v} are in bijection with intervals $[x,y]$ in the usual Tamari lattice such that $x$ and $y$ are of the same type. Such intervals in the usual Tamari lattice are called a \mydef{synchronized interval}. A bijection from intervals in generalized Tamari lattices to synchronized intervals is given in \cite{PRV2014extension}. Our problem can thus be reduced to the enumeration of synchronized intervals, which is more adapted to our purpose.

Our main contribution in this chapter is a bijection between synchronized intervals (thus intervals in \tam{v}) and non-separable planar maps. To describe this bijection, we need an intermediate structure called \emph{decorated tree}, which is a kind of rooted trees with labels on their leaves that satisfy certain conditions. We then show that an exploration process gives a bijection between non-separable planar maps and decorated trees, and there is a bijection between decorated trees and synchronized intervals.

As a consequence of our bijection, we give the following enumeration formula for the number of intervals in $\bigcup_v \textsc{Tam}(v)$.
\begin{thm}\label{thm:5:enum}
The total number of intervals in \tam{v} over all possible $v$ of length $n-1$ is given by
\begin{equation} \label{eq:5:cnt} \sum_{v \in (N,E)^{n-1}} \mathrm{Int}(\textsc{Tam}(v)) = \frac{2 (3n)!}{(n+1)! (2n+1)!}. \end{equation}
\end{thm}
This enumeration formula was first obtained by Tutte in \cite{Tutte:census} for non-separable planar maps. As an example, when $n=4$, the formula \eqref{eq:5:cnt} is evaluated to 22, which is exactly the number of intervals in all generalized Tamari lattices with a canopy of length $3$, which are illustrated in the lower part of Figure~\ref{fig:5:tam-partition}. At the end of this chapter, we will also discuss some other enumerative consequences of our bijection.

\section{Recursive decompositions}

Starting from this section, we will revert to the usual representation of Dyck paths by up steps $u=(1,1)$ and down steps $d=(1,-1)$. We can also view a Dyck path as a word formed by letters $\{u,d\}$. However, the type $\mathit{Type}(P)$ of a Dyck path $P$ still uses letters $N$ and $E$. The choice of using diagonal steps $\{u,d\}$ for Dyck paths is intentional. Dyck paths are elements in the usual Tamari lattice, while the type of Dyck path, formed by north and east steps $\{N,E\}$, corresponds to the canopy of a generalized Tamari lattice. We choose to use different notations for these two kinds of objects to underline the fact that they are elements in \textbf{different} lattices.

We are now interested in the link between two families of objects: synchronized intervals in the usual Tamari lattice on Dyck paths of size $n$, and non-separable planar maps with $n+1$ edges. In fact, their enumerations are governed by the same functional equation. In this section, we show how to decompose recursively these two families of combinatorial objects to obtain a functional equation of their generating functions. We reiterate that our main contribution, which is the enumeration of generalized intervals via a non-recursive bijection, will be described explicitly in the next section.

\subsection{Recursive decomposition of synchronized intervals}

We define a \mydef{properly pointed Dyck path} to be a Dyck path $P = P^\ell P^r$ such that $P^\ell$ and $P^r$ are Dyck paths, and $P^\ell$ is not empty unless $P$ is itself empty. Furthermore, a \mydef{properly pointed synchronized interval} $[P^\ell P^r, Q]$ is a synchronized interval where the lower Dyck path is properly pointed. We recall that $\mathcal{I}_n$ is the set of synchronized intervals of size $n$, and we denote by $\mathcal{I}^\bullet_n$ the set of properly pointed synchronized intervals of size $n$.

Before giving a recursive decomposition of synchronized intervals, we need to introduce some additional notation and borrow a lemma from \cite{bousquet-fusy-preville}. 

There is a natural matching between up steps and down steps in a Dyck path defined as follows: let $u_i$ be an up step of a Dyck path $P$, we draw a horizontal ray from the middle of $u_i$ to the right until it meets a down step $d_j$, and we say that $u_i$ is \mydef{matched} with $d_j$. We denote by $\ell_P(u_i)$ the distance from $u_i$ to $d_j$ in $P$ considered as a word, which is defined as the number of letters between $u_i$ and $d_j$ plus 1. For example, in $P=uududd$, we have $\ell_P(u_1)=5$, since its matching letter $d$ is the one at the end. We define the \mydef{distance function} $D_P$ by $D_P(i) = \ell_P(u_i)$, where $u_i$ is the $i^{\rm th}$ up step in $P$.

\begin{lem}[Proposition~5 in \cite{bousquet-fusy-preville}]\label{lem:5:ericmirlouis}
Let $P$ and $Q$ be two Dyck paths of size $n$. Then $P \leq Q$ in the Tamari lattice if and only if $D_{P}(i) \leq D_{Q}(i)$ for all $1 \leq i \leq n$. 
\end{lem}

We can now describe a way to construct a larger synchronized interval from a smaller synchronized interval and a properly-pointed synchronized interval.

\begin{prop} \label{prop:5:rec-construct}
Let $I_1 = [P_1^\ell P_1^r, Q_1]$ be a properly pointed synchronized interval and $I_2 = [P_2, Q_2]$ a synchronized interval. We construct the Dyck paths
\[ P = u P_1^\ell d P_1^r P_2, \qquad Q = u Q_1 d Q_2. \]
Then $I = [P,Q]$ is a synchronized interval. Moreover, this transformation from $(I_1, I_2)$ to $I$ is a bijection between $\cup_{n \geq 0} \mathcal{I}^\bullet_n \times \cup_{n \geq 0} \mathcal{I}_n$ and $\cup_{n > 0} \mathcal{I}_n$.
\end{prop}
\begin{proof} 
An illustration of the construction of $I$ is given in Figure~\ref{fig:5:interval-decomp}. To show that $I$ is a synchronized interval, we only need to show that $P$ and $Q$ have the same type, and $[P,Q]$ is an interval in the Tamari lattice. To show that $\mathrm{Type}(P)=\mathrm{Type}(Q)$, we notice that $P_2$ and $Q_2$ are of the same type since they form a synchronized interval. It is clear that, for two Dyck paths $P_\ell, P_r$, the type of their concatenation $P_\ell P_r$ is the concatenation of $\mathrm{Type}(P_\ell)$ and $\mathrm{Type}(P_r)$. Let $W_1 = \mathrm{Type}(u P_1^\ell d P_1^r)$ and $W_2 = \mathrm{Type}(u Q_1 d)$. We only need to show that $W_1 = W_2$. We write $W_1 = w_1 W_1'$ and $W_2 = w_2 W_2'$, where $w_1$ (resp. $w_2$) is the first letter of $W_1$ (resp. $W_2$), and $W_1'$ (resp. $W_2'$) is the rest of the word $W_1$ (resp. $W_2$). We clearly have $w_1 = w_2$, since $w_1 = N$ if and only if $P_1^\ell$ is empty, which is equivalent to $Q_1$ being empty, which occurs if and only if $w_2=N$. For the rests $W_1'$, since $P_1^\ell$ is a Dyck path which ends in $d$ by definition, we have $W_1' = \mathrm{Type}(P_1^\ell)\mathrm{Type}(P_1^r) = \mathrm{Type}(P_1^\ell,P_1^r)$. The same argument applies to $W_2'$ and $Q_1$, which leads to $W_2' = \mathrm{Type}(Q_1)$. Since $[P_1,Q_1]$ is a synchronized interval, we have $W_1'=W_2'$. We conclude that $W_1=W_2$. Therefore, $P$ and $Q$ are of the same type. It is not difficult to show that $[P,Q]$ is a Tamari interval using Lemma~\ref{lem:5:ericmirlouis} and the fact that both $I_1 = [P_1^\ell P_1^r,Q_1]$ and $I_2=[P_2,Q_2]$ are also Tamari intervals.



To show that the transformation we described is indeed a bijection, we only need to show that we can decompose any non-empty Tamari interval $I=[P,Q]$ back into $(I_1,I_2)$. To go back from $I=[P,Q]$ to $(I_1,I_2)$, we only need to split $P$ and $Q$ into $P=u P_1^\ell d P_1^r P_2$ and $Q=u Q_1 d Q_2$ such that $P_1^\ell, P_1^r, P_2, Q_1, Q_2$ are all Dyck paths with $P_2, Q_2$ of the same length. This can be done by first cutting $Q$ at the first place that it touches again the $x$-axis, where $P$ also touches the $x$-axis. This cutting breaks $Q$ into $u Q_1 d$ and $Q_2$, and $P$ into $P_1$ and $P_2$. We then perform the same operation on $P_1$ to cut it into $u P_1^\ell d$ and $P_1^r$. We thus conclude that we indeed have a bijection between $\cup_{n \geq 0} \mathcal{I}^\bullet_n \times \cup_{n \geq 0} \mathcal{I}_n$ and $\cup_{n > 0} \mathcal{I}_n$.
\end{proof}

\begin{figure}
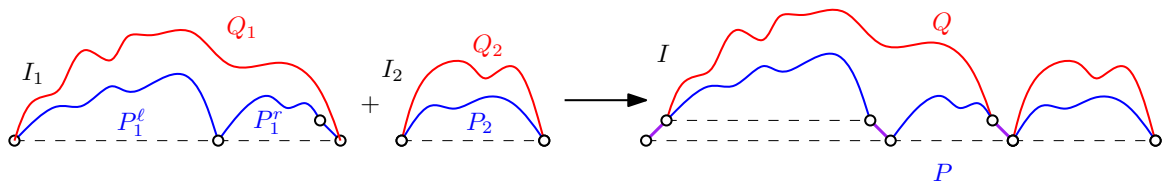

\centering
\insertfigure[0.95]{ch5-fig.pdf}{4}
\caption{Composition of synchronized intervals} \label{fig:5:interval-decomp}
\end{figure}

Since the construction in Proposition~\ref{prop:5:rec-construct} is a bijection, we can also see it in the reverse direction as a recursive decomposition. We now translate the recursive decomposition in Proposition~\ref{prop:5:rec-construct} into a functional equation for the generating function of synchronized intervals. To this end, we need to investigate another statistic on synchronized intervals, which will give us a suitable catalytic variable for our functional equation. A \mydef{contact} of a Dyck path $P$ is an intersection of $P$ with the $x$-axis. Both endpoints of $P$ are also considered as contacts. Let $\operatorname{contacts}(P)$ be the number of contacts of $P$. We define $F(x,t)$ as the following generating function of synchronized intervals: 
\[ F(x,t) = \sum_{n \geq 1} \sum_{[P,Q] \in \mathcal{I}_n} t^n x^{\operatorname{contacts}(P)-1}. \]

From Proposition \ref{prop:5:rec-construct}, we know that a non-empty synchronized interval $I=[P,Q]$ can be decomposed into $P=u P_1^\ell d P_1^r P_2$ and $Q = u Q_1 d Q_2$, where $P_1^\ell, P_1^r, P_2, Q_1, Q_2$ are all Dyck paths, and both $[P_1^\ell P_1^r, Q_1]$ and $[P_2,Q_2]$ are synchronized intervals. The generating functions for the intervals of the form $[u P_1^\ell d P_1^r,u Q_1 d]$ is given by $xt\left( 1 + \frac{F(x,t) - F(1,t)}{x-1} \right)$, where the divided difference accounts for \emph{pointing} each non-initial contact (individually) over all elements in $\mathcal{I}_n$ to obtain properly pointed intervals of the form $I^\bullet=[P_1^\ell P_1^r,Q_1]$. Indeed, a non-empty Dyck path $P_1$ with $k+1$ contacts and length $2n$ has a contribution of $x^k t^n$ to the generating function $F(x,t)$. Furthermore, there are exactly $k$ ways to turn $P_1$ into a properly pointed Dyck path $P_1^\ell P_1^r$, by cutting at one of the contacts, except the first one. When lifted from $P_1^\ell P_1^r$ to $u P_1^\ell d  P_1^r$, all contacts inside $P_1^\ell$ except the initial one are lost, which leads to $k$ properly pointed Dyck paths with length $2n$ and a number of contacts from $2$ to $k+1$. These properly pointed Dyck paths obtained from $P_1$ thus have a total contribution of $t^n(x+x^2+\cdots+x^k)=x\frac{t^n x^k - t^n}{x-1}$, which accounts for the divided difference. The term $xt$ is for the added steps. On the other hand, we observe that the path $P = u P_1^\ell d P_1^r P_2$  has $\operatorname{contacts}(P) -1= (\operatorname{contacts}(u P_1^\ell d P_1^r) \! - \!1) + (\operatorname{contacts}(P_2) \! - \! 1)$. Therefore, from Proposition \ref{prop:5:rec-construct}, we obtain the functional equation 
\begin{equation} \label{eq:5:interval} F(x,t) = xt(1 + F(x,t)) \left( 1 + \frac{F(x,t) - F(1,t)}{x-1} \right).
\end{equation}

\subsection{Recursive decomposition of non-separable planar maps}

We now turn to non-separable planar maps, which were first enumerated by Tutte in \cite{Tutte:census} using algebraic methods, then by Jacquard and Schaeffer in \cite{JS1998bijective} using a bijection based on their recursive decomposition. We recall that a \mydef{non-separable planar map} is a planar map without cut vertex that contains at least two edges. Figure~\ref{fig:5:non-sep-planar-map} gives an example of such a map. Note that we exclude the two one-edge maps. We have the following interesting property of non-separable planar maps.

\begin{figure}
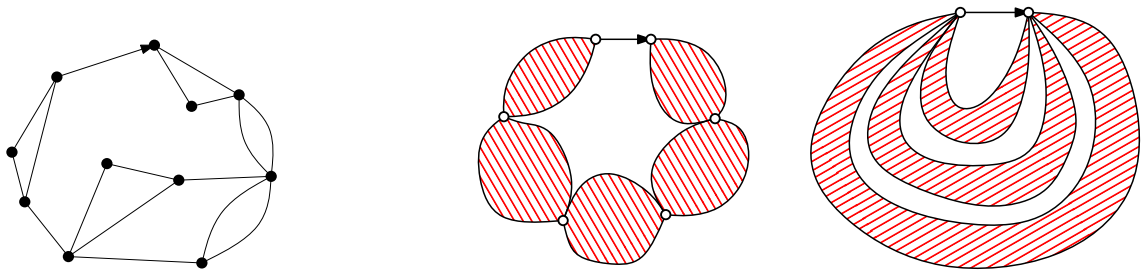

\centering
\insertfigure[0.7]{ch5-fig.pdf}{1} \quad \quad \quad \quad \quad \quad \insertfigure[0.7]{ch5-fig.pdf}{5}
\caption{A non-separable planar map, and series/parallel decompositions of non-separable planar maps} \label{fig:5:non-sep-planar-map} \label{fig:5:map-decomp}
\end{figure}

\begin{prop}[Corollary II in \cite{Tutte:census}] \label{prop:5:non-sep-self-dual}
The dual of a non-separable planar map is also non-separable.
\end{prop}

There are two ways to decompose non-separable planar maps recursively. We will call them ``series'' and ``parallel'' decompositions respectively. We only need one decomposition for the functional equation, but describing both leads to a more thorough understanding. In Figure~\ref{fig:5:map-decomp}, we sketch how larger maps can be built with smaller maps in both series and parallel ways.

For the series decomposition of a non-separable planar map $M$, we delete its root, and the remaining map $M'$ may cease to be non-separable. In $M'$, every cut vertex splits $M'$ into two parts, each containing an endpoint of the root. The remaining map is thus a series of non-separable planar maps (and possibly single edges) linked by cut vertices (see Figure~\ref{fig:5:map-decomp}, middle). Let $\mathcal{M}_n$ be the set of non-separable planar maps with $n+1$ edges, and $M_s(x,t)$ the generating function of non-separable planar maps defined as
\[ M_s(x,t) = \sum_{n \geq 1} \sum_{M \in \mathcal{M}_n} t^n x^{\deg(\textrm{outer face}(M)) - 1}. \]

For a component in the series, we root it at its first edge adjacent to the outer face in clockwise order to obtain a non-separable planar map such that the root vertex is one of the linking vertices in the chain. Conversely, from a non-separable planar map with $n+1$ edges and the outer face of degree $k+1$ (therefore of contribution $t^n x^k$ in $M_s(x,t)$), there are $k$ choices for the cut vertex other than the root vertex to obtain a component, each adding a value from $1$ to $k$ to the outer face degree. These choices thus have a total contribution of $t^{n+1} \sum_{i=1}^k x^i = tx \frac{t^n x^k - t^n}{x-1}$. On the other hand, the contribution of a single edge between two cut vertices is $xt$. Therefore, by the series decomposition we have
\begin{equation} \label{eq:5:map}
M_s(x,t) = \frac{xt + xt\frac{M_s(x,t)-M_s(1,t)}{x-1}}{1 - \left(xt + xt\frac{M_s(x,t)-M_s(1,t)}{x-1} \right)}.
\end{equation}
A reordering gives the same functional equation as \eqref{eq:5:interval}.

For the parallel decomposition, we consider the effect of contracting the root. Let $M'$ be the map obtained from contracting the root edge of a non-separable planar map $M$, and $u$ the vertex of the map $M'$ resulting from the contraction of the root. The only possible cut vertex in $M'$ is $u$. By deleting $u$ and attaching a new vertex of each edge adjacent to $u$, we have an ordered list of non-separable planar components (and possibly single edges) that come in parallel (see Figure~\ref{fig:5:map-decomp}, right). By identifying the newly-added vertices in each connected component, we obtain an ordered list of non-separable planar maps (and possibly loops). Let $M_p(x,t)$ be the generating function of non-separable planar maps as follows:
\[ M_p(x,t) = \sum_{M\in \mathcal{M}_n} t^n x^{\deg(v) - 1}. \]
Here, $v$ is the root vertex of $M$

To obtain a non-separable component from a non-separable planar map, we only need to split the root vertex into two, that is to say to partition edges adjacent to the root vertex into two non-empty sets formed by consecutive edges. For a root vertex $v$ of degree $k$, this can be done by choosing a corner of $v$ other than the root corner, and splitting the edges by the chosen corner and the root corner. There are exactly $k-1$ choices. From a non-separable planar map with root vertex degree $k$, we can thus obtain $k-1$ different non-separable components that we can use in the parallel decomposition, each with root vertex degree from $1$ to $k-1$. We can thus write a functional equation for $M_p$, with the degree of the root vertex minus 1 as the statistics of the catalytic variable. We leave readers to check that the parallel decomposition leads to the same equation as (\ref{eq:5:interval}). 

Since $F, M_s, M_p$ all obey the same functional equation, we have $F = M_s = M_p$, therefore these objects are equi-enumerated under the specified statistics, which invites us to search for a bijective proof. Observe that $M_s = M_p$ already has a simple bijective explanation by duality. For instance, for a non-separable planar map $M$, we take its dual $M^\dagger$ and root it in a way such that the root vertex of $M^\dagger$ is the dual of the outer face of $M$, and the outer face of $M^\dagger$ is the dual of the root vertex of $M$. This bijection preserves the number of edges and transfers the degree of the outer face to the degree of the vertex from which the root points, which implies $M_s=M_p$.

\section{Bijections}

Before we present our main contribution, which is a bijection between synchronized intervals and non-separable planar maps, we will first give some intuition on how we came up with this bijection. 

In Section~\ref{sec:1:bij}, we have seen quite a few bijections from planar maps to trees. The main ideas of these bijections is more or less the same: using a certain exploration process on a planar map (or its dual in some cases), we can construct a spanning tree of the map by cutting open some edges, alongside with some labels or blossoms to remember which edges to reconnect when restoring the map from the tree. We can thus imagine such an exploration process that turns a non-separable planar map into a plane tree with some extra labels. On the other hand, a synchronized interval is just a pair of Dyck paths, and a Dyck path is not far from a plane tree. Indeed, there is a well-known bijection that extracts a Dyck path from a plane tree: we do a traversal starting from the root along the contour of the plane tree from left to right, and the variation of depths in the traversal forms a Dyck path. Therefore, from a plane tree with some extra labels, we can easily extract one Dyck path, and for the other we need to make use of the extra labels. To summarize, our strategy to establish the wanted bijection is to first devise an exploration process to turn non-separable planar maps into trees with labels, then given such a tree with labels, we extract a Dyck path from its tree structure, and another using also its labels, and we will argue that the two Dyck paths form a synchronized interval.

We now present our main contribution. To describe our bijection from synchronized intervals to non-separable planar maps, we first introduce a family of trees. We take the convention that the root of a tree is of depth $0$. The \mydef{traversal order} on the leaves of a tree is simply the left-to-right order. A \mydef{decorated tree} is a rooted plane tree with an integer not smaller than $-1$ attached to each leaf as a label, satisfying the following conditions:
\begin{enumerate}
\item For a leaf $\ell$ adjacent to a vertex of depth $p$, the label of $\ell$ is strictly smaller than $p$.
\item For each internal node of depth $p>0$, there is at least one leaf in its descendants with label at most $p-2$.
\item For $t$ a node of depth $p$ and $T'$ a sub-tree rooted at a child of $t$, consider leaves of $T'$ in traversal order. If a leaf $\ell$ is labeled $p$ (which is the depth of $t$), each leaf in $T'$ coming before $\ell$ has a label at least $p$.
\end{enumerate}
The right side of Figure~\ref{fig:5:bij-map-tree} gives an example of a decorated tree. In a decorated tree, a leaf labeled with $-1$ is called a \mydef{free leaf}. We denote by $\mathcal{T}_n$ the set of decorated trees with $n$ edges (internal and external).

The definition of decorated trees may not be very intuitive, but after the introduction of the exploration process, we will see that each condition captures an important aspect of non-separable planar maps, and together they characterize trees that we obtain from non-separable planar maps via the exploration process we will define.


\subsection{From maps to trees}

We start with a bijection from non-separable planar maps to decorated trees which relies on the following exploration procedure. For a non-separable planar map $M$ with a root pointing from $v$ to $u$, we perform a depth-first exploration of vertices in clockwise order around each vertex, starting from $v$ and the root. When the exploration along an edge adjacent to the current vertex $w$ encounters an already visited vertex $x$, we replace the edge by a leaf attached to $w$ labeled with the depth of $x$ in the tree, with the convention that the depth of $v$ is $-1$. Since the map is non-separable, this exploration gives a spanning tree whose root $v$ has degree $1$, or else $v$ will be a cut vertex of the map. We then delete the edge $(v,u)$ to obtain $\mathrm{T}(M)$. Figure~\ref{fig:5:bij-map-tree} gives an instance of the transformation $\mathrm{T}$.

\begin{figure}
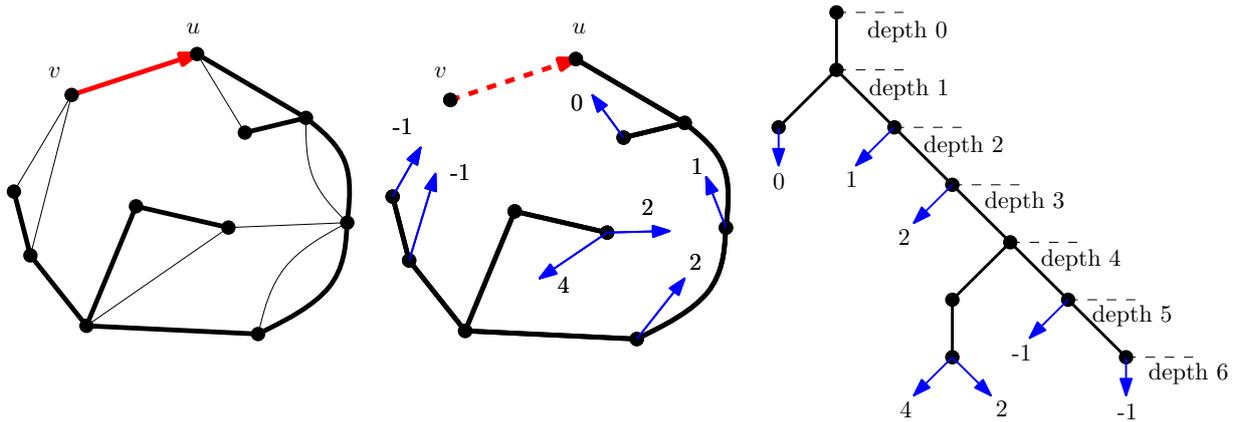

\centering
\insertfigure[0.9]{ch5-fig.pdf}{2}
\caption{An example of the bijection $\mathrm{T}$ from non-separable planar maps to decorated trees} \label{fig:5:bij-map-tree}
\end{figure}

By abuse of notation, we identify internal nodes of $\mathrm{T}(M)$ with corresponding vertices in $M$. We notice that, for children of the same vertex in the tree, the ones being visited first in the map come last in the traversal order. Readers familiar with graph algorithms will notice that this exploration procedure is very close to an algorithm proposed by Hopcroft and Tarjan in \cite{hopcroft} that finds 2-connected components of an undirected graph. Indeed, our exploration procedure can be seen as an adaptation of that algorithm in the case of planar maps, where there is a natural cyclic order for edges adjacent to a given vertex.

We now present the inverse $\mathrm{S}$ of $\mathrm{T}$, with an example illustrated in Figure~\ref{fig:5:bij-S}. For a decorated tree $T$ rooted at $u$, we define $\mathrm{S}(T)$ as the map obtained according to the following steps.
\begin{enumerate}
\item Attach an edge $\{u,v\}$ to $u$ with a new vertex $v$, and make it the root, pointing from $v$ to $u$.
\item In clockwise order, for each leaf $\ell$ in the tree starting from the last leaf in traversal order, do the following. Let $t$ be the parent of $\ell$ and $p$ the label of $\ell$. Let $s$ be the ancestor of $\ell$ of depth $p$, and $e$ be the first edge of the path from $s$ to $\ell$ (thus an edge adjacent to $s$). We replace $\ell$ by an edge from $t$ to $s$ by attaching its other end to $s$ just after $e$ in clockwise order around $s$. See also the right-hand side of Figure~\ref{fig:5:ST}
\end{enumerate}

\begin{figure}
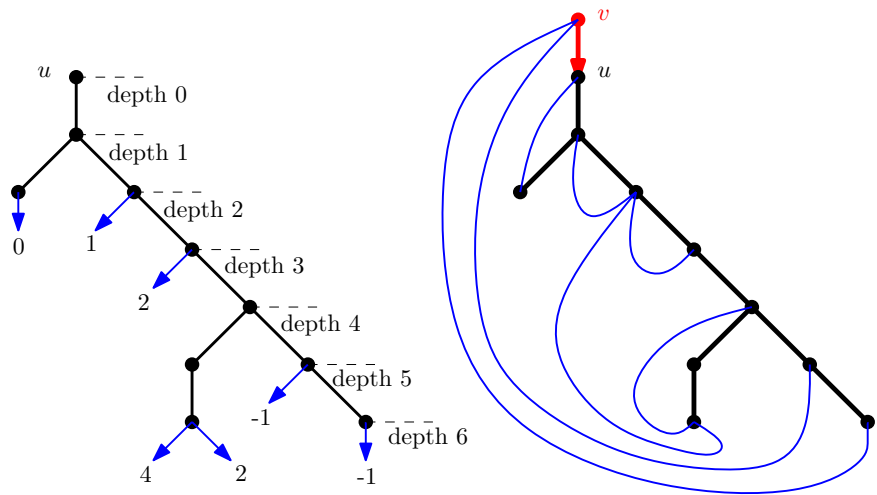

  \centering
  \insertfigure[0.9]{ch5-fig.pdf}{29}
  \caption{An example of the bijection $\mathrm{S}$ from decorated trees to non-separable planar maps}
  \label{fig:5:bij-S}
\end{figure}

From the definitions of $\mathrm{T}$, we will show that the first and third conditions of decorated trees guarantee that $\mathrm{T}$ is an exploration tree of a certain planar map, and the second guarantees the map is non-separable. We now give detailed proofs that $\mathrm{T}$ and $\mathrm{S}$ are well-defined transformations between $\mathcal{M}$ and $\mathcal{T}$, and that they are in fact bijective and inverses of each other. 

\begin{prop} \label{prop:5:inclusionTMnintoTn} 
$\mathrm{T}(\mathcal{M}_n) \subset \mathcal{T}_n$.
\end{prop}
\begin{proof}
Let $M \in \mathcal{M}_n$ be a non-separable planar map with $n+1$ edges. It is clear that $\mathrm{T}(M)$ has $n$ edges. We suppose that the root of $M$ points from $v$ to $u$. We need to check the three conditions of decorated trees on $\mathrm{T}(M)$. For the first condition, let $\ell$ be a leaf adjacent to a node $t$ of depth $p$ in $\mathrm{T}(M)$, resulting from the exploration of the edge $e = \{t, t'\}$. It follows from the exploration order that $t'$ is a vertex of $M$ that was not completely explored at the moment that $e$ was visited. Thus, $t'$ is an ancestor of $t$ in $\mathrm{T}(M)$, and the label of $\ell$ is strictly smaller than $p$. For the second condition, let $t$ be a node of $\mathrm{T}(M)$ of depth $p>0$. If all leaves in the sub-tree induced by $t$ have labels at least $p-1$, then any path linking $t$ and the root vertex $v$ goes through the parent of $t$, making it a cut vertex, which is forbidden. For the third condition, let $t$ be a node of depth $p$ and $T'$ one of the sub-trees rooted at a child of $t$, and suppose that there is a leaf $\ell$ labeled $p$ in $T'$. By the exploration order, the cycle formed by the edge corresponding to $\ell$ and the path from $\ell$ to $t$ in $\mathrm{T}(M)$ encloses all leaves in $T'$ coming before $\ell$. Therefore, any leaf coming before $\ell$ in $T'$ cannot have a label strictly less than $p$, or else there will be a crossing in $M$ that makes it not planar. With all conditions satisfied, $\mathrm{T}(M)$ is a decorated tree. 
\end{proof}

\begin{prop}\label{prop:5:inclusionSTnintoMn} 
$\mathrm{S}(\mathcal{T}_n) \subset \mathcal{M}_n$.
\end{prop}
\begin{proof}
Let $T \in \mathcal{T}_n$ be a decorated tree with $n$ edges. It is clear that $\mathrm{S}(T)$ is a map with $n+1$ edges. We first prove that $\mathrm{S}(T)$ is a non-separable planar map. It is not difficult to show from the first and the third conditions of the definition of decorated trees and from the definition of $\mathrm{S}$ that $\mathrm{S}(T)$ is planar. We suppose that $\mathrm{S}(T)$ is separable and $t$ is a cut vertex. We cannot have $t=v$ since $T$ is already a connected spanning tree of all vertices besides $v$, therefore $t$ is a vertex of $T$. Suppose that $t$ has depth $p$. We consider the connected component of $\mathrm{S}(T)$ containing $v$ after removing $t$. It must contain all vertices that are not descendants of $t$. Therefore, for at least one sub-tree of $t$ rooted at a child $t'$ of $t$, there is no leaf linking to ancestors of $t$, or equivalently this sub-tree only contains leaves with labels greater than or equal to $p$, which violates the definition of decorated trees since $t'$ has depth $p+1$. 
\end{proof}

\begin{prop}\label{prop:5:compositionST}
For any non-separable planar map $M$, we have $\mathrm{S}(\mathrm{T}(M)) = M$.
\end{prop}
\begin{proof}
Using leaf labels, it is clear from the definitions that $\mathrm{S}(\mathrm{T}(M))$ is equal to $M$ as a graph, and we only need to show that they have the same cyclic order of edges around each vertex. Let $t$ be an internal node of depth $p$ in $\mathrm{T}(M)$. We consider its descendant leaves of label $p$ in one of its sub-trees $T'$ induced by a descendant edge $e$ adjacent to $t$. Let $\ell_i$ be such a leaf. When reconnecting, the new edge corresponding to $\ell_i$ should come before $e$ by construction of $\mathrm{T}(M)$, and it cannot encompass other sub-trees rooted at a child of $t$, or else $t$ will be a cut vertex (see the left part of Figure~\ref{fig:5:ST}). If there are multiple such leaves in $T'$, their order is fixed by planarity (see the right part of Figure~\ref{fig:5:ST}). The reasoning also works for the extra vertex $v$ that is not in $\mathrm{T}(M)$. There is thus only one way to recover a planar map from $\mathrm{T}(M)$, and we have $\mathrm{S}(\mathrm{T}(M)) = M$. 
\end{proof}

\begin{figure}
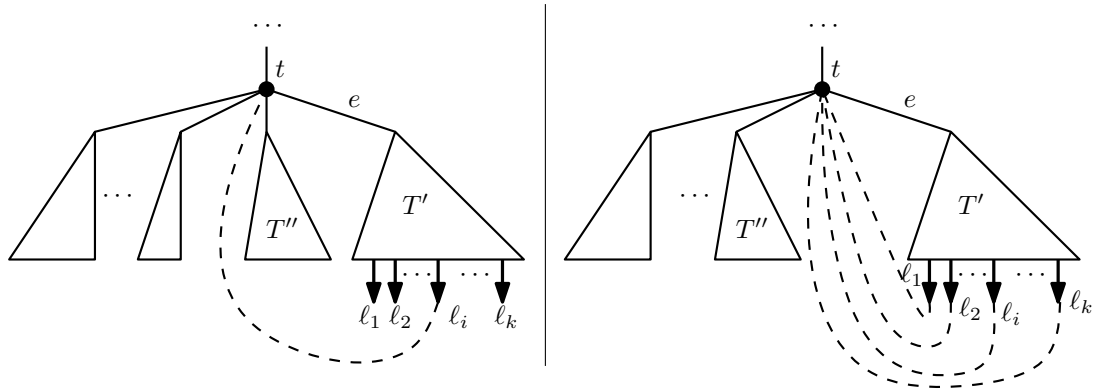

\centering
\insertfigure{ch5-fig.pdf}{10}
\caption{Illustration of the proof of Proposition~\ref{prop:5:compositionST}} \label{fig:5:ST}
\end{figure}

\begin{prop}\label{prop:5:compositionTS}
For any decorated tree $T$, we have $\mathrm{T}(\mathrm{S}(T)) = T$.
\end{prop}
\begin{proof}
Let $M = \mathrm{S}(T)$. We only need to show that the exploration tree $T'$ of $M$ is $T$ without labels. Closing each leaf one by one in the procedure $\mathrm{S}(T)$, it is clear that the exploration tree remains the same, therefore $T' = T$.
\end{proof}

\begin{thm}\label{thm:5:bijectMnTn}
The transformation $\mathrm{T}$ is a bijection from the set of non-separable planar maps $\mathcal{M}_n$ to the set of decorated trees $\mathcal{T}_n$ for any $n>0$, and $\mathrm{S}$ is its inverse.
\end{thm}
\begin{proof}
This is a consequence of Propositions~\ref{prop:5:inclusionTMnintoTn},~\ref{prop:5:inclusionSTnintoMn},~\ref{prop:5:compositionST}~and~\ref{prop:5:compositionTS}.
\end{proof}

\subsection{From trees to intervals}

We now construct a bijection from decorated trees to synchronized intervals. For a decorated tree $T$, we want to construct a synchronized interval $[\mathrm{P}(T), \mathrm{Q}(T)]$. For the upper path, we simply define $\mathrm{Q}(T)$ as the transformation from the tree $T$ to a Dyck path by taking the depth evolution in the tree traversal. The definition of $\mathrm{P}$ is more complicated. We need to define a quantity on leaves of the tree $T$ called the \mydef{charge}. The transformation $\mathrm{P}$ takes the following steps.

\begin{enumerate}
\item Every leaf has an initial charge $0$. For each internal vertex $v$ of depth $p>0$, we add $1$ to the charge of the first leaf in its descendants (in traversal order) with label at most $p-2$. We observe that the total number of charges is exactly the number of internal vertices.
\item We perform a traversal of the tree in order to construct a word in $u,d$. When we first visit an internal edge, we append $u$ to the word. When we first visit a leaf with charge $k$, we append $ud^{1+k}$ to the word. We thus obtain the word $\mathrm{P}(T)$.
\end{enumerate}

\begin{figure}
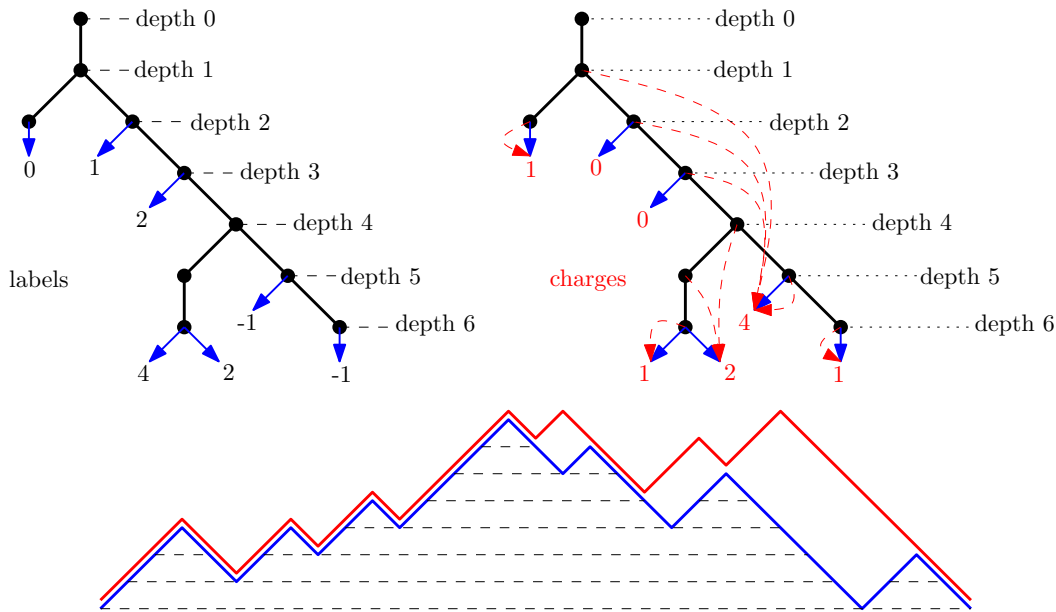

\centering
\insertfigure[0.9]{ch5-fig.pdf}{3}
\caption{An example of a decorated tree $T$, with the charges on its leaves and the corresponding interval $[\mathrm{P}(T),\mathrm{Q}(T)]$} \label{fig:5:bij-tree-interval}
\end{figure}

An example of the whole process is given in Figure~\ref{fig:5:bij-tree-interval}. We now prove that the transformations $\mathrm{P}$ and $\mathrm{Q}$ send a decorated tree to a synchronized interval.

\begin{prop} \label{prop:5:PQ-conform}
For a decorated tree $T$, the paths $\mathrm{P}(T)$ and $\mathrm{Q}(T)$ are Dyck paths, and $[\mathrm{P}(T), \mathrm{Q}(T)]$ is a synchronized interval.
\end{prop}
\begin{proof}
Since $\mathrm{Q}(T)$ is the depth evolution of the traversal of $T$, it is a well-defined Dyck path of length $2n$, where $n$ is the number of edges in $T$, which equals to the number of internal vertices of $T$. By the charging process, there are $n$ charges on leaves in total, and it is clear that $\mathrm{P}(T)$ is also of length $2n$ with $n$ up steps. We need to show that $\mathrm{P}(T)$ is positive. Consider a letter $d$ in $\mathrm{P}(T)$. The charge that gives rise to this letter $d$ comes from a non-root vertex $t$ and goes onto a descendant leaf $\ell$ of $t$. Let $e$ be the edge from $t$ to its parent. We pair up this letter $d$ to the letter $u$ in $\mathrm{P}(T)$ given by traversal of $e$. All letters in $\mathrm{P}(T)$ can be paired up in this way, and by the traversal rule, in a pair, the letter $u$ always comes before the letter $d$ since $\ell$ is a descendant of $t$. Therefore, $\mathrm{P}(T)$ is positive, thus also a Dyck path. We can also easily see that $\mathit{Type}(\mathrm{P}(T)) = \mathit{Type}(\mathrm{Q}(T))$, since in both $\mathrm{P}(T)$ and $\mathrm{Q}(T)$, a letter $u$ is followed by a letter $d$ if and only if it corresponds to a leaf.

We now need to show that $[\mathrm{P}(T), \mathrm{Q}(T)]$ is a Tamari interval. Let $u_i^{Q}$ be the $i^{\rm th}$ up step in $\mathrm{Q}(T)$, and $e$ the edge in $T$ that gives rise to $u_i^{Q}$ in the construction of $\mathrm{Q}(T)$. By the definition of $\mathrm{P}$ and $\mathrm{Q}$, it is clear that $e$ also gives rise to the $i^{\rm th}$ up step $u_i^{P}$ in $\mathrm{P}(T)$. If we can show that $D_P(i) \leq D_Q(i)$ for all $i$, then by Lemma~\ref{lem:5:ericmirlouis}, we know that $[\mathrm{P}(T), \mathrm{Q}(T)]$ is a Tamari interval.

Let $v$ be the lower endpoint of the edge $e$, $T'$ the sub-tree of $T$ rooted at $v$, $\ell$ the descendant leaf that $v$ charges and $p$ the depth of $v$. By the charging process, $\ell$ has a label at most $p-2$. Let $d_j^{P}$ be the matching down step of $u_i^{P}$ in $\mathrm{P}(T)$. We prove that $d_j^{P}$ is generated during the traversal of $\ell$, from which it follows that $D_P(i) \leq D_Q(i)$ by the definition of the distance function. Let $k \geq 1$ be the number of charges of $\ell$. Consider the segment $P(e,\ell) = u_i^{P}Wd^{k+1}$ of $P$ from $e$ to $\ell$ in the traversal for the construction of $\mathrm{P}(T)$. We first show that $|P(e,\ell)|_u \leq |P(e,\ell)|_d$, which implies that $d_j^P$ is in $P(e,\ell)$. To this end, we only need to show that we can always pair an up step in $P(e,\ell)$ with a down step also in $P(e,\ell)$. An up step is generated either by an internal edge or a leaf. For an up step $u_*$ generated by an internal edge $e'$ visited in $P(e,\ell)$, let $v'$ be its lower endpoint and $\ell'$ the leaf charged by $v'$. We know that $\ell' = \ell$ or $\ell'$ precedes $\ell$, therefore the down step $d_*$ produced by the charge added to $\ell'$ by the internal vertex $v'$ is already in $P(e,\ell)$. We pair $u_*$ with $d_*$. For an up step arisen from visiting a leaf, we can pair it with the first letter $d$ given by the leaf. Therefore, $|P(e,\ell)|_u \leq |P(e,\ell)|_d$. We now show that $d_j^{P}$ is not in $W$. Indeed, since $T$ is a decorated tree and $\ell$ has a label at most $p-2$, by the third condition of decorated trees, a descendant leaf $\ell'$ of $v$ that precedes $\ell$ has a label at least $p-1$. All charges of $\ell'$ come from internal nodes in $T'$ other than $v$ and they are all ancestors of $\ell'$, which means that they are visited before $\ell'$ in the traversal. Therefore, there are more $u$'s than $d$'s in any prefix of $W$, so $d_j^{P}$ cannot be in $W$, thus it must be among the down steps $d^{k+1}$ produced by $\ell$, which completes the proof.
\end{proof}

We now describe the inverse transformation $\mathrm{R}$, which sends a synchronized interval $[P,Q]$ to a decorated tree $T = \mathrm{R}([P,Q])$ by the following steps. A partial example is illustrated in Figure~\ref{fig:5:bij-tree-interval-leaf}. We should note that the following definition of $\mathrm{R}$ does not use the notion of charges in the definition of $\mathrm{P}$. However, we will use the notion of charges to prove that $\mathrm{R}$ is indeed the inverse transformation of $[\mathrm{P},\mathrm{Q}]$. More precisely, we will show how to read off, from a synchronized interval, the vertices that charge a given leaf on the corresponding decorated tree without any knowledge on the labels.

\begin{enumerate}
\item We construct the tree structure of $T$ from $Q$.
\item We perform the following procedure on each leaf, as illustrated in Figure~\ref{fig:5:bij-tree-interval-leaf}. Let $\ell$ be a leaf. Suppose that $\ell$ gives rise to the $i^{\rm th}$ up step in $Q$. We look at the lowest point $u$ of the consecutive down steps that come after the $i^{\rm th}$ up step in $P$, and we draw a ray from $u$ to the left until intersecting the midpoint of two consecutive up steps in $P$. Suppose that the lower up step is the $j^{\rm th}$ up step in $P$. We take $e$ the edge in $T$ that gives rise to the $j^{\rm th}$ up step in $Q$. Let $p$ be the depth of the shallower end point of $e$. We label the leaf $\ell$ with $p$. In the case that no such intersection exists, $\ell$ is labeled $-1$.
\end{enumerate}

\begin{figure}
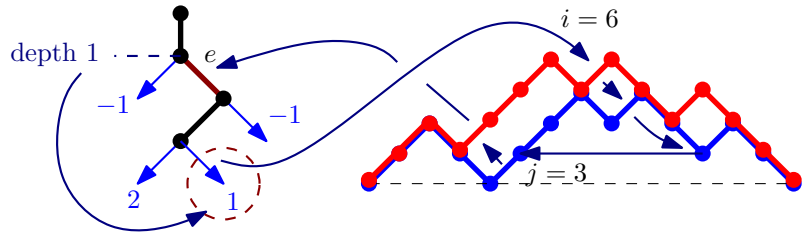

\centering
\insertfigure{ch5-fig.pdf}{36}
\caption{An example of how to recover leaf labels using the lower path $P$ (here, leaf with label $1$)} \label{fig:5:bij-tree-interval-leaf}
\end{figure}

We start by a property of Tamari intervals, which is a corollary of Lemma~\ref{lem:5:ericmirlouis}.

\begin{lem}\label{lem:5:sync-inc}
Let $[P,Q]$ be a Tamari interval. For $A \in \{P,Q\}$, we denote by $u_j^A$ the $j^{\rm th}$ up step in $A$, $d_j^A$ its matching down step, and $A^{[j]}$ the segment of $A$ between $u_j^A$ and $d_j^A$, excluding both ends. For any index $i, j$ such that the $i^{\rm th}$ up step $u_i^P$ of $P$ is in $P^{[j]}$, the $i^{\rm th}$ up step $u_i^Q$ in $Q$ is also in $Q^{[j]}$.
\end{lem}
\begin{proof}
We observe that, for any Dyck path $A$ and index $j$, the segment $A^{[j]}$ is a Dyck path of size $(D_A(j)-1)/2$. Since $A^{[j]}$ is formed by consecutive letters in $A$, the up steps in $A^{[j]}$ have consecutive indices starting from $j+1$. Therefore, $u_i^A$ is in $A^{[j]}$ if and only if $j+1 \leq i \leq j+1+(D_A(j)-1)/2$. By Lemma~\ref{lem:5:ericmirlouis}, $D_P(j) \leq D_Q(j)$. Therefore, $u_i^P$ is in $P^{[j]}$ implies that $u_i^Q$ is in $Q^{[j]}$.
\end{proof}

In the following proofs, for a tree $T$, the sub-tree induced by an edge $e$ is the sub-tree obtained by cutting $e$.

\begin{prop}\label{prop:5:R-conform}
For $[P,Q]$ a synchronized interval, the tree $T = \mathrm{R}([P,Q])$ is a decorated tree.
\end{prop}
\begin{proof}
We need to verify that $T$ satisfies the three conditions of decorated trees. We first look at the first condition for the case $P=Q$. In this case, to show that the label of a leaf $\ell$ attached to a vertex $u$ of depth $p$ is strictly smaller than $p$, we consider the next newly visited edge $e$ after $\ell$ in the tree traversal. If no such $e$ exists, $\ell$ is labeled $-1$. Let $v$ be the vertex adjacent to $e$ with smaller depth, and $v$ must be an ancestor of $\ell$. In the case $P=Q$, the label of $\ell$ is the depth of $v$ minus $1$, which must be strictly smaller than $p$. The first condition is also satisfied for any other $P$ since going down in the lattice weakly reduces the labels in $\mathrm{R}([P,Q])$. Therefore, the first condition is satisfied for all $\mathrm{R}([P,Q])$.

For the second condition, let $v$ be an internal node of depth $p$ in $T$ that is not the root and $e$ the edge from $v$ to its parent. We need to exhibit one descendant leaf of $v$ that has a label of value at most $p-2$. Suppose that $e$ corresponds to the $j^{\rm th}$ up step $u_j^Q$ in $Q$. On the path $P$, let $u_j^P$ be the $j^{\rm th}$ up step and $d_j^P$ be the matching down step. Let $d_i^P$ be the first down step of the consecutive down steps containing $d_j^P$, and $u_i^P$ be its matching up step, which is also the $i^{\rm th}$ up step of $P$. It is clear that $u_i^P$ is between $u_j^P$ and $d_j^P$ in $P$. Let $\ell$ be the leaf that gives rise to $u_i^Q$. We now show that the label of $\ell$ is at most $p-2$. Since $[P,Q]$ is a synchronized interval, thus also a Tamari interval, by Lemma~\ref{lem:5:sync-inc}, the $i^{\rm th}$ up step $u_i^Q$ of $Q$ must be between the $j^{\rm th}$ up step $u_j^Q$ and its matching down step in $Q$. Furthermore, the edge $e$ gives rise to $u_j^Q$, therefore $\ell$ must be a descendant leaf of $e$. Consider the lowest down step $d_k^P$ corresponding to $\ell$ in $P$. Let $e'$ be the edge that gives label to $\ell$. By the definition of $\mathrm{R}$, the segment of $P$ between the corresponding up step of $e'$ and $d_k^P$ contains $u_j^P$, which implies that the edge $e'$ induces a sub-tree containing $e$ (or the entire tree when the label is $-1$) according to Lemma~\ref{lem:5:sync-inc}. Let $w$ be the endpoint with a smaller depth of $e'$. The depth of $w$ is thus at most that of $v$ minus $2$, therefore the label of $\ell$ is at most $p-2$. The tree $T$ thus satisfies the second condition.

For the third condition, let $\ell$ be a leaf in $T$ with label $p$. We consider the edge $e$ that gives a label to $\ell$ in the construction of $T$. The edge $e$ links two vertices of depth $p$ and $p+1$. Therefore, its induced sub-tree $T'$ is one of the sub-trees of a vertex $v$ of depth $p$ as in the third condition. The leaf $\ell$ is in $T'$ by the construction of $T$ and Lemma~\ref{lem:5:sync-inc}. We only need to prove that there is no leaf with label strictly less than $p$ before $\ell$ in $T'$. Let $\ell_1$ be a leaf in $T'$ that comes before $\ell$. We suppose that $\ell$, $e$ and $\ell_1$ give rise to the $i^{\rm th}$, $j^{\rm th}$ and $i_1^{\rm th}$ up step in $Q$ respectively. We have $j < i_1 < i$. We now look at the corresponding up steps in $P$. By construction of the horizontal ray, the lowest point of the consecutive down steps that comes after the $i_1^{\rm th}$ up step in $P$ cannot be lower than that of the $i^{\rm th}$ step, or else the ray would be blocked. Therefore, $\ell_1$ receives a label from an edge in $T'$ or from $e$, which give it a label at least $p$. The third condition is thus satisfied, and we conclude that $T$ is a decorated tree.
\end{proof}

We now show that the transformations $[\mathrm{P}, \mathrm{Q}]$ and $\mathrm{R}$ are inverse of each other. 

\begin{prop}\label{prop:5:inverse-R}
Let $[P,Q]$ be a synchronized interval, and $T = \mathrm{R}([P,Q])$. We have $\mathrm{P}(T)=P$ and $\mathrm{Q}(T) = Q$.
\end{prop}
\begin{proof}
For $Q$ it is clear. We only need to prove the part for $P$. From Proposition~\ref{prop:5:R-conform}, we know that $[\mathrm{P}(T),\mathrm{Q}(T)]$ is a synchronized interval. Therefore, given the path $Q$, the Dyck path $\mathrm{P}(T)$ is totally determined by the charge of each leaf in $T$. We only need to show that each leaf in $T$ receives the correct amount of charge, which is one less than the length of the corresponding consecutive down steps in $P$.

We will first investigate the charging process. Let $v$ be a non-root vertex of depth $p$ in $T$, $e$ the edge linking $v$ to its parent, and $u_i^P$ the up step in $P$ that comes from $e$, which is also the $i^{\rm}$ up step of $P$. Let $\ell$ be the leaf that gives rise to the matching down step $d_i^P$ of $u_i^P$ in $P$. We now show that $v$ charges $\ell$ by showing that $\ell$ has a label of value at most $p-2$ and showing that $\ell$ has the first such label.

To show that the label of $\ell$ is at most $p-2$, we consider the labeling process on $\ell$. We draw a ray to the left from the lowest point of the down steps of $\ell$, which hits a double up step. Suppose that the lower one of the double up step is the $j^{\rm th}$ up step $u_j^P$ of $P$. Let $e'$ be the edge giving rise to $u_j^P$, and $v'$ the endpoint of $e'$ with a smaller depth. It is clear that $i \neq j$ and $u_i^P$ is in the segment of $P$ between $u_j^P$ and its matching down step. Therefore, by Lemma~\ref{lem:5:sync-inc}, $e$ must also be in the sub-tree induced by $e'$. Therefore, $v'$ is of depth at most $p-2$, thus $\ell$ has a label at most $p-2$. 

To show that $\ell$ is the first descendant leaf of $v$ with a label at most $p-2$, we consider a descendant leaf $\ell'$ of $v$ that comes before $\ell$ in the traversal order. Let $d_k^P$ be the last down steps in $P$ that comes from $\ell'$. Since $\ell'$ comes before $\ell$, $d_k^P$ is strictly between $u_i^P$ and $d_i^P$, and the horizontal ray from the lower point of $d_k^P$ lays strictly above that from $u_i^P$ to $d_i^P$. By the labeling process, the double up steps that corresponds to $d_k^P$ (thus to $\ell'$) is in the segment from $u_i^P$ to $d_i^P$ (both ends included), therefore the label of $\ell'$ is at least $p-1$. We conclude that $\ell$ is the first leaf in the sub-tree of $v$ that has a label at most $p-2$, and thus $v$ charges $\ell$.

To count the number of charges of $\ell$, we notice that each down step in $P$ that comes from $\ell$ corresponds to a charge, except for the highest one. To see this, we only need to consider their matching up steps. It is clear that highest down step of $\ell$ in $P$ is matched with the only up step in $P$ that comes from $\ell$. For a down step of $\ell$ that is not the highest, it is impossible that its matching up step is immediately followed by a down step, therefore the matching up step is the lower part of a double up step, corresponding to an internal vertex in $T$, and we can see from the argument above that this internal vertex charges $\ell$. We thus conclude that $\ell$ receives the correct number of charges in the construction of $\mathrm{P}(T)$, which implies $\mathrm{P}(T)=P$.
\end{proof}

We now show that the transformation $[\mathrm{P},\mathrm{Q}]$ is an injection.

\begin{prop}\label{prop:5:PQ-injective}
Let $T_1, T_2$ be two decorated trees. If $\mathrm{P}(T_1) = \mathrm{P}(T_2)$ and $\mathrm{Q}(T_1) = \mathrm{Q}(T_2)$, then $T_1 = T_2$.
\end{prop}
\begin{proof}
Suppose that $T_1 \neq T_2$. Since $\mathrm{Q}(T_1) = \mathrm{Q}(T_2)$, the only difference between $T_1$ and $T_2$ must be on labels. Let $\ell$ be the first leaf in the traversal order that $T_1$ and $T_2$ differ in label. We suppose that $\ell$ has a label $k_1$ in $T_1$ and label $k_2$ in $T_2$, with $k_1 > k_2 \geq -1$. We have $k_1 \geq 0$. It is clear that all nodes charging $\ell$ in $T_1$ also charge $\ell$ in $T_2$. Since $\mathrm{P}(T_1) = \mathrm{P}(T_2)$, we know that $\ell$ receives the same number of charges in $T_1$ and $T_2$, thus $\ell$ is also charged by the same set of vertices in $T_1$ and $T_2$. Let $u$ be the ancestor of $\ell$ of depth $k_1 + 1$ in $T_1$ and $T_2$. The vertex $u$ has a parent since $k_1 \geq 0$. The existence of $u$ is guaranteed by the first condition of decorated trees. In $T_1$, by construction, $u$ does not charge $\ell$, therefore $u$ should not charge $\ell$ in $T_2$ either. Therefore, in $T_2$ there must be a descendant leaf $\ell'$ of $u$ that has a label at most $k_2$ and comes before $\ell$, to prevent $u$ from charging $\ell$. By the minimality of $\ell$, we know that $\ell'$ also has a label $k_2 < k_1$ in $T_1$, violating the third condition of decorated trees on the parent of $u$, which is impossible. Therefore, $T_1 = T_2$.
\end{proof}

We now prove that $\mathrm{R}$ is a bijection between decorated trees and synchronized intervals.

\begin{thm}
The transformation $\mathrm{R}$ is a bijection from $\mathcal{I}_n$ to $\mathcal{T}_n$ for all $n \geq 1$, with $[\mathrm{P}, \mathrm{Q}]$ its inverse.
\end{thm}
\begin{proof}
It is clear that $\mathrm{R}$ preserves the size $n$. By Proposition~\ref{prop:5:PQ-conform} and Proposition~\ref{prop:5:R-conform}, we have $[\mathrm{P},\mathrm{Q}](\mathcal{T}_n) \subset \mathcal{I}_n$ and $\mathrm{R}(\mathcal{I}_n) \subset \mathcal{T}_n$. And by Proposition~\ref{prop:5:inverse-R} and Proposition \ref{prop:5:PQ-injective}, both transformations are injective, therefore they are bijections between $\mathcal{I}_n$ and $\mathcal{T}_n$, and they are the inverse of each other.
\end{proof}

By composing the two bijections $\mathrm{T}$ and $[\mathrm{P},\mathrm{Q}]$, we obtain a natural bijection from non-separable planar maps with $n+1$ edges to canopy intervals of size $n-1$ via decorated trees with $n$ edges, therefore these two kinds of objects are enumerated by the same formula. By Tutte's enumeration result on non-separable planar maps in \cite{Tutte:census}, we obtain Theorem~\ref{thm:5:enum}.

\section{Discussion} 

Other than the enumeration of intervals in generalized Tamari lattices, our bijection has further structural implications on generalized Tamari lattices and non-separable maps. For instance, it is not difficult to see that some statistics on synchronized intervals are transferred to statistics on non-separable planar maps by our bijection, which leads to refined enumeration results. For instance, we can prove that the number of synchronized intervals of size $n$ where the last descent of its upper path has length $k$ is the same as that of non-separable planar maps of $n+1$ edges with outer face degree $k+1$. Furthermore, we can relate these structures to some other combinatorial structures that have the same enumeration formula. In the following, we will give an informal summary of our current work in this direction, without giving technical details.

Our bijection between non-separable planar maps and synchronized intervals may seem complicated and unnatural at first sight, and the proofs are indeed technical at places. But in fact, these bijections come naturally from recursive decompositions of both kinds of objects. There is a version of parallel decomposition of non-separable planar maps (see the right-most part of Figure~\ref{fig:5:map-decomp}) that just removes the innermost non-separable component. This version of parallel decomposition is isomorphic to the one underlying Proposition~\ref{prop:5:rec-construct} for synchronized intervals. Thus, there exists a ``canonical'', recursively defined bijection between $\mathcal{M}_n$ and $\mathcal{I}_n$. We prove that our bijection coincides with the canonical bijection, which allows us to tap into the recursive structures of these objects using a non-recursive bijection. 

Using the fact that our bijection is canonical with respect to some recursive decompositions, we were able to unveil an unexpected relation between map duality and generalized Tamari lattices. Recall that Pr\'eville-Ratelle and Viennot discovered in \cite{PRV2014extension} an isomorphism between the generalized Tamari lattice \tam{v} and the dual of \tam{\overleftarrow{v}} (see Theorem~\ref{thm:5:tamari-symmetry}). This isomorphism induces an involution on synchronized intervals. In the lens of our bijection, we obtain the following theorem.
\begin{thm}
The involution on synchronized intervals induced by the isomorphisms between \tam{v} and \tam{\overleftarrow{v}} for all possible $v$ is exactly the duality of non-separable planar maps under conjugation of our bijection $\mathrm{S} \circ \mathrm{R}$.
\end{thm}
The proof of this theorem makes heavy use of recursive decompositions, and relies on the fact that our bijection is canonical with respect to these decompositions. 

We have also tried to find bijections from synchronized intervals and non-separable planar maps to other combinatorial objects enumerated by the same formula as in Theorem~\ref{thm:5:enum}. We have looked at a class of labeled trees called $\beta(1,0)$-trees, which are closely related to non-separable planar maps and were used in \cite{desc-tree, JS1998bijective} under the name ``description trees'' for bijective enumeration of non-separable planar maps. It turns out that $\beta(1,0)$-trees are also related to several classes of permutations with forbidden pattern (see \cite[Chapter~2]{Kitaev2011} for a survey on these relations). In \cite{kitaev-involution}, an involution $h$ on $\beta(1,0)$-trees was introduced by Claesson, Kitaev and Steingr{\'{\i}}msson. Then it was proved in \cite{kitaev-nsp} by Kitaev and de Mier using generating function method that the number of fixed points of the involution $h$ on the set of $\beta(1,0)$-trees with $2n$ nodes is exactly the same as that of self-dual (\textit{i.e.} isomorphic to their own duals) non-separable planar maps with $2n$ edges. They asked for a bijective explanation. By relating decorated trees to $\beta(1,0)$-trees using a bijection, we are able to prove that the involution $h$ on $\beta(1,0)$-trees is the same as map duality on non-separable planar maps under our bijection, which answers positively the problem proposed by Kitaev and de Mier. The proof again uses the fact that our bijections are canonical with respect to some recursive decompositions of these objects.

Other than generalizations to other objects, it is also interesting to see what our bijection gives when restricted to intervals in generalized Tamari lattices with special canopies. We know that \tam{(NE)^n} is isomorphic to the usual Tamari lattice, and in \cite{BB2009intervals}, Bernardi and Bonichon gave a bijection between Tamari intervals and planar rooted triangulations. It is thus natural to look for a similar bijection as a specialization of our bijection, and eventually a generalization to \tam{(NE^m)^n}, which is isomorphic to the $m$-Tamari lattice. Such a bijection from intervals of the $m$-Tamari lattice to a natural class of planar maps will explain why the enumeration formula of these intervals is similar to those of planar maps. Furthermore, using trees with blossoms to encode planar maps bijectively is a common practice, but we rarely see a DFS tree in this interplay. Triangulations may be a step towards the extension of this new approach.

As a final remark, decorated trees do not seem quite suitable for direct enumeration, due to their complicated definition. However, the exploration process of non-separable planar maps to obtain decorated trees can be applied to other classes of maps, which may give rise to some variants of decorated trees that can be used for map enumeration.

\chapter{Counting graphs with maps}

Maps have their root in graph theory, but since gained independence to become an interesting subject on their own ever since. By definition, maps can be seen as graphs embedded into a certain surface \Sg, with an extra structure that indicates how edges are embedded around a vertex, which gives maps a connection to factorizations in the symmetric group via rotation systems. Indeed, in previous chapters we have seen exact and asymptotic enumeration formulas for various kinds of maps, including planar maps, planar constellations, bipartite maps of higher genus and non-separable planar maps. These formulas are rather nice in general, and sometimes come with beautiful bijections. In contrast, when the information about how exactly edges are embedded are removed, what we have are graphs embeddable into a certain surface \Sg. The enumeration of these graphs is more difficult in general, and exact formulas appear only in very special cases, probably due to the lack of structure of general graphs. However, this lack of exact formulas is not considered as a problem in general, since nowadays graph enumeration is mostly motivated by the study of large random graphs, which only needs asymptotic results on graph enumeration. Comparing the status of maps and graphs in the enumeration aspect, the following idea becomes natural: is it possible to use already established map enumeration results to reach graph enumeration results? This is indeed possible, thanks to a theorem of Robertson and Vitray. There is also some previous effort, such as \cite{graphfive, bender2011asymptotic} on this line of thought. In this chapter, we will see an example of how map enumeration can help us to count graphs.

This chapter is based on a submitted paper \cite{fang-graz} in collaboration with Mihyun Kang, Michael Mo{\ss}hammer and Philipp Spr{\" u}ssel. An extended abstract \cite{fang-cubicgraph} was published in Proceedings of European Conference on Combinatorics, Graph Theory and Applications 2015 (Eurocomb 2015). In this chapter, we will enumerate asymptotically cubic graphs strongly embeddable into the surface \Sg{} of given genus $g$ via the asymptotic enumeration results on several kinds of triangulation of higher genera. 

\section{From random graphs to cubic graphs} \label{sec:6:intro}

We will start by a reminder of the type of graphs we study. In this chapter, we will allow graphs to have loops and/or multiple edges. In other words, we will consider ``finite multigraphs'' as mentioned at the beginning of Chapter~1. A graph is called \mydef{simple} if it has no loop nor multiple edges.

In some sense, the study of \emph{random graphs}, \textit{i.e.} natural probability distributions on graphs, is the study of the generic behavior of a graph obtained from a given random procedure of graph construction, sometimes subjected to possible acceptance criteria. In general, many random graph models focus on vertex-labeled graphs. The first introduction of random graphs was due to Erd{\H{o}}s and R\'enyi \cite{erdos-renyi}, and also independently to Gilbert \cite{gilbert}. There are two \mydef{Erd\H{o}s-R\'enyi random graph models}, denoted by $G(n,p)$ and $G(n,M)$. In the $G(n,p)$ model, we consider a graph with $n$ vertices labeled from $1$ to $n$, and for each pair $\{i,j\}$ of vertices, we include the edge linking $i$ and $j$ with probability $p$. In the $G(n,M)$ model, we also consider a graph with $n$ labeled vertices, but for edges, we uniformly select a set of $M$ edges among all $\binom{n}{2}$ possibilities. There are also many other random graph models, some inspired by phenomena in the nature and in human societies (\textit{cf.} \cite{barabasi, bollobas,fan-lu}).

Random graphs exhibit many interesting phenomena. The most interesting ones are various \emph{phase transitions}. A \mydef{phase transition} of a certain property $P$ in a probability model with a parameter is the phenomenon that, when the size of the model tends to infinity, the probability that the model has property $P$ varies drastically (often from $0$ to $1$) when the parameter varies around some \mydef{critical value}. For a graph property $P$, we say that $P$ is \mydef{asymptotically almost surely} (or simply \mydef{a.a.s.}) satisfied in a random graph model if the probability that the random graph has the property $P$ tends to $1$ when the size of the graph tends to infinity. In \cite{erdos-renyi}, Erd\H{o}s and R\'enyi mentioned the first such phase transition result by Erd\H{o}s and Whitney, which is about the connectivity of the $G(n,M)$ model: the phase transition of the property that $G(n,M)$ is connected occurs at $M^* = \frac1{2}n\log n$, \textit{i.e.} for any $\epsilon >0$, the random graph $G(n,(1-\epsilon)M^*)$ is a.a.s. disconnected, and $G(n,(1+\epsilon)M^*)$ is a.a.s. connected. More precise behavior of this phase transition was further given in \cite{erdos-renyi}. Later, it was found that many interesting and important graph-theoretic properties and statistics, such as existence of Hamiltonian cycles and diameter, also have phase transitions in many random graph models. Readers can find more information about random graphs in \cite{bollobas} and in references therein.


Later, we will be interested in the model $G(n,M)$ with extra condition on embeddability. Since the model $G(n,M)$ is a model of random \textbf{vertex-labeled} graph, in the following we will also enumerate \textbf{vertex-labeled} cubic graphs strongly embeddable into \Sg{} for a given genus $g$. There is another classical phase transition on random graphs called the \emph{emergence of the giant component}, first proved in \cite{evol-er} for $G(n,M)$. In the regime $M = o(n\log n)$, the random graph is almost surely disconnected, and we want to study its connected components. We say that a random graph model $G$ with $n$ vertices has a \mydef{giant component} if there is a connected component $C$ in $G$ containing at least $pn$ vertices with some fixed constant $p>0$. The emergence of the giant component for $G(n,M)$ occurs at $M=n/2+O(n^{2/3})$. For a more precise description, we take $M=n/2+s$ for some $s=O(n)$. If $s n^{-2/3} \to -\infty$, then $G(n,M)$ a.a.s. has no giant component; if $s n^{-2/3} \to +\infty$, then $G(n,M)$ a.a.s. has one unique component with $\Omega(s)$ edges, which is the largest, and when $s$ is positive and of order $n$, there is a giant component. A detailed description of the evolution of $G(n,M)$ around this critical value $M=n/2$ can be found in \cite{giant-component}. 

The study of random graphs often involves asymptotic enumerations of various classes of graphs, which can be used to determine whether a subclass of graphs is predominant in a random graph model with a given parameter. Therefore, we are also interested in asymptotic enumeration of graphs. In recent years, researchers have been interested in graphs \emph{embeddable into a fixed surface}, both in the context of random graphs and of asymptotic enumeration. A \emph{connected} graph $G$ is said to be \mydef{embeddable into \Sg{}} (resp. \mydef{strongly embeddable into \Sg{}}) if there is a map $M$ of genus $g' \leq g$ (resp. exactly $g$) whose underlying graph is $G$. We notice that a connected graph can be strongly embeddable on surfaces with different genera. Figure~\ref{fig:6:embeddings} shows such an example. A \emph{general} graph $G$ is said to be embeddable (resp. strongly embeddable) into \Sg{} if all its connected components $G_1, G_2, \ldots, G_k$ are embeddable (resp. strongly embeddable) into surfaces $\mathbb{S}_{g_1}, \mathbb{S}_{g_2}, \ldots, \mathbb{S}_{g_k}$ respectively, and the sum of all genera $\sum_{i=1}^k g_i$ is $g$. A graph $G$ is called \mydef{planar} if it is embeddable into the sphere $\mathbb{S}_0$, which means that all its connected components are planar. We can also define the \mydef{genus} of a graph $G$ to be the minimal genus $g$ such that $G$ is embeddable into \Sg{}, while the \mydef{maximal genus} of $G$ is the maximal genus that satisfies the same condition.

\begin{figure}
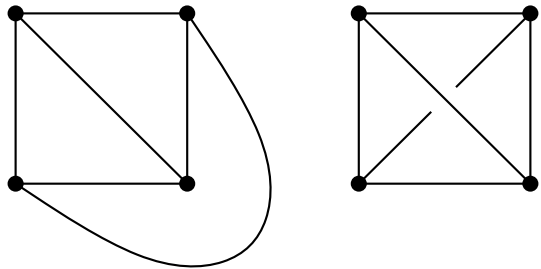

  \centering
  \insertfigure{ch6-fig.pdf}{2}
  \caption{Two maps of the complete graph $K_4$, one on the sphere, the other on the torus}
  \label{fig:6:embeddings}
\end{figure}


In \cite{planar-law}, Gim\'enez and Noy first gave the asymptotic number of vertex-labeled simple planar graphs and a more precise result on the asymptotic law of the number of edges among these simple planar graphs. For $1 < \mu < 3$, they also gave a formula for the asymptotic number of simple planar graphs with $n$ vertices and $\mu n$ edges. Similar results for simple graphs embeddable into \Sg{} for a fixed genus $g$ were obtained in \cite{graphfive} by Chapuy, Fusy, Gim\'enez, Mohar and Noy. The asymptotic result without control on the number of edges was also obtained independently in \cite{bender2011asymptotic}. It is worth noting that, in the asymptotic results above, the genus only has an effect on the \emph{sub-exponential growth}. More precisely, the number $a_n^{(g)}$ of simple graphs embeddable into \Sg{} for a fixed genus $g$ has the asymptotic form
\[ a_n^{(g)} \sim c^{(g)}n^{5(g-1)/2-1}\gamma^n n!. \]
We thus observe that the genus $g$ only influences the sub-exponential part $c^{(g)} n^{5(g-1)/2-1}$ of the asymptotic formula, but not the more prominent exponential growth $\gamma^n$ (the part $n!$ accounts for vertex-labeling). The same phenomenon also occurs when we control the number of edges.

On the random graph side, the previous work mostly concentrates on planarity. We know from \cite{evol-er} that, in the sub-critical regime $\mu < 1/2$, the random graph $G(n,\mu n)$ contains a.a.s. no component with more than one cycle, which implies that $G(n,\mu n)$ is almost surely planar in this regime. On the other hand, we know from \cite{giant-component, planar-prob} that, in the super-critical regime $\mu > 1/2$, the random graph $G(n,\mu n)$ almost surely has the non-planar graph $K_{3,3}$ (the complete bipartite graph with 3 vertices on both side) as a topological minor, therefore is not planar. We thus know that the phase transition of planarity on $G(n,M)$ occurs at the same place $\mu=1/2$ as the emergence of the giant component. The fine variation of the probability of $G(n,M)$ being planar was also studied by {\L}uczak, Pittel and Wierman in \cite{planar-prob}. It was shown there that, for $M=n/2+cn^{2/3}$ with a constant $c$, the probability that $G(n,M)$ is planar converges to a value $p(c)$. The limit probability $p(c)$ was later determined by Noy, Ravelomanana and Ru\'e in \cite{planar-prob-det} using an exact enumeration of planar cubic graphs. Here a \mydef{cubic graph} is a graph whose vertices are all of degree $3$. Later we will explain how to use enumeration results on cubic graphs to obtain results in random graphs with embeddability constraints.

We can also define random graph models directly on graphs embeddable into a certain surface \Sg{}. In \cite{planar-phase}, Kang and {\L}uczak studied the random graph model $P(n,M)$ of planar graphs, which is simply $G(n,M)$ conditioned on the event that the graph obtained is planar. That is to say, the model $P(n,M)$ is a uniform distribution over all vertex-labeled planar graphs with $n$ vertices and $M$ edges. They found that the emergence of the giant component in $P(n,M)$ also occurs at $M=n/2+O(n^{2/3})$, the same place and window as $G(n,M)$. But the planarity in $P(n,M)$ brings itself another phase transition at $M=n+O(n^{3/5})$, where the growth of the giant component changes. Readers are referred to \cite{planar-phase} of more details on this extra phase transition.

It is now natural to try to obtain results of similar nature for random graphs embeddable into surfaces of higher genus. For instance, the method used in \cite{planar-prob-det} to determine the limit probability $p(c)$ of $G(n,M)$ being in the class of planar graphs for $M=n/2+cn^{2/3}$ can be adapted in principle to a large variety of graph classes, including graphs embeddable into \Sg{} for any fixed $g$. The only obstacle for this generalization to higher genera is that we don't have an exact enumeration of cubic graphs embeddable into surfaces of higher genera. We can also extend the study of phase transitions of $P(n,M)$ in \cite{planar-phase} to the random graph model $P_g(n,M)$ with a fixed genus $g$, which is $G(n,M)$ conditioned on the event of the graph being embeddable into \Sg{}. However, this line of thought also needs the enumeration of cubic graphs embeddable into \Sg{} for a fixed $g$, although an asymptotic one will suffice in this case. But how cubic graphs enter in all these studies of graphs embeddable into surfaces?

In fact, results in \cite{planar-prob, planar-prob-det, planar-phase} all rely on a common set of notions in graph theory to simplify the study of embeddability. For a graph $G$, its \mydef{core} $\core(G)$ is obtained by removing all vertices of degree $1$ repeatedly until there is no such vertices. The removal process needs to be repeated because the removal of a vertex of degree $1$ may reduce the degree of other vertices to $1$. The \mydef{kernel} $\kernel(G)$ of a graph $G$ is obtained by ``smoothing out'' all vertices of degree $2$ in the core of $G$. More precisely, for a vertex $u$ of degree $2$, let $v,w$ be its adjacent vertices (here $v$ is not necessarily different from $w$), we ``smooth out'' $u$ by deleting $u$ and its adjacent edges, then adding a new edge $\{v, w\}$ to the graph. Figure~\ref{fig:6:kernel} shows the whole procedure to obtain the kernel of a graph. Conversely, all graphs $G$ with a given kernel $G' = \kernel(G)$ can be obtained by first inserting a sequence of vertices of degree $2$ into each edge of $G'$, in the direction from the vertex with smaller label to the other, then attaching a tree graph (\textit{i.e.} acyclic graph) to each vertex. We notice that, when going back from the kernel to the core of a graph adding back vertices of degree $2$, there may be some symmetry issues to be dealt with. The reason is that, in the scenario of graphs, vertices are \emph{labeled} but edges are \emph{not ordered} around a vertex. In this case, different sequences of vertex-adding on multiple edges or loops may result in the same graph. More precisely, if we add back a sequence of vertices with labels $\ell_1, \ell_2, \ldots, \ell_k$ to a loop in the kernel, we will obtain the same core as adding back the sequence with labels $\ell_k, \ell_{k-1}, \ldots, \ell_1$. Similarly, for multiple edges in the kernel, we can exchange the sequence of vertices we will add back while always obtaining the same core. Therefore, we need a \emph{compensation factor} for multiple edges and loops. For cubic graphs, we can only have double edges or triple edges as multiple edges. Each set of double edges has a compensation factor $1/2$, while a set of triple edges has a compensation factor $1/6$. These compensation factors accounts for possible permutations of sequences of degree $2$ vertices we will add back to obtain the core. On the other hand, a loop has a compensation factor $1/2$, accounting for the fact that we can reverse the sequence of degree $2$ vertices to add back while obtaining the same core. Therefore, we are in fact interested in \emph{weighted} cubic graphs. But these factors are easy to deal with, either by a change of variables of by adding some coefficients in a recursive decomposition. Later, we will simply talk about the enumeration of cubic graphs, without stating the fact that they are weighted. Details will be given in the proofs.

\begin{figure}
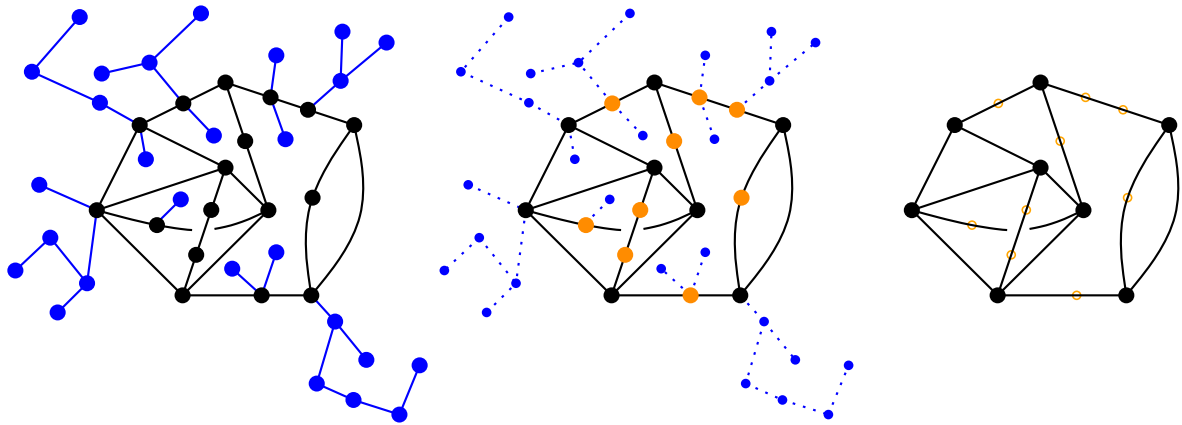

  \centering
  \insertfigure{ch6-fig.pdf}{1}
  \caption{An example of a graph, its core and its kernel}
  \label{fig:6:kernel}
\end{figure}

We observe that, if $G$ is embeddable into \Sg{}, then $\kernel(G)$ is also embeddable into \Sg{}, and \textit{vice versa}. Therefore, to study graphs embeddable into \Sg{}, it suffices to study all possible kernels embeddable into the same surface. It is immediate from the definition that all vertices in $\kernel(G)$ (if any) have degree at least $3$. When the graph $G$ has very few edges, we may expect that its kernel $\kernel(G)$ would not have any vertex of high degree. For graphs with $n$ vertices and $M=n/2+O(n^{2/3})$ edges, which is also the critical window of the emergence of the giant component in $G(n,M)$, the intuition above can be quantified by the following result from \cite[Theorem~4(I)]{planar-prob}.
\begin{lem}
For $M=n/2+O(n^{2/3})$, the random graph $G(n,M)$ a.a.s has a kernel that is a cubic graph.
\end{lem}
Therefore, in this critical window, by studying cubic graphs that are embeddable into \Sg{}, we can study the embeddability of random graphs. Furthermore, when kernels do not deviate too much from being cubic, we can still bound their number using cubic graphs (\textit{cf.} \cite[Lemma~3]{planar-phase}). This technique was used in \cite{planar-phase} to study the random planar graph model $P(n,M)$ for $M=n+O(n^{3/5})$, where the new phase transition occurs, and it can potentially be extended to the random graph model $P_g(n,M)$ for graphs embeddable into \Sg{}. We can thus say that the study of random graphs with embeddability restriction with few edges relies on the enumeration of cubic graphs embeddable into \Sg{} with a given genus $g$. This is our motivation to study cubic graphs embeddable on \Sg{}.

\section{Using triangulations to count cubic graphs}

But how do we count cubic graphs embeddable on \Sg{} for a given genus $g$? This is a tricky problem. Although we know how to count various kinds of triangulations (\textit{c.f.} \cite{Gao1992-2-connected-surface, Gao1993-pattern, Gao}), which are duals of maps with cubic graphs as underlying graphs, we cannot directly transfer these enumeration results to cubic graphs. The reason is that a cubic graph can correspond to a very large number of embeddings. For instance, we can construct gadgets (\textit{i.e.} fragments of graphs) that have two different local embeddings, such as the one on the left side of Figure~\ref{fig:6:gadget}. Then using these gadgets we can construct graphs corresponding to a large number (at least exponential in the number of vertices) of unrooted maps, because the embedding choice of each gadget is independent. The right side of Figure~\ref{fig:6:gadget} shows an example of an unrooted map with such a graph as its underlying graph. Therefore, if we want to transfer map enumeration results to graphs embeddable into a surface with fixed genus, we need to find a sub-class of these graphs within which we can somehow control the number of embeddings of graphs.

\begin{figure}
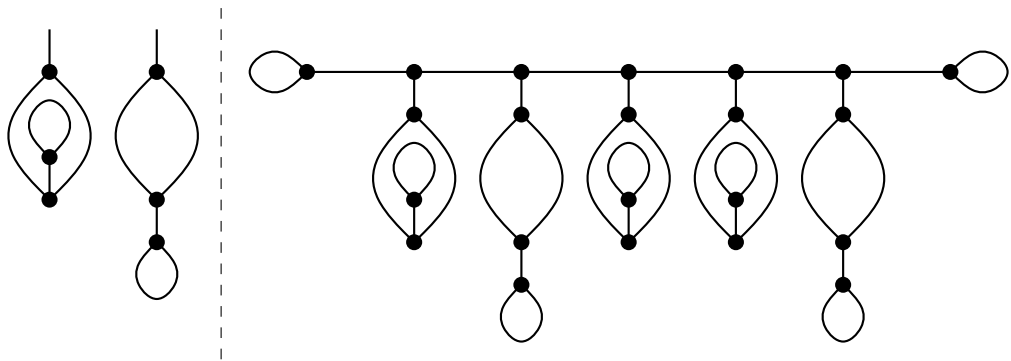

  \centering
  \insertfigure{ch6-fig.pdf}{3}
  \caption{A gadget with 2 local embeddings, and an unrooted map with its underlying graph constructed with this gadget}
  \label{fig:6:gadget}
\end{figure}

\begin{figure}
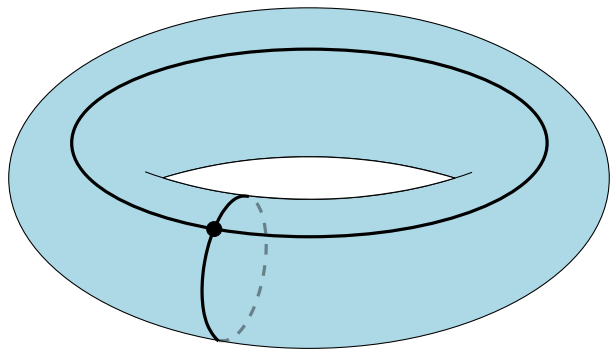

  \centering
  \insertfigure{ch6-fig.pdf}{4}
  \caption{Non-contractible circles on the torus}
  \label{fig:6:torus-non-contractible}
\end{figure}

From now on, we will consider a map alongside with the surface into which it embeds, because in the following we will perform topological surgeries on surfaces to obtain new maps. This viewpoint is also convenient for us to introduce notions in topological graph theory that are necessary for our results. Let $M$ be an unrooted map on \Sg{} with underlying graph $G$. We identify vertices, edges and faces of $M$ respectively with the corresponding points, curves and regions on the surface \Sg{}. For edges, we include their adjacent vertices on the corresponding curve to make them a closed set. The same applies to faces, that is, we include adjacent vertices and edges of a face in its corresponding region to make it closed. A \mydef{non-contractible circle} on \Sg{} is a closed curve on \Sg{} that cannot be contracted to a point. The \mydef{facewidth} of $M$, denoted by $\fw(M)$, is the minimal number of faces whose union contains a non-contractible circle. Figure~\ref{fig:6:facewidth} shows a portion of a map on a torus, and we can see that the map has facewidth at most $3$, since the union of the $3$ faces on the left contains a non-contractible circle. Similarly, we define the \mydef{edgewidth} of $M$, denoted by $\ew(M)$, as the minimal number of edges whose union contains a non-contractible circle, which is equivalent to the minimal length of cycles in $M$ that are non-contractible. Always taking the example in Figure~\ref{fig:6:facewidth}, the map has edgewidth at most $4$, since on the right there is a non-contractible cycle of edges of length 4. When $M$ is planar, we take the convention that $\fw(M)=\ew(M)=\infty$. We can easily extend the definitions of facewidth and edgewidth to rooted maps. For a graph $G$ that is embeddable into \Sg{}, we define its \mydef{facewidth on \Sg{}} $\fw_g(G)$ to be the maximal facewidth of all maps of genus $g$ with $G$ as their underlying graph.


\begin{figure}
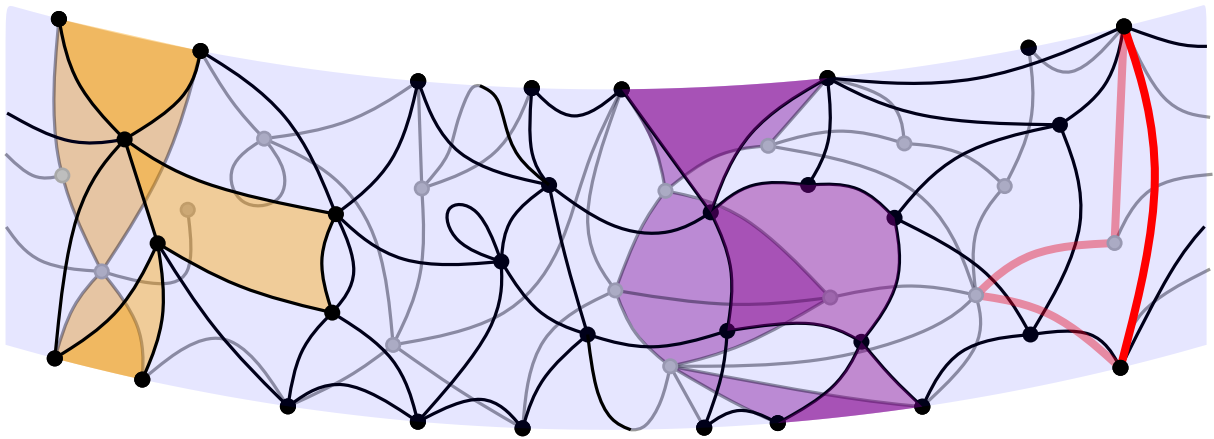

  \centering
  \insertfigure{ch6-fig.pdf}{6}
  \caption{Part of a map on a torus, containing two sets of faces whose unions contain non-contractible circles, and a non-contractible cycle}
  \label{fig:6:facewidth}
\end{figure}

We now follow the approach in \cite{graphfive, bender2011asymptotic}, which relies on the notion of facewidth and a series of theorems in topological graph theory. Our starting point is the following theorem by Robertson and Vitray \cite{robertson-vitray} which proposes a class of graphs with a unique cellular embedding (up to orientation).

\begin{thm} \label{thm:6:robertson-vitray}
For $g>0$ and $G$ a 3-connected graph embeddable into \Sg{}, if $\fw_g(G) \geq 2g+3$, then there is a unique unrooted vertex-labeled map $M$ (up to orientation) as embedding of $G$ on \Sg{}. Equivalently, in this case if there are two unrooted maps $M_1, M_2$ on \Sg{} both with $G$ as underlying graph, then there is a (not necessarily orientation-preserving) homeomorphism that sends $M_1$ to $M_2$.
\end{thm}

Many classes of maps have large facewidth. For instance, it was proved in \cite{log-facewidth} that when $n$ tends to infinity, almost all rooted maps of size $n$ have facewidth at least $\log(n)$, which is larger than any constant. Therefore, for asymptotic enumeration of maps, we can often safely drop the facewidth constraint. It is thus reasonable to expect that the same holds for graphs, that is, for any fixed constant $c$, when the size tends to infinity, almost all graphs will also have facewidth larger than $c$. Indeed, this fact will be proved later for the class of graphs that we study. Therefore, with some extra technical analysis, we can transfer asymptotic enumeration results of 3-connected maps to 3-connected graphs. By Theorem~\ref{thm:6:robertson-vitray}, we can thus obtain asymptotic enumeration results of the ``hard'' case of graphs from the ``easy'' case of maps, at least in the 3-connected case.

On the graph side, once the asymptotic number of 3-connected graphs is known, we can obtain the asymptotic enumeration of connected graphs and general graphs by a \emph{decomposition along connectivity}. The definitions of notions related to connectivity can be found in Chapter~2.1. It is folklore that a general graph can be seen as a set of connected graphs. For a connected graph $G$, we can split all its cut vertices to obtain a sequence of sub-graphs that are all 2-connected. This decomposition of $G$ into 2-connected graphs is unique. The sub-graphs obtained in the decomposition are called the \mydef{2-connected components} of $G$, and they form a bipartite tree structure if we regard each cut vertex and each component as vertex, and their incidences as edges. Figure~\ref{fig:6:conn-decomp} shows an example of this decomposition. Similarly, we can decompose a 2-connected graph by its pairs of separating vertices into a tree structure, with three types of components: cycles, multiple edges and 3-connected graphs with at least $4$ vertices (also called \emph{3-connected components}). Figure~\ref{fig:6:2-conn-decomp} shows an example of this decomposition. These decompositions are due to Tutte \cite{tutte-conn}, and \cite{conn-survey} is a survey on their applications to graph enumeration. 

\begin{figure}
  \centering
  \insertfigure[0.7]{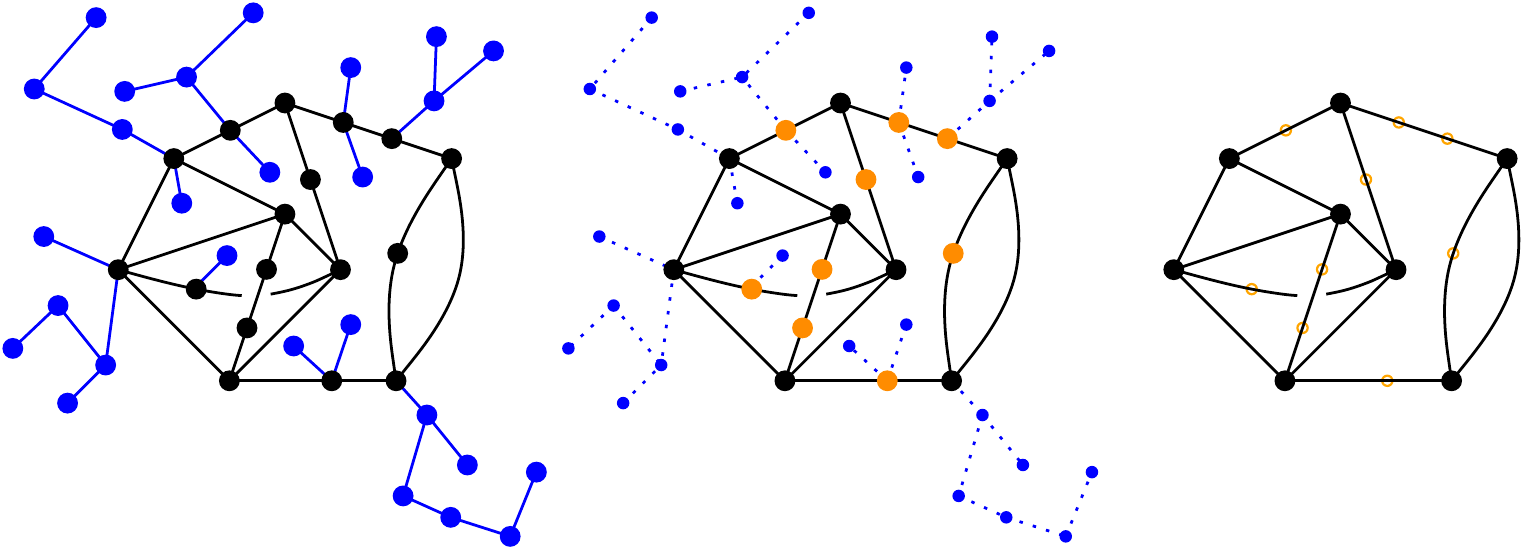}{7}
  \caption{Decomposition of a connected graph into 2-connected components}
  \label{fig:6:conn-decomp}
\end{figure}

\begin{figure}
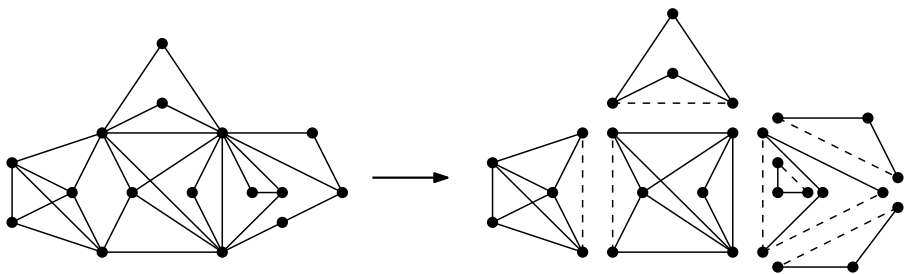

  \centering
  \insertfigure[0.7]{ch6-fig.pdf}{16}
  \caption{Decomposition of a 2-connected graph into components}
  \label{fig:6:2-conn-decomp}
\end{figure}

Again, if we add a facewidth restriction, we can have a better control on these decompositions. The following result was proved in \cite{robertson-vitray}. 

\begin{thm} \label{thm:6:conn-decomp}
Let $g>0$ and let $M$ be a map on \Sg{} with underlying graph $G$.
\begin{enumerate}
\item If $G$ is 2-connected, then $M$ has facewidth $\fw(M)=k \geq 3$ if and only if $M$ has a unique 3-connected component embedded on \Sg{} with facewidth $k$ and all other 3-connected components of $M$ are planar.
\item We have $\fw(M) = k \geq 2$ if and only if $M$ has a unique 2-connected component embedded on \Sg{} with facewidth $k$ and all other 2-connected components of $M$ are planar.
\end{enumerate}
\end{thm}

We can also formulate the theorem above in terms of graphs as in the following corollary.

\begin{coro}\label{coro:6:conn-decomp}
Let $g>0$ and let $G$ be a graph embeddable on \Sg{}.
  \begin{enumerate}
  \item If $G$ is 2-connected and $\fw_g(G) \geq 3$, then $G$ has a unique 3-connected component that is non-planar and embeddable on \Sg{} with facewidth $\fw_g(G)$, and all other 3-connected components are planar.
  \item If $\fw_g(G) \geq 2$, then $G$ has a unique 2-connected component that is non-planar and embeddable on \Sg{} with facewidth $\fw_g(G)$, and all other 2-connected components are planar.
  \end{enumerate}
\end{coro}

Using this result, we can express the generating function of connected graphs using that of 2-connected graphs (with a suitable variable substitution that will be explained later), which in turns can be expressed using that of 3-connected graphs (again with a suitable variable substitution), whose asymptotic behavior is supposed to be known from map enumeration results. We hope to obtain in this way the asymptotic number of connected graphs using this decomposition along connectivity. Again, with some technical but easy analysis, we can drop the facewidth constraint in the asymptotic enumeration.

What we have discussed above is the general procedure using the decomposition along connectivity discovered by Tutte to obtain asymptotic enumeration results of connected graphs from that of maps (\textit{cf.} \cite{planar-law, graphfive, bender2011asymptotic}). Its application also needs some careful adaptation to the class of graphs we work on, which is the class of cubic graphs in our case. This adaptation is not trivial at all and requires a genuine effort.

But there is a last missing puzzle piece. We know that 3-connected cubic graphs have unique embeddings, which are 3-connected cubic maps, but we don't know their asymptotic behavior! Although triangulations on a general surface \Sg{}, which are duals of cubic maps on the same surface, have been studied quite thoroughly (\textit{cf.} \cite{Gao1992-2-connected-surface, Gao1993-pattern, Gao}), none of the variants correspond to the dual of 3-connected cubic maps, which allow loops and double edges in some cases but not all (the exact condition will be given later). To count these triangulations, we will perform \emph{surgeries} on them to ``cut out'' the loops and double edges that we don't want, in order to deduce their asymptotic enumeration formula from those of other types of triangulations that are already known. The idea of using surgeries to tailor graph embeddings for enumeration problems is classical. For instance, it has been used in \cite{log-facewidth, 3conn-fw-ew} to obtain asymptotic enumeration results on maps with small facewidth or edgewidth, which lead to asymptotic properties of facewidth and edgewidth of maps.

To summarize our strategy, Theorem~\ref{thm:6:robertson-vitray} of Robertson and Vitray relates cubic graphs and cubic maps that are 3-connected. On the map side, we consider the dual triangulations, whose asymptotic enumeration is obtained using surgeries, and then transferred to 3-connected cubic graphs. On the graph side, we use the decomposition along connectivity to transfer the asymptotic enumeration of 3-connected cubic graphs to 2-connected and then to connected cubic graphs, which is our goal. Figure~\ref{fig:6:strategy} illustrates our strategy.

\begin{figure}
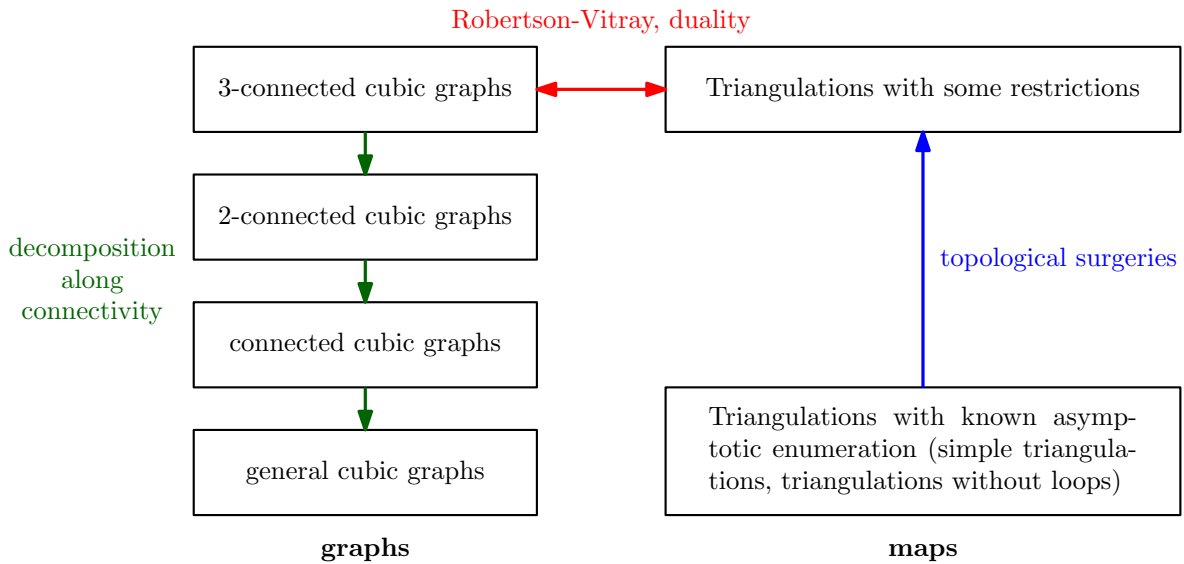

  \centering
  \insertfigure{ch6-fig.pdf}{8}
  \caption{Our strategy to enumerate cubic graphs using triangulations}
  \label{fig:6:strategy}
\end{figure}

Using this strategy, we obtain the following asymptotic enumeration result of \emph{weighted} cubic graphs. We recall that, for the sake of enumeration of random graphs via their kernels, we are interested in cubic graphs weighted with \emph{compensation factors} of $1/2$ per loop and double edge and $1/6$ per triple edge. We denote by $w_g(n)$ the total weight of vertex-labeled cubic graphs strongly embeddable into \Sg{} with $2n$ vertices under the compensation factors. We should note that the cubic graphs here are not necessarily connected. For a cubic graph $G$ with $k$ connected component, $G$ is said to be strongly embeddable into \Sg{} if its connected components are strongly embeddable into $\torus{g_1}, \torus{g_2}, \ldots, \torus{g_k}$ respectively such that $\sum_{i=1}^k g_i=g$.

\begin{thm} \label{thm:6:main}
For $g\geq 0$, the total weight $w_g(n)$ of vertex-labeled cubic graphs with $n$ vertices that are strongly embeddable into \Sg{} with compensation factors has the following asymptotic expansion:
\[ w_g(n) = \left( 1 + O(n^{-1/4}) \right) e_g n^{5/2(g-1)-1} \gamma_w^{2n}(2n)!. \]
Here, $\gamma_w = 79^{3/4} 54^{-1/2} \simeq 3.606$ and $e_g$ is a constant that only depends on the genus $g$.
\end{thm}

It is important to note that Theorem~\ref{thm:6:main} is about graphs that are \emph{strongly embeddable} into \Sg{}, which is different from the notion of graphs \emph{embeddable} into \Sg{}, which allows embeddings that are not maps, \textit{i.e.} their faces are not necessarily topological disks. Although the two notions of embeddability are different, the same asymptotic formula holds for both notions. It is because when the size tends to infinity, the number of cubic graphs strongly embeddable into $\mathbb{S}_{g'}$ is negligible comparing to those strongly embeddable into \Sg{} for any $g'<g$, according to Theorem~\ref{thm:6:main}. Therefore, the number of graphs embeddable into \Sg{} has the same asymptotically dominant term as that of graphs strongly embeddable into \Sg{}.

\section{Map surgeries}

As we have mentioned previously, we are going to enumerate embeddings of 3-connected cubic graphs on \Sg{} by enumerating their duals, which form some class of triangulations of genus $g$ that we will precise later. Before proceeding to the enumeration of these triangulations, we need some tools. For enumeration, we need to consider maps alongside with the surface into which they embeds. In this viewpoint, we can perform a \emph{topological surgery} (or simply a \emph{surgery}) on the surface into which a map embeds in order to obtain a new map. Surgery is a heavily used tool in topology and geometry, and it has been introduced in the enumeration of maps long ago in \cite{log-facewidth, 3conn-fw-ew}. Although it is possible to formally define surgeries in the language of topology, I choose to define the surgeries we will use later in an informal fashion, in the hope of making the definitions more accessible. 

During surgeries, we may temporarily encounter \mydef{maps with boundaries}, which are maps with distinguished faces called \mydef{boundaries}. Furthermore, there is an order over all boundaries, and the root corner is not contained in any boundary. A boundary $f$ is called \mydef{simple} if each of its adjacent edges also borders a face or a boundary other than $f$. A \mydef{map with simple boundaries} is a map with boundaries which are all simple. For a map $M$ of genus $g$ with simple boundaries $f_1, f_2, \ldots, f_k$, we consider it to be embedded into a \mydef{surface with holes} $\mathbb{S}$, obtained by excluding the interior of all $f_i$'s from the surface \Sg{} on which $M$ embeds. The genus of $S$ is defined to be the same as \Sg{}. Each interior of $f_i$ is called a \mydef{hole}. Figure~\ref{fig:6:map-boundaries} shows a map with simple boundaries, and the surface with holes into which it embeds. Maps with boundaries can be seen as a generalization of maps. Since we will only encounter simple boundaries later, we assume that all boundaries are simple hereinafter. A \mydef{triangulation with simple boundaries} is thus a map with simple boundaries such that all faces (excluding boundaries) are of degree $3$.

\begin{figure}
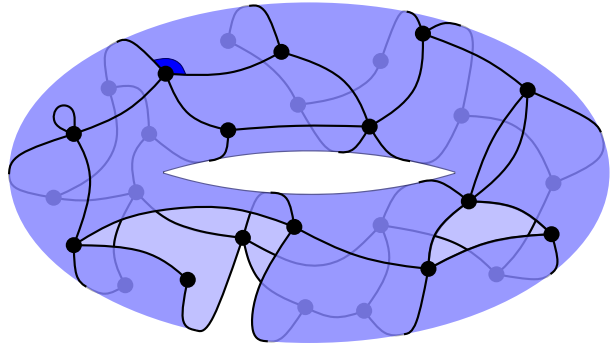

  \centering
  \insertfigure{ch6-fig.pdf}{10}
  \caption{A map with two boundaries on a torus}
  \label{fig:6:map-boundaries}
\end{figure}

We will now describe some surgeries on surfaces with (or without) holes, on which maps with (or without) simple boundaries live. A \mydef{simple cycle} of length $k$ in a map $M$ is a sequence $(v_1, e_1, v_2, e_2, \ldots, v_k, e_k)$ with vertices $v_i$'s and edges $e_i$'s such that all $v_i$'s are distinct, and $e_i$ is an edge adjacent to $v_i$ and $v_{i+1}$, with $v_{k+1}=v_1$. For convenience, we often omit the vertices and only write $(e_1, e_2, \ldots, e_k)$ for a simple cycle. Let $M$ be a map with simple boundaries $f_1, \ldots, f_k$ embedded into a surface $\mathbb{S} = \mathbb{S}_g-(f_1,\ldots,f_k)$ with holes, and $C = (e_1, \ldots, e_k)$ be a simple cycle in $M$. For the cycle $C$, we fix an orientation, with which we can distinguish the left and the right side of faces and edges adjacent to elements in $C$. The surgery of \mydef{cutting $M$ along $C$} consists of locally splitting the surface $S$ along $C$ by duplicating all vertices and edges in $C$ and making each original side of $C$ adjacent to one copy of $C$. This cutting surgeries creates two holes, each delimited by a copy of $C$. Figure~\ref{fig:6:cutting-along} shows two examples of the effect of cutting along a cycle on a surface with holes given by the map with boundaries in Figure~\ref{fig:6:map-boundaries}. We notice that in the example on the left, one of the new holes borders an existing hole by two edges after the cutting, and these edges are still considered a part of the surface with holes we obtain. After a cutting, there are two cases: either $S$ remains connected (left of Figure~\ref{fig:6:cutting-along}), or it is cut into two surfaces with holes (right of Figure~\ref{fig:6:cutting-along}), each containing a map with boundaries. In both cases, we have two new holes in the resulting surfaces, both of the same size and with disjoint vertices and edges. In the first case, the cycle $C$ is called \mydef{non-separating}, while in the second case, $C$ is called \mydef{separating}. In the separating case, we should note that only one of the components contains the root, and later when we enumerate maps using this surgery, we will specify how we root the unrooted component. Concerning the genus of surfaces we obtain, suppose that we cut a surface of genus $g$ along a simple cycle. In the non-separating case, the surface we obtain is of genus $g-1$. In the separating case, the two surfaces we obtain have total genus $g$. The variation of genus can be seen using Euler's relation. Let us take the non-separating case as an example. After cutting, we can fill the holes for them to become faces, which gives two extra faces, while we have the same number of new vertices and faces coming from the duplication in cutting. Therefore, by Euler's relation $v-e+f=2-2g$, we must decrease by $1$ in the genus. The analysis for the non-separable case is similar.

\begin{figure}
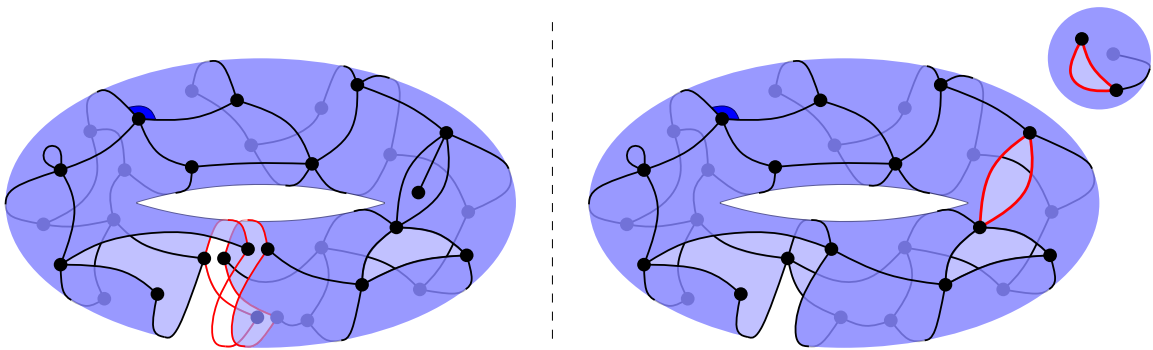

  \centering
  \insertfigure[0.85]{ch6-fig.pdf}{11}
  \caption{Two examples of cutting along a cycle}
  \label{fig:6:cutting-along}
\end{figure}

The inverse of the cutting surgery is called \mydef{gluing}, and it comes in two types, each corresponding to one case in the cutting surgery. Let $M$ be a map with simple boundaries embedded into a surface $\mathbb{S}$, with two boundaries $f_1, f_2$ of the same degree. Let $C_1, C_2$ be the simple cycles that enclose $f_1$ and $f_2$ respectively. In the case where $C_1$ and $C_2$ are disjoint, the surgery of \mydef{gluing $C_1$ and $C_2$} on $M$ consists of identifying vertices and edges of $C_1$ and $C_2$ with respect to the orientation of $\mathbb{S}$. For $C_1$ and $C_2$ of length $k$, there are $k$ ways to glue them. With the condition that $C_1$ and $C_2$ are disjoint, we avoid pathological cases and make this gluing surgery the inverse of the cutting surgery in the non-separating case. For the inverse of the cutting surgery in the separating case, we consider two maps with simple boundaries $M, M'$ embedded into two surfaces $\mathbb{S}, \mathbb{S}'$ respectively, and we suppose that $M$ has a boundary $f$ that is of the same degree as another boundary $f'$ of $M'$. Let $C, C'$ be the simple cycles that enclose $f$ and $f'$ respectively. We can thus similarly define the surgery of \mydef{gluing $M$ and $M'$ by $C$ and $C'$}, and we see that this surgery is the inverse of the separating case of the cutting surgery. We should also note that there is also a rooting issue for these surgeries, which will be clarified later. For examples, we can look at Figure~\ref{fig:6:cutting-along} in another way around: by gluing the distinguished cycles, we can obtain the map with simple boundaries in Figure~\ref{fig:6:map-boundaries}.


We now introduce two last surgeries, which only apply to boundaries formed by double edges. In the following, we refer to a pair of edges that share both distinct endpoints as a single entity called a \mydef{double edge}. Let $M$ be a map with simple boundaries, where one of them is formed by a double edge $d = \{ e_1, e_2 \}$. The surgery of \mydef{zipping the double edge $d$} on $M$ consists of identifying $e_1$ and $e_2$ with respect to their extremities and discard the boundary enclosed by $e_1,e_2$. In the reverse direction, given a map with simple boundaries $M$ and one of its edges $e$ that is not a loop, the surgery of \mydef{cutting open the edge $e$} on $M$ consists of duplicating $e$ and add a boundary enclosed by the resulting double edge. Figure~\ref{fig:6:zipping} shows an example of these two surgeries. 

\begin{figure}
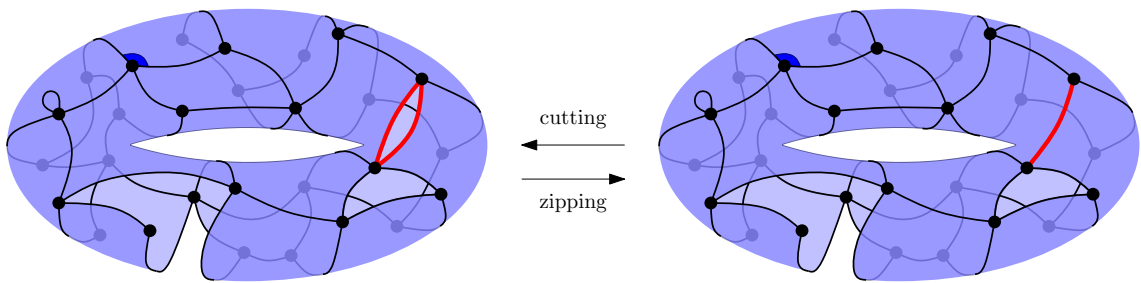

  \centering
  \insertfigure[0.8]{ch6-fig.pdf}{12}
  \caption{Cutting an edge and zipping a double edge on a map with simple boundaries}
  \label{fig:6:zipping}
\end{figure}

In the following, for the sake of simplicity, when talking about surgeries on maps with simple boundaries, we will stop mentioning the fact that these surgeries are in fact performed on the surface with holes on which maps live. In other words, we will assume that each map comes with the surface into which it embeds. We also notice that all these surgeries preserve faces in maps (boundaries excluded). Since we will be performing these surgeries on triangulations, any map without holes we obtain will be a triangulation.

\begin{figure}
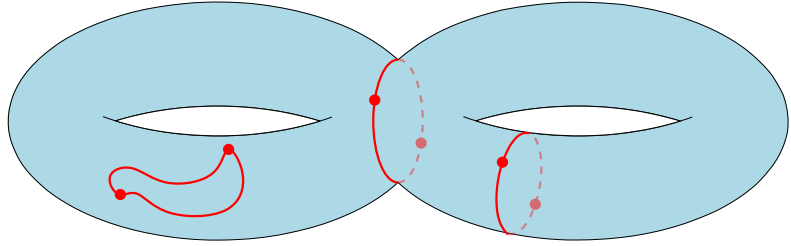

  \centering
  \insertfigure[0.7]{ch6-fig.pdf}{19}
  \caption{Examples of planar separating, non-planar separating and non-separating double edges (from left to right) on $\torus{2}$}
  \label{fig:6:circle-type}
\end{figure}

As a final remark, we can use surgeries to classify double edges and loops in a map. Let $C=(e_1, e_2)$ be a double edge on a map $M$. We know that $C$ is also a simple cycle. We now cut along $C$ on $M$. If the surgery does not disconnect $M$, then the double edge $C$ is called \mydef{non-separating}. If the surgery disconnects $M$ into two maps with one of them planar, then $C$ is called \mydef{planar separating} (or simply \mydef{planar}), otherwise $C$ is called \mydef{non-planar separating}. Figure~\ref{fig:6:circle-type} shows examples of all three cases. Similarly, for a loop $e_\ell$ in a map $M$, since $e_\ell$ is a simple cycle, we say that $e_\ell$ is \mydef{separating} if cutting along it disconnects the map, and otherwise \mydef{non-separating}. We will also consider cutting along a pair of loops. For a pair of loops $e_1, e_2$ in $M$, we say that they are \mydef{a separating pair} if none of them is separating but cutting along both of them disconnects the map. We notice that the cutting surgeries along two different loops commute, except in the case where the two loops share the same vertex, and when cutting one of the loops, the endpoints of the other locate on different duplication of the vertex. Figure~\ref{fig:6:two-loops-case} shows an example of this case. But in this exceptional case, after the first surgery, the other loop is transformed into an edge $e'$ with its endpoints on the two boundaries of length $1$ (see Figure~\ref{fig:6:two-loops-case}, middle), and by cutting along $e'$ we merge the two boundaries, which does not alter the genus (see Figure~\ref{fig:6:two-loops-case}, right). Therefore, when considering separating pairs of loops, we do not need to consider this situation, and we can assume that the cutting surgeries of the two loops are performed simultaneously.

\begin{figure}
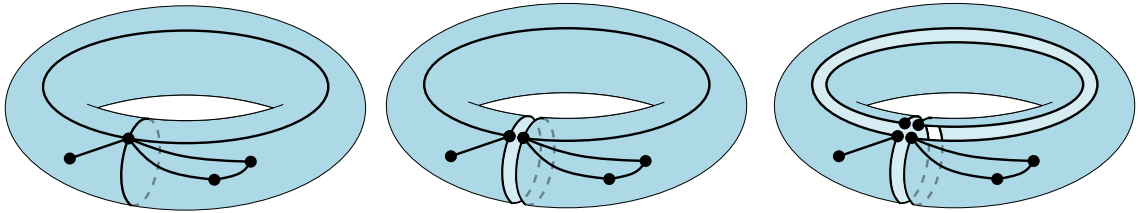

  \centering
  \insertfigure[0.75]{ch6-fig.pdf}{20}
  \caption{Example of the special case in cutting along two loops}
  \label{fig:6:two-loops-case}
\end{figure}

\section{Counting 3-connected cubic maps with surgeries}

We begin the implementation of our plan in Figure~\ref{fig:6:strategy} by the enumeration of 3-connected cubic maps, which corresponds to the right-hand side of Figure~\ref{fig:6:strategy}. We know that the dual of a cubic map is a triangulation. Since much is known about various triangulations, such as in \cite{Gao1991-2-connected-projective, Gao1992-2-connected-surface, Gao1993-pattern}, it may be easier to base our enumeration effort on these known results. Therefore, instead of directly counting 3-connected cubic maps, we count their duals, which are triangulations with constraints described in the following proposition.

\begin{figure}
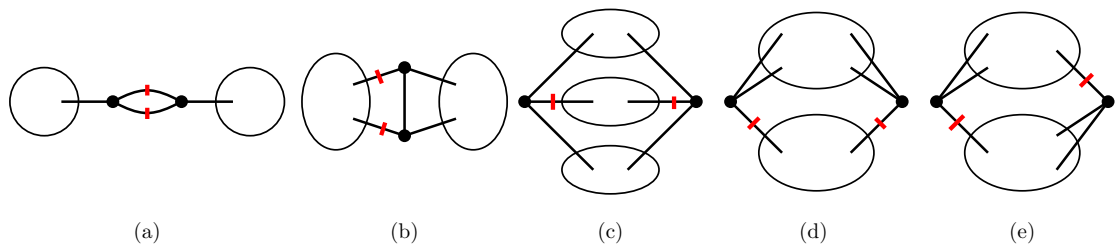

  \centering
  \insertfigure[0.8]{ch6-fig.pdf}{21}
  \caption{Possible configurations of a vertex separator of size 2 in a cubic graph, alongside with a possible edge separator of size 2 in each case}
  \label{fig:6:cubic-sep}
\end{figure}

\begin{prop} \label{prop:6:3-conn-dual}
Let $M$ be a cubic map on \Sg{} with at least $4$ vertices and $M^*$ its dual triangulation. Then $M$ is 3-connected if and only if $M^*$ is a triangulation without separating loop, separating double edge nor separating pair of loops.
\end{prop}
\begin{proof}
We first prove that, for cubic graphs $G$ with at least $4$ vertices, 3-connectivity and 3-edge-connectivity coincide. Suppose that we can disconnect $G$ by deleting two edges $e_1,e_2$. In this case, by choosing one endpoint of each of $e_1$ and $e_2$, we get two vertices whose deletion will also disconnect $G$, and since $G$ has at least $4$ vertices, each component after deletion is not trivial. Conversely, suppose that we can disconnect $G$ by deleting two vertices $u,v$. Since $G$ is cubic, $u$ and $v$ are both of degree $3$. As $G$ has at least $4$ vertices, there are at most $2$ edges between $u$ and $v$. We now divide our analysis into cases with different number of edges between $u$ and $v$. These cases are illustrated in Figure~\ref{fig:6:cubic-sep}.

If there are two edges between $u$ and $v$ (Figure~\ref{fig:6:cubic-sep}(a)), then by deleting these two edges, we disconnect $G$. 

If there is only one edge between $u$ and $v$, then there are $4$ edges that link components to $u$ and $v$. If there is a component $C$ that only links to one of the vertices $u$ and $v$, for instance $u$, there are at most $2$ edges from $C$ to $u$, and their deletion will disconnect the graph. We now suppose that each component are linked to both $u$ and $v$ at the same time. Since there are at least 2 components after deletion of $u,v$, we know that there are exactly 2 components, each connected to $u$ and $v$ as in Figure~\ref{fig:6:cubic-sep}(b). We pick one of the components, the deletion of its two edges linking to $u$ and $v$ will disconnect $G$. 

The last case is that there is no edge between $u$ and $v$. If there are 3 components after deletion of $u$ and $v$ (Figure~\ref{fig:6:cubic-sep}(c)), then we can apply the same reasoning as in the previous case. Otherwise, there are only two components, and among the 3 edges adjacent to $u$, there is one connected to a different component from the other two (see Figure~\ref{fig:6:cubic-sep}(d,e)). By deleting this edge and the edge in the same situation for $v$, we disconnect $G$. 

Therefore, we have proved that $G$ has a 2-vertex separator if and only if it has a 2-edge separator, which means that 3-connectivity and 3-edge connectivity coincide. We will use 3-edge-connectivity hereinafter. 

We now define the dual graph $G^*(M_0)$ of a map with boundaries $M_0$. We first transform $M_0$ into a map without boundary $M_1$ by regarding its boundaries as normal faces. Then we take the underlying graph $G_1^*$ of the dual $M_1^*$ of $M_1$, and delete from $G_1^*$ the dual vertices and edges corresponding to faces in $M_1$ that are boundaries in $M_0$. We thus obtain $G^*(M_0)$. 

Let $T$ be a triangulation with boundaries and $e$ one of its edges. By cutting along $e$, we obtain a triangulation with boundaries $T'$. We observe that $G^*(T')$ can be obtained from $G^*(T)$ by deleting the dual edge that corresponds to $e$. Indeed, suppose that $e$ borders two faces $f_1, f_2$ in $T$. After cutting along $e$, we have a new boundary $b$ that borders both $f_1$ and $f_2$ by an edge, while $f_1$ and $f_2$ now lose the common border $e$. This change is reflected in the dual graph as the deletion of the edge corresponding to $e$ that connects the dual vertices of $f_1$ and $f_2$, since the dual vertex of $b$ and its adjacent edges are deleted. Conversely, the deletion of an edge $e$ in $G^*(T)$ also corresponds to cutting along the dual edge of $e$ in $T$.

Now, for a cubic map $M$, let $G$ be its underlying graph and $M^*$ its dual triangulation. We suppose that $M$ is not 3-edge-connected, which means that there are two edges $e_1, e_2$ in $G$ whose deletion will disconnect the graph. Therefore, their dual edges $e_1^*, e_2^*$ in $M$, cutting along them will disconnect the surface. There are several cases for $e_1^*$ and $e_2^*$, either they form a double edge, or they are two simple edges, or at least one of them is a loop. If they form a double edge, by definition, they form a separating double edge. It is impossible for both of them to be simple edges, because in this case cutting along both of them will never disconnects the surface. If only one of them is a loop, then this loop must be separating, because cutting along the other, which is a simple edge, cannot disconnect the surface. If both of them are loops, either one of them is separating, in which case we have a separating loop; or both of them are not separating, but together they form a separating pair of loops. In conclusion, one of the following structures must exists: separating loop, separating double edge and separating pair of loops.


\end{proof}

Let $\mathcal{M}_g$ be the class of triangulations on \Sg{} without separating loops, separating double edges nor separating pair of loops. We note that these triangulations are exactly duals of 3-connected cubic maps with at least 4 vertices or a triangulation with exactly 3 edges, as indicated in Proposition~\ref{prop:6:3-conn-dual}, and we want to obtain their asymptotic enumeration. To this end, we need some other types of triangulations. Let $\mathcal{N}_g$ be the class of triangulations on \Sg{} without loops nor separating double edges, $\mathcal{R}_g$ be the class of triangulations on \Sg{} without loops or planar double edges, $\mathcal{T}_g$ be the class of triangulations on \Sg{} without loops, and $\mathcal{S}_g$ be the generating function of simple triangulations on \Sg{}, \textit{i.e.} without loop nor double edges. The size statistic of these classes of triangulations is the number of edges. We denote their OGFs by $M_g(t), N_g(t), R_g(t), T_g(t)$ and $S_g(t)$ respectively, all with $t$ marking the number of edges. The empty triangulation is not counted in any class here. Table~\ref{tab:6:triangulations} summarizes the characterization of these triangulations. We notice that $\mathcal{S}_g \subseteq \mathcal{N}_g \subseteq \mathcal{R}_g \subseteq \mathcal{T}_g$ and $\mathcal{N}_g \subseteq \mathcal{M}_g$.

\begin{table}
\centering
\begin{minipage}[c]{\textwidth}
\begin{tabular}{ccccc}
\toprule
\quad & \quad & \multicolumn{3}{c}{double edges} \\ \cline{3-5}
\quad & loops & non-separating & planar separating & non-planar separating \\
\midrule
$\mathcal{M}_g$ & Some cases\footnote{All cases except separating loops and separating pairs of loops.} & Yes & No & No \\
\midrule
$\mathcal{N}_g$ & No & Yes & No & No \\
\midrule
$\mathcal{R}_g$ & No & Yes & No & Yes \\
\midrule
$\mathcal{T}_g$ & No & Yes & Yes & Yes \\
\midrule
$\mathcal{S}_g$ & No & No & No & No \\
\bottomrule
\end{tabular}
\end{minipage}
\caption{Comparison of different classes of triangulations}
\label{tab:6:triangulations}
\end{table}

In the following, we will adopt the following convention of notation: we will use calligraphic letters such as $\mathcal{M}_g$ to denote classes of maps, and they always come with at least the genus parameter $g$; their corresponding generating functions will be denoted by letters in italic such as $M_g(t)$, also always coming with at least the genus $g$ as subscript, but also the variable $t$; for an element in a class of maps, we will also use normal letters, for instance $M$, potentially with a superscript and a subscript, but never the genus $g$ in the subscript, and we never append the variable $t$.

We already have asymptotic enumeration results on two classes $\mathcal{T}_g$ and $\mathcal{S}_g$, first given in \cite{Gao1992-2-connected-surface} and \cite{Gao1993-pattern}. The following proposition presents these results in the form of dominant singularities of the corresponding OGFs. We recall that for a power series $f(t)$ with positive coefficients, we write $f(t) \cong g(t) + O(h(t))$ if there are two $\Delta$-analytic power series $f_-(t)$ and $f_+(t)$ with non-negative coefficients and the same asymptotic behavior $g(t)+O(h(T))$ such that $f_-(t) \preceq f(t) \preceq f_+(t)$.

\begin{prop} \label{prop:6:asympt-S}
The OGF $S_0(t)$ has the following rational parametrization:
\[ S_0(t) = s(1-2s), \;\; \textrm{where} \;\; t^3 = s(1-s)^3.\]
Therefore, $S_0$ is $\Delta$-analytic, with the following expansion near the dominant singularity $\rho_S = 2^{-8/3} \cdot 3$:
\begin{equation} \label{eq:6:S0}
S_0(t) = \frac1{8} - \frac{9}{16}(1-\rho_S^{-1}t) + \frac{3}{2^{5/2}}(1-\rho_S^{-1}t)^{3/2} + O\left( (1-\rho_S^{-1}t)^2\right).
\end{equation}
For $g \geq 1$, the OGF $S_g(t)$ has the same dominant singularity $\rho_S$ as $S_0(t)$. The expansion of $S_g(t)$ near $t=\rho_S$ is given by
\begin{equation} \label{eq:6:Sg}
S_g(t) \cong c_g (1-\rho_S^{-1}t)^{-5(g-1)/2-1} \left(1+O\left( (1-\rho_S^{-1}t)^{1/4} \right) \right).
\end{equation}
Here, $c_g = 3t_g\Gamma(5(1-g)/2)\beta_S^{5(g-1)/2}$, where $t_g$ is the universal coefficient in \cite{BC0} that depends only on $g$, and $\beta_S = (4/3)6^{1/5}$.
\end{prop}

\begin{prop} \label{prop:6:asympt-T}
The OGF $T_0(t)$ has the following rational parametrization:
\[ T_0(t) = \frac{s(1-4s)}{(1-2s)^2}, \;\; \textrm{where} \;\; t^3 = s(1-2s)^2.\]
Therefore, $T_0$ is $\Delta$-analytic, with the following expansion near the dominant singularity $\rho_T = 2^{1/3}\cdot 3$:
\begin{equation} \label{eq:6:T0}
T_0(t) = \frac1{8} - \frac{9}{8}(1-\rho_T^{-1}t) + 3(1-\rho_T^{-1}t)^{3/2} + O\left( (1-\rho_T^{-1}t)^2 \right).
\end{equation}
For $g \geq 1$, the OGF $T_g(t)$ has the same dominant singularity $\rho_T$ as $T_0(t)$. The expansion of $T_g(t)$ near $t=\rho_T$ is given by
\begin{equation} \label{eq:6:Tg}
T_g(t) \cong c'_g (1-\rho_T^{-1}t)^{-5(g-1)/2-1} \left(1+O\left( (1-\rho_T^{-1}t)^{1/4} \right) \right).
\end{equation}
Here, $c'_g = 3t_g\Gamma(5(1-g)-2)\beta_T^{5(g-1)/2}$, where $t_g$ is the same universal coefficient as in Proposition~\ref{prop:6:asympt-S}, and $\beta_T = (2/3)6^{1/5}$.
\end{prop}

The parametrization of $S_0$ and $T_0$ comes directly from \cite{Gao1993-pattern} and \cite{Gao1991-2-connected-projective} respectively, from which we can deduce the expansion in \eqref{eq:6:S0} and \eqref{eq:6:T0}. The dominant terms in \eqref{eq:6:Sg} and \eqref{eq:6:Tg} were already given in \cite{Gao1992-2-connected-surface} and \cite{Gao1993-pattern}. The error terms can be obtained by an induction on genus on the functional equations for $S_g$ and $T_g$ given in the same papers. Details are omitted here, and interested readers are referred to Appendix~A of \cite{fang-graz}.

We note that the number of edges of a triangulation must be divisible by $3$, since every face is of degree $3$ and each edge is counted twice in face degrees. Therefore, all the OGFs of triangulations exhibit coefficient periodicity, and there are other singularities that have the same modulus as the dominant one. However, by a change of variable $u=t^3$, we can see that these companion singularities does not affect the asymptotic behavior, and we can focus our analysis on the dominant one. In this chapter, we will still keep the number of edges as the size parameter, as in other chapters.

Although the coefficients $c_g$ and $c'_g$ seem complicated, they are in fact closely related. We observe that $\beta_S = 2\beta_T$, therefore we have
\begin{equation} \label{eq:6:coeff-multiplier}
c_g = 2^{5(g-1)/2} c'_g.
\end{equation}
We also notice that $\rho_S = (9/8)\rho_T$.

Our aim now is to relate the class $\mathcal{M}_g$ to $\mathcal{S}_g$ using surgeries in order to obtain the asymptotic enumeration of $\mathcal{M}_g$. More precisely, we will first prove that $R_g(t)$ has the same asymptotic behavior as $S_g(t)$ (Proposition~\ref{prop:6:R-asymptotic}). This is done by establishing with surgeries a relation between $R_g(t)$ and $T_g(t)$, which involves the functional composition $R_g(t(1+T_0(t)))$. By analyzing carefully this functional composition, we are able to determine the dominant singularity and the asymptotic behavior of $R_g(t)$, which coincide with those of $S_g(t)$. We then prove that $N_g(t)$ has the same asymptotic behavior as $R_g(t)$ (Proposition~\ref{prop:6:N-asymptotic}), again using surgeries. The idea of the proof is that, since the only difference between $\mathcal{R}_g$ and $\mathcal{N}_g$ is that non-planar separating double edges are allowed in $\mathcal{R}_g$, by performing cutting surgeries on triangulations in $\mathcal{R}_g$ and by an induction on the genus $g$, we show that triangulations containing at least a non-planar separating double edge form a negligible part in $\mathcal{R}_g$. Finally, we prove that $M_g(t)$ has the same asymptotic behavior as $N_g(t)$ (Proposition~\ref{prop:6:M-asymptotic}), this time by performing surgeries on loops on triangulations in $\mathcal{M}_g$. In \cite{fang-graz}, this procedure is shortened to passing directly from $\mathcal{R}_g$ to $\mathcal{M}_g$ by performing the two sets of surgeries at once. But here I choose to split into more steps for a potentially more accessible proof.


This chain of congruence of asymptotic behavior leads to our main result in this section, which is also the main result on the map side of our plan (see the right-hand side of Figure~\ref{fig:6:strategy}).

\begin{prop} \label{prop:6:M-asymptotic}
The OGFs $M_g(t)$ and $S_g(t)$ have the same dominant singularity $\rho_S$. Furthermore, we have $M_0(t)=S_0(t)$, and for $g\geq1$, the OGF $M_g(t)$ has the same asymptotic behavior as $S_g(t)$ near the dominant singularity, namely
\[
M_g(t) \cong c_g(1-\rho_S^{-1}t)^{-5(g-1)/2-1}\left(1+O\left(1-\rho_S^{-1}t)^{1/4}\right)\right).
\]
\end{prop}

We now present a topological surgery acting on double edges on which we will rely heavily later.

\begin{lem} \label{lem:6:surgery-2}
Let $T$ be a triangulation on \Sg{} with $n$ edges, and $D=\{ e_1, e_2 \}$ a double edge. Then by cutting along $D$ and zipping the two copies of $D$ duplicated in the cutting on both sides, the result must fall in one of the following cases:
\begin{itemize}
\item A triangulation $T'$ on $\torus{g-1}$ with $n$ edges and two marked edges that share no common vertex (see upper part of Figure~\ref{fig:6:zip-dbl-edge});
\item Two triangulations $T_1$ on $\torus{g_1}$ and $T_2$ on $\torus{g_2}$ with $n$ edges in total, satisfying $g=g_1+g_2$, where $T_1$ contains the original root corner and an extra marked edge, and $T_2$ is a rooted triangulation without marked edge (see lower part of Figure~\ref{fig:6:zip-dbl-edge}).
\end{itemize}
In the first case, it is a 2-to-1 correspondence. In the second case, it is a 2-to-2 correspondence.
\end{lem}
\begin{proof}
It is clear that the result will be triangulations without boundaries, since both surgeries we perform do not affect faces, and both holes created in the cutting are closed by zipping. When the double edge is non-separating, we fall in the first case. Otherwise, we fall in the second case. We now deal with roots and marked edges of these triangulations. In the first case, we simply keep the root corner and mark the two edges coming from zipping the holes. In the second case, we take the component containing the root as $T_1$, and the other as $T_2$. In $T_1$, we simply mark the edge coming from zipping the hole. In $T_2$, since there is no root, we root at the edge coming from zipping the hole, with two possible orientations.

We now look at the other direction. In the first case, we can simply cut open the two marked edges, then glue the two double edges we obtained by cutting. This is always possible since the marked edges share no common vertex. There are two ways of gluing, depending on which pairs of vertices we identify. In the second case, we cut open the marked edge on $T_1$ and the root edge on $T_2$, then glue the two double edges we obtained. Again, there are two ways of gluing. Therefore, we have a 2-to-1 correspondence in the first case, and a 2-to-2 correspondence in the second.
\end{proof}

\begin{figure}
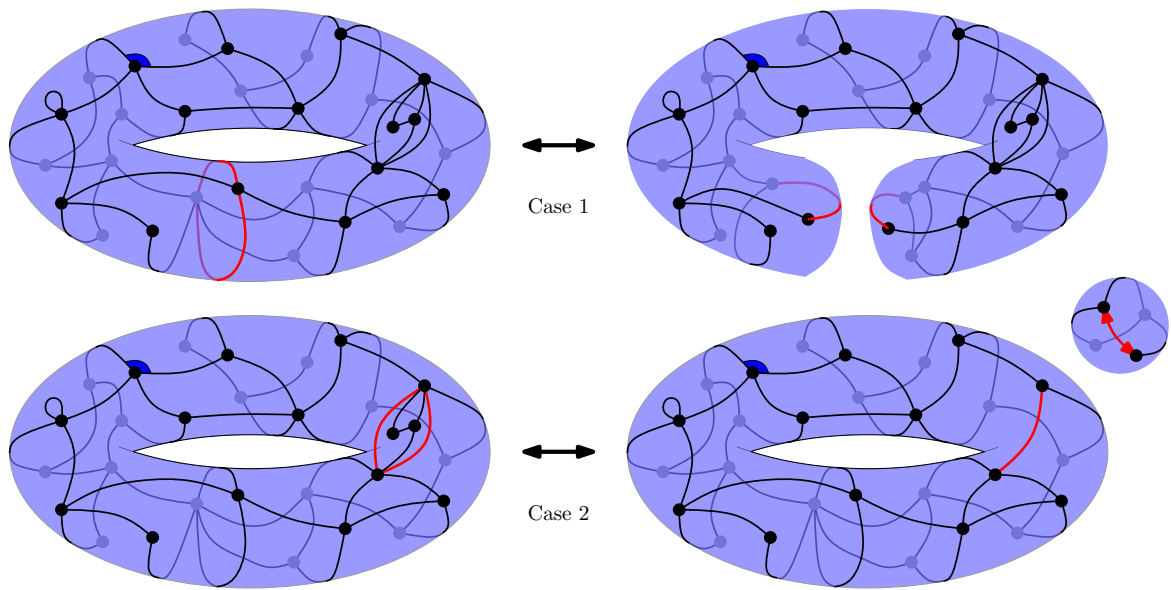

  \centering
  \insertfigure[0.8]{ch6-fig.pdf}{22}
  \caption{Two cases of cutting along and zipping a double edge}
  \label{fig:6:zip-dbl-edge}
\end{figure}

We now consider the implication of Lemma~\ref{lem:6:surgery-2} in map enumeration using generation functions. For the non-separating case, we notice that the two marked edges on the triangulation $T'$ obtained after the surgery must not share any common vertex. This extra constraint is difficult to be expressed in terms of generating functions, therefore it seems to be difficult to apply the non-separating case to obtain equalities of generating functions. However, if we remove this constraint of forbidden common vertex, $T'$ will simply be a triangulation with two marked edges. In terms of generating functions, two marked edges can be translated to partial differentiations by the size variable $t$, which are simple to manipulate. However, in exchange of simplicity, we no longer have an exact correspondence. Since we removed a constraint, what we have is now an upper bound of the number (or coefficient-wise dominance for generating functions) of triangulations of a given genus in the non-separating case of Lemma~\ref{lem:6:surgery-2}. Therefore, although we cannot use the non-separating case of Lemma~\ref{lem:6:surgery-2} to obtain an equality of OGFs, we can still use it to obtain a coefficient-wise domination of OGFs. The separating case of Lemme~\ref{lem:6:surgery-2}, however, can be translated directly into an equality of OGFs without any further problem.


In the surgery presented in Lemma~\ref{lem:6:surgery-2}, the sub-case of a planar double edge in the separating case is particularly important. In other cases, we obtain one or two maps of strictly lower genus. But when the pair of double edges is planar, we obtain a planar map and a smaller map with the same genus. We now present a property of these planar double edges. Let $T$ be a triangulation with boundaries on \Sg{} with $g>0$, and $d_1, d_2$ two planar double edges of $T$. We say that $d_1$ is contained in $d_2$ on $T$ (denoted by $d_1 <_T d_2$) if the disk we obtain by cutting along $d_2$, when not yet cut from $T$, contains $d_1$. See Figure~\ref{fig:6:dbl-edge-order} for examples of the relation $<_T$. We have the following lemma on properties of the relation $<_T$, which was already known in \cite{3conn-fw-ew} in the context of quadrangulations.


\begin{figure}
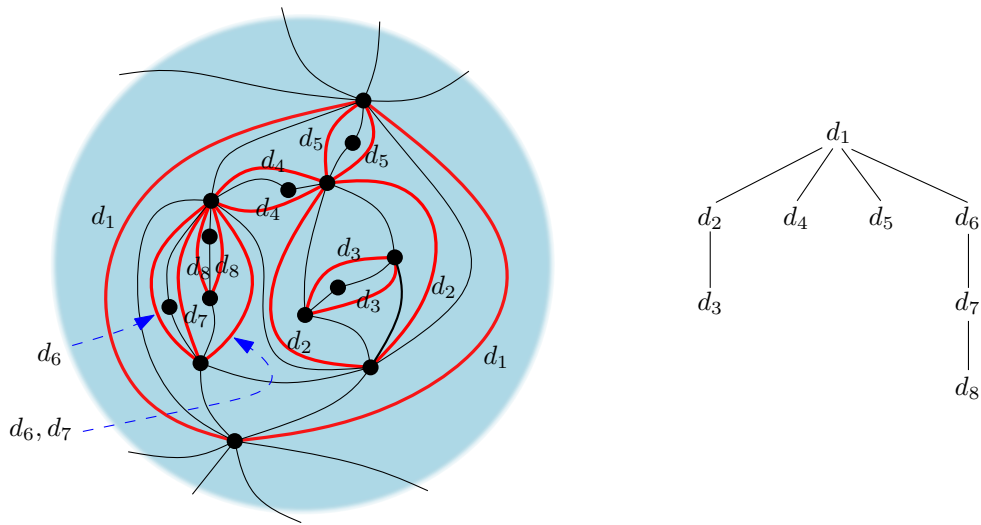

  \centering
  \insertfigure{ch6-fig.pdf}{23}
  \caption{Planar double edges on part of a triangulation, and the Hasse diagram of the partial order $<_T$ on them}
  \label{fig:6:dbl-edge-order}
\end{figure}

\begin{lem} \label{lem:6:planar-dbl-inclusion}
Let $T$ be a triangulation on \Sg{} with $g>0$, and $D$ the set of planar double edges in $T$. The relation $<_T$ defined on $D$ is a partial order on $D$.  

Furthermore, for $d \in D$, the suborder $(D', <_T)$, with $D'$ the set of elements greater than $d$, has a greatest element and a smallest element (namely $d$).
\end{lem}
\begin{proof}
Let $d_1, d_2$ be two double edges, and $P_1, P_2$ be the set of faces in the planar component we obtain by cutting along $d_1$ and $d_2$ respectively. Since $T$ is of genus $g>0$, there is a unique assignment from planar double edges and their corresponding planar component. We have $d_1 <_T d_2$ if and only if $P_1 \subseteq P_2$, since $T$ is not planar. Therefore, since the set inclusion order is a partial order, we conclude that $<_T$ is also a partial order over $D$.

For the second part, suppose that $D'$ has at least two maximal elements, namely $d_1$ and $d_2$. We define $P_1$ and $P_2$ as before. Since $P_1$ and $P_2$ contain the same double edge $d$, they have a non-empty intersection. Since both $d_1$ and $d_2$ are maximal elements, by Jordan curve theorem they must share their vertices. In this case, by taking the union of $P_1$ and $P_2$, we obtain a larger planar disk with boundary of length 2, therefore a double edge that is larger than both $d_1$ and $d_2$, which is a contradiction. 
\end{proof}

\subsection[The asymptotic behavior of $R_g(t)$]{The asymptotic behavior of $\boldsymbol{R_{g}(t)}$}

We start by relating $\mathcal{T}_g$ and $\mathcal{R}_g$ using surgeries. We observe that the only difference between $\mathcal{T}_g$ and $\mathcal{R}_g$ is that planar double edges are allowed in the former but not in the latter. Let $T$ be a triangulation in $\mathcal{T}_g$, we want to cut and zip the maximal planar double edges in $T$ such that the remaining triangulation $R$ of genus $g$ will be in $\mathcal{R}_g$. Each double edge in $T$ thus gives rise to an edge in $R$ and a planar triangulation in $\mathcal{T}_0$, since the planar part may contain other planar double edges. Now, when we look at $R$, an edge in $R$ may come from cutting and zipping a planar double edge in $T$, or from a simple edge in $T$. We can thus write an equality involving $T_g(t)$, $R_g(t)$ and $T_0(t)$ that will be used to deduce the asymptotic behavior of $R_g(t)$, which happens to coincide with that of $S_g(t)$. Using this line of reasoning, we have the following result.

\begin{prop} \label{prop:6:R-asymptotic}
The OGFs $R_g(t)$ and $S_g(t)$ have the same dominant singularity $\rho_S$. Furthermore, we have $R_0(t) = S_0(t)$, and for $g \geq 1$, the OGF $R_g(t)$ has the same asymptotic behavior as $S_g(t)$ near the dominant singularity, namely
\[
R_g(t) \cong c_g(1-\rho_S^{-1}t)^{-5(g-1)/2-1}\left(1+O\left((1-\rho_S^{-1}t)^{1/4}\right)\right).
\]
\end{prop}
\begin{proof}
In the planar case $g=0$, since all double edges are separating and planar, we have $\mathcal{R}_0 = \mathcal{S}_0$, which leads to $R_0(t)=S_0(t)$.

We now deal with the case $g \geq 1$. For a triangulation $T$ in $\mathcal{T}_g$, let $D_{\max}$ be the set of maximal planar double edges with respect to $<_T$. Lemma~\ref{lem:6:planar-dbl-inclusion} implies that double edges in $D_{\max}$ cannot have common edges. We perform the following surgery as in Lemma~\ref{lem:6:surgery-2} in arbitrary order on each double edge $d$ in $D_{\max}$: we first cut along $d$, then zip the holes created by the cutting. Since all double edges in $D_{\max}$ are planar, we obtain a planar triangulation in $\mathcal{T}_0$ for each surgery. Moreover, all double edges in $D_{\max}$ are maximal in the order $<_T$, the remaining part will thus contain no planar double edge, which means that it is in $\mathcal{R}_g$. We thus obtain a triangulation $R$ in $\mathcal{R}_g$ with marked edges, while each marked edge $e$ has a corresponding planar triangulation $T_e$ in $\mathcal{T}_0$.

We now deal with roots and marked edges on these triangulations $T_e$ and $R$. Let $e$ be a marked edge on $R$. According to Lemma~\ref{lem:6:surgery-2}, if $T_e$ does not contain the original root corner, then $T_e$ is rooted at the edge resulting from zipping the hole in the surgery. Otherwise, $T_e$ contains the original root corner, and $T_e$ is rooted at that corner and we mark the edge on $T_e$ resulting from hole-zipping. We thus have two cases: either the original root corner remains in the component $R$, or it is contained in one of the $T_e$'s. In the first case, all $T_e$'s have no extra marked edge. In the second case, let $e_r$ be the root of $R$, then $e_r$ must be a marked edge, and the corresponding triangulation $T_{e_r}$ has a marked edge, while $T_e$ for any other marked edge $e$ has no extra marked edge. We also notice that the surgeries we perform do not change the total number of edges in these triangulations. Therefore, in the first case, $T$ can be considered as $R$ with edges equipped with either an empty map or a triangulation in $\mathcal{T}_0$, and in the second case, $T$ can be considered as $R$ with the same extra load on edges as before, except for the root edge, where it must be equipped with a triangulation in $\mathcal{T}_0$ with a marked edge. We combine the two cases by saying that the root edge of $R$ can be equipped with either an empty triangulation or a triangulation $\mathcal{T}_0$ with or without a marked edge.

Since the surgeries we perform on double edges in $D_{\max}$ are 2-to-2 correspondence according to Lemma~\ref{lem:6:surgery-2}, using the case analysis above, we obtain the following expression of the OGF $T_g$ of triangulations in $\mathcal{T}_g$:
\begin{equation} \label{eq:6:S-to-R}
T_g(t) = \left( 1+T_0(t)+tT_0'(t) \right) \frac{R_g(t(1+T_0(t)))}{1+T_0(t)}.
\end{equation}
Here, the term $\frac{R_g(t(1+T_0(t)))}{1+T_0(t)}$ corresponds to a triangulation in $\mathcal{R}_g$ with edges equipped with either the empty triangulation or a triangulation in $\mathcal{T}_0$, except for the root edge, whose load is represented by the factor $1+T_0(t)+tT_0'(t)$. We can rearrange \eqref{eq:6:S-to-R} as
\begin{equation} \label{eq:6:S-to-R-alt}
R_g(t(1+T_0(t))) = T_g(t) \frac{1+T_0(t)}{1+T_0(t)+tT_0'(t)}.
\end{equation}

By Proposition~\ref{prop:6:asympt-T}, the dominant singularity of $T_g(t)$ occurs at $\rho_T$ for all $g$, and both $T_0(t)$ and $tT_0'(t)$ are finite and non-negative at the dominant singularity. In particular, we have $tT_0'(t) = 9/8 + O((1-\rho^{-1}_T t)^{1/2})$ near $t=\rho_T$, and the rational factor on the right-hand side of \eqref{eq:6:S-to-R-alt} is $\Delta$-analytic and has the following expansion at $t=\rho_T$:
\[
\frac{1+T_0(t)}{1+T_0(t)+tT_0'(t)} = \frac1{2} + O\left((1-\rho_T^{-1}t)^{1/2}\right).
\]
Therefore, according to \eqref{eq:6:S-to-R-alt}, for any $t^*<\rho_T$, the value of $R(t^*(1+T_0(t^*)))$ must be finite. Let $\rho_R$ be the dominant singularity of $R_g(t)$, which is a positive real number according to Pringsheim's Theorem (see Section~\ref{sec:2:analytic}). We thus have
\[
\rho_R \geq \rho_T(1+T_0(\rho_T)) = \frac{9}{8} \rho_T = \rho_S.
\]
But since $\mathcal{S}_g \subseteq \mathcal{R}_g$, we have $\rho_R \leq \rho_S$. We thus have $\rho_R = \rho_S$, that is, the series $R_g(t)$ has the same dominant singularity as $S_g(t)$.

Since the congruence relation of formal power series with non-negative coefficients is stable by multiplication of a $\Delta$-analytic function on both sides, by Proposition~\ref{prop:6:asympt-T}, we thus have the following relation for $t$ near $\rho_T$:
\begin{equation} \label{eq:6:R-expr}
R_g(t(1+T_0(t))) \cong \frac1{2} c'_g(1-\rho_T^{-1})^{-5(g-1)/2-1} \left(1+O\left( (1-\rho_T^{-1}t)^{1/4}\right)\right).
\end{equation}
Let $u=t(1+T_0(t))$. We know that when $t=\rho_T$, we have $u=\rho_S$. We also know that $\rho_S = (9/8)\rho_T$. Therefore, we now study the behavior of $(1-\rho_S^{-1}u)$ for $t$ near $\rho_T$:
\begin{align*}
1-\rho_S^{-1}u &= 1-\rho_S^{-1} t \left( \frac{9}{8} - \frac{9}{8}(1-\rho_T^{-1}t) + O\left( (1-\rho_T^{-1}t)^{3/2} \right) \right) \\
&= 1 - \rho_T^{-} t + \rho_T^{-1} t (1-\rho_T^{-1}t) \left( 1 + O\left( (1-\rho_T^{-1}t)^{1/2} \right) \right) \\
&= 1 - \rho_T^{-} t + \left( (1-\rho_T^{-1}t) - (1-\rho_T^{-1}t)^2 \right) \left( 1 + O\left( (1-\rho_T^{-1}t)^{1/2} \right) \right) \\
&= 2(1-\rho_T^{-1}t) \left( 1 + O\left( (1-\rho_T^{-1}t)^{1/2} \right) \right).
\end{align*}
Therefore, we also have
\begin{equation} \label{eq:6:infinitesimal}
1-\rho_T^{-1}t = \frac1{2} (1-\rho_S^{-1}u) \left( 1+O\left((1-\rho_S^{-1}t)^{1/2}\right)\right).
\end{equation}

We now want to perform the change of variable $u=t(1+T_0(t))$ in both sides of \eqref{eq:6:R-expr}, but we need to prove that this change of variable does not break the $\Delta$-analyticity of related functions, because the congruence relation implies the existence of two $\Delta$-analytic functions, which should remain $\Delta$-analytic after the change of variable to keep the congruence relation. More precisely, we want to prove that $R_g(u)$, which is known to have non-negative real coefficients beforehand, also has radius of convergence $\rho_S$ and bounded by $\Delta$-analytic functions with matching asymptotic behavior at $\rho_S$. 

Since $R_g(u)$ has non-negative real coefficients, by Pringsheim's theorem, one of its dominant singularities lies on the real positive axis. On the other hand, $T_0(t)$ also has non-negative real coefficients, therefore $u(t)$ is increasing on the segment $[0,\rho_T]$. We further observe that $u(0)=0$ and $u(\rho_T)=\rho_S$. We can thus study the behavior of $R_g(u)$ for $u \in [0,\rho_S]$ from that of $R_g(u(t))$ for $t \in [0,\rho_T]$. From the right-hand side of \eqref{eq:6:S-to-R-alt} and the fact that $\rho_S > \rho_T$, we deduce that $R_g(u(t))$ has its first singularity at $t=\rho_T$. Therefore, $R_g(u)$ also has its first singularity at $u=\rho_S$, and it is the dominant singularity.



To prove that $R_g(u)$ is also bounded by two $\Delta$-analytic functions with matching asymptotic behavior, we first observe from Proposition~\ref{prop:6:asympt-S} that we have a congruence relation of $S_g(t)$. Combining with (\ref{eq:6:S-to-R-alt}), the series $R_g(u(t))$ is coefficient-wise bounded by two $\Delta$-analytic functions $f(t), g(t)$ for $t$ with matching asymptotic behavior. We can suppose in addition that $f(t),g(t)$ are of non-negative coefficients and share a $\Delta$-domain $D=\Delta(\rho_T,R,\theta)$. By definition, the image of $D$ under complex conjugate is itself. We now consider the image of $D$ by $u(t)$. Since $u(t)$ is a power series with real coefficients, the image $u(D)$ under complex conjugate is itself. We know that $u(t)$ is $\Delta$-analytic. By a simple computation, we deduce that $u'(t)$ is non-zero near $t=\rho_T$, which implies that $u(t)$ is conformal near $t=\rho_T$ in its $\Delta$-domain. With $u(\rho_T)=\rho_S$ and the fact that $u(t)$ is conformal near $t=\rho_T$, we deduce that the boundary of $u(D)$ near $u=u(\rho_T)=\rho_S$ is formed by two curves that are sent to each other by complex conjugate, and each of them is in an angle $\theta$ with the real axis because the function $u(t)$ is conformal at $t=\rho_T$. We can thus find a $\Delta$-domain $D' = \Delta(\rho_S,\theta',R')$ in $u(D)$, where $\theta' > \theta$. Therefore, both $f(t(u))$ and $g(t(u))$ as functions in $u$ are analytic in $D'$. 


We can now perform the substitution $u=t(1+T_0(t))$ to \eqref{eq:6:R-expr}, using \eqref{eq:6:infinitesimal} on the right-hand side for terms in the expansion. With \eqref{eq:6:coeff-multiplier} in mind, we thus have
\[
R_g(u) = c_g(1-\rho_S^{-1}u)^{-5(g-1)/2-1}\left(1+O\left((1-\rho_S^{-1}u)^{1/4}\right)\right).
\]
We also know that $f(t(u))$ and $g(t(u))$ have this asymptotic behavior near $u=\rho_S$. Therefore, their coefficients must be ultimately positive. If we apply the same reasoning on the error term, then at the cost of adding appropriate $\Delta$-analytic functions to $f(t(u))$ and $g(t(u))$, we know that the two functions bound $R_g(u)$ coefficient-wise.

We recall that, by the combinatorial definition of $R_g(u)$, it is in fact a power series in $U=u^3$ with positive coefficients. Therefore, a dominant singularity of $R_g(u)$ in $U$ translates to three dominant singularities in $u$, differing by a factor $e^{2i\pi/3}$. By the change of variable $U=u^3$ in the analysis, we can pretend that $R_g(u)$ has only one dominant singularity, without further periodic behavior. We can further assume that the two bounding functions have the same periodicity and can be treated in the same way. In this case, we don't need to be concerned about the positivity of the two bounding functions after substitution, since by their asymptotic behaviors, their coefficients must be ultimately positive. Therefore, by adding a polynomial, which does not change the asymptotic behavior nor $\Delta$-analyticity, we can assume that the two bounding functions of $R_g(u)$ have positive coefficients. We thus have the congruence
\[
R_g(u) \cong c_g(1-\rho_S^{-1}u)^{-5(g-1)/2-1}\left(1+O\left((1-\rho_S^{-1}u)^{1/4}\right)\right).
\]

We finish the proof by comparing this expression with that of $S_g$ in Proposition~\ref{prop:6:asympt-S}.
\end{proof}

Proposition~\ref{prop:6:R-asymptotic} leads to the following structural corollary. We recall that the class $\mathcal{S}_g$ is a subset of $\mathcal{R}_g$ (\textit{cf.} Table~\ref{tab:6:triangulations}).

\begin{coro} \label{coro:6:R-asymptotic}
For $g>0$, let $R$ be a triangulation chosen uniformly among triangulations with $3n$ edges in $\mathcal{R}_g$. When $n \to \infty$, the probability for $R$ to be simple is $1-O(n^{-1/4})$.
\end{coro}
\begin{proof}
We clearly have $S_g \preceq R_g$ since $\mathcal{S}_g \subseteq \mathcal{R}_g$. By applying the transfer theorem (Theorem~\ref{thm:2:transfer} in Section~\ref{sec:2:analytic}) to both $R_g(t)-S_g(t)$ and $R_g(t)$, we conclude that $[t^{3n}](R_g-S_g)(t)/[t^{3n}]R_g(t) = O(n^{-1/4})$, which is the probability that $R$ fails to be in $\mathcal{S}_g$.
\end{proof}

\subsection[Reducing $N_g(t)$ to $R_g(t)$]{Reducing $\boldsymbol{N_g(t)}$ to $\boldsymbol{R_g(t)}$}

We now prove that $N_g(t)$ has the same asymptotic behavior as $R_g(t)$, although $\mathcal{N}_g \subset \mathcal{R}_g$. This is in fact a consequence of Proposition~\ref{prop:6:R-asymptotic}. Since $\mathcal{S}_g \subseteq \mathcal{N}_g \subseteq \mathcal{R}_g$, we have $S_g(t) \preceq N_g(t) \preceq R_g(t)$. By Proposition~\ref{prop:6:R-asymptotic}, the OGFs $S_g(t)$ and $R_g(t)$ have the same behavior near their common dominant singularity, which implies that $N_g(t)$ also shares the same dominant singularity and the same behavior. However, we will prove it in another way, in order to illustrate a scheme to perform surgeries on double edges such that the resulting maps do not contain planar double edges. This surgery scheme will be very useful in our attack to the final goal $\mathcal{M}_g$. 


\begin{prop} \label{prop:6:N-asymptotic}
The OGFs $R_g(t)$ and $N_g(t)$ have the same dominant singularity $\rho_S$. Furthermore, we have $R_0(t)=N_0(t)$, and for $g \geq 1$, the OGF $N_g(t)$ has the same asymptotic behavior as $R_g(t)$ near the dominant singularity, namely
\[
N_g(t) \cong c_g(1-\rho_S^{-1}t)^{-5(g-1)/2-1}\left( 1+O\left((1-\rho_S^{-1}t)^{1/4}\right)\right).
\]
\end{prop}
\begin{proof}
For the planar case $g=0$, we have $\mathcal{N}_0 = \mathcal{R}_0$, since the only difference between $\mathcal{N}_0$ and $\mathcal{R}_0$ is that non-planar separating double edges are allowed in $\mathcal{R}_0$ but not in $\mathcal{N}_0$, but such double edges do not exist in the planar case anyway. For the case $g=1$, we also have $\mathcal{N}_1 = \mathcal{R}_1$ by the same argument. We now consider the case $g>1$.

Let $R$ be a triangulation in $\mathcal{R}_g$. If there is no non-planar separating double edge in $R$, then $R$ is also in $\mathcal{N}_g$. We now suppose that $R$ contains at least one non-planar separating double edge, denoted by $d$. We want to transform $R$ into triangulations in $\mathcal{N}_{g'}$ of smaller genus by surgeries. 

We will not directly cut along $d$ to obtain smaller triangulations, but we will use this surgery to determine a double edge that we will cut along. We temporarily cut along $d$ and zip the holes. Since $d$ is non-planar separating, the cutting splits $R$ into two components $R^{(1)}$ and $R^{(2)}$ of genus $g_1>0$ and $g_2>0$ respectively. We have $g=g_1+g_2$. Let $e_1$ be the marked edge resulted from zipping the hole on $R^{(1)}$. We now consider the set $D_1$ of planar double edges on $R^{(1)}$ that contain $e_1$. We suppose that $D_1$ is not empty. By temporarily duplicating $e_1$ to form a planar double edge, we can apply Lemma~\ref{lem:6:planar-dbl-inclusion}, which states that the set $D_1$ has a greatest element, denoted by $d_*$. Although technically $d_*$ is a double edge on $R^{(1)}$, we identify it with the double edge on $R$ before the surgery. When $D_1$ is empty, we take $d_*=d$. The upper part of Figure~\ref{fig:6:seq-planar} shows an example of this process to find $d_*$. On this figure, no matter which double edge is taken as $d$, we will always find the same $d_*$.

Once we find the correct double edge $d_*$ to cut along, we restore the triangulation $R$ and perform the surgery for $d_*$: cutting along $d_*$ and zipping the two holes created in the cutting. We thus obtain two triangulations $R^{(1)}_*$ and $R^{(2)}_*$, where $R^{(1)}_*$ is contained in $R^{(1)}$ if considered as set of faces of $R$. Let $e_*^{(1)}$ and $e_*^{(2)}$ be the edges obtained from zipping the hole respectively on $R^{(1)}_*$ and $R^{(2)}_*$. We will now analyze these two triangulations.

The case of $R^{(1)}$ is relatively simple. By maximality of $d_*$, if any planar double edge exists on $R^{(1)}_*$, it can not contain $e_*^{(1)}$, therefore is already planar in $R$, which is impossible. Therefore, $R^{(1)}_*$ is an element in $\mathcal{R}_{g_1}$. 

The case for $R^{(2)}_*$ is more complicated, since it may contain planar double edges that need to be eliminated. Figure~\ref{fig:6:seq-planar} illustrates the surgeries we will be performing on $R^{(2)}_*$. If such double edges exist, their corresponding planar component contains $e_*^{(2)}$, or else they would already be planar double edges on $R$. We then pick a minimal double edge, denoted by $d_{*,1}$, with respect to the order $<_{R_*^{(2)}}$. We then cut along $d_{*,1}$ and zip the two holes, and we obtain a planar triangulation $S_1$ and a triangulation $R^{(2)}_{*,1}$ of genus $g_2$. By minimality of $d_{*,1}$, there is no planar double edge in $S_1$, therefore $S_1$ is in $\mathcal{S}_0$. We then perform this procedure repeatedly. We consider the surgery on $d_*$ as the $0^{\rm th}$ surgery. For the triangulation $R^{(2)}_{*,i}$ of genus $g_2$ obtained after performing the $i^{\rm th}$ surgery, let $e_{*,i-1}$ be the edge on $R^{(2)}_{*,i}$ obtained from zipping the hole, we cut along a minimal planar double edge containing $e_{*,i-1}$ in the partial order $<_{R_*^{(2)}}$ described in Lemma~\ref{lem:6:planar-dbl-inclusion}, then zip the holes to obtain a planar triangulation $S_{i+1}$ in $\mathcal{S}_0$ and a triangulation $R^{(2)}_{*,i+1}$ possibly with planar double edges. This procedure is repeated until we obtain some $R^{(2)}_{*,k}$ that has no planar double edge, which means that $R^{(2)}_{*,k}$ is in $\mathcal{R}_{g_2}$. When $k=0$, \textit{i.e.} $R^{(2)}_*$ is already in $\mathcal{R}_{g_2}$, we just take $R^{(2)}_{*,0} = R^{(2)}_*$. We thus obtain from $R^{(2)}_*$ a (possibly empty) sequence of planar triangulations $S_1,S_2,\ldots,S_k$ and a triangulation $R^{(2)}_{*,k}$ in $\mathcal{R}_{g_2}$. Figure~\ref{fig:6:seq-planar} illustrates how these triangulations are obtained. We also observe that the total number of edges and faces remains unchanged after these surgeries.

We now discuss roots and markings on triangulations $R^{(1)}_*$, $(S_1,\ldots,S_k)$ and $R^{(2)}_{*,k}$ obtained from this series of surgeries. Since we only cut along separating double edges, the second case of Lemma~\ref{lem:6:surgery-2} always applies, and we will have a 2-to-2 correspondence for each surgery. Without loss of generality, we can assume that the root corner is contained in $R^{(1)}$ obtained from tentative surgery of cutting along $d$, which implies that $R^{(2)}_{*,k}$ does not contain the original root corner. Therefore, the original root corner must be contained in either $R^{(1)}_*$ or in one of the planar components $S_i$. In the first case, $R^{(1)}_*$ is rooted at the original root corner with one marked edge, each $S_i$ is rooted at the zipped edge of the $(i-1)^{\rm th}$ surgery and has one marked edge that is the zipped edge of the $i^{\rm th}$ surgery, and finally $R^{(2)}_{*,k}$ is rooted at the zipped edge of the $k^{\rm th}$ surgery. Similarly, in the second case, both $R^{(1)}_*$ and $R^{(2)}_{*,k}$ are rooted at the zipped edges of surgeries, and each $S_i$ is similarly rooted and has one marked edge, except the $S_i$ that contains the original root corner, which is rooted at the original root corner and contains two distinct marked edges. We have thus finished describing a procedure to dissect a triangulation $R$ in $\mathcal{R}_g \setminus \mathcal{N}_g$ with $g>1$ and a non-planar separating double edge $d$ in $R$ to obtain two triangulations $R^{(1)}_*$ and $R^{(2)}_{*,k}$ in $\mathcal{R}_{g_1}$ and $\mathcal{R}_{g_2}$ respectively and a sequence of planar triangulations $(S_1, \ldots, S_k)$ in $\mathcal{S}_0$, such that $g_1, g_2>1$ and $g_1+g_2=g$. By Lemma~\ref{lem:6:surgery-2} and the fact that we can choose the starting double edge $d$, this surgery procedure is an injection, which gives us the following bound of the series $R_g(t)-N_g(t)$:
\begin{equation} \label{eq:6:R-to-N}
0 \preceq R_g(t) - N_g(t) \preceq \sum_{\substack{g_1+g_2=g \\ g_1,g_2\geq 1}} \left( tR_{g_1}'(t) R_{g_2}(t) \frac1{1-tS_0'(t)} + R_{g_1}(t)R_{g_2}(t) \frac{t^2S_0''(t)}{(1-tS_0'(t))^2} \right).
\end{equation}

Here, the first term in the sum corresponds to the case where the original root corner is contained in $R^{(1)}_*$, and the second term to the other case. The OGF of triangulations with a marked edge is obtained by first differentiating by $t$ then multiplying by $t$. For the special case of two distinct marked edges, we simply differentiate twice by $t$ and then multiply by $t^2$ to make up the right power. In the first case, the factor $\frac1{1-tS_0'(t)}$ stands for the (possibly empty) sequence of planar triangulations $(S_1, \ldots, S_k)$. In the second case, we need to figure out the contribution of $(S_1, \ldots, S_k)$ with a certain $S_i$ containing two distinct marked edges, which is equivalent to two sequences $(S_1, \ldots, S_{i-1})$ and $(S_{i+1},\ldots,S_k)$ along with $S_i$. This gives the factor $\frac{t^2S_0''(t)}{(1-tS_0'(t))^2}$.

We now analyze the terms on the right-hand side of \eqref{eq:6:R-to-N}. Proposition~\ref{prop:6:R-asymptotic} gives the following asymptotic behavior of $R_g(t)$ near the dominant singularity $\rho_S$: 
\[
R_g(t) \cong \Theta\left((1-\rho_S^{-1}t)^{-5(g-1)/2-1}\right).
\]
Since the congruence relation is stable by differentiation, we have
\[
R_g'(t) \cong \Theta\left((1-\rho_S^{-1}t)^{-5(g-1)/2-2}\right).
\]
From \eqref{eq:6:S0} in Proposition~\ref{prop:6:asympt-S}, with a small computation, near the singularity $\rho_S$, we have
\begin{align}
S_0'(t) &= \frac{9}{16}\rho_S^{-1} - \frac{9}{2^{7/2}} \rho_S^{-1} (1-\rho_S^{-1}t)^{1/2} + O(1-\rho_S^{-1}t), \label{eq:6:S0-prime} \\
S_0''(t) &= \frac{9}{2^{9/2}} \rho_S^{-2} (1-\rho_S^{-1}t)^{-1/2} + O(1). \label{eq:6:S0-prime2}
\end{align}
We remark that $S_0'(t)$ is finite at the dominant singularity, and $\rho_S S_0'(\rho_S) = 9/16 < 1$. Since $tS_0'(t)$ is a formal power series with non-negative coefficients, it is increasing from $0$ to $\rho_S$, meaning that $(1-tS_0'(t))^{-1}$ is finite in the interval $[0,\rho_S]$, including the dominant singularity $\rho_S$. These series are also $\Delta$-analytic. We thereby substitute these asymptotic formulas into \eqref{eq:6:R-to-N}, which gives
\begin{align*}
0 \preceq R_g(t) - N_g(t) &\preceq \sum_{\substack{g_1+g_2=g \\ g_1,g_2\geq1}} \bigg( \Theta\left( (1-\rho_S^{-1}t)^{(-5(g_1-1)/2-2) + (-5(g_2-1)/2-1)} \right) \\
&\quad + \Theta\left( (1-\rho_S^{-1}t)^{(-5(g_1-1)/2-1) + (-5(g_1-1)/2-1) -1/2} \right) \bigg) \\
&= (g-1) \left( \Theta\left( (1-\rho_S^{-1}t)^{-5(g-1)/2-1/2} \right) + \Theta\left( (1-\rho_S^{-1}t)^{-5(g-1)/2} \right) \right) \\
&= \Theta\left( (1-\rho_S^{-1}t)^{-5(g-1)/2-1/2} \right). \\
\end{align*}
Comparing to the asymptotic of $R_g$ in Proposition~\ref{prop:6:R-asymptotic}, we see that $R_g(t)-N_g(t)$ is dominated by the error term in the asymptotic expression of $R_g$, which is $O\left( (1-\rho_S^{-1}t)^{-5(g-1)-3/4}\right)$. Therefore, $N_g(t) \cong R_g(t)$, and we conclude the proof by comparing to Proposition~\ref{prop:6:R-asymptotic}.
\end{proof}

\begin{figure}
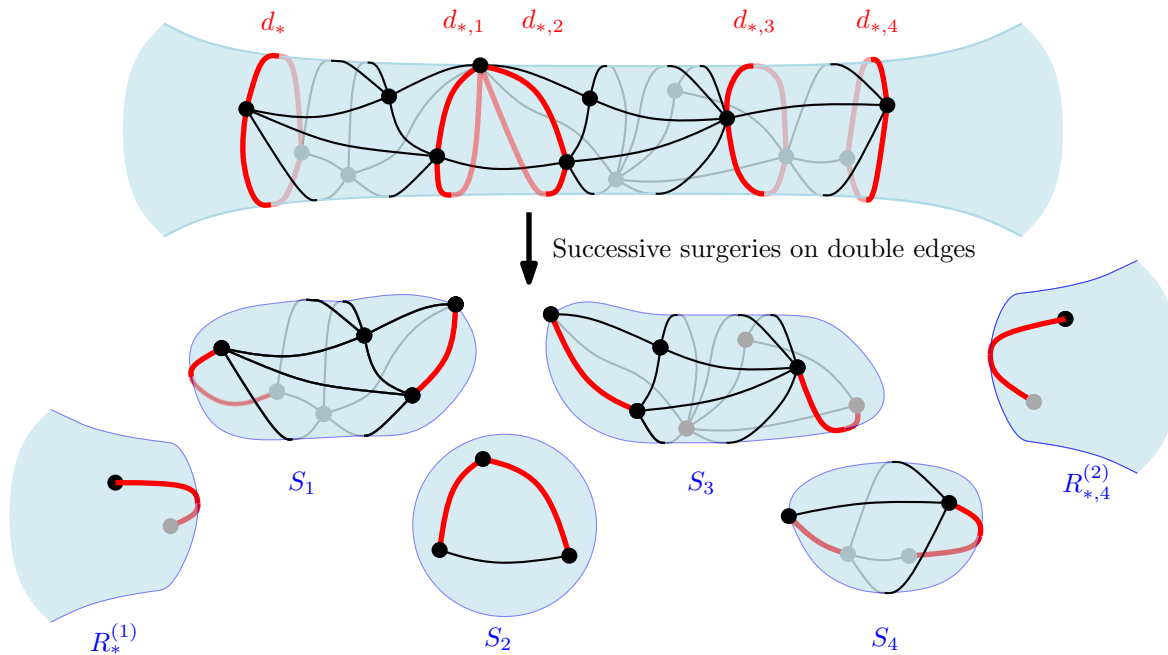

  \centering
  \insertfigure{ch6-fig.pdf}{13}
  \caption{Cutting along a double edge and chasing the planar parts}
  \label{fig:6:seq-planar}
\end{figure}

In the proof of Proposition~\ref{prop:6:N-asymptotic}, after cutting along a non-planar separating double edge, we obtain two components with marked edges. We have chosen the double edge to cut along such that only one component may contain planar double edges after the surgery. For the other component, all double edges that become planar after a surgery must contain the new edge $e$ that comes from zipping a double edge in the surgery. By successively taking a minimal element with respect to the inclusion order among all such double edges and cutting along them, we obtain a sequence of planar triangulations, and the remaining part of the component no longer has any planar double edges. We refer to this method of eliminating planar double edges after a surgery as \mydef{chasing the planar parts} after the edge $e$, which will be used frequently in the proof of the asymptotic enumeration of $\mathcal{M}_g$.

\subsection[Reducing $M_g(t)$ to $N_g(t)$ : loop elimination]{Reducing $\boldsymbol{M_g(t)}$ to $\boldsymbol{N_g(t)}$: loop elimination}

In this subsection, we will prove the asymptotic behavior of $M_g(t)$ in Proposition~\ref{prop:6:M-asymptotic}, using results from previous subsection. We start by introducing a family of triangulations with boundaries that will be used in the proof of the asymptotic enumeration of $\mathcal{M}_g$. We recall that a triangulation with boundaries is a map with boundaries such that all faces (excluding boundaries) are of degree $3$. For double edges on maps with boundaries, we say that they are planar if by cutting along them and zipping a hole, we obtain two connected components and one of them is a planar map \emph{without boundaries}. In other words, a planar double edge cannot contain any boundary in its planar component. We denote by $\mathcal{R}_{g,h}$ the set of triangulations of genus $g$ having exactly $h\geq0$ \emph{ordered} boundaries, all of length 1, while containing no planar double edge nor loop that is not a boundary. We have $\mathcal{R}_{g,0}=\mathcal{R}_g$. Similarly, for $h>0$ we denote by $\mathcal{R}_{g,h}^*$ the subset of $\mathcal{R}_{g,h}$ formed by triangulations rooted at their minimal boundary loop with the clockwise orientation. The particular choice of the orientation is to make sure that the corresponding root corner is not inside the boundary. The OGFs of $\mathcal{R}_{g,h}$ and $\mathcal{R}_{g,h}^*$ are denoted by $R_{g,h}$ and $R_{g,h}^*$ respectively.

The reason of introducing these families of maps with boundaries is that, in the following, we are going to eliminate loops from triangulations in $\mathcal{M}_g$ by cutting along them, which gives triangulations with boundaries of length $1$, captured by the classes $\mathcal{R}_{g,h}$. Since these triangulations are the building blocks of our maps in $\mathcal{M}_g$, it will be convenient later to know about their asymptotic enumeration. We now look at the asymptotic behavior of $R_{g,h}$ and $R_{g,h}^*$.

\begin{prop} \label{prop:6:R-holes-asymptotic}
For $g\geq0$ and $h>0$, the OGFs $R_{g,h}(t)$, $R_{g,h}^*(t)$ and $R_g(t)$ have the same dominant singularity $\rho_S$. Furthermore, we have 
\begin{align*}
R_{g,h}(t) &\cong \Theta \left( (1-\rho_S^{-1}t)^{-5(g-1)/2-h-1} \right), \\
R_{g,h}^*(t) &\cong \Theta \left( (1-\rho_S^{-1}t)^{-5(g-1)/2-h} \right).
\end{align*}
\end{prop}
\begin{proof}
Given a triangulation with boundaries $R \in \mathcal{R}_{g,h}$, we first analyze the triangular face adjacent to a boundary delimited by a loop $\ell$ adjacent to a vertex $u$ on $R$. Since the loop forms one side of the triangle, the other two sides must also be adjacent to $u$. There are three cases for these two sides: they are either two loops, or a double edge, or one single edge. Given a triangular faces in one of these three cases, by cutting along all its adjacent edges, we obtain a triangulation with boundaries with only one face, and the three possibilities are denoted by $\Psi$, $\Lambda_0$ and $\Lambda$ respectively. Figure~\ref{fig:6:one-loop-triangle} illustrates these three cases. We can see that we can obtain $\Lambda$ from $\Lambda_0$ by zipping the double edge.

\begin{figure}
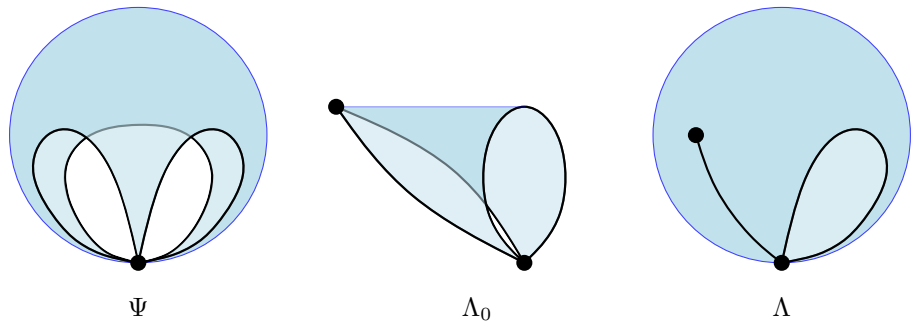

  \centering
  \insertfigure{ch6-fig.pdf}{14}
  \caption{Three cases of the triangular face with a loop as its side}
  \label{fig:6:one-loop-triangle}
\end{figure}

Except for the case $g=0,h=3$ where we can have $\Psi \in \mathcal{R}_{0,3}$, and the case $g=0,h=1$ where we have $\Lambda \in \mathcal{R}_{0,1}$, a loop in a map in $\mathcal{R}_{g,h}$ must be adjacent to a triangular face of the form $\Lambda_0$.

For any fixed $g \geq 0$, we first analyze the asymptotic behavior of $R_{g,h}(t)$ with an induction on the number of holes $h$. The base case $h=0$ is trivial, since we have $\mathcal{R}_{g,0}=\mathcal{R}_g$ in this case. We now fix the value of $h$, and we suppose that our proposition is valid for all $R_{g,h-1}(t)$, that is,
\[
R_{g,h-1}(t) \cong \Theta \left( (1-\rho_S^{-1}t)^{-5(g-1)/2-h} \right).
\]

Let $R$ be a map in $\mathcal{R}_{g,h}$, and $\ell$ the loop that delimits the first boundary. We suppose that $R$ is neither $\Psi$ nor $\Lambda$. The face adjacent to $\ell$ must be $\Lambda_0$ and the other two sides form a double edge $d$. We now perform some surgeries to obtain a counting result. We first cut along $d$ and zip the two holes, obtaining $\Lambda$ from the triangular face and $R'$ the remaining part with a marked edge $e$ resulted from zipping. Since we eliminate a hole using surgery, there may now be planar double edges in $R'$, which must contain $e$ in their planar component. We then chase the planar parts after the edge $e$, as in the proof of Proposition~\ref{prop:6:N-asymptotic}. We thus obtain a (possibly empty) sequence $(S_1,S_2,\ldots,S_k)$ of planar triangulations in $\mathcal{S}_0$ and a triangulation $R''$ with one less loop but the same genus and without loops nor planar double edges. Therefore, $R''$, when rooted, is in $\mathcal{R}_{g,h-1}$.

We now deal with the rooting issue. There are three cases for the location of the original root corner of $R$:
\begin{enumerate}[(a)]
\item On $\Lambda$;
\item On one of the planar triangulations $S_i$;
\item On $R''$.
\end{enumerate}

We denote by $R_{g,h}^{(a)}(t),R_{g,h}^{(b)}(t),R_{g,h}^{(c)}(t)$ the OGFs of triangulations in each case respectively. We thus have
\begin{equation} \label{eq:6:R-holes-decomp}
R_{g,h}(t) = t^3\boldsymbol{1}_{g=0,h=3} + t^2\boldsymbol{1}_{g=0,h=1} + R_{g,h}^{(a)}(t) + R_{g,h}^{(b)}(t) + R_{g,h}^{(c)}(t).
\end{equation}
We now analyze the asymptotic behavior of the OGFs for all cases, and we will see that Case~(c) is dominant and gives the asymptotic behavior of $R_{g,h}(t)$.

We now analyze the rooting of each case. We recall that we always perform the surgery in Lemma~\ref{lem:6:surgery-2} on separating double edges, then zip the holes created in the surgery. This sequence of surgeries does not alter the total number of edges. Therefore, when we look at each edge resulted from zipping, the component containing the original root corner takes it as a marked edge, while the other component is rooted at that edge. We also recall that the surgery in Lemma~\ref{lem:6:surgery-2} is a 2-to-2 correspondence. In all three cases, we will chase the planar parts (\textit{cf.} Proposition~\ref{prop:6:N-asymptotic} and the explanation that follows) by applying the surgery in Lemma~\ref{lem:6:surgery-2} repeatedly.

In Case~(a), we obtain from surgeries a rooted copy of $\Lambda$ with the non-loop edge marked, a sequence of rooted planar triangulations in $\mathcal{S}_0$ with a marked edge, and $R''$ a rooted triangulation in $\mathcal{R}_{g,h-1}$. We thus have
\[
R_{g,h}^{(a)}(t) = \frac{3t^2}{1-tS_0'(t)} R_{g,h-1}(t).
\]
The term $3t^2$ accounts for the 3 root possibilities of $\Lambda$, which has two edges and 3 possible root corners. Since the marked edge on $\Lambda$ is fixed, there is no extra factor for its choice. The term $(1-tS_0'(t))^{-1}$ accounts for elements in $\mathcal{S}_0$ with a marked edge. The term $R_{g,h-1}(t)$ is for $R''$. From Proposition~\ref{prop:6:asympt-S} (see \eqref{eq:6:S0-prime}), we know that $(1-tS_0'(t))^{-1}$ is finite at the dominant singularity $\rho_S$. Therefore, by the induction hypothesis, near the singularity $\rho_S$, we have
\begin{equation} \label{eq:6:R-holes-a}
R_{g,h}^{(a)}(t) \cong \Theta\left( (1-\rho_S^{-1}t)^{-5(g-1)/2-h} \right).
\end{equation}

In Case~(b), we obtain from surgeries a copy of $\Lambda$ that is rooted at the non-loop edge, a first sequence of rooted planar triangulations in $\mathcal{S}_0$ with a marked edge, the planar component in $\mathcal{S}_0$ that is rooted at the original root corner and has two marked edges, a second sequence of rooted planar triangulations with a marked edge, and $R''$ a rooted triangulation in $\mathcal{R}_{g,h-1}$. We thus have
\[
R_{g,h}^{(b)}(t) = 2t^2 \frac{t^2S''_0(t)}{(1-tS_0'(t))^2} R_{g,h-1}(t).
\]
The factor $2t^2$ is the contribution of $\Lambda$, which can be rooted in two ways at the non-loop edges. The factor $(1-tS_0'(t))^{-2}$ accounts for the two sequences of planar triangulations with marked edge, while the factor $t^2S''_0(t)$ is for the planar component with two marked edges. Finally, the possibilities of $R''$ are covered by $R_{g,h-1}(t)$. Again from Proposition~\ref{prop:6:asympt-S} (see \eqref{eq:6:S0-prime2}), we know that $S_0''(t) = \Theta\left( (1-\rho_S^{-1}t)^{-1/2} \right)$ near the dominant singularity $\rho_S$, and $(1-tS_0'(t))^{-2}$ is finite at its dominant singularity $\rho_S$. Therefore, by the induction hypothesis, after accounting for $S_0''(t)$, near the singularity $\rho_S$, we have
\begin{equation} \label{eq:6:R-holes-b}
R_{g,h}^{(b)}(t) \cong \Theta \left( (1-\rho_S^{-1}t)^{-5(g-1)/2-h-1/2} \right).
\end{equation}

Finally, in Case~(c), we obtain from surgeries a copy of $\Lambda$ that is rooted at the non-loop edge, a sequence of rooted planar triangulations in $\mathcal{S}_0$ with a marked edge, and $R''$ a rooted triangulation in $\mathcal{R}_{g,h-1}$ with a marked edge. We thus have
\[
R_{g,h}^{(c)}(t) = \frac{2t^2}{1-tS_0'(t)} t R_{g,h-1}'(t).
\]
The composition of the factors should be clear by comparing to previous cases. For the behavior near the singularity $\rho_S$, always by the induction hypothesis, after accounting for the differentiation of $R_{g,h-1}'(t)$, we have
\begin{equation} \label{eq:6:R-holes-c}
R_{g,h}^{(c)}(t) \cong \Theta \left( (1-\rho_S^{-1}t)^{-5(g-1)/2-h-1} \right).
\end{equation}

Substituting \eqref{eq:6:R-holes-a}, \eqref{eq:6:R-holes-b} and \eqref{eq:6:R-holes-c} back to \eqref{eq:6:R-holes-decomp}, we see that the dominant term is $R_{g,h}^{(c)}$. Since all OGFs here are either $\Delta$-analytic or congruent to a $\Delta$-analytic function, we have
\[
R_{g,h}(t) \cong \Theta \left( (1-\rho_S^{-1}t)^{-5(g-1)/2-h-1} \right).
\]
We thus conclude the proof for the estimate of $R_{g,h}(t)$.

For the case of $R_{g,h}^*(t)$, we simply observe that it accounts for one third of the possibilities in Case~(a) of the analysis of $R_{g,h}(t)$, and the result follows from the previous analysis of $R_{g,h}^{(a)}(t)$.
\end{proof}

Using Proposition~\ref{prop:6:R-holes-asymptotic}, we can prove the main result of this section, which is the asymptotic enumeration of $\mathcal{M}_g$.

\begin{proof}[Proof of Proposition~\ref{prop:6:M-asymptotic}]
By Proposition~\ref{prop:6:R-asymptotic} and Proposition~\ref{prop:6:N-asymptotic}, the series $N_g$ have the same asymptotic as $S_g$. Therefore, we only need to compare $M_g$ with $N_g$.

For the planar case, since any loop will be separating, we have $\mathcal{M}_0 = \mathcal{N}_0$, which implies $M_0(t)=N_0(t)$. For the case $g \geq 1$, the only difference between $\mathcal{M}_g$ and $\mathcal{N}_g$ is that some loops are allowed in triangulations in $\mathcal{M}_g$, namely those that are not separating nor forming a separating pair. We recall that a separating pair of loop is a pair of non-separating loops such that cutting along both separate the surface. We thus want to eliminate any existing loop by surgeries. 

Let $M$ be a triangulation in $\mathcal{M}_g \setminus \mathcal{N}_g$. By definition, $M$ contains at least a loop. Let $L = (\ell_1, \ell_2, \ldots, \ell_m)$ be all the loops in $M$, ordered arbitrarily. We now cut along these loops one by one in the order of $L$. The surgery of cutting along a loop may split the two ends of another loop, which then ceases to be a loop. In this case, we don't perform surgeries on it. Since no loop can be created in these surgeries, at the end we obtain a set $P(M)$ of triangulations with boundaries of length 1, which are all delimited by loops. We order the boundaries by the time they are created from surgeries, and for two boundaries created from the same surgery, we order them arbitrarily.


We observe that elements in $P(M)$ are all triangulations with boundaries of size $1$. Let $M^{(0)}$ be an element in $P(M)$. By definition, all loops in $M^{(0)}$ delimit a boundary. For planar double edges on maps with boundaries, we say that their planar components cannot contain any boundary. Under this definition, there is no planar double edge in $M^{(0)}$. Therefore, $M^{(0)}$ is an element in $\mathcal{R}_{g',h'}$ for some $g'<g$ and $h'$. 

We now discuss the possibilities for $P(M)$. There are two cases: either $P(M)$ has only one element, or it has multiple elements. Let $\mathcal{M}_g^{(1)}$ be the set of triangulations in $\mathcal{M}_g$ in the first case, and $\mathcal{M}_g^{(2)}$ the set of those in the second case. We denote by $M_g^{(1)}$ and $M_g^{(2)}$ their respective OGFs. We thus have $\mathcal{M}_g = \mathcal{N}_g \uplus \mathcal{M}_g^{(1)} \uplus \mathcal{M}_g^{(2)}$, which leads to
\begin{equation} \label{eq:6:decomp-M-1}
M_g(t) = N_g(t) + M_g^{(1)}(t) + M_g^{(2)}(t).
\end{equation}

In the case where $P(M)$ has only one element, we denote this element by $M^{(1)}$. Let $h$ be the number of loops that we have cut along, then $M^{(1)}$ has $2h$ ordered boundaries and is of genus $g-h$. Therefore, we have $h \leq g$, and $M^{(1)}$ is an element of $\mathcal{R}_{g-h,2h}$. By Proposition~\ref{prop:6:R-holes-asymptotic}, we have
\begin{align} \label{eq:6:M-1}
\begin{split}
0 \preceq M_g^{(1)}(t) &\preceq \sum_{h=1}^{g} R_{g-h,2h}(t) = \sum_{h=1}^g \Theta\left( (1-\rho_S^{-1}t)^{-5(g-1)/2-1+h/2} \right) \\
&= \Theta\left( (1-\rho_S^{-1}t)^{-5(g-1)/2-1/2} \right).
\end{split}
\end{align}

We now deal with the case where $P(M)$ contains more than one element. We denote its elements by $M^{(1)} ,\ldots, M^{(k)}$ for $k \geq 2$, each of genus $g_1, \ldots, g_k$, with $h_1, \ldots, h_k$ holes of size 1 on each. The total number of loops that are cut along is $h = (h_1 + \cdots + h_k)/2$.

We now consider the number of boundaries that each component $M^{(i)}$ contains. For each $M^{(i)}$, it is clear that it cannot contain only one hole, or else the associated loop is separating in $M$, which should not exist due to the definition of $\mathcal{M}_g$. If $M^{(i)}$ contains two boundaries, there are two cases: they either come from the same loop or from two different loops. In the first case, the loop is non-separable, and there is only one component in $P(M)$, which contradicts our assumption. In the second case, the two different loops thus form a separating pair, which is against the definition of $\mathcal{M}_g$. Therefore, the two cases do not stand, and each $M^{(i)}$ contains at least three boundaries. We thus have
\begin{equation} \label{eq:6:M-comp-holes}
h = \frac1{2} \sum_{i=1}^k h_i \geq \frac{3k}{2}.
\end{equation}
On the other hand, since there are $k$ connected components, there must be $k-1$ loops that are separating when we cut along the loops. Therefore, there are $h-k+1$ loops that are non-separating when we cut along them, and each of them decreases the total genus by 1. Therefore, we have the following relation on the total genus after all the surgeries:
\begin{equation} \label{eq:6:M-comp-genus}
\sum_{i=0}^k g_i = g-h+k-1.
\end{equation}
Combining \eqref{eq:6:M-comp-holes} and \eqref{eq:6:M-comp-genus}, by observing that all the genera $g_i$'s are non-negative, we have
\begin{align} \label{eq:6:M-comp-bound}
g \geq h-k+1 \geq k/2+1, \quad \; k \leq 2(g-1), \quad \; h \leq 3(g-1).
\end{align}
Therefore, $P(M)$ has more than 2 elements only if $g \geq 2$, and in this case it has at most $2(g-1)$ elements. It is also easy to see that $h \geq 3k$, since each $M^{(i)}$ contains at least three boundaries.

We now come to the discussion of roots on each components $M^{(i)}$. There will be one element that contains the original root corner, and without loss of generality, we can suppose that it is $M^{(1)}$. Therefore, $M^{(1)}$ is an element in $\mathcal{R}_{g_1,h_1}$ as defined previously and analyzed in Proposition~\ref{prop:6:R-holes-asymptotic}. For other components, we first recall that there is an order on their boundaries. For the sake of indicating the orientation of a component $M^{(i)}$ with $i \neq 1$, we root it at its minimal boundary loop such that the root corner is not inside the boundary. We thus see that $M^{(i)}$ is in $\mathcal{R}_{g_i,h_i}^*$. For the reconstruction of $M$ from $P(M)$, we still need the information about how boundaries are paired up. Since boundaries on each $M^{(i)}$ are ordered, the number of ways to combine them into an ordered list $b_1, b_2, \ldots, b_{2h}$ is counted by the multinomial $\binom{2h}{h_1, h_2, \ldots, h_k}$, and we say that we will glue up the boundaries $b_{2k-1}$ and $b_{2k}$ for all $1 \leq k \leq h$. It is clear that every possible gluing reconstruction from $P(M)$ to $M$ can be obtained in this way. We should notice that we over-count the number of ways to reconstruct $M$, because some possible boundary pairings may repeat, and some pairings we obtain will not give a connected map. For instance, if we have several components in $P(M)$, one of them with four holes, then any pairing that leads to gluing these four holes up against each other will not be valid. Nevertheless, this over-counting will be sufficient for an upper bound of $M_g^{(2)}(t)$. We recall from Proposition~\ref{prop:6:R-holes-asymptotic} that $R_{g,h}(t) \cong \Theta\left( (1-\rho_S^{-1}t)^{-5(g-1)/2-h-1} \right)$ and $R_{g,h}^*(t) \cong \Theta\left( (1-\rho_S^{-1}t)^{-5(g-1)/2-h} \right)$ near their dominant singularity $\rho_S$. Using \eqref{eq:6:M-comp-bound}, we thus have
\begin{align} 
\begin{split}\label{eq:6:M-2}
M_g^{(2)}(t) &\preceq \sum_{k=2}^{2(g-1)} \sum_{h=\lceil 3k/2 \rceil}^{\lfloor 3(g-1)/2 \rfloor} \sum_{\substack{g_1+\cdots+g_k = g-h+k-1, g_i \geq 0 \\ h_1+\cdots+h_k = h, h_i \geq 3}} \binom{2h}{h_1, h_2, \ldots, h_k} R_{g_1,h_1}(t) \prod_{i=2}^k R_{g_i,h_i}^*(t) \\
&\preceq \sum_{k=2}^{2(g-1)} \sum_{h=\lceil 3k/2 \rceil}^{\lfloor 3(g-1)/2 \rfloor} \sum_{\substack{g_1+\cdots+g_k = g-h+k-1, g_i \geq 0 \\ h_1+\cdots+h_k = 2h, h_i \geq 3}} \Theta\left( (1-\rho_S^{-1}t)^{-5\left(\sum_{i=1}^k g_i -k \right)/2 - \sum_{i=1}^k h_i -1 } \right) \\
&= \sum_{k=2}^{2(g-1)} \sum_{h=\lceil 3k/2 \rceil}^{\lfloor 3(g-1)/2 \rfloor} \Theta\left( (1-\rho_S^{-1}t)^{-5(g-h-1)/2 - 2h+1} \right) \\
&= \Theta\left( (1-\rho_S^{-1}t)^{-5(g-1)/2 + 1/2} \right).
\end{split}
\end{align}
We should note that $R_{g,h}(t)$ and $R_{g,h}^*(t)$ are all bounded by some $\Delta$-analytic function. 

We recall that the three terms that add up to $M_g(t)$ in \eqref{eq:6:decomp-M-1} are $N_g(t)$, $M_g^{(1)}(t)$ and $M_g^{(2)}(t)$. By comparing Proposition~\ref{prop:6:N-asymptotic} to \eqref{eq:6:M-1} and \eqref{eq:6:M-2}, we see that they all have the same dominant singularity $\rho_S$, and the dominant term is $N_g(t)$. Therefore,
\[
N_g(t) \preceq M_g(t) \preceq N_g(t) + \Theta \left( (1-\rho_S^{-1}t)^{-5(g-1)/2-1/2} \right).
\]
But Proposition~\ref{prop:6:N-asymptotic} states that
\[
N_g(t) \cong c_g(1-\rho_S^{-1}t)^{-5(g-1)/2-1}\left( 1+O\left((1-\rho_S^{-1}t)^{1/4}\right)\right).
\]
We observe that even the error term dominates $M_g^{(1)}(t)$ and $M_g^{(2)}(t)$. We therefore conclude the proof.
\end{proof}

Similar to Corollary~\ref{coro:6:R-asymptotic}, we have the following corollary for the probability of a triangulation in $\mathcal{M}_g$ to be simple for $g>0$. The case $g=0$ is clear, since $\mathcal{M}_0 = \mathcal{N}_0 = \mathcal{S}_0$, therefore every triangulation in $\mathcal{M}_0$ is simple.

\begin{coro} \label{coro:6:M-asymptotic}
For $g>0$, let $M$ be a triangulation chosen uniformly among triangulations with $3n$ edges in $M_g$. When $n \to \infty$, the probability for $M$ to be simple is $1-O(n^{-1/4})$.
\end{coro}
\begin{proof}
We clearly have $\mathcal{S}_g \subset \mathcal{M}_g$. By applying the transfer theorem (Theorem~\ref{thm:2:transfer} in Section~\ref{sec:2:analytic}), we conclude that $[t^{3n}](M_g(t)-S_g(t))/[t^{3n}]M_g(t) = O(n^{-1/4})$, which is the probability that $M$ fails to be in $\mathcal{S}_g$.
\end{proof}


\section{Controlling widths}

Proposition~\ref{prop:6:M-asymptotic} gives us the asymptotic enumeration result on $\mathcal{M}_g$. According to our strategy, which is illustrated in Figure~\ref{fig:6:strategy}, the next step, which is represented by the only arrow between the two sides, will be using the unique embedding theorem of Robertson and Vitray (Theorem~\ref{thm:6:robertson-vitray}) to transfer the asymptotic enumeration of maps to graphs. However, Theorem~\ref{thm:6:robertson-vitray} only applies to cubic graphs embeddable on \Sg{} that have facewidth at least $2g+3$. Therefore, we need to control the facewidth of the embeddings of these cubic graphs, which are duals of triangulations that we have counted in Proposition~\ref{prop:6:M-asymptotic}. This is done by the following proposition.

\begin{prop} \label{prop:6:facewidth-dual}
Let $M$ be a cubic map and $M^*$ its dual triangulation. We have $\fw(M) = \ew(M^*)$.
\end{prop}
\begin{proof}
We first observe that, in a cubic map $M$, if two faces $f_1, f_2$ share an adjacent vertex, then they also share an adjacent edge. Let $S$ be a set of faces in $M$ that contains a non-contractible circle $C$. Suppose that $C$ passes through a certain vertex $v$ to go from a face $f_1$ to another face $f_2$. By the previous observation, $f_1,f_2$ share an adjacent edge $e$, and we can modify $C$ such that it goes from $f_1$ to $f_2$ by crossing $e$. By applying this modification whenever $C$ passes through a vertex, we obtain a non-contractible circle $C'$ that does not pass through any vertex. Therefore, we now only consider non-contractible circles that does not cross any vertex. Let $S^*$ be the set of vertices of $M^*$ that are dual of the faces in $S$. Then the circle $C$ that crosses faces and edges on $M$ corresponds to a cycle $C^*$ on dual vertices and edges of $M^*$. It is clear that $C^*$ is also non-contractible on $M^*$, and the length of $C^*$ is at most the size of $S$, since we can always make $C$ to cross a face just once. Therefore, we have $\fw(M) \geq \ew(M^*)$. On the other hand, for a non-contractible cycle $C^*$ on vertices of $M^*$, if we consider it on the cubic map $M$, it is contained in the set of faces that are dual to the vertices it contains on $M^*$. Therefore, $\fw(M) \leq \ew(M^*)$ and we have the equality.
\end{proof}

By Proposition~\ref{prop:6:facewidth-dual}, cubic maps of facewidth at least $2g+3$ correspond to triangulations of edgewidth at least $2g+3$. We denote by $\mathcal{M}^{\ew \geq k}_g$ the set of triangulations in $\mathcal{M}_g$ with edgewidth at least $k$, and by $M^{\ew \geq k}_g(t)$ its OGF. We are thus interested in the asymptotic enumeration of $\mathcal{M}_g^{\ew \geq 2g+3}$, since by Theorem~\ref{thm:6:robertson-vitray}, their duals, which are 3-connected cubic maps with facewidth at least $2g+3$, are unique embeddings of their underlying graphs. The following proposition shows that this restriction on edgewidth has no effect on the asymptotic number of triangulations that we count. Similar but much stronger results for other kinds of maps can be found in \cite{log-facewidth, 3conn-fw-ew}.

\begin{prop} \label{prop:6:M-ew-asymptotic}
For fixed integers $g>0$ and $k>3$, the OGF $M_g^{\ew \geq k}(t)$ has the same dominant singularity $\rho_S$ as $M_g(t)$, and also the same asymptotic behavior near $\rho_S$, namely
\[
M_g^{\ew \geq k}(t) \cong c_g(1-\rho_S^{-1}t)^{-5(g-1)/2-1}\left(1+O\left(1-\rho_S^{-1}t)^{1/4}\right)\right).
\]
\end{prop}
\begin{proof}
We will first prove the corresponding result on the class $\mathcal{S}_g$. Let $\mathcal{S}^{C=i}_g$ be the set of triangulations in $\mathcal{S}_g$ with a marked non-contractible cycle of length $i$ for $i \geq 3$, and $S^{C=i}_g(t)$ its OGF. We use similar notations for triangulations of $\mathcal{S}_g$ with a given edgewidth. It is clear that $S^{\ew=i}_g(x) \preceq S^{C=i}_g(x)$, since a triangulation of edgewidth $i$ must contain at least one non-contractible cycle of length $i$. 

Let $S$ be a triangulation counted in $\mathcal{S}^{C=i}$ and $C$ its marked non-contractible cycle. We perform the following surgery, illustrated in Figure~\ref{fig:6:surgery-ew}: we first cut along $C$, then glue each hole with ``a wheel of size $i$'', which is the only triangulation with a boundary of length $i$ and an extra vertex that is adjacent to all vertices on the boundary, just like a wheel. When $i=3$, we don't need the extra vertex and just recognize the holes as faces. We then pick arbitrarily an edge on $C$, mark its copies after the surgery and orient them such that the new added faces are on their right. Given the marked and oriented edges, it is easy to revert the surgery. We notice that the number of edges increases by $3i$ after the surgery.

\begin{figure}
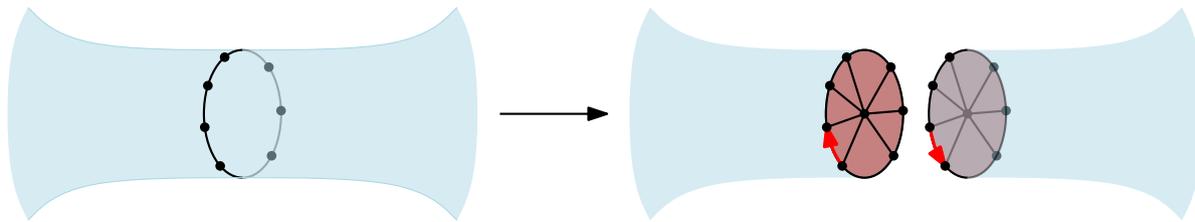

  \centering
  \insertfigure{ch6-fig.pdf}{15}
  \caption{Surgery on a non-contractible cycle}
  \label{fig:6:surgery-ew}
\end{figure}

We now have two cases: either $C$ is separating or not. If $C$ is separating, we obtain two triangulations $S^{(1)}$ and $S^{(2)}$ of genus $g_1$ and $g_2$ such that $g_1 + g_2 = g$, and we can suppose that $S^{(1)}$ contains the original root corner. In this case, we turn the marked edge of $S^{(2)}$ into its root. Otherwise, $C$ is not separating, we obtain a triangulation $M'$ of genus $g-1$ with two marked and oriented edges. The triangulations $S^{(1)},S^{(2)}$ are in $\mathcal{S}_{g_1}, \mathcal{S}_{g_2}$ respectively, because the surgery cannot create loops nor double edges. Since the surgery can be reverted, for $i >3$, we have the following dominance relation:
\begin{equation} \label{eq:6:M-C}
t^{3i} S^{C=i}_g(t) \preceq \sum_{\substack{g_1 + g_2 = g \\ g_i \geq 1}} tS'_{g_1}(t) S_{g_2}(t) + t^2 S''_{g-1}(t).
\end{equation}
The factor $t^{3i}$ here accounts for the extra $3i$ edges we obtain from the surgery. The case $i=3$ is a bit special, since no extra vertex is added. In this case, we have
\begin{equation} \label{eq:6:M-C-3}
t^{3} S^{C=3}_g(t) \preceq \sum_{\substack{g_1 + g_2 = g \\ g_i \geq 1}} tS'_{g_1}(t) S_{g_2}(t) + t^2 S''_{g-1}(t).
\end{equation}
The factor $t^3$ accounts for the extra $3$ edges that comes from duplication of $C$.

From Proposition~\ref{prop:6:asympt-S}, we know that $S_g(t) \cong \Theta \left( (1-\rho_S^{-1}t)^{-5(g-1)/2-1} \right)$. Using this fact on the right-hand side of \eqref{eq:6:M-C} and \eqref{eq:6:M-C-3}, for all $i \geq 3$, we have
\[
S^{C=i}_g(t) \preceq \Theta \left( (1-\rho_S^{-1}t)^{-5(g-1)/2-1/2} \right),
\]
which is dominated by the error term of $S_g(t)$. This is valid for all constant $i$. We thus have
\[ S_g(t) \succeq S^{\ew \geq k}_g(t) = S_g(t) - \sum_{i=3}^{k-1} S^{\ew = i}_g(t) \succeq S_g(t) - \sum_{i=3}^{k-1} S^{C=i}_g(t). \]
Since all $S^{C=i}_g(t)$ are negligible compared to $S_g(t)$, the series $S^{\ew \geq k}_g(t)$ has the same dominant singularity and the same asymptotic behavior as $S_g(t)$.

To transfer the result to $M^{\ew \geq k}_g(t)$, we notice the following chain of coefficient-wise dominance relations:
\[
S^{\ew \geq k}_g(t) \preceq M^{\ew \geq k}_g(t) \preceq M_g(t) \cong S_g(t).
\]
We thus conclude that $M^{\ew \geq k}_g(t)$ has the same dominant singularity and the same asymptotic behavior as $M_g(t)$.
\end{proof}

With all these preparations, we are ready to transfer our result on triangulations to cubic graphs.

\section{Decompositions of cubic graphs}

We now enter the world of graphs, where the size of a graph is given by the number of vertices, and these vertices are given distinct labels. We recall that graphs are not rooted. In such a labeled world, we need to use EGFs for enumeration. Therefore, in this section, there will be a transition from OGFs for maps to EGFs for vertex-labeled graphs.

We first define several classes of graphs. We denote by $\mathcal{C}_g$ (respectively $\mathcal{B}_g$ and $\mathcal{D}_g$) the class of connected (respectively 2-connected and 3-connected) cubic vertex-labeled graphs without triple edges that are strongly embeddable into \Sg{}. For a little help to memorize the notation, it is clear that the letter C in $\mathcal{C}_g$ and the letter B in $\mathcal{B}_g$ comes from the words ``connected'' and ``bi-connected'', and the letter D in $\mathcal{D}_g$ can be seen as coming from ``\textit{drei}-connected'', where \textit{drei} means ``three'' in German. Since we will be using facewidth (and edgewidth) in the decomposition of graphs, we will need classes of graphs with restrictions on these widths. These restrictions will be expressed by a superscript on the notation of graph classes. For instance, $\mathcal{D}_g^{\mathrm{fw}\geq 3}$ means the class of 3-connected cubic vertex-labeled graphs with facewidth at least 3. 

For EGFs of these graph classes, we use $x$ to mark the size parameter, that is, the number of vertices. We will also need to mark some extra statistics. We mark the number of simple edges by $y$, the number of double edges by $z$, and the number of loops by $w$. In other words, for any class $\mathcal{F}$ consisting of connected vertex-labeled cubic graphs without triple edges, its EGF $F(x,y,z,w)$ is defined by
\[
F(x,y,z,w) = \sum_{M \in \mathcal{F}} \frac{x^{{\rm \#vertices}(M)}}{(\mathrm{\# vertices}(M))!} y^{{\rm \#simple\;edges}(M)} z^{{\rm \#double\;edges}(M)} w^{{\rm \#loops}(M)}.
\]
The reason that we don't use the variable $t$ here for edges is that we want to separate simple edges from double edges and loops. Since there is no triple edge, the statistics of the number of double edges is well-defined. When some statistics are trivially null in a class family (for example, the number of double edges and loops in a class of 3-connected cubic graphs), we just drop their corresponding variables in the EGF.

By taking all these statistics into account, once we obtain an expression for $F(x,y,z,w)$, we can put weights on double edges and loops. This control is important for our purpose, since we are mainly interested in the enumeration of \emph{weighted} cubic graphs with a weight (or \emph{compensation factor}) $1/2$ for each loop and double edge. For this purpose, we define the following specialization of our EGF $F(x,y,z,w)$ using a new variable $v$:
\begin{equation} \label{eq:6:v-change}
F(v) \eqdef F\left(v^{1/4}, v^{1/6}, \frac{v^{1/3}}{2}, \frac{v^{1/6}}{2}\right).
\end{equation}
We can then express the total weight of cubic graphs with compensation factors. We recall that the number of vertices in a cubic graph is always divisible by $2$, and a cubic graph with $2n$ vertices has $3n$ edges.

\begin{prop} \label{prop:6:weighted-sum-gf}
Let $\mathcal{F}$ be a class of connected cubic graphs without triple edges. For an integer $n>0$, we denote by $\mathcal{F}(2n)$ the set of cubic graphs with $2n$ vertices. We can express the weighted sum of cubic graphs in $\mathcal{F}(2n)$ in terms of coefficients of $F(v)$ as
\[
\sum_{G \in \mathcal{F}(2n)} 2^{-{\rm \#double\;edges}(G)-{\rm \#loops}(G)} = (2n)![v^n]F(v).
\]
\end{prop}
\begin{proof}
We recall that $F(v) = F(v^{1/4}, v^{1/6}, v^{1/3}/2, v^{1/6}/2)$. Given a cubic graph $G$ in $\mathcal{F}(2n)$, it has $3n$ edges. Suppose that $G$ has $a$ simple edges, $b$ double edges and $c$ loops, we then have $a+2b+c=3n$. The contribution of $G$ to $F(v)$ will then be $((2n)!))^{-1}v^{2n/4+a/6+b/3+c/6}2^{-b-c} = ((2n)!))^{-1}v^n2^{-b-c}$. On the other hand, the compensation factor of $G$ also happens to be $2^{-b-c}$. Therefore, the contribution of $G$ to both sides of the equation is the same, which ensures the equality.
\end{proof}

By taking other specializations of $F(x,y,z,w)$, we can also answer enumeration questions about cubic graphs with different weights. The reasoning is similar to that in Proposition~\ref{prop:6:weighted-sum-gf}. For instance, we can look at $F(v^{1/4}, v^{1/6}, v^{1/3}, v^{1/6})$ for the enumeration of unweighted cubic graphs, and $F(v^{1/4},v^{1/6},0,0)$ for simple cubic graphs (\textit{i.e.} without double edges nor loops).

As a remark, it may seem to be much ado about nothing to have a variable for vertices in $F(x,y,z,w)$, since a cubic graph with $3n$ edges must have $2n$ vertices. However, we keep the variable for the clarity of our exposition of graph decompositions that we will use later.

Having defined the classes of graphs we will work on, we now transfer the asymptotic enumeration result of triangulations to 3-connected cubic graphs, using Theorem~\ref{thm:6:robertson-vitray}. 

\begin{prop} \label{prop:6:D-asymptotic}
For given $g\geq 0$ and $k \leq 2g+3$, the generating function $D_g^{\fw \geq k}(v)$ has the dominant singularity at $v=\rho_D = \rho_S^3 = \frac{27}{256}$. Furthermore, we have the following asymptotic behavior near the dominant singularity:
\begin{align*}
D_0^{\fw \geq k}(v) &\cong c_0 (1-\rho_D^{-1}v)^{5/2} + O\left( (1-\rho_D^{-1}v)^{3} \right), \\
D_1^{\fw \geq k}(v) &\cong c_1 \log(1-\rho_D^{-1}v) + O\left( (1-\rho_D^{-1}v)^{1/4} \right), \\
D_g^{\fw \geq k}(v) &\cong c_g (1-\rho_D^{-1}v)^{-5(g-1)/2} + O\left( (1-\rho_D^{-1}v)^{-5(g-1)/2+1/4} \right).
\end{align*}
Here, all $c_g$'s are explicit constants independent of $k$ as in Proposition~\ref{prop:6:asympt-S}.
\end{prop}
\begin{proof}
Let $\overline{\mathcal{M}}_g$ be the set of \emph{unrooted} maps in $\mathcal{M}_g$ with edges labeled from $1$ to the number of edges. We denote by $\overline{M}_g(t)$ the EGF of $\overline{\mathcal{M}}_g$ with the number of edges as size parameter. We use EGF here because we have a labeled class. For a triangulation $M \in \mathcal{M}_g$ with $n$ edges, there are $n!$ possible ways to label its edges. Conversely, for a triangulation $\overline{M} \in \overline{\mathcal{M}}_g$ with $n$ edges, there are $2n$ possible ways to specify its root corner. Therefore, we have
\[ [t^n]M_g(t) = 2n [t^n]\overline{M}_g(t), \]
and thus
\begin{equation} \label{eq:6:M-overline}
M_g(t) = 2t\overline{M}_g'(t), \;\;\; \overline{M}_g(t) = \int_0^t (2s)^{-1}M_g(s) ds.
\end{equation}

Now we introduce the class $\overline{\mathcal{D}}_g$, which is the set of cubic graphs in $\mathcal{D}_g$ with distinct labels on edges, from $1$ to the number of edges. We denote by $\overline{D}_g$ its EGF, with the number of edges as size parameter. By definition of $\mathcal{D}_g$, every graph $G \in \mathcal{D}_g$ is strongly embeddable on \Sg{}, therefore $G$ has at least two embeddings (which are unrooted maps) on \Sg{}, one by existence, the other by changing the orientation. Since vertices are labeled, the two embeddings cannot be isomorphic, or else each face would be mapped to another face with exactly the same vertices and adjacent edges but opposite orientation, which can only happen in a map with exactly two faces. By Proposition~\ref{prop:6:3-conn-dual}, the unrooted maps obtained from embedding elements of $\mathcal{D}_g$ are exactly the duals of the triangulations in $\overline{\mathcal{M}}_g$, with labels on edges (but not labels on vertices). Since a cubic graph with $2n$ vertices has $3n$ edges, by Proposition~\ref{prop:6:weighted-sum-gf}, we have
\[ 2\overline{D}_g(v) \preceq \overline{M}_g(v^{1/3}).\]
The substitution $t=v^{1/3}$ is due to the fact that the coefficient of $v^n$ accounts for cubic graphs with $3n$ edges in $\overline{D}_g(v)$. 

By Proposition~\ref{prop:6:facewidth-dual}, for a cubic map $M$ on \Sg{}, its facewidth $\fw(M)$ is equal to the edgewidth $\ew(M^*)$ of its dual triangulation $M^*$. Furthermore, from Theorem~\ref{thm:6:robertson-vitray} we know that every cubic map in $\overline{\mathcal{D}}_g^{\fw \geq 2g+3}$ has exactly two embedding. Combining the two facts, we have
\[ 2\overline{D}_g^{\fw \geq 2g+3}(v) = \overline{M}_g^{\ew \geq 2g+3}(v^{1/3}). \]
Obviously, we have $\overline{D}_g^{\fw \geq 2g+3}(v) \preceq \overline{D}_g(v)$, therefore
\begin{equation} \label{eq:6:squeeze-D-overline}
2\overline{M}_g^{\ew \geq 2g+3}(v^{1/3}) = \overline{D}_g^{\fw \geq 2g+3}(v) \preceq \overline{D}^{\fw \geq k}_g(v) \preceq \overline{D}_g(v) \preceq 2\overline{M}_g(v^{1/3}).
\end{equation}

On the other hand, if we compare the two EGFs $D_g(v)$ and $\overline{D}_g(v)$, we can see that they are identical. Indeed, if we consider 3-connected cubic graphs with both edge labelings and vertex labelings, they can be obtained by adding an edge labeling to an element in $\mathcal{D}_g$, where there is already a vertex labeling. The number of such cubic graphs with $2n$ vertices and $3n$ edges is thus $(3n)!(2n)![v^n]D_g(v)$, where the factor $(2n)!$ accounts for the fact that $D_g(v)$ is derived from an EGF with the number of vertices as size parameter, and $(3n)!$ the number of possible edge labelings. Equivalently, these graphs with labelings on edges and vertices can also be obtained by adding vertex labelings to elements in $\overline{\mathcal{D}}_g$. By a similar reasoning, the number of such cubic graphs with $2n$ vertices and $3n$ edges is $(2n)!(3n)![v^n]\overline{D}_g(v)$. We thus have $[v^n]D_g(v)=[v^n]\overline{D}_g(v)$ for all $n$, therefore $D_g(v) = \overline{D}_g(v)$. Substituting into \eqref{eq:6:squeeze-D-overline}, we have
\begin{equation} \label{eq:6:squeeze-D}
2\overline{M}_g^{\ew \geq 2g+3}(v^{1/3}) \preceq D_g^{\fw \geq k}(v) \preceq 2\overline{M}_g(v^{1/3}).
\end{equation}

By applying \eqref{eq:6:M-overline} to Proposition~\ref{prop:6:M-asymptotic} and Proposition~\ref{prop:6:M-ew-asymptotic}, we have
\begin{align*}
\overline{M}_0(t) = \overline{M}_0^{\ew \geq 3}(t) &\cong \frac{c_0}{2}(1-\rho_S^{-1}t)^{5/2} + O\left( (1-\rho_S^{-1}t)^3 \right), \\
\overline{M}_1(t) = \overline{M}_1^{\ew \geq 5}(t) &\cong \frac{c_1}{2} \log(1-\rho_S^{-1}t) + O\left( (1-\rho_S^{-1}t)^{1/4} \right), \\
\overline{M}_g(t) = \overline{M}_g^{\ew \geq 2g+3}(t) &\cong \frac{c_g}{2} (1-\rho_S^{-1}t)^{-5(g-1)/2} + O\left( (1-\rho_S^{-1}t)^{-5(g-1)/2+1/4}) \right), \;\; \mathrm{for} \;\; g \geq 2,
\end{align*}
with fixed constants $c_g$. We then conclude the proof by combining them with \eqref{eq:6:squeeze-D}.
\end{proof}

We now pass from 3-connected cubic graphs to 2-connected cubic graphs using Corollary~\ref{coro:6:conn-decomp}, in which we will need to deal with planar 2-connected components. We thus introduce the class $\mathcal{N}^\circ$ of 2-connected vertex-labeled cubic planar graphs with a marked and oriented edge, which will be the class of these planar components. Cubic graphs in $\mathcal{N}^\circ$ are also called \mydef{networks}. We should not confuse $\mathcal{N}^\circ$ with the class $\mathcal{N}_g$, which consists of some types of triangulations of \Sg{}. We denote by $N^\circ(x,y,z)$ its generating function, where the marked edge is never counted as a part of a double edge, that is, if one of the edges in a double edge is marked, the contribution of the double edge will be $y^2$ instead of $z$. In the following, by abuse of notation, we extend the meaning of $\preceq$ to multivariate power series, where it compares the coefficients of each possible monomial.

\begin{prop} \label{prop:6:D-to-B}
For $g \geq 0$, we have the following relation between EGFs of 3-connected cubic graphs in $\mathcal{D}_g$ and 2-connected cubic graphs in $\mathcal{B}_g$:
\begin{equation} \label{eq:6:D-to-B}
D_g^{\fw \geq 3}(x,{\bar y}) - D_0(x,\bar y) \preceq B_g^{\fw \geq 3}(x,y,z) \preceq D_g^{\fw \geq 3}(x, \bar y),
\end{equation}
where $\bar y = y(1+N^\circ(x,y,z))$.
\end{prop}
\begin{proof}
The case $g=0$ clearly holds, since planar graphs have infinite facewidth. We now suppose that $g\geq 1$.

Given a genus $g \geq 1$, we denote by $\overline{\mathcal{D}}_g$ the class formed by elements of the form $(D, L)$, where $D$ is a graph in $\mathcal{D}_g$ and $L = [L_e]_{e \in E(D)}$ is a list indexed by edges in $D$ whose elements are either an empty graph or a network. It is clear that the EGF of $\overline{\mathcal{D}}_g$ is given by $D_g(x,\bar y)$. We can also add a facewidth constraint to $\overline{\mathcal{D}}_g$, which will then be transferred to a facewidth constraint on $D$ for an element $(D,L)$.


We first establish an injection from $\mathcal{B}_g^{\fw \geq 3}$ to $\overline{\mathcal{D}}_g^{\fw \geq 3}$, which gives the upper bound of $B_g^{\fw \geq 3}$. By Corollary~\ref{coro:6:conn-decomp}, $G$ has a unique 3-connected component $T$ strongly embeddable into \Sg{} with the same facewidth, and all other 3-connected components are planar. Since $G$ is a cubic graph, we can suppose that the two vertices in a separator that separate $T$ from other components have two edges linking to vertices in $T$, as in Figure~\ref{fig:6:conn-comp-sep}(a). We can then disconnect the other components as in Figure~\ref{fig:6:conn-comp-sep} (there $u$ is the vertex with a larger label, which fixes the orientation of $\{u',v'\}$), and each disconnection gives $T$ a new edge $\{u,v\}$. Since we will only consider separators directly adjacent to $T$, they separate $T$ from 2-connected planar sub-graphs (thus networks) that can be further decomposed. By disconnecting all other 2-connected components from $T$ using the procedure illustrated in Figure~\ref{fig:6:conn-comp-sep}, we obtain a graph $T'$ in $\mathcal{D}_g^{\fw \geq 3}$ with its edges either comes from the 3-connected component $T$ or from disconnections. We can also say that edges of $T'$ are associated to either nothing or a network, or equivalently, we have a list $L$ indexed by edges in $T'$ whose elements are either the empty graph or a network. The pair $(T',L)$ is thus an element in $\overline{\mathcal{D}}_g^{\fw \geq 3}$. To show that this procedure is indeed an injection, we will show how to reconstruct a graph $G$ in $\mathcal{B}_g^{\fw \geq 3}$ from a given $(T',L)$ in $\overline{\mathcal{D}_g}^{\fw \geq 3}$. For each edge $e=\{u, v\}$ in $T'$ where $u$ is the vertex with a larger label, either the associate element $L_e$ in $L$ is empty, in which case we do nothing; or $L_e$ is a network with an oriented marked edge $e'=(u',v')$, in which case we perform the operation illustrated in Figure~\ref{fig:6:conn-comp-sep}(a) in the reverse direction, first deleting both $e$ and $e'$, then adding two new edges $\{u, u'\}$ and $\{v, v'\}$. The result of this operation is clearly in $\mathcal{B}_g^{\fw \geq 3}$. Therefore, the correspondence from $\mathcal{B}_g^{\fw \geq 3}$ to $\overline{\mathcal{D}}_g^{\fw \geq 3}$ is indeed injective, and we have the upper bound.

For the lower bound, we only need to establish an injection from the set $\overline{\mathcal{D}}_g^{\fw \geq 3} \setminus \overline{\mathcal{D}}_0$ to $\mathcal{B}_g^{\fw \geq 3}$. In fact, this set is formed by elements of the form $(T,L)$ such that $T$ is strongly embeddable into \Sg{} of facewidth at least $3$ but not planar. We recall that a graph can be strongly embeddable into surfaces of different genus, see Figure~\ref{fig:6:embeddings} for an example. The injection is still provided by the procedure described in the previous paragraph, but this time we can uniquely determine $(T,L)$ from $G$, since now $G$ contains only one 3-connected component that is not planar. We thus have the lower bound.


\end{proof}

\begin{figure}
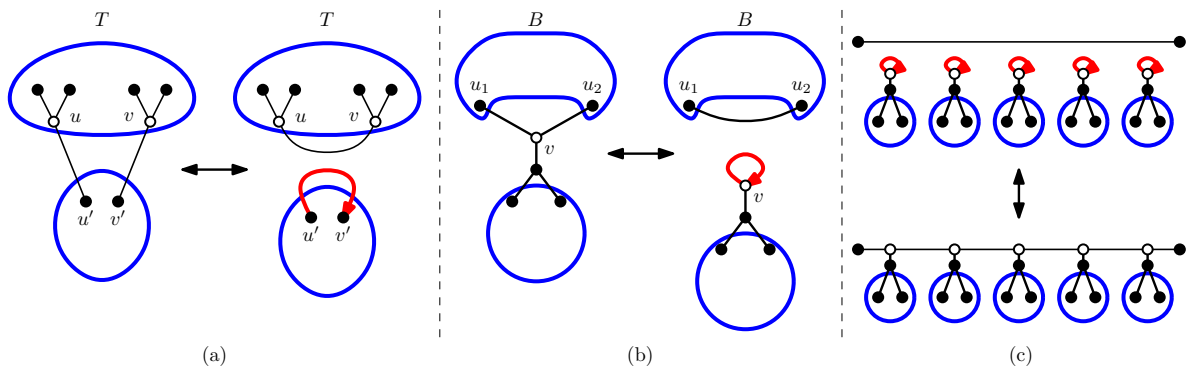

  \centering
  \insertfigure[0.75]{ch6-fig.pdf}{17}
  \caption{3-connected and 2-connected component separation in cubic graphs, with separators distinguished}
  \label{fig:6:conn-comp-sep}
\end{figure}

We continue the decomposition along connectivity to pass from 2-connected cubic graphs to connected graphs, which requires yet another class of cubic graphs. We denote by $\mathcal{Q}$ the class of vertex-labeled connected cubic planar graphs with a marked \emph{loop}. The EGF of $\mathcal{Q}$ is denoted by $Q(x,y,z,w)$, where the marked loop is regarded as a simple edge and contributes a factor $y$.

\begin{prop} \label{prop:6:B-to-C}
For $g \geq 0$, we have the following relation between EGFs of 2-connected cubic graphs in $\mathcal{B}_g$ and connected cubic graphs in $\mathcal{C}_g$:
\begin{equation} \label{eq:6:B-to-C}
B_g^{\fw \geq 2}(x,{\bar y},\bar z) - B_0(x,\bar y,\bar z) \preceq C_g^{\fw \geq 2}(x,y,z,w) \preceq B_g^{\fw \geq 2}(x, \bar y,\bar z),
\end{equation}
where
\[
\bar y = \frac{y}{1-Q(x,y,z,w)}, \quad \bar z = \frac1{2} \left( \frac{y}{1-Q(x,y,z,w)} \right)^2 - \frac{y^2}{2} + z.
\]
\end{prop}
\begin{proof}
The case $g=0$ clearly holds, since planar graphs have infinite facewidth. We now suppose that $g \geq 1$.

Given a genus $g\geq 1$, we denote by $\overline{\mathcal{B}}_g$ the class formed by elements of the form $(B,L^y,L^z)$, where $B$ is a graph in $\mathcal{B}_g$, $L^y = [L_e^y]_{e}$ is a list indexed by simple edges in $B$ whose elements are (possibly empty) lists of graphs in $\mathcal{Q}$, and $L^z=[L_d^y]_{d}$ is a list indexed by double edges in $B$ whose elements are multisets containing two (possibly empty) lists of graphs in $\mathcal{Q}$. When writing the EGF of $\overline{\mathcal{B}}_g$, a double edge $d$ in $B$ has weight $z$ if and only if the corresponding item $L_d$ consists of two empty lists, or else it has weight $y^2/2$. We will see the reason of this weighting scheme later. Under this convention, the EGF of $\overline{\mathcal{B}}_g$ is clearly given by $B_g(x,\bar y, \bar z)$. We can also add a facewidth constraint to $\overline{\mathcal{B}}_g$, which will then be transferred to a facewidth constraint on $B$ for an element $(B,L^y,L^z)$.

We first establish an injection from $\mathcal{C}_g^{\fw \geq 2}$ to $\overline{\mathcal{B}}_g^{\fw \geq 2}$. By Corollary~\ref{coro:6:conn-decomp}, $G$ has a unique 2-connected component $B$ strongly embeddable into \Sg{} with the same facewidth, and all other 2-connected components are planar. Since $G$ is a cubic graph, as illustrated in Figure~\ref{fig:6:conn-comp-sep}(b), for a cut vertex $v$ that separates $B$ from other components, $v$ must have two edges in $B$, linked to two vertices $u_1, u_2$. We can see that $u_1$ cannot be the same as $u_2$, or else $u_1$ would be a cut vertex. As illustrated in Figure~\ref{fig:6:conn-comp-sep}(b), we delete the two edges $\{v, u_1\}, \{v, u_2\}$, attach a loop to $v$ and add an edge $\{u_1,u_2\}$. This process maintains the degree of each vertex and detaches the component from $B$. The detached component is a connected cubic graph, since it may contain other 2-connected components. By marking the added loop, we can see the detached component as an element in $\mathcal{Q}$. By detaching all such components, we obtain a graph $B'$ in $\mathcal{B}_g^{\fw \geq 2}$, whose edges are either already in $B$ or added during the procedure in Figure~\ref{fig:6:conn-comp-sep}(b). For an edge in $B'$ that was not in $B$, it may corresponds to several different components (see Figure~\ref{fig:6:conn-comp-sep}(c)). We can thus record them as a list of elements in $\mathcal{Q}$, in the order that goes from the vertex with smaller label to that of larger label. For edges of $B$ that were already in $B$, we associate the empty list. For simple edges in $B'$, this is the whole story, and we have the list $L^y$. For double edges in $B'$, since its two edges are indistinguishable, we have a multiset of two lists for each double edge, which forms the list $L^z$. A double edge $d$ in $B'$ corresponds to a double edge in $B$ if and only if the corresponding item $L_d^z$ contains two empty lists, that is, nothing is inserted in either edges. It is for dealing with this case that we introduced the unusual weighting scheme for the EGF of $\overline{\mathcal{B}}_g$, which is designed to make the total weight invariant after the process above that disconnects components. We now have a correspondence from $G$ to $(B',L^y,L^z)$. To show that this process is injective, it suffices to see that we can reconstruct $G$ from $(B',L^y,L^z)$ by reversing the cutting operation illustrated in Figure~\ref{fig:6:conn-comp-sep}(c). More precisely, to insert a component $Q$ into an edge $e$, we first remove the marked loop of $Q$, then insert the remaining vertex of degree $1$ in the middle of $e$. Since the reconstruction is possible, the correspondence is an injection, and we have the upper bound.

For the lower bound, we only need to establish an injection from the set $\overline{\mathcal{B}}_g^{\fw \geq 2} \setminus \overline{\mathcal{B}}_0$ to $\mathcal{C}_g^{\fw \geq 2}$. This set is formed by elements $(B',L^y,L^z)$ such that $B'$ is 2-connected, strongly embeddable into \Sg{} with facewidth at least $3$ but not planar. The injection is still provided by the procedure described above, but this time we can uniquely determine $(B',L^y,L^z)$ from $G$, because now $G$ has only one 2-connected component that is not planar. We thus have the lower bound.


\end{proof}

Although in Proposition~\ref{prop:6:D-to-B} and Proposition~\ref{prop:6:B-to-C} we can bound the EGF of a class of cubic graphs of lower connectivity by another of higher connectivity, these bounds are not exactly what we want. We need to enumerate the class $\mathcal{C}_g$, but Proposition~\ref{prop:6:B-to-C} is only valid for the class $\mathcal{C}_g^{\fw \geq 2}$, and we lack the part $\mathcal{C}_g^{\fw = 1}$. And in Proposition~\ref{prop:6:B-to-C}, we use $B_g^{\fw \geq 2}(x,y,z)$ to bound $C_g^{\fw \geq 2}(x,y,z,w)$, but the bound in Proposition~\ref{prop:6:D-to-B} concerns $B_g^{\fw \geq 3}(x,y,z)$, which does not include 2-connected cubic graphs of facewidth $2$. The following propositions fill up the gaps by giving an upper bound to the EGFs of the two classes of cubic graphs at discrepancy. In both cases, we exploit the fact that the graphs we count have small facewidth to break them into smaller graphs strongly embeddable into surfaces with strictly smaller genus. In the following, by abuse of notation, we will extend the coefficient-wise dominance relation $\preceq$ to multivariate Laurent series.

\begin{figure}
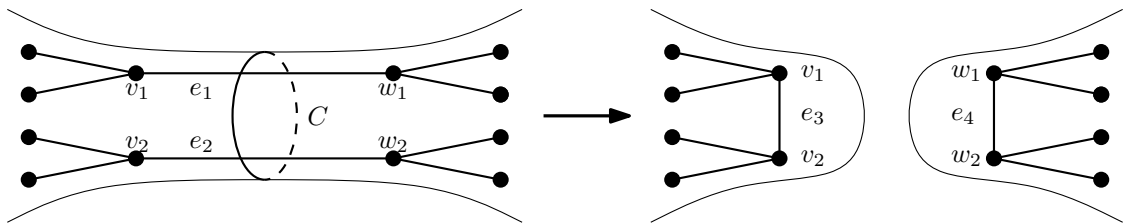

  \centering
  \insertfigure{ch6-fig.pdf}{18}
  \caption{Altering edges on a 2-connected cubic graph of facewidth 2}
  \label{fig:6:2-conn-fw-2-op}
\end{figure}

\begin{prop} \label{prop:6:fw-gap-2-conn}
For $g \geq 1$, we have
\begin{align} \label{eq:6:2-conn-fw-2}
\begin{split}
&B_g^{\fw=2}(x,y,z) \preceq 2 \left( y + \frac{z}{y} \right)^2 \left( \frac1{y} + \frac{y}{z} \right)^2 
\Bigg( 
(\delta_y + \delta_z)^2 B_{g-1}^{\fw \geq 2}(x,y,z) \\
&\quad \quad \quad \quad+ \sum_{\substack{g_1+g_2=g \\ g_1, g_2 > 0}} \left( (\delta_y + \delta_z)B_{g_1}^{\fw\geq2}(x,y,z) \right) \left( (\delta_y + \delta_z)B_{g_2}^{\fw\geq2}(x,y,z) \right)
\Bigg),
\end{split}
\end{align}
where $\delta_y = \frac{y\partial}{\partial y}$ and $\delta_z = \frac{z\partial}{\partial z}$.
\end{prop}
\begin{proof}
It is clear that 2-connectivity and 2-edge-connectivity are the same in cubic graphs. Indeed, if an edge is a separator in a cubic graph, then one of its vertices must be a cut vertex. Conversely, if there is a cut vertex $v$, all its edges cannot be linked to the same component, and since $v$ is of degree $3$, there must be at least a component that is linked to $v$ only by an edge $e$, which is a separator.

Let $B$ be a cubic graph in $\mathcal{B}_g^{\fw=2}$. By the definition of facewidth, there is an embedding $M_B$ of $B$ with facewidth $2$, \textit{i.e.} there are two faces $f_1,f_2$ whose union contains a non-contractible circle $C$. We can choose $C$ to first cross from $f_1$ to $f_2$, then from $f_2$ to $f_1$. Since $B$ is a cubic graph, we can suppose that $C$ crosses edges instead of passing by vertices, and let $e_1 = \{v_1, w_1\}, e_2 = \{v_2, w_2\}$ be the two edges that $C$ crosses, such that $v_1, v_2$ are on the same side of $C$. It is clear that $e_1$ and $e_2$ do not share any vertex, or else one of $f_1,f_2$ would already contain a non-contractible circle. The following construction is illustrated in Figure~\ref{fig:6:2-conn-fw-2-op}. Given the two edges, we first delete $e_1,e_2$ from $B$, then add $e_3 = \{ v_1, v_2\}, e_4=\{w_1,w_2\}$ to obtain a cubic graph $B'$, coming with an induced embedding. These operations can also be explained in terms of surgeries on dual triangulations. Consider the dual triangulation $M_B^*$ of $M_B$, in which the dual edges $e_1^*, e_2^*$ form a non-contractible double edge $d$. By first cutting along $d$ then zipping the holes, we may obtain one or two triangulations, whose dual is an embedding of $B'$.

Now, depending on the nature of $C$, there are two cases: either $C$ is separating, in which case $B'$ will contain two connected components; or $C$ is non-separating, then $B'$ is connected. We now discuss the cubic graph $B'$ we obtain in these two cases.

In the first case, we obtain two connected cubic graphs $B_1, B_2$ strongly embeddable into $\torus{g_1}$, $\torus{g_2}$ respectively, with $g_1+g_2=g$. We have $g_1>0, g_2>0$ since $C$ is not contractible. We mark the edges $e_3, e_4$ on $B_1$ and $B_2$ respectively. We first prove that $B_1$ is 2-connected. Indeed, let $P$ be a path with in $B$ linking two vertices $u,v$ in $B_1$, then we can modify $P$ to give a path $P'$ in $B_1$ by replacing any sub-path in $B\setminus B_1$ by $e_3$. Since $B$ is 2-connected, we see that $B_1$ is also 2-connected. By the same argument, $B_2$ is also 2-connected. We now consider the facewidth. We can assume that the embeddings of $B_1$ and $B_2$ induced from $M_B$ are of minimal facewidth, since we can choose $M_B$ freely. Since $B$ is of facewidth $2$, no face in $B_1$, except $f_1, f_2$ that are changed, cannot contain any non-contractible circle. For $f_1$ and $f_2$, we can see that the operation does not create new non-contractible circles. Therefore, $B_1$ (and similarly $B_2$) is of facewidth at least $2$. We thus have an upper bound on the EGF of graphs $B_1$, $B_2$ that we obtain in this case, which is
\[
\sum_{\substack{g_1+g_2=g \\ g_1, g_2 > 0}} \left( (\delta_y + \delta_z)B_{g_1}^{\fw\geq2}(x,y,z) \right) \left( (\delta_y + \delta_z)B_{g_2}^{\fw\geq2}(x,y,z) \right).
\]
Here, since the marked edge can be simple or part of a double edge, we use $\delta_y + \delta_z$ to give an upper bound of the EGF of these cubic graphs with marked edges.

In the second case, we obtain a connected cubic graph $B'$, strongly embeddable into $\torus{g-1}$, with marked edges $e_3, e_4$. We denote by $M_{B'}$ its embedding derived from $M_B$. We first prove that $B'$ is 2-connected. Suppose that $e$ is a bridge of $B'$. The edge $e$ cannot be $e_3$ nor $e_4$. The faces on the two sides of $e$ on $M_{B'}$ must be the same, or else their adjacent edges would prevent $e$ to be a bridge. Let $f$ be the only face adjacent to $e$ on $M_{B'}$, and $C_e$ a circle in $f$ that crosses $e$. We now consider $C_e$ in $M_B$. Since we go back from $M_{B'}$ to $M_{B}$ by merging some faces, $C_e$ is well-defined on $M_B$. It cannot be contractible, or else $e$ would be a bridge in $B$, which contradicts the fact that $B$ is 2-connected. It cannot be non-contractible, or else $M_B$ would have facewidth $1$, which is impossible. Therefore, $e$ does not exist, and $B'$ is 2-connected. We now prove that $B'$ is of facewidth at least $2$. Since we have the freedom to choose the embedding $M_B$, we can assume that $M_{B'}$ has minimal facewidth by choosing $M_B$ accordingly. Suppose that there is a face $f$ in $M_{B'}$ that contains a non-contractible circle $C$. We can assume that $C$ crosses only one edge $e$, which is adjacent to $f$. Then by the same argument as above, $C$ is either contractible in $M_B$, which contradicts the 2-connectivity of $B$, or non-contractible in $M_B$, which contradicts the fact that $\fw(B)=2$. Therefore, such a circle cannot exist, and the facewidth of $B'$ is at least $2$. We thus have an upper bound of the EGF of $B'$ with marked edges, which is
\[
(\delta_y + \delta_z)^2 B_{g-1}^{\fw \geq 2}(x,y,z).
\]
Here, by the same reason as in the previous case, we use $\delta_y + \delta_z$ for the marked edges.

We now explain the factor $2(y+z/y)^2(1/y+y/z)^2$, which accounts for the change from cubic graphs with marked edges to cubic graphs in $\mathcal{B}_g^{\fw=2}$. The factor $2$ accounts for the fact that there are two ways to reconnect $e_1$ and $e_2$ given the marked edges $e_3,e_4$. The deletion of $e_3$ (respectively $e_4$) may either eliminate a simple edge or turn a double edge into a simple edge, and its effect on the EGF is bounded from above by the factor $(1/y+y/z)$. The addition of $e_1$ (respectively $e_2$) may either add a new simple edge or turn an existing simple edge to a double edge, and its effect on the EGF is bounded from above by the factor $(y+z/y)$. We thus have the upper bound of $B_g^{\fw=2}(x,y,z)$ in \eqref{eq:6:2-conn-fw-2}.
\end{proof}

\begin{prop} \label{prop:6:fw-gap-1-conn}
For $g\geq1$, we have
\begin{align}\label{eq:6:1-conn-fw-1}
  \begin{split}
    &C_g^{\fw=1}(x,y,z,w) \preceq (xyw)^{-2} \left( y + \frac{z}{y} \right) \Bigg( \left(\frac{w\partial}{\partial w}\right)^2 C_{g-1}(x,y,z,w) \\
    &\quad \quad \quad \quad + \sum_{\substack{g_1+g_2=g \\ g_1, g_2 >0}} \left( \frac{w\partial}{\partial w} C_{g_1}(x,y,z,w) \right) \left( \frac{w\partial}{\partial w} C_{g_2}(x,y,z,w) \right) \Bigg).
  \end{split}
\end{align}
\end{prop}
\begin{proof}
Let $G$ be a cubic graph in $\mathcal{G}_g^{\fw=1}$. By definition of the facewidth, there is an embedding $M_G$ of $G$ with facewidth $1$, \textit{i.e.} there is a face $f$ that contains a non-contractible circle $C$, which can be assumed to cross only one edge $e$. We first suppose that $e$ is a loop, and let $v$ be its adjacent vertex. Since $G$ is cubic, $v$ is adjacent to only one edge other than $e$, and one side of the loop $e$ is thus a face of size $1$, which makes $C$ contractible and contradicts our assumption. Therefore, $e$ cannot be a loop, and we denote by $v_1, v_2$ its adjacent vertices. We now perform the following operation illustrated in Figure~\ref{fig:6:1-conn-fw-1-op}: we first delete $e$, then for $v_1$ and $v_2$, we attach on each vertex a new edge leading to a new vertex with a marked loop. On the embedding $M_G$, this is done by first adding the circle $C$ as an edge and its crossing with $e$ as a vertex, then cutting along the loop $C$ (see Figure~\ref{fig:6:1-conn-fw-1-op}). 

Now, according to the nature of $C$, there are two cases: either $C$ is separating or not. If $C$ is separating, we obtain two connected cubic graphs $G_1, G_2$ strongly embeddable into $\torus{g_1}, \torus{g_2}$ respectively, with $g_1+g_2=g$ and $g_1 >0, g_2>0$. Both $G_1, G_2$ contain a marked loop. If $C$ is non-separating, we obtain a connected cubic graph $G'$ strongly embeddable into $\torus{g-1}$. The EGF of cubic graphs obtained in this way is thus bounded from above by
\[
\left( \frac{w\partial}{\partial w}\right)^2 C_{g-1}(x,y,z,w) + \sum_{\substack{g_1+g_2=g \\ g_1, g_2 >0}} \left( \frac{w\partial}{\partial w} C_{g_1}(x,y,z,w) \right) \left( \frac{w\partial}{\partial w} C_{g_2}(x,y,z,w) \right).
\]
Here, the differential operators are for marking loops.

To go in the reverse direction, we only need to delete the two marked loops, their adjacent vertices and the edges adjacent to these vertices, and then add an edge between the two vertices of degree $2$ that remains. In other words, we first delete two vertices, two loops and two simple edges, which accounts for the factor $(xyw)^{-2}$. We then add a new edge, which may be a new simple edge or turn an existing simple edge into a double edge. Its effect on the EGF is bounded from above by the factor $(y+z/y)$. We thus have the upper bound of $C_g^{\fw=1}$ in (\ref{eq:6:1-conn-fw-1}).
\end{proof}

\begin{figure}
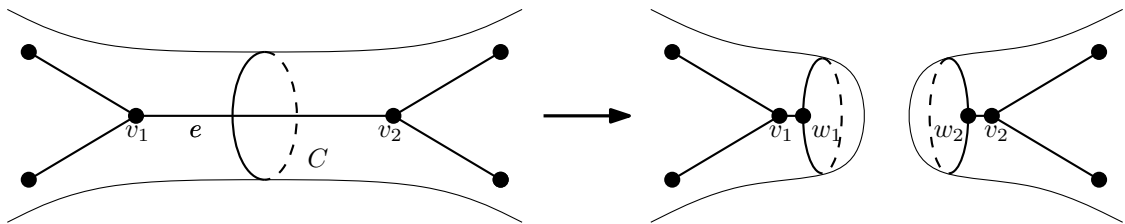

  \centering
  \insertfigure{ch6-fig.pdf}{24}
  \caption{Altering edges on a connected cubic graph of facewidth 1}
  \label{fig:6:1-conn-fw-1-op}
\end{figure}

Later, we will use Proposition~\ref{prop:6:fw-gap-2-conn} and Proposition~\ref{prop:6:fw-gap-1-conn} to show that the number of multigraphs with small facewidth is negligible compared to their counterparts without restriction on facewidth. Now the only thing missing on the path to push asymptotic enumeration results from 3-connected cubic graphs to connected cubic graphs is the expressions of the unknown series $N^\circ(x,y,z)$ of networks (used in Proposition~\ref{prop:6:D-to-B}) and $Q(x,y,z,w)$ of connected cubic planar graphs rooted at a loop (used in Proposition~\ref{prop:6:B-to-C}). Similar expressions of the series of networks and connected cubic planar graphs rooted at a loop keeping track of fewer statistics are already obtained in \cite{planar-phase, planar-prob-det}, and with minor modifications of weights in the arguments in \cite{planar-phase}, we obtain the following expressions of $N^\circ(x,y,z)$ and $Q(x,y,z,w)$. In the following proposition, we will temporarily omit the arguments of generating functions.

\begin{prop}\label{prop:6:graph-system}
The generating function $Q(x,y,z,w)$ of the class $\mathcal{Q}$ satisfies
\begin{align}
  \begin{split} \label{eq:6:Q-sys}
    Q &= \frac{x^2y^3}{2} A + \frac{Q^2}{2} + x^2y^2w, \\
    A &= Q + S + P + H, \\
    S &= \frac{A^2}{A+1}, \\
    P &= \frac{x^2y^3}{2} A^2 + x^2y^3A + x^2yz, \\
    2H(1+A) &= U(1-2U) - U(1-U)^3, \\
    x^2y^3(1+A)^3 &= U(1-U)^3,
  \end{split}
\end{align}
where $A, S, P, H, U$ are all formal power series in the variables $x,y,z,w$.

The generating function $N^\circ(x,y,z)$ of the class $\mathcal{N}^\circ$ satisfies
\begin{align}
  \begin{split} \label{eq:6:N0-sys}
    N^\circ &= \frac1{2} \left( V(1-2V) - x^2y(y^2-2z)(1+N^\circ) \right), \\
    x^2y^3(1+N^\circ)^3 &= V(1-V)^3,
  \end{split}
\end{align}
where $V$ is a formal power series in variables $x,y,z$.
\end{prop}

It is worth mentioning that the series $Q(x,y,z,w)$ and $N^\circ(x,y,z)$ are algebraic, which is a rare situation for series of graphs. In comparison, similar series in the setting of general planar graphs are not algebraic (\textit{cf.} \cite{planar-law}).

As in Section~3 of \cite{planar-phase}, the proof of \eqref{eq:6:Q-sys} relies on a recursive decomposition of edge-rooted connected cubic graphs. The only difference is that we account properly for loops and double edges in the initial conditions in a way that is appropriate to our needs. To give readers some ideas on how the decomposition works, we will just describe the possible cases of the decomposition. Let $G$ be a connected cubic graph rooted at an edge, \textit{i.e.} with a marked and oriented edge. There are five possibilities of the placement of the root:
\begin{enumerate}
\item the root is a loop;
\item the root is a bridge;
\item the root is part of a minimal edge separator of size 2;
\item the vertices adjacent to the root separate the graph;
\item the root is part of a 3-connected component.
\end{enumerate}
It is shown in \cite{planar-phase} that the list above is exhaustive. We then assign a graph class for each case, and we see that the first case corresponds to the class $\mathcal{Q}$. For the first four cases, we can decompose the graph by deleting the root and its adjacent vertices in an appropriate way, which can be used to write functional equations. The last case involves a more complicated parametric equation for $H$ in terms of $U$. The detail proof is omitted here.

As a remark, although the recursive decomposition above for edge-rooted connected cubic graphs and the Tutte decomposition for the enumeration of maps share the same idea of deleting the root, they are of very different nature. In the case of maps, the extra embedding information tells us that the root deletion can be treated as merging or splitting faces, which leads to simple functional equations with one catalytic variable marking the degree of a special face (the outer face). In the case of graphs, we do not have such information. It is thus difficult to construct a recursive decomposition that is bijective for enumeration, since the root may potentially connect any two vertices in the graph. We must resort to extra information, such as connectivity, for a decomposition of graphs. Indeed, the recursive decomposition of edge-rooted planar cubic graphs above relies on the decomposition along connectivity of graphs, the unique embedding theorem of Robertson and Vitray (Theorem~\ref{thm:6:robertson-vitray}) and the enumeration of cubic planar maps, which can be seen in how the cases are split.

To obtain \eqref{eq:6:N0-sys}, we start with \eqref{eq:6:Q-sys}. Since $\mathcal{N}^\circ$ is the class of 2-connected planar cubic graphs with a marked and oriented edge (thus rooted), Case~1 and Case~2 in the decomposition above never occurs. By suppressing these cases, we can obtain a system that simplifies into \eqref{eq:6:N0-sys}.

Although these systems of equations does not allow us to express $Q(x,y,z,w)$ and $N^\circ(x,y,z)$ in a simple explicit form, they are sufficient for asymptotic enumeration, which is what we need.

\section{Asymptotic enumeration of cubic graphs}

In this section, we will use asymptotic analysis to obtain the asymptotic behavior of the generating function $C_g(v)$ of connected cubic graphs of genus $g$, which is then used for the asymptotic enumeration of general cubic graphs. Our analysis builds on various bounds, expressions and systems established in the previous sections. In the following, we will specialize previous results to generating functions in the variable $v$. We recall that, for a formal power series $F(x,y,z,w)$, the corresponding function $F(v)$ is the specialization $F(v)=F(x_*=v^{1/4},y_*=v^{1/6},z_*=v^{1/3}/2, w_*=v^{1/6}/2)$. We can see that the specialized variables are closely related, for instance $z_*=y_*^2/2$ and $w_*=y_*/2$. Using these relations, we can establish the following lemmas, which will be useful to reduce differentiation with respect to other variables to that of $v$.

\begin{lem} \label{lem:6:reduce-to-v}
Let $F(x,y,z,w)$ be a formal power series in $x,y,z,w$ with non-negative coefficients. We have
\begin{equation} \label{eq:6:dw-to-dy}
\left. \left(\frac{w\partial}{\partial w}F(x,y,z,w)\right) \right|_{z=\frac{y^2}{2},w=\frac{y}{2}} \preceq \frac{y\partial}{\partial y} \left(F\left(x,y,\frac{y^2}{2},\frac{y}{2}\right)\right).
\end{equation}

Let $F(x,y,z)$ be a formal power series in $x,y,z$ with non-negative coefficients. We have
\begin{equation} \label{eq:6:dz-to-dy}
\left. \left( \frac{z\partial}{\partial z}F(x,y,z) \right) \right|_{z=\frac{y^2}{2}} \preceq \frac{y\partial}{\partial y} \left( F\left(x,y,\frac{y^2}{2}\right) \right).
\end{equation}

Let $F(x,y)$ be a formal power series in $x,y$ with non-negative coefficients formed by monomials of the form $x^{2k} y^{3k}$. We have
\begin{equation} \label{eq:6:dy-to-dv}
3\frac{vd}{dv}F(v) = \left. \left( \frac{y\partial}{\partial y}F(x,y) \right) \right|_{x=v^{1/4}, y=v^{1/6}}.
\end{equation}
\end{lem}
\begin{proof}
The first two relations \eqref{eq:6:dw-to-dy} and \eqref{eq:6:dz-to-dy} are of the same flavor and are rather straight-forward. We now only analyze \eqref{eq:6:dz-to-dy}, and it should then be clear for \eqref{eq:6:dw-to-dy}. We take $F(x,y,z)=x^a y^b z^c$ a monomial and substitute it into both sides of \eqref{eq:6:dz-to-dy}. On the left-hand side, we obtain $c x^a y^b (y^2/2)^c$, and the right-hand side becomes $(b+2c) x^a y^b (y^2/2)^c$, and we clearly have the coefficient-wise dominance relation. Since the relation holds for each monomial, it also holds for formal power series with non-negative coefficients.

The third relation \eqref{eq:6:dy-to-dv} can be proved by a straight-forward computation. Since $F(x,y)$ is formed by monomials of the form $x^{2k} y^{3k}$, we have $3\frac{x\partial}{\partial x}F(x,y) = 2\frac{y\partial}{\partial y}F(x,y)$, therefore,
\begin{align*}
3\frac{v\partial}{\partial v}F(v) &= \left. \left( 3v \frac{\partial F(x,y)}{\partial x} \frac{\partial x}{\partial v} + 3v \frac{\partial F(x,y)}{\partial y} \frac{\partial y}{\partial v} \right) \right|_{x=v^{1/4},y=v^{1/6}} \\
&= \left. \left( \frac{3}{4} v^{1/4} \frac{\partial F(x,y)}{\partial x} + \frac1{2}v^{1/6} \frac{\partial F(x,y)}{\partial y} \right) \right|_{x=v^{1/4},y=v^{1/6}} \\
&= \left. \left( \frac{y\partial}{\partial y}F(x,y) \right) \right|_{x=v^{1/4}, y=v^{1/6}}. \qedhere
\end{align*}
\end{proof}

As a remark, the substitutions in \eqref{eq:6:dw-to-dy} and \eqref{eq:6:dz-to-dy} will appear in our study of generating functions in $v$. Indeed, in the specialization $F(v)$, we always have $z_*=y_*^2/2$ and $w_*=y_*/2$, and later we will first do the substitution $z=y^2/2$ and $w=y/2$ as in \eqref{eq:6:dw-to-dy} and \eqref{eq:6:dz-to-dy} in the generating functions of cubic graphs. We also remark that these substitutions will lead to formal power series satisfying the restrictions of \eqref{eq:6:dy-to-dv}. Therefore, the restriction on the formal power series for \eqref{eq:6:dy-to-dv} will always hold for our case.

We now proceed to the analysis of asymptotic behaviors of generating functions in $v$. We start by the generating function $N^\circ(v)$ of the class $\mathcal{N}^\circ$.

\begin{prop} \label{prop:6:asym-Nv}
The dominant singularity of $N^\circ(v)$ occurs at $\rho_N=\frac{2^4 3^3}{17^3}$. Furthermore, $N^\circ(v)$ is $\Delta$-analytic, and near its dominant singularity, we have
\[
N^\circ(v) = \frac1{16} - \frac{51}{400}(1-\rho_N^{-1}v) + \frac{17^{5/2}}{2 \cdot 3^{1/2} \cdot 5^5}(1-\rho_N^{-1}v)^{3/2} + O\left( (1-\rho_N^{-1}v)^2 \right).
\]
\end{prop}
\begin{proof}
By first eliminating $v$ from \eqref{eq:6:N0-sys} and then performing the substitution of variable as in \eqref{eq:6:v-change}, we obtain the following implicit equation for $N=N^\circ(v)$:
\begin{align*}
0 &= 16v^2N^6 + v(32+96v)N^5 + (16+56v+240v^2)N^4 + (24-25v+320v^2)N^3 \\
&\quad + (12-91v+240v^2)N^2 + (2-43v+96v^2)N - v(1-16v).
\end{align*}
We can then use standard methods of singularity analysis of general algebraic functions (see \cite[Section~VII.7.1]{flajolet}) and preferably a computer algebra system to determine the dominant singularity $\rho_N$ of $N^\circ(v)$ and its expansion near $\rho_N$. The discriminant is $1024 v^5 (4913 v-432)^3 (729 v^2+16)$. The root $v=0$ is not a dominant singularity, since the dominant coefficient is $16v^2$, which becomes $0$ when $v=0$. But by Pringsheim's Theorem, $N$ must have a dominant singularity that is real. Therefore, its dominant singularity has to be $\rho_N = \frac{432}{4913}$. The asymptotic behavior can be easily computed using a computer algebra system.
\end{proof}

Using a similar approach, we can also obtain the asymptotic behavior of $Q(v)$. We omit the proof of the following proposition.

\begin{prop} \label{prop:6:asym-Qv}
The dominant singularity of $Q(v)$ is $\rho_Q = \frac{54}{79^{3/2}}$. Furthermore, $Q(v)$ is $\Delta$-analytic, and near $\rho_Q$ we have
\[
Q(v) = q_0 - q_1(1-\rho_Q^{-1}v) + q_2(1-\rho_Q^{-1}v)^{3/2} + O\left( (1-\rho_Q^{-1}v)^2 \right),
\]
where
\[
q_0 = 1-\frac{17}{2\sqrt{79}}, q_1 = \frac{189}{298\sqrt{79}}, q_2 = \frac{79 \cdot 2^{3/2} \cdot 3^{5/2}}{199^{5/2}}.
\]
\end{prop}

Knowing the dominant singularities and asymptotic behaviors of $Q(v)$ and $N^\circ(v)$, we are now ready to analyze the functions $B_g(v)$ and $C_g(v)$.

\begin{prop} \label{prop:6:asym-Bv}
For all $g \geq 0$, the dominant singularity of the power series $B_g^{\fw\geq2}(v)$ is the same $\rho_N = \frac{2^4 3^3}{17^3}$ as that of $N^\circ(v)$. The power series $B_0^{\fw\geq2}(v)$ is $\Delta$-analytic and satisfies
\[ 
B_0^{\fw\geq2}(v) = a_0 + a_1(1-\rho_N^{-1}v) + a_2(1-\rho_N^{-1}v)^2 + d_0(1-\rho_N^{-1}v)^{5/2} + O\left( (1-\rho_N^{-1}v)^3 \right), 
\]
where $a_0, a_1, a_2,b_0$ are constants. Furthermore, we have
\begin{align*}
B_1^{\fw\geq2}(v) &\cong d_1 \log(1-\rho_N^{-1}v) + O\left( (1-\rho_N^{-1}v)^{1/4} \right), \\
B_g^{\fw\geq2}(v) &\cong d_g (1-\rho_N^{-1}v)^{-5(g-1)/2} + O\left( (1-\rho_N^{-1}v)^{-5(g-1)/2+1/4}\right), \quad \forall g \geq 2.
\end{align*}
Here, all $d_g$'s are constants.
\end{prop}
\begin{proof}
For the planar case $g=0$, since planar graphs have infinite facewidth, we have $B_0^{\fw\geq2}(v)=B_0(v)$. We observe that networks are just edge-rooted 2-connected cubic graphs. Therefore, we can directly relate $B_0(x,y,z)$ to $N^\circ(x,y,z)$ by
\[
N^\circ(x,y,z) = 2\frac{y\partial}{\partial y} B_0(x,y,z) + 2\frac{y^2\partial}{\partial z} B_0(x,y,z).
\] 
With the substitution $z=y^2/2$, by a simple computation, we have
\[
N^\circ\left(x,y,\frac{y^2}{2}\right) = 2\frac{yd}{dy} \left(B_0\left(x,y,\frac{y^2}{2}\right)\right).
\]
By substituting $x=x(v)=v^{1/4}$ and $y=y(v)=v^{1/6}$, we have
\[ B_0(v') = \int_0^{v'} \frac{N^\circ(v)}{2y(v)} dy(v) = \int_0^{v'} \frac{N^\circ(v)}{16v} dv.\]
We thus conclude by an integration from Proposition~\ref{prop:6:asym-Nv}, whose validity is guaranteed by the fact that $N^\circ(v)$ is $\Delta$-analytic.

For the non-planar case $g \geq 1$, we first look at $D_g^{\fw\geq3}(x,\bar y) = D_g^{\fw\geq3}(x,y(1+N^\circ(x,y,z)))$ in Proposition~\ref{prop:6:D-to-B}. Since $D_g^{\fw\geq3}(x,y)$ is the EGF of a class of cubic graphs with only simple edges, it is formed by monomials of the form $x^{2n}y^{3n}$, which become $v^n$ under the change of variable in \eqref{eq:6:v-change}. Therefore, under the same change of variable, $D_g^{\fw\geq3}(x,\bar y)$ becomes $D_g^{\fw\geq3}(v(1+N^\circ(v))^3)$, and \eqref{eq:6:D-to-B} becomes
\begin{equation} \label{eq:6:D-to-B-v}
D_g^{\fw\geq3}\left(v(1+N^\circ(v))^3\right) - D_0\left(v(1+N^\circ(v))^3\right) \preceq B_g^{\fw\geq3}(v) \preceq D_g^{\fw\geq3}\left(v(1+N^\circ(v))^3\right).
\end{equation}
The dominant singularity of $D_g^{\fw\geq3}(v(1+N^\circ(v))^3)$ either comes from the dominant singularity $\rho_N$ of $N^\circ(v)$ or from a solution of $v(1+N(v))^3 = \rho_D$ not larger than $\rho_N$, where $\rho_D = \rho_S^3$ is the dominant singularity of $D_g^{\fw\geq3}(v)$. By verifying their values in Proposition~\ref{prop:6:asympt-S} and Proposition~\ref{prop:6:asym-Nv}, we have $\rho_N(1+\rho_N)^3 = \rho_D$. Therefore, $v=\rho_N$ is also a solution of $v(1+N(v))^3 = \rho_D$. Furthermore, it is also the unique solution within $[0,\rho_N]$, since $v(1+N(v))$ is strictly increasing on $[0,\rho_N]$. We thus conclude that $\rho_N$ is the dominant singularity of $D_g^{\fw\geq3}(v(1+N^\circ(v))^3)$, and the composition is critical (\textit{cf.} \cite[Section~VI.9]{flajolet}). By Proposition~\ref{prop:6:D-asymptotic}, Proposition~\ref{prop:6:asym-Nv} and the fact that $v(1+N^\circ(v))^3$ has only non-negative coefficients, we have
\begin{align*}
D_0^{\fw\geq3}\left(v(1+N^\circ(v))^3\right) &\cong d_0 (1-\rho_N^{-1}v)^{5/2} + O\left( (1-\rho_N^{-1}v)^3 \right), \\
D_1^{\fw\geq3}\left(v(1+N^\circ(v))^3\right) &\cong d_1 \log(1-\rho_N^{-1}v) + O\left((1-\rho_N^{-1})^{1/4}\right)
\end{align*}
for $g=0$ and $g=1$. Then for all $g\geq2$, we have
\[
D_g^{\fw\geq3}\left(v(1+N^\circ(v))^3\right) \cong d_g(1-\rho_N^{-1}v)^{-5(g-1)/2} + O\left( (1-\rho_N^{-1}v)^{-5(g-1)/2+1/4}\right).
\]
Combining with \eqref{eq:6:D-to-B-v}, we have
\begin{align}
\begin{split} \label{eq:6:B-fw-3}
B_1^{\fw\geq3}(v) &\cong d_1 \log(1-\rho_N^{-1}v) + O\left( (1-\rho_N^{-1}v)^{1/4} \right), \\
B_g^{\fw\geq3}(v) &\cong d_g (1-\rho_N^{-1}v)^{-5(g-1)/2} + O\left( (1-\rho_N^{-1}v)^{-5(g-1)/2+1/4} \right), \forall g \geq 2.
\end{split}
\end{align}

To conclude the proof, it suffices to show that $B_g^{\fw=2}(v) = B_g^{\fw\geq2}(v) - B_g^{\fw\geq3}(v)$ can be absorbed in the error term, namely
\begin{equation} \label{eq:6:B-fw-2-err}
B_g^{\fw=2}(v) \cong O\left( (1-\rho_N^{-1}v)^{-5(g-1)/2+1/4} \right).
\end{equation}
We proceed by induction on $g$. The base case $g=0$ is trivial, since $B_g^{\fw=2}(v)=0$ in this case. We now suppose that \eqref{eq:6:B-fw-2-err} holds for all genera $g'<g$, and we will prove that it also holds for genus $g$.

We first perform the substitution $z=y^2/2$ on (\ref{eq:6:2-conn-fw-2}) in Proposition~\ref{prop:6:fw-gap-2-conn}. By replacing the differential operator $z\partial/\partial z$ by $y\partial/\partial y$ using \eqref{eq:6:dz-to-dy} in Lemma~\ref{lem:6:reduce-to-v} to loosen the upper bound, we have
\begin{align*}
&B_g^{\fw=2}\left(x,y,\frac{y^2}{2}\right) \preceq 2\left(y+\frac{y}{2}\right)^2\left(\frac1{y}+\frac2{y}\right)^2
\Bigg(
4\left(\frac{y\partial}{\partial y}\right)^2B_{g-1}^{\fw\geq2}\left(x,y,\frac{y^2}{2}\right) \\
&\quad \quad \quad \quad \quad \quad + \sum_{\substack{g_1+g_2=g\\g_1,g_2>0}} \left( 2\frac{y\partial}{\partial y}B_{g_1}^{\fw\geq2}\left(x,y,\frac{y^2}{2}\right) \right) \left( 2\frac{y\partial}{\partial y}B_{g_2}^{\fw\geq2}\left(x,y,\frac{y^2}{2}\right) \right)
\Bigg).
\end{align*}
We now perform the change of variable $x=v^{1/4}, y=v^{1/6}$ to obtain a bound on $B_g^{\fw=2}(v)$, using \eqref{eq:6:dy-to-dv} in Lemma~\ref{lem:6:reduce-to-v} to replace $y\partial/\partial y$ by $vd/dv$ after the substitution. By the definition of $B_g^{\fw=2}(x,y)$, its coefficients are clearly non-negative, therefore \eqref{eq:6:dy-to-dv} applies.  We thus have
\begin{equation} \label{eq:6:B-fw-2-v}
B_g^{\fw=2}(v) \preceq 2^2 \cdot 3^4 \cdot \left( \frac{vd}{dv}B_{g-1}^{\fw\geq2}(v) + \sum_{\substack{g_1+g_2=g \\ g_1, g_2 >0}} \left( \frac{vd}{dv}B_{g_1}^{\fw\geq2}(v)\right) \left( \frac{vd}{dv}B_{g_2}^{\fw\geq2}(v)\right) \right).
\end{equation}
By induction hypothesis and \eqref{eq:6:B-fw-3}, for all $g'<g$, we have
\[
\frac{vd}{dv}B_{g'}^{\fw\geq2}(v) \cong O\left( (1-\rho_N^{-1})^{-5(g'-1)/2-1} \right).
\]
We also have
\[
\left(\frac{vd}{dv}\right)^2 B_{g-1}^{\fw\geq2}(v) \cong O\left( (1-\rho_N^{-1})^{-5(g-2)/2-2} \right).
\]
Substituting the two congruence relations into \eqref{eq:6:B-fw-2-v}, we have
\[
B_g^{\fw=2}(v) \preceq O\left( (1-\rho_N^{-1}v)^{-5(g-1)/2+1/2} \right),
\]
which implies \eqref{eq:6:B-fw-2-err}. By induction, the relation \eqref{eq:6:B-fw-2-err} holds for all $g\geq0$, which concludes the proof.
\end{proof}

As a corollary, we can obtain the asymptotic enumeration of weighted 2-connected cubic graphs with facewidth at least 2, using the transfer theorem (Theorem~\ref{thm:2:transfer}) on $B_g^{\fw\geq2}$.

\begin{coro} \label{coro:6:asym-Bv}
The asymptotic number of 2-connected vertex-labeled weighted cubic graphs strongly embeddable into \Sg{} with facewidth at least 2 and $2n$ vertices is given by
\[ (2n)![v^n]B_g^{\fw\geq2}(v) = (2n)! \left( 1+O\left(n^{-1/4}\right)\right)c_g n^{5(g-1)/2-1}\rho_N^{-n}. \]
Here, the constant $c_g$ depends only on the genus.
\end{coro}

We now proceed to the asymptotic analysis of the series $C_g(v)$ for connected cubic graphs, using the result on the 2-connected series $B_g(v)$ we just obtained. This asymptotic analysis leads to the following result whose proof has a structure very similar to that of Proposition~\ref{prop:6:asym-Bv}.

\begin{thm} \label{thm:6:asym-Cv}
For all $g\geq0$, the dominant singularity of the power series $C_g(v)$ of connected vertex-labeled weighted cubic graphs strongly embeddable into \Sg{} is the same $\rho_Q=\frac{54}{79^{3/2}}$ as that of $Q(v)$. The power series $C_0(v)$ is $\Delta$-analytic, and satisfies
\[
C_0(v) = a_0' + a_1'(1-\rho_Q^{-1}v) + a_2'(1-\rho_Q^{-1}v)^2 + d_0'(1-\rho_Q^{-1}v)^{5/2} + O\left( (1-\rho_Q^{-1}v)^3 \right),
\]
where $a_0', a_1', a_2', d_0'$ are constants. Furthermore, we have
\begin{align*}
C_1(v) &\cong c_1 \log(1-\rho_Q^{-1}v) + O\left( (1-\rho_Q^{-1}v)^{1/4} \right), \\
C_g(v) &\cong c_g (1-\rho_Q^{-1}v)^{-5(g-1)/2} + O\left( (1-\rho_Q^{-1}v)^{-5(g-1)/2+1/4} \right), \forall g \geq 2.
\end{align*}
Here, $d_g'$ is a constant that only depends on $g$.
\end{thm}
\begin{proof}
The planar case $g=0$ can be obtained by unrooting the class of edge-rooted connected planar cubic graphs, which can be expressed as a sum of the classes in \eqref{eq:6:Q-sys}. The proof is omitted here, and readers are referred to \cite{random-cubic-planar-graphs} or \cite{planar-phase} for more details.

For the non-planar case $g\geq1$, we first look at 
\[
B_g^{\fw\geq2}(x,\bar y,\bar z) = B_g^{\fw\geq2}\left(x,\frac{y}{1-Q(x,y,z,w)},\frac{y^2}{2(1-Q(x,y,z,w))} + z - \frac{y^2}{z}\right)
\]
as in Proposition~\ref{prop:6:B-to-C}. Since $B_g^{\fw\geq2}(x,y,z)$ is the EGF of a class of cubic graphs with simple and double edges, it is formed by monomials of the form $x^{2n} y^{3n-2b} z^b$, which become $2^{-b} v^n$ under the change of variable in \eqref{eq:6:v-change}. Therefore, under the same change of variable, $B_g^{\fw\geq2}(x,\bar y,\bar z)$ becomes $B_g^{\fw\geq2}(v(1-Q(v))^{-3})$, and \eqref{eq:6:B-to-C} becomes
\begin{equation} \label{eq:6:B-to-C-v}
B_g^{\fw\geq2}\left(v(1-Q(v))^{-3}\right) - B_0\left(v(1-Q(v)\right)^{-3} \preceq C_g^{\fw\geq2}(v) \preceq B_g^{\fw\geq2}\left(v(1-Q(v)\right)^{-3}).
\end{equation}
The dominant singularity of $B_g^{\fw\geq2}(v(1-Q(v))^{-3})$ either comes from the dominant singularity $\rho_Q$ of $Q(v)$ as in Proposition~\ref{prop:6:asym-Qv}, or from a positive solution not larger than $\rho_Q$ of $v(1-Q(v))^{-3}=\rho_N$ or $Q(v)=1$, where $\rho_N$ is the dominant singularity of $B_g^{\fw\geq2}(v)$. By Proposition~\ref{prop:6:asym-Qv} and Proposition~\ref{prop:6:asym-Bv}, we verify that $\rho_Q(1-Q(\rho_Q))^{-3} = \rho_N$. Since $Q(v)$ is a power series without constant term but with positive coefficients, $v(1-Q(v))^{-3}$ is thus also a power series with positive coefficients, which implies that $v(1-Q(v))^{-3}$ is strictly increasing in $[0,\rho_Q]$. Therefore, $v=\rho_Q$ is the only solution of $v(1-Q(v))^{-3}=\rho_N$ in $[0,\rho_Q]$. We also know that $Q(v)=1$ has no solution in $[0,\rho_Q]$, since $Q(v)$ is strictly increasing in this interval, and $Q(\rho_Q)=q_0<1$ according to Proposition~\ref{prop:6:asym-Qv}. Therefore, $v=\rho_Q$ is the dominant singularity of $B_g^{\fw\geq2}(v(1-Q(v))^{-3})$, and the composition is critical (\textit{cf.} \cite[Section~VI.9]{flajolet}). By Proposition~\ref{prop:6:asym-Bv} and Proposition~\ref{prop:6:asym-Qv}, we have
\begin{align*}
B_0\left(v(1-Q(v))^{-3}\right) &\cong d_0'(1-\rho_Q^{-1}v)^{5/2} + O\left( (1-\rho_Q^{-1}v)^3 \right), \\
B_1^{\fw\geq2}\left(v(1-Q(v))^{-3}\right) &\cong d_1' \log(1-\rho_Q^{-1}v) + O\left( (1-\rho_Q^{-1}v)^{1/4} \right)
\end{align*}
for $g=0$ and $g=1$. Then for all $g\geq2$, we have
\[
B_g^{\fw\geq2}\left(v(1-Q(v))^{-3}\right) \cong d_g' (1-\rho_Q^{-1}v)^{-5(g-1)/2} +O\left( (1-\rho_Q^{-1}v)^{-5(g-1)/2+1/4} \right).
\]
Combining with \eqref{eq:6:B-to-C-v}, we have
\begin{align}
  \begin{split} \label{eq:6:C-fw-2}
    C_1^{\fw\geq2}(v) &\cong d_1' \log(1-\rho_Q^{-1}v) + O\left((1-\rho_Q^{-1}v)^{1/4}\right), \\
    C_g^{\fw\geq2}(v) &\cong d_g'(1-\rho_Q^{-1}v)^{-5(g-1)/2} + O\left((1-\rho_Q^{-1}v)^{-5(g-1)/2+1/4}\right), \forall g\geq2.
  \end{split}
\end{align}

To conclude the proof, it suffices to show that $C_g^{\fw=1}(v)=C_g(v)-C_g^{\fw\geq2}(v)$ can be absorbed in the error term, namely
\begin{equation} \label{eq:6:C-fw-1-err}
C_g^{\fw=1} \cong O\left( (1-\rho_Q^{-1}v)^{-5(g-1)/2+1/4} \right).
\end{equation}
We proceed by induction on $g$. The base case $g=0$ is trivial, since $C_g^{\fw=1}(v)=0$ in this case. We now suppose that \eqref{eq:6:C-fw-1-err} holds for all genera $g'<g$, and we will prove that it also holds for genus $g$.
 
We first perform the substitutions $z=y^2/2, w=y/2$ on (\ref{eq:6:1-conn-fw-1}) in Proposition~\ref{prop:6:fw-gap-1-conn}. Since $C_g^{\fw=1}(x,y,z,w)$ is the EGF of a class of cubic graphs, we can apply \eqref{eq:6:dw-to-dy} in Lemma~\ref{lem:6:reduce-to-v} to get a relaxed upper bound. We thus have
\begin{align*}
&C_g^{\fw=1}\left(x,y,\frac{y^2}{2},\frac{y}{2}\right) \preceq 4x^{-2}y^{-4} \left( y + \frac{y}{2} \right) \Bigg( 
\left( \frac{y\partial}{\partial y} \right)^2 \left(C_{g-1}\left(x,y,\frac{y^2}{2},\frac{y}{2}\right)\right) \\
&\quad\quad\quad\quad+\sum_{\substack{g_1+g_2=g\\g_1,g_2>0}} \left[ \frac{y\partial}{\partial y} \left(C_{g_1}\left(x,y,\frac{y^2}{2},\frac{y}{2}\right)\right)\right]\left[ \frac{y\partial}{\partial y} \left(C_{g_2}\left(x,y,\frac{y^2}{2},\frac{y}{2}\right)\right)
\right]
\Bigg).
\end{align*}
We now perform the change of variable $x=v^{1/4},y=v^{1/6}$ to obtain a bound for $C_g^{\fw=1}(v)$. Since $C_g^{\fw=1}(x,y,z,w)$ is the EGF of a class of cubic graphs, it has non-negative coefficients, and so does $C_g^{\fw=1}(v)$. We can thus use \eqref{eq:6:dy-to-dv} in Lemma~\ref{lem:6:reduce-to-v} to replace $y\partial/\partial y$ by $vd/dv$ after the change of variable, which gives
\begin{equation} \label{eq:6:C-fw-1-v}
C_g^{\fw=1}(v) \preceq 54v^{-1} \left( \frac{vd}{dv}C_{g-1}(v) + \sum_{\substack{g_1+g_2=g\\g_1,g_2>0}} \left(\frac{vd}{dv}C_{g_1}(v) \right) \left(\frac{vd}{dv}C_{g_2}(v) \right) \right).
\end{equation}
It is clear that the right-hand side is also a formal power series, without singularity at $v=0$. By induction hypothesis and \eqref{eq:6:C-fw-2}, for all $g'<g$, we have
\[
\frac{vd}{dv}C_{g'}(v) =O\left( (1-\rho_Q^{-1}v)^{-5(g'-1)/2-1} \right).
\]
We also have
\[
\left(\frac{vd}{dv}\right)^2 C_{g'}(v) =O\left( (1-\rho_Q^{-1}v)^{-5(g-2)/2-2} \right).
\]
Substituting the two congruence relations into \eqref{eq:6:C-fw-1-v}, we have
\[
C_g^{\fw=1}(v) \preceq O\left( (1-\rho_Q^{-1}v)^{-5(g-1)/2+1/2} \right),
\]
which implies \eqref{eq:6:C-fw-1-err}. By induction, the relation \eqref{eq:6:C-fw-1-err} holds for all $g\geq0$, which concludes the proof.
\end{proof}

We can obtain the asymptotic enumeration result on connected weighted cubic graphs as a corollary in a similar way as in Corollary~\ref{coro:6:asym-Bv}.

\begin{coro}
The asymptotic number of connected vertex-labeled weighted cubic graphs strongly embeddable into \Sg{} with $2n$ vertices is given by
\[
(2n)![v^n]C_g(v) = (2n)!\left(1+O(n^{1/4})\right) c_g n^{5(g-1)2-1} \rho_Q^{-n}.
\]
Here, $c_g$ is a constant depending only on $g$, and $\rho_Q = \frac{54}{79^{3/2}}$.
\end{coro}

We are now finally ready to prove our main result in this chapter, Theorem~\ref{thm:6:main}. We denote by $\mathcal{W}_g$ the class of vertex-labeled cubic graphs strongly embeddable into \Sg{}. We recall that a cubic graph in $\mathcal{W}_g$ is not necessarily connected, and such a cubic graph is said to be strongly embeddable into \Sg{} if the genera of its connected components sum up to $g$. Since a cubic graph always has an even number of vertices, we can define the EGF $W_g(v)$ of $\mathcal{W}_g$ as
\[
W_g(v) = \sum_{G \in \mathcal{W}_g} \frac{v^{\rm \#vertices/2}}{({\rm \#vertices})!} 2^{-{\rm \#double\;edges}-{\rm \#loops}}6^{-{\rm \#triple\;edges}}.
\]
The number $w_g(n)$ of such weighted vertex-labeled cubic graphs with $2n$ vertices is thus simply the coefficient $(2n)![v^n]W_g(v)$. By analyzing the behavior of $W_g(v)$ near its dominant singularity, we can obtain the asymptotic behavior of $w_g(n)$. Since a cubic graph in $\mathcal{W}_g$ can be seen as a set of connected cubic graphs, we can use $C_g(v)$ to express $W_g(v)$, or at least to bound it.

\begin{proof}[Proof of Theorem~\ref{thm:6:main}]
By the definition of $\mathcal{W}_g$, we have
\begin{equation} \label{eq:6:w-decomp}
W_g(v) \preceq \sum_{k\geq1} \sum_{g_1+\cdots+g_k=g} \frac1{k!} \prod_{i=1}^k \left( C_{g_i}(v) + \frac{v}{6} \right).
\end{equation}
The term $v/6$ stands for the special cubic graph formed by a triple edge, which is not presented in any $\mathcal{C}_g$. Here we only have an upper bound instead of an equality, because the same cubic graph can be strongly embeddable into surfaces of several genera, giving rise to multiple presence in different classes of the form $\mathcal{C}_g$. Later, we will establish a lower bound that matches the asymptotics of the upper bound.

For the planar case $g=0$, since there is no over-counting in this case, \eqref{eq:6:w-decomp} becomes $W_0(v) = \exp(C_0(v)+v/6)$, which coincides with Theorem~1 of \cite{planar-phase}. We can thus easily conclude the proof in this case by a substitution.

We now deal with the non-planar case $g\geq1$. We first single out all planar components to obtain
\begin{equation*} 
W_g(v) \preceq \sum_{k=1}^g \sum_{\substack{g_1+\cdots+g_k=g \\ g_i\geq1}} \frac1{k!}\prod_{i=1}^k \left(C_{g_i}(v)+\frac{v}{6}\right)
\sum_{j\geq0} \frac{k!}{(k+j)!}\left(C_0(v)+\frac{v}{6}\right).
\end{equation*}
It is clear that $W_g(v)$ is singular at $v=\rho_Q$, which is the dominant singularity of all $C_g(v)$'s. Since $C_0(\rho_Q)$ is finite, we can see that the term representing planar components can at most contribute a constant factor $c_{\rm planar}$ for any $v$ of modulus at most $\rho_Q$. We thus have
\begin{equation} \label{eq:6:w-decomp-planar}
W_g(v) \preceq c_{\rm planar} \sum_{k=1}^g \sum_{\substack{g_1+\cdots+g_k=g \\ g_i\geq1}} \frac1{k!}\prod_{i=1}^k \left(C_{g_i}(v)+\frac{v}{6}\right).
\end{equation}
Since this is a finite positive sum, the dominant singularity $W_g(v)$ is also at $v=\rho_Q$. For a given sequence $g_1,\ldots,g_k$, we define
\[
A_{(g_1,\ldots,g_k)} \eqdef \frac1{k!} \prod_{i=1}^k \left( C_{g_i}(v) + \frac{v}{6} \right).
\]
We now only need to find the sequences $g_1,\ldots,g_k$ that lead to a predominant $A_{(g_1,\ldots,g_k)}$.

For $g=1$, there is only one possibility $k=1$ and $g_1=1$, and the only term that appears in \eqref{eq:6:w-decomp-planar} is $C_1(v)$, which only appears once. By Proposition~\ref{thm:6:asym-Cv}, we have
\[
W_g(v) \preceq P_1(v) + d_1 \log(1-\rho_Q^{-1}v) + O\left( (1-\rho_Q^{-1}v)^{1/4} \right),
\]
where $P_1(v)$ is a polynomial in $v$, and $d_1$ is a constant.

For $g\geq2$, without loss of generality, we can choose the sequence $g_1, \ldots, g_k$ to be decreasing. Suppose that $g_1 = \cdots = g_\ell = 1$ are the $g_i$'s that are equal to $1$. We then have
\begin{align*}
A_{(g_1,\ldots,g_k)} &\cong \left(1+O\left((1-\rho_Q^{-1}v)^{1/4}\right)\right) \frac1{k!} c_1^\ell \left( \log(1-\rho_Q^{-1}v) \right)^\ell \prod_{i=\ell+1}^k c_{g_i} (1-\rho_Q^{-1}v)^{-5(g_i-1)/2} \\
&= d_{(g_1,\ldots,g_k)} \left(1+O\left((1-\rho_Q^{-1}v)^{1/4}\right)\right) \left( \log(1-\rho_Q^{-1}v) \right)^\ell (1-\rho_Q^{-1}v)^{-5(g-k)/2}.
\end{align*}
Here $d_{(g_1,\ldots,g_k)}$ is a constant. 

For the case $k=1$ and $g_1=g$ (thus $\ell=0$), we have 
\[
A_{(g)} = C_g(v) + \frac{v}{6} \cong c_g (1-\rho_Q^{-1}v)^{-5(g-1)/2} + O\left( (1-\rho_Q^{-1}v)^{-5(g-1)/2+1/4}\right).
\]
For other sequences with $k\geq2$, we have
\[
A_{(g_1,\ldots,g_k)} \cong O\left( (1-\rho_Q^{-1}v)^{-5(g-1)/2+2} \right).
\]
Therefore, the term $A_{(g)}$ is indeed dominant, and we have
\[
W_g(v) \preceq e_g (1-\rho_Q^{-1}v)^{-5(g-1)/2} + O\left( (1-\rho_Q^{-1}v)^{-5(g-1)/2+1/4}\right),
\]
where $e_g = c_g \sum_{j\geq0} \frac1{(j+1)!} \left(C_0(v)+v/6\right).$

We now construct a lower bound for $W_g(v)$ for all $g\geq1$. Let $\tilde{\mathcal{W}}_g$ be the subclass of $\mathcal{W}_g$ consisting of cubic graphs with one component of genus $g$ and all others planar. We denote by $\tilde{W}_g(v)$ its EGF. The choice of $\tilde{\mathcal{W}}_g$ comes from the intuition that, in the sparse regime of random graphs, if there is a giant component, all other connected components ought to be mostly trees and unicycles, which are all planar. We thus have
\begin{align*}
\sum_{j\geq0} \frac1{(j+1)!} \left(C_0(v)+\frac{v}{6}\right) &\succeq \tilde{W}_g(v) \\
&\succeq C_g(v) \sum_{j\geq0} \frac1{(j+1)!} \left(C_0(v)+\frac{v}{6}\right) - \sum_{j\geq0} \frac{(j+1)}{(j+1)!} \left(C_0(v)+\frac{v}{6}\right).
\end{align*}
Indeed, the only possibility that we count an element of $\tilde{\mathcal{W}}_g$ multiple times in the upper bound of $\tilde{W}_g$ is when its component of genus $g$ is also planar. In this case, there will be at most $j+1$ choices of the component of genus $g$, if there are $j+1$ components. By subtracting the generating function for this case, we obtain the lower bound. We thus have
\[
W_g(v) \succeq \tilde{W}_g(v) \succeq C_g(v) \sum_{j\geq0} \frac1{(j+1)!} \left(C_0(v)+\frac{v}{6}\right) - \sum_{j\geq0} \frac{(j+1)}{(j+1)!} \left(C_0(v)+\frac{v}{6}\right).
\]
We can see that the term we subtract is finite at the dominant singularity $\rho_Q$. Furthermore, the other term is predominant in the upper bound. Therefore, the two bounds match, and with a certain constant $e_g$ we have
\[
W_g(v) \cong e_g (1-\rho_Q^{-1}v)^{-5(g-1)/2} + O\left( (1-\rho_Q^{-1}v)^{-5(g-1)/2+1/4}\right).
\]
We conclude by applying the transfer theorem (Theorem~\ref{thm:2:transfer}) to extract the asymptotic behavior of the coefficients $w_g(n)=[v^n]W_g(v)$ of $W_g(v)$.
\end{proof}

\horizrule

As we have mentioned above, with minor modifications, our method can also apply to the enumeration of cubic graphs without weight and simple cubic graphs. We have also obtained asymptotic enumeration results in these cases. 

In the perspective of random graphs, our main theorem (Theorem~\ref{thm:6:main}) is a first step towards the study of random graphs embeddable into \Sg{} with a given $g$. As we mentioned at the end of Section~\ref{sec:6:intro}, our main motivation is to study random graphs embeddable into \Sg{} in the sparse regime $\mu<1$. Even with constraints on embeddability, the kernel of a random graph in this regime does not deviate too much from being cubic. We can thus obtain the generating function of these graphs by appropriate variable substitutions to the generating function of cubic graphs. We can thus study the structure of random graphs embeddable into \Sg{} with a fixed genus $g$. We are starting to look into random graphs in this direction, and we expect to find a second phase transition as the one found in \cite{planar-phase} for the planar case.

\chapter{Towards the uncharted}

\begin{flushright}
\textit{Wir müssen wissen -- wir werden wissen!}

(We must know -- we will know!)

--- David Hilbert
\end{flushright}

\medskip

In previous chapters, we have seen some enumeration results on maps and related combinatorial objects. These results are obtained using methods that vary from the algebraic ones such as characters and functional equations, to the bijective ones such as bijections and surgeries, and to the analytic ones such as asymptotic analysis. I hope that I have well conveyed the richness of combinatorial maps in their relations to other fields of mathematics, in the variety of methods, and in its applications to other enumerative problems. Nevertheless, what we have seen is just the tip of an iceberg. There are still a handful of prominent aspects about maps that were not treated in my thesis.

Such important but untreated topics include random maps and statistic physic models (for instance the Potts model and fully-packed loops) on maps. These topics are more in the realm of probability than combinatorics, though they are also closely related to the asymptotic enumeration of maps, and use some of the tools we have talked about in the previous chapters. It is thus interesting to explore these related subjects to see how to transfer ideas and methods between different domains related to map enumeration. 

Another topic that was not treated is the enumeration of \emph{non-orientable maps}, \textit{i.e.} graph embeddings on non-orientable surfaces. Non-orientable maps are more difficult to capture than orientable maps due to a lack of a common orientation. Some of the methods that we have previously seen still apply, for example, it is still possible to write and solve Tutte equations for non-orientable maps with only minor tweaking (\textit{cf.} \cite{BC0, Gao, Gao1993-pattern}), and some bijections for maps on orientable surfaces can be naturally extended to non-orientable cases (\textit{cf.} \cite{chapuy-dolega}). However, the representation-theoretic approach to the enumeration of non-orientable maps becomes much more difficult, due to a radical change in their rotation systems. These rotation systems are now defined in a more complicated algebraic structure called the \emph{double coset algebra} of the symmetric group, due to a lack of global orientation in the map (\textit{cf.} \cite{non-orient-map}). Non-orientable maps are thus related to deeper topics in the study of symmetric functions, such as zonal polynomials and Jack polynomials (\textit{cf.} \cite{non-orient-map, jack-b-conj, lacroix, jack-poly}). 

Even for topics treated in the previous chapters, there are still many open problems to be explored. We now reiterate several of these problems, alongside with some new problems that may be farther from our reach.

\paragraph{Generalized quadrangulation relation (Chapter~3)} ~\\
Given the simplicity of the relation in Corollary~\ref{coro:3:general-counting-relation-in-numbers}, is it possible to give a bijective proof? We should however keep in mind that the original quadrangulation relation, even in its simplest form, has resisted all attempts for a bijective proof in full generality. Therefore, it seems to be more difficult to find a bijective proof for our generalization. But in our generalization, there are some non-negative coefficients $c^{(m)}_{k_1,\ldots,k_{m-1}}$, which may have some combinatorial meaning that can be exploited in the search of a bijection.

\paragraph{Enumeration of constellations (Chapter~4)} ~\\
In the resolution of the planar case, we used the differential-catalytic method developed in \cite{BMCPR2013representation} for the enumeration of intervals in the $m$-Tamari lattice. Why is this method also effective on the Tutte equation that we have written? The functional equations for the generating functions of planar constellations and $m$-Tamari intervals have a vague resemblance, for example they both contain a flavor of divided difference operator, and they have operators iterated an arbitrary number of times. Are there other interesting combinatorial classes whose generating function satisfies a similar functional equation which allows the same method to be applied? Which kind of functional equations can we solve using the differential-catalytic method? Since the functional equations of both $m$-constellations and $m$-Tamari intervals seem to be closely related, can they be described in a unified framework?

\paragraph{Bijection between intervals in generalized Tamari lattice and non-separable planar maps (Chapter~5)} ~\\
We have discussed how the map duality on non-separable planar maps is transferred to an involution in the set of intervals in generalized Tamari lattice. Are there other correspondences for other symmetries and statistics? For example, changing the orientation of the root edge in a non-separable map gives an involution on the map side. What is its counterpart on the Tamari side under our bijection, and does it have interesting properties? These questions can be asked for many other symmetries and statistics, most of which may not be of great value, but there may still be interesting correspondences. 

If we zoom out from objects (maps and intervals) to relations between objects, we can see that it is possible to define relations on the class of non-separable planar maps by lifting existing relations on the class of intervals in generalized Tamari lattices. For instance, we can define a partial order on intervals in generalized Tamari lattices by interval inclusion, which can be then lifted to a partial order on non-separable planar maps. Is this partial order interesting? A lot of other structures such as Hopf algebras can be built upon the Tamari lattice and its generalizations (\textit{cf.} \cite{loday1998hopf, m-dyck-alg}), are there any lifting of these structures to the maps side that would be interesting?


\bigskip

Besides the relatively specific questions that are directly related to the work in the previous chapters, we can also ask ``big questions'' in the field of map enumeration. The ``big question'' that fascinates me the most is the following:

\medskip

\noindent \textbf{Question}: Can we understand, in a purely combinatorial way, results and methods that come from deep algebraic structures, such as characters of the symmetric group and the KP hierarchy? Conversely, how can we translate map bijections into the algebraic language of factorizations in the symmetric group $S_n$, and what do they imply in the algebraic study of $S_n$?

\medskip

In my humble opinion, this question is intriguing and important, yet barely explored. It is actually about the interaction between the two descriptions of maps: the topological description of maps as graph embeddings and the algebraic description of maps as transitive factorizations in the symmetric group. A better understanding of this interaction will lead to a transfer of techniques between bijective and algebraic study of maps, in the form of new types of bijections or non-trivial theorems about irreducible characters of $S_n$ and symmetric functions.

Of course, the question above is somehow vague, but we can point out concrete instances that are interesting. For example, as stated above in the open problems related to results in Chapter~3, is there a bijective proof of the quadrangulation relation and our generalization to constellations and hypermaps? Since our proof of the generalized quadrangulation relation essentially relies on a character factorization theorem of Littlewood, if an independent bijective proof exists, it must somehow ``act out'' bijectively the character factorization. Such a bijection will lead to a better understanding of the role of characters factorization in map enumeration. Another example is the simple recurrences on triangulations \cite{goulden2008kp} and quadrangulations \cite{kp-quad} obtained via differential equations in the KP hierarchy. Here, a bijective proof would lead to more combinatorial insight of these differential equations, and eventually to new recurrences of the same type, alongside with new bijections. For instances in the other direction, we can analyze well-known bijections in map enumeration under the framework of rotation systems to see if they lead to interesting transformations of permutation factorizations in $S_n$.

Readers familiar with map enumeration may comment that these concrete instances are all rather difficult to solve, which undermines the practicality of pursuing the research direction indicated by our big question. Indeed, as we have mentioned, the original quadrangulation relation was published more than 20 years ago, and a bijective explanation is known only in special cases. The recurrence on triangulations, being much more recent, shares the same situation. But there are also successes in this direction. The Harer-Zagier \emph{formula} and the Harer-Zagier \emph{recurrence} for unicellular maps were first obtained in \cite{HZ} using matrix integral techniques, which are based on the rotation system representation of maps. At that time, they were as mysterious as the quadrangulation relation. But later, combinatorial interpretations and bijective proofs of both relations were discovered \cite{lass, goulden-nica, unicellular-tree}, and many interesting bijections were invented during this exploration of combinatorial interpretations of these algebraic relations (\textit{cf.} \cite{trisection,OB}). The Harer-Zagier formula and recurrence thus set a precedence of success in the pursuit of clarification of the interaction between the two descriptions of maps. If we can solve a lot of these concrete instances of the big question, we may be able to puzzle up a more complete view of how the embedding description and the rotation system description of maps interact. For now, it is still a largely uncharted territory to be explored, in the amazingly rich domain of map enumeration.

\bibliographystyle{plain}
\bibliography{thesis}

\end{document}